\documentclass[11pt]{amsart}
\usepackage{graphics}
\usepackage{cite}
\usepackage{amsmath}
\usepackage{amsfonts}
\usepackage{amssymb}
\usepackage{float}
\usepackage{epic}
\usepackage{eepic}
\usepackage{color}
\usepackage{ae}
\usepackage{mathrsfs,array,graphicx}
\usepackage[all,ps]{xy}
\usepackage{epstopdf}
\usepackage{a4wide}
\usepackage{enumerate}
\usepackage{hyperref} \hypersetup{colorlinks,allcolors=[rgb]{0,0,0}}
\usepackage{arydshln}
\usepackage[a4paper,margin=1in]{geometry}
\usepackage{stmaryrd}
\usepackage{enumitem}
\usepackage{mathtools}

\theoremstyle{theorem}
\newtheorem{thmx}{Theorem}

\newtheorem{theorem}{Theorem}
\newtheorem{proposition}[theorem]{Proposition}

\newtheorem{conjecture}{Conjecture}
\newtheorem{hypothesis}[theorem]{Hypothesis}
\newtheorem{corollary}[theorem]{Corollary}
\newtheorem{lemma}[theorem]{Lemma}
\theoremstyle{definition}
\newtheorem{definition}[theorem]{Definition}
\theoremstyle{remark}
\newtheorem{remark}[theorem]{Remark}

\newtheorem{example}[theorem]{\textbf{Example}}

\numberwithin{equation}{section}
\numberwithin{theorem}{section}

\def\hat{\widehat}
\let\li\overline

\newcommand{\tu}[1]{\textup{#1}}

\newcommand{\Ad}{\tu{Ad}}
\newcommand{\Aut}{\tu{Aut}}
\newcommand{\A}{\mathbb A}
\newcommand{\Bbad}{B_{\textup{bad}}}
\newcommand{\Bgood}{B_{\textup{good}}}
\newcommand{\BGSO}{\tu{B}_{\tu{GSO}}}
\newcommand{\BGSpin}{\tu{B}_{\GSpin}}

\newcommand{\BSO}{\tu{B}_{\SO}}
\newcommand{\C}{\mathbb C}
\newcommand{\cpt}{\tu{cmpt}}
\newcommand{\End}{\tu{End}}
\newcommand{\Frob}{\tu{Frob}}
\newcommand{\GL}{{\tu{GL}}}
\newcommand{\GO}{\tu{GO}}
\newcommand{\GPin}{\tu{GPin}}
\newcommand{\GSO}{{\tu{GSO}}}
\newcommand{\GSPin}{\tu{GSpin}}
\newcommand{\GSpin}{\tu{GSpin}}
\newcommand{\GSp}{\tu{GSp}}
\newcommand{\Gal}{{\tu{Gal}}}

\newcommand{\Gm}{\mathbb{G}_{\tu{m}}}
\newcommand{\G}{\mathbb G}
\newcommand{\HT}{\tu{HT}}
\newcommand{\Hodge}{\tu{Hodge}}
\newcommand{\Hom}{\tu{Hom}}
\newcommand{\Id}{\tu{id}}
\newcommand{\Ind}{\tu{Ind}}
\newcommand{\Int}{{\tu{Int}}}
\newcommand{\Ker}{\tu{ker}}
\newcommand{\LG}{{}^L G}

\newcommand{\Lie}{\tu{Lie}\,}
\newcommand{\OO}{\tu{O}}
\newcommand{\Out}{\tu{Out}}
\newcommand{\PGL}{\tu{PGL}}
\newcommand{\PO}{\tu{PO}}
\newcommand{\PSO}{\tu{PSO}}
\newcommand{\Q}{\mathbb Q}
\newcommand{\Res}{\tu{Res}}
\newcommand{\R}{\mathbb R}
\newcommand{\SEbad}{S^E_{\tu{bad}}}
\newcommand{\SFbad}{S^F_{\tu{bad}}}
\newcommand{\SL}{\tu{SL}}
\newcommand{\SO}{{\tu{SO}}}
\newcommand{\SU}{{\tu{SU}}}
\newcommand{\Sbad}{S_{\tu{bad}}}
\newcommand{\Sh}{\tu{Sh}}
\newcommand{\Spin}{\tu{Spin}}
\newcommand{\Sp}{\tu{Sp}}
\newcommand{\St}{\tu{St}}
\newcommand{\Sym}{{\tu{Sym}}}
\newcommand{\TGSO}{\uT_{\GSO}}
\newcommand{\TGSpin}{\uT_{\GSpin}}

\newcommand{\TGspin}{\uT_{\GSpin}}

\newcommand{\TSO}{\uT_{\SO}}
\newcommand{\TSpin}{\uT_{\Spin}}

\newcommand{\Tr}{\tu{Tr}\,}
\newcommand{\WGSO}{\tu{W}_{\GSO}}
\newcommand{\ZGSpin}{Z(\GSpin_{2n})}
\newcommand{\Zhat}{\widehat{\mathbb{Z}}}
\newcommand{\Zspin}{Z(\Spin_{2n})}
\newcommand{\Z}{\mathbb Z}
\newcommand{\ad}{\tu{ad}}
\newcommand{\bad}{\tu{bad}}
\newcommand{\cE}{\mathcal E}
\newcommand{\cH}{\mathcal H}
\newcommand{\cL}{\mathcal L}
\newcommand{\cN}{\mathcal N}
\newcommand{\cO}{\mathcal O}
\newcommand{\cS}{\mathcal S}
\newcommand{\cV}{\mathcal V}
\newcommand{\centSO}{(\tu{cent$^\circ$})}
\newcommand{\cmpt}{\tu{cpt}} 
\newcommand{\coh}{\tu{shim}}  
\newcommand{\cusp}{\tu{cusp}}
\newcommand{\der}{\tu{der}}
\newcommand{\diag}{{\tu{diag}}}

\newcommand{\eps}{\varepsilon}
\newcommand{\ep}{\tu{ep}}

\newcommand{\fke}{\mathfrak e}
\newcommand{\fkp}{\mathfrak{p}}
\newcommand{\grootvierkant}[4]{{\lhk\begin{matrix} #1 & #2 \cr #3 & #4 \end{matrix}\rhk}}
\newcommand{\hra}{\hookrightarrow}
\newcommand{\iS}{\mathfrak S}
\newcommand{\ig}{\mathfrak g}

\newcommand{\inv}{^{-1}}
\newcommand{\isomto}{\overset \sim \to}
\newcommand{\isom}{\stackrel{\sim}{\to}}

\newcommand{\istd}{i_{\std}}

\newcommand{\lbr}{\left\lbrace}

\newcommand{\lhk}{\left(}
\newcommand{\lql}{\overline{\Q}_\ell}
\newcommand{\lqp}{\overline{\Q}_p}
\newcommand{\ol}{\overline}
\newcommand{\one}{\textbf{\tu{1}}}
\newcommand{\osim}{\overset {\circ} \sim}
\newcommand{\pp}{{\mathfrak p}}
\newcommand{\pri}{\tu{pri}}
\newcommand{\pr}{\tu{pr}}
\newcommand{\p}{\mathfrak p}
\newcommand{\ql}{\Q_\ell}
\newcommand{\qp}{{\mathbb{Q}_p}}
\newcommand{\qq}{{\mathfrak q}}
\newcommand{\qst}{{\qq_\tu{St}}}
\newcommand{\quand}{\quad \tu{and} \quad}
\newcommand{\ra}{\rightarrow}
\newcommand{\rbr}{\right\rbrace}
\newcommand{\reg}{\tu{reg}}
\newcommand{\rhk}{\right)}

\newcommand{\simil}{\tu{sim}}
\newcommand{\spin}{\tu{spin}} 
\newcommand{\ssimple}{ {\tu{ss}} }

\newcommand{\std}{\tu{std}}

\newcommand{\surjects}{\twoheadrightarrow}

\newcommand{\tra}{{}^\tt}
\newcommand{\tspin}{\widetilde{\spin}}
\newcommand{\uH}{\tu{H}}
\newcommand{\uM}{\tu{M}}
\newcommand{\uO}{\tu{O}}
\newcommand{\uT}{\tu{T}}
\newcommand{\uc}{\tu{c}}
\newcommand{\vierkant}[4]{{\lhk \begin{smallmatrix} #1 & #2 \cr #3 & #4 \end{smallmatrix} \rhk } }
\newcommand{\vol}{\tu{vol}}
\newcommand{\wh}[1]{\widehat{#1}}
\newcommand{\wotimes}{\widehat \otimes}
\newcommand{\wo}{\wotimes}
\newcommand{\wt}{\widetilde}
\renewcommand{\tt}{\tu{t}}
\renewcommand{\SS}{{\mathbb S}}
\newcommand{\PSp}{\textup{PSp}}

\usepackage[continuous, page]{pagenote}
\makepagenote

\title{Galois representations for even general special orthogonal groups}

\author{Arno Kret \and Sug Woo Shin}

\begin{document}
\date{\today}

\begin{abstract}
We prove the existence of $\GSpin_{2n}$-valued Galois representations corresponding to cohomological cuspidal automorphic representations of certain quasi-split forms of $\GSO_{2n}$ under the local hypotheses that there is a Steinberg component and that the archimedean parameters are regular for the standard representation. This is based on the cohomology of Shimura varieties of abelian type, of type $D^{\mathbb{H}}$, arising from forms of $\GSO_{2n}$. As an application, under similar hypotheses, we compute automorphic multiplicities, prove meromorphic continuation of (half) spin $L$-functions, and improve on the construction of $\SO_{2n}$-valued Galois representations by removing the outer automorphism ambiguity.
\end{abstract}

\maketitle
\tableofcontents

\newpage

\section*{Introduction}

Inspired by conjectures of Langlands and Clozel's work~\cite{ClozelAnnArbor} for the group $G = \GL_n$, Buzzard--Gee~\cite[Conj.~5.16]{BuzzardGee} formulate the following version of the Langlands correspondence (in one direction) for an arbitrary connected reductive group $G$ over a number field $F$. Let $\A_F$ denote the ring of ad\`eles over $F$. Write $\hat G$ (resp.~$^L G$) for the Langlands dual group (resp.~$L$-group) of $G$ over $\lql$. When $g\in {}^L G(\lql)$, let $g_{\tu{ss}}$ denote its semisimple part.

\begin{conjecture}\label{conj:BuzzardGee}
Let $\ell$ be a prime number and fix an isomorphism $\iota \colon \C \isomto \lql$. Let $\pi$ be a cuspidal $L$-algebraic automorphic representation of $G(\A_F)$. Then there exists a Galois representation
$$
\rho_\pi = \rho_{\pi, \iota} \colon \Gal(\li F/F) \to {}^L G(\lql),
$$
such that for all but finitely many primes $\qq$ of $F$ (excluding $\qq|\ell$ and those such that $\pi_{\qq}$ are ramified), the $\hat G$-conjugacy class of $\rho_\pi(\Frob_{\qq})_{\tu{ss}} \in {}^L G(\lql)$ is the Satake parameter of $\pi_{\qq}$ via $\iota$.
\end{conjecture}

The conjecture of Buzzard-Gee is more precise (and does not assume cuspidality). They describe the image of each complex conjugation element and $\ell$-adic Hodge-theoretic properties of $\rho_\pi$. Moreover they predict \cite[Conj.~5.17]{BuzzardGee} that the compatibility holds at every $\qq$ coprime to $\ell$ such that $\pi_\qq$ is unramified. In fact $\rho_\pi(\Frob_{\qq})$, instead of its semisimple part, appears in their conjecture. While $\rho_\pi(\Frob_{\qq})$ is expected to be always semisimple, this seems to be a problem of different nature and out of reach. Thus we state the conjecture with $\rho_\pi(\Frob_{\qq})_{\tu{ss}}$.

For most recent results on Conjecture~\ref{conj:BuzzardGee} for $\GL_n$ (in the regular case), we refer to \cite{ScholzeTorsion, HLTT} and the references therein. Arthur's endoscopic classification \cite{ArthurBook} (see \cite{Mok,KMSW} for unitary groups)\footnote{The endoscopic classification is conditional in the following sense. At this time, the postponed articles [A25], [A26] and [A27] in the bibliography of \cite{ArthurBook} have not appeared. The proof of the weighted fundamental lemma for non-split groups has not become available yet either.} provides a crucial input for constructing Galois representations as in the conjecture for symplectic, special orthogonal, and unitary groups by reducing the question to the case of general linear groups. When the group is $\SO_{2n}$, however, such an approach proves only a weaker local-global compatibility up to outer automorphisms (see (SO-i) in Theorem \ref{thm:ExistGaloisSO} below), falling short of proving Conjecture \ref{conj:BuzzardGee} (even under local hypotheses); we will return to this point as an application of our main theorem. 

Our goal is to prove Conjecture~\ref{conj:BuzzardGee} for a quasi-split form $G^*$ of $\GSO_{2n}$ over a totally real field under certain local hypotheses, as a sequel to our work \cite{GSp}  where we proved the conjecture for $\GSp_{2n}$ under similar local hypotheses. The group $\GSO_{2n}$ is closely related to the classical group $\SO_{2n}$, just like $\GSp_{2n}$ is to $\Sp_{2n}$, but the similitude groups may well be regarded as non-classical groups. An important reason is that the Langlands dual groups of $\GSO_{2n}$ and $\GSp_{2n}$, namely the general spin groups $\GSpin_{2n}$ and $\GSpin_{2n+1}$, do not admit standard \emph{embeddings} (into general linear groups of proportional rank). This makes the problem both nontrivial and interesting. Furthermore, since the groups $\GSp_{2n}$ and $\GSO_{2n}$ appear as endoscopic groups of each other for varying $n$ \cite[Sect.~2.1]{BinXuLPackets}, results for the one group likely have applications for the other, especially if one tries to prove cases of Conjecture \ref{conj:BuzzardGee} without local hypotheses.

To be more precise, we set up some notation. Let $F$ be a totally real number field, and $n \in \Z_{\geq 3}$. Let $\GSO_{2n}$ denote the connected split reductive group over $F$ which is the identity component of the orthogonal similitude group $\GO_{2n}$. (See \S\ref{sect:RootDatum} below for an explicit definition.)

Our setup depends on the parity of $n$:
\begin{itemize}[leftmargin=+1in]
  \item[($n$ \textbf{even})] $E=F$, and $G^*=\GSO_{2n}$ (the split form over $F$),
  \item[($n$ \textbf{odd})] $E$ is a totally imaginary quadratic extension of $F$, and $G^*$ is a non-split quasi-split form of $\GSO_{2n}$ relative to $E/F$ (explicitly given as \eqref{eq:QuasiSplitGSO}).
\end{itemize}
We write $\GSO^{E/F}_{2n}$ for the $F$-group $G^*$ in either case. The setup is naturally designed so that there are Shimura varieties for (an inner twist of) $\Res_{F/\Q} G^*$. In particular $G^*(F_y)$ has discrete series at every infinite place $y$ of $F$. (Indeed $G^*(F_y)$ has no discrete series if we swap the parity of $n$ above.) There is a short exact sequence of $F$-groups
$$1\ra \SO^{E/F}_{2n} \longrightarrow \GSO^{E/F}_{2n} \stackrel{\simil}{\longrightarrow} \G_m \ra 1,$$
where $\SO^{E/F}_{2n}$ is a quasi-split form of $\SO_{2n}$, defined similarly as $\GSO^{E/F}_{2n}$, and $\simil$ denotes the similitude character.
It is convenient to use the version of $L$-group relative to $E/F$, with coefficients in either $\C$ or $\lql$:
$$
^L G^* = \hat G^* \rtimes \Gal(E/F) = \GSpin_{2n} \rtimes \Gal(E/F),
$$
where the nontrivial element of $\Gal(E/F)$ acts non-trivially on $\GSpin_{2n}$. (This identifies $^L G^*$ with $\GPin_{2n}$ if $[E:F]=2$.)
An important feature of the (general) spin groups $\GSpin_m$ ($m \in \Z_{\geq 2}$) is their spin representation $\spin_m \colon \GSpin_m \rightarrow \GL_{2^{\lfloor m/2 \rfloor}}$. In case $m$ is even, this representation is reducible and splits up into a direct sum $\spin_m = \spin_m^+ \oplus \spin_m^-$ of two irreducible representations of dimension $2^{\lfloor m/2 \rfloor -1}$. These representations $\spin_m^\pm$ are called the \emph{half-spin} representations. Two other important representations are the \emph{standard representation} and the \emph{spinor norm} (see Lemma~\ref{lem:SurjectionOntoGSO} for $\pr^\circ$)
$$
\std \colon \GSpin_m \stackrel{\pr^\circ}{\to} \SO_m \to \GL_m, \quand \cN \colon \GSpin_m \to \GL_1.
$$
If $m$ is odd, $\spin$ is faithful. In the even case $m = 2n$, none of the representations $\spin^+, \spin^-, \std$, or $\cN$ is faithful, but $\spin$ is faithful.

Let $\pi$ be a cuspidal automorphic representation of $\GSO^{E/F}_{2n}(\A_F)$. Consider the following hypotheses on $\pi$, where $|\simil|$ denotes the composite $\GSO^{E/F}_{2n}(F\otimes_\Q \R)\stackrel{\simil}{\ra} (F\otimes \R)^\times \stackrel{|\cdot|}{\ra} \R^\times_{>0}$:
\begin{itemize}[leftmargin=+1in]
\item[\textbf{(St)}] There is a finite $F$-place $\qst$ such that $\pi_{\qst}$ is the Steinberg representation of $G^*(F_{\qst})$ twisted by a character.
\item[\textbf{(L-coh)}] $\pi_\infty|\simil|^{-n(n-1)/4}$ is $\xi$-cohomological for an irreducible algebraic representation $\xi=\otimes_{y:F\hra \C} \xi_y$ of the group $(\Res_{F/\Q} G^*) \otimes_\Q \C\simeq \prod_{y:F\hra \C} (G^*\otimes_{F,y}\C)$.
\item[\textbf{(std-reg)}] The infinitesimal character of $\xi_y$ for every $y:F\hra \C$, which is a regular Weyl group orbit in the Lie algebra of $\hat{G}^*=\GSpin_{2n}(\C)$, remains regular under the standard representation $\GSpin_{2n}\ra \GL_{2n}$.
\end{itemize}
In (L-coh), `$\xi$-cohomological' means that the tensor product with $\xi$ has nonvanishing relative Lie algebra cohomology in some degree (\S\ref{sect:Notation} below). Condition (L-coh) implies that $\pi$ is L-algebraic. The other two conditions should be superfluous as they do not appear in Conjecture \ref{conj:BuzzardGee}. Condition (St) plays an essential role in our argument, and would take significant new ideas and effort to get rid of. We assume (std-reg) for the reason that certain results for regular-algebraic self-dual cuspidal automorphic representations of $\GL_N$, $N>2$, are missing in the non-regular case. However we need less than (std-reg) for our argument to work. The necessary input for us to proceed without (std-reg) is formulated as Hypothesis~\ref{hypo:non-std-reg}, which we expect to be quite nontrivial but within reach nonetheless. Thus we assume either (std-reg) or Hypothesis \ref{hypo:non-std-reg} in the main theorem, hoping that (std-reg) will be removed as soon as the hypothesis is verified.

Let $S_{\mathrm{bad}} = S_{\bad}(\pi)$ denote the finite set of rational primes $p$ such that either $p=2$, $p$ ramifies in $F$, or $\pi_\qq$ ramifies at a place $\qq$ of $F$ above $p$. The following theorem assigns an $\ell$-adic Galois representation to $\pi$  for each prime number $\ell$ and each isomorphism $\iota \colon \C \isomto \lql$.

\begin{thmx}\label{thm:A}
Assume that $\pi$ satisfies conditions (St) and (L-coh). If (std-reg) does not hold for $\pi$, further assume Hypothesis \ref{hypo:non-std-reg} (for an $\SO_{2n}(\A_F)$-subrepresentation of $\pi$). Then there exists, up to $\hat G$-conjugation, a unique semisimple Galois representation attached to $\pi$ and $\iota$
$$
\rho_{\pi} = \rho_{\pi,\iota} \colon \Gal(\li F/F) \to \LG^*
$$
such that the following hold.
\begin{enumerate}[label=(A\arabic*)]
\item\label{thm:Ai} For every prime $\qq$ of $F$ not above $S_{\mathrm{bad}}\cup \{\ell\}$, $\rho_{\pi}(\Frob_\qq)_{\tu{ss}}$ is $\hat G^*$-conjugate to $\iota\phi_{\pi_\qq}(\Frob_\qq)$, where $\phi_{\pi_\qq}$ is the unramified Langlands parameter of $\pi_{\qq}$.
\item\label{thm:Aii} The composition 
$$
\Gal(\li F/F) \overset {\rho_{\pi}} \to \LG^* \overset {\pr^\circ} \to \SO_{2n}(\lql)  \rtimes \Gal(E/F)
$$
corresponds to a cuspidal automorphic $\SO_{2n}^{E/F}(\A_F)$-{sub\-re\-pre\-sen\-ta\-tion} $\pi^\flat$ contained in $\pi$ in that $\pr^\circ(\rho_{\pi}(\Frob_{\qq})_{\tu{ss}})$ is $\SO_{2n}(\lql)$-conjugate to the Satake parameter of $\pi^\flat_\qq$ via $\iota$ at every $\qq$ not above $S_{\mathrm{bad}}\cup \{\ell\}$.
Further, the composition
$$
\Gal(\li F/F) \overset {\rho_{\pi}} \to \LG^*  \overset {\cN} \to \GL_1(\lql) 
$$
corresponds to the central character of $\pi$ via class field theory and $\iota$.

\item\label{thm:Aiii} For every $\qq|\ell$, the representation $\rho_{\pi, \qq}$ is de Rham (in the sense that $r \circ \rho_{\pi,\qq}$ is de Rham for all representations $r$ of $\widehat G^*$). Moreover
\begin{itemize}
\item[(a)] The Hodge--Tate cocharacter of $\rho_{\pi, \qq}$ is explicitly determined by $\xi$. More precisely, for all $y \colon F \to \C$ such that $\iota y$ induces $\qq$, we have
$$
\mu_{\tu{HT}}(\rho_{\pi, \qq}, \iota y) = \iota \mu_{\tu{Hodge}}(\xi_y) - \frac{n(n-1)}{4} \simil. 
$$
(We still write $\simil$ to mean the cocharacter of $\GSpin_{2n}$ dual to
$\simil:G^*\ra \G_m$. See \S\ref{sect:Notation} below for the Hodge--Tate and Hodge cocharacters $\mu_{\tu{HT}}$ and $\mu_{\tu{Hodge}}$.\footnote{More precisely, the Hodge cocharacter is a half-integral cocharacter, but subtracting $n(n-1)/4$ times $\tu{sim}$ makes it integral. The two cocharacters in (A3)(a) are well-defined up to conjugacy (i.e., they are conjugacy classes of cocharacters), but the formula makes sense because $\simil$ is a central cocharacter.})
\item[(b)] If $\pi_\qq$ has nonzero invariants under a hyperspecial (resp.~Iwahori) subgroup of $G^*(F_\qq)$ then either $\rho_{\pi, \qq}$ or a quadratic character twist is crystalline (resp.~semistable).
\item[(c)] If $\ell \notin S_{\mathrm{bad}}$ then $\rho_{\pi, \qq}$ is crystalline.
\end{itemize}
\item\label{thm:Aiv} For every $v|\infty$, $\rho_{\pi,v}$ is odd (see \S\ref{sect:Notation} and Remark~\ref{rem:complex-conj-GSO} below).
\item\label{thm:Av}
The Zariski closure of the image of $\rho_\pi(\Gal(\ol F/E))$ in $\PSO_{2n}$ maps onto one of the following four subgroups of
$\PSO_{2n}$:
\begin{itemize}
\item[(a)] $\PSO_{2n}$,
\item[(b)] $\PSO_{2n-1}$ (as a reducible subgroup),
\item[(c)] the image of a principal $\SL_2$ in $\PSO_{2n}$, or
\item[(d)] (only when $n=4$) $G_2$ (embedded in $\SO_7 \subset \PSO_8$) or $\SO_7$ (as an irreducible subgroup via the projective spin representation).
\end{itemize}
\item\label{thm:Avi}
If $\rho' \colon \Gal(\li F/F) \to \LG^*$ is another semisimple Galois representation such that, for almost all finite $F$-places $\qq$ where $\rho'$ and $\rho_\pi$ are unramified, the semisimple parts $\rho'(\Frob_\qq)_{\textup{ss}}$ and $\rho_{\pi}(\Frob_\qq)_{\textup{ss}}$ are conjugate, then $\rho$ and $\rho'$ are conjugate.
\end{enumerate}
\end{thmx}

\begin{remark}
The proof of the above theorem relies crucially on the main results of Arthur's book \cite{ArthurBook}, which are currently conditional as explained in footnote 2. In particular Theorem A, and in turn Theorems B, C and D are conditional on the same results mentioned in this footnote. 
\end{remark}

As explained below Conjecture \ref{conj:BuzzardGee}, the existence of Galois representations
\begin{equation}\label{eq:IntroRhoPiFlat}
\rho_{\pi^\flat} \colon \Gal(\li F/F) \to \SO_{2n}(\lql) \rtimes \Gal(E/F)
\end{equation}
in a weaker form is known for cuspidal automorphic representations $\pi^\flat$ of $\SO_{2n}^{E/F}(\A_F)$ satisfying (coh$^\circ$), (St$^\circ$), and (std-reg$^\circ$) (see Section~\ref{sect:GSO-valued-Galois} for these conditions), and possibly a larger class of representations though we have not worked it out. The main ingredients are Arthur's transfer \cite[Thm.~1.5.2]{ArthurBook} from $\SO_{2n}^{E/F}(\A_F)$ to $\GL_{2n}(\A_F)$, and collective results on the Langlands correspondence for $\GL_{2n}(\A_F)$ in the self-dual case. Statements (SO-i)--(SO-v) of Theorem~\ref{thm:ExistGaloisSO} below summarize what we know about $\rho_{\pi^\flat}$. A main drawback of Theorem~\ref{thm:ExistGaloisSO} is that the conjugacy class of each $\rho_{\pi^\flat}(\Frob_{\qq})_{\tu{ss}}$ is determined only up to $\tu{O}_{2n}$-conjugacy, rather than $\SO_{2n}$-conjugacy.

Using Theorem \ref{thm:A} we can upgrade Theorem~\ref{thm:ExistGaloisSO} and remove this ``outer'' ambiguity (coming from the outer automorphism) as long as $\pi^\flat$ can be extended to a cohomological representation $\pi$ of $\GSO_{2n}^{E/F}$. If $\pi$ is $\xi$-cohomological then $\xi$ must satisfy condition (cent) of \S\ref{sect:Shimura}, so a necessary condition for such a cohomological extension to exist is the following condition (which is void for $F=\Q$):
\begin{itemize}
  \item[\textbf{\centSO}] the central character $ \{\pm 1\}=\mu_2(F_y)\ra \C^\times$ of $\pi_y$ at each infinite place $y$ of $F$ is independent of $y$.
\end{itemize}

\begin{thmx}\label{thm:B}
Let $\pi^\flat$ be a cuspidal automorphic representation of $\SO^{E/F}_{2n}(\A_F)$ satisfying \centSO, (coh$^\circ$), (St$^\circ$), and (std-reg$^\circ$). Then Conjecture \ref{conj:BuzzardGee} holds (for every $\ell$ and $\iota$). The associated Galois representation $\rho_{\pi^\flat}$ is characterized uniquely up to $\SO_{2n}(\lql)$-conjugation.
\end{thmx}

See Theorem \ref{thm:ExistGaloisSO-Plus} below for a precise and stronger statement. The crux of the argument lies in showing that $\pi^\flat$ extends to an automorphic representation $\pi$ of $\GSO_{2n}^{E/F}(\A_F)$ satisfying conditions of Theorem~\ref{thm:A}. As Theorem~\ref{thm:A} has no outer ambiguity, this yields Theorem~\ref{thm:B}.

Theorem \ref{thm:B} offers a new perspective on the local Langlands correspondence for quasi-split forms of $\SO_{2n}$ over $p$-adic fields. By localizing the theorem at finite places, we get a candidate for the correspondence, not just up to $\OO_{2n}$-conjugacy as in \cite{ArthurBook}. More precisely, let $H$ denote a quasi-split form of $\SO_{2n}$ over a $p$-adic field $k$, assumed to be split if $n$ is even. Then we can find $E/F$ as above (depending on the parity of $n$) and a prime $\qq$ of $F$ such that $F_\qq\simeq k$ and $\SO^{E/F}_{2n,\qq} \simeq H$. If $\sigma$ is an irreducible discrete series representation of $H(k)$ then a candidate for the $L$-parameter for $\sigma$ is described by the following procedures.
\begin{enumerate}
  \item Find $\pi^\flat$ satisfying \centSO, (coh$^\circ$), (St$^\circ$), and (std-reg$^\circ$) such that $\pi^\flat_\qq\simeq \sigma$.
  \item Obtain $\rho_{\pi^\flat}$ from Theorem \ref{thm:B} (which relies on Theorem \ref{thm:A}).
  \item Take $\tu{WD}(\rho_{\pi^\flat}|_{\Gamma_{F_\qq}})$, which can be viewed as an $L$-parameter for $H(k)$.
\end{enumerate}
The globalization in (1) is possible by a standard trace formula argument proving the limit multiplicity formula. See \S\ref{sect:Notation} below for the definition of WD. The $L$-parameter resulting from the above is in the $\OO_{2n}$-orbit of the $L$-parameter in \cite{ArthurBook} by Theorem \ref{thm:ExistGaloisSO} (SO-i), but could a priori depend on various choices. It is an interesting problem to relate the global construction here to the purely local constructions by Kaletha \cite{KalRegSC,KalSC} and Fargues--Scholze \cite{FarguesScholze}. In fact all this can be mimicked for $\GSO_{2n}$ in place of $\SO_{2n}$, using Theorem \ref{thm:A} rather than Theorem \ref{thm:B}, so a similar question may be asked in the $\GSO_{2n}$-case.

As another application of Theorem \ref{thm:A}, we compute the automorphic multiplicities $m(\pi)$ for certain automorphic representations $\pi$ of $\GSO_{2n}^{E/F}(\A_F)$.

\begin{thmx}\label{thm:AutomMult}
Let $\pi$ be a cuspidal automorphic representation of $\GSO_{2n}^{E/F}(\A_F)$
satisfying (L-coh), (St) and (std-reg). Then we  have $m(\pi) = 1$.
\end{thmx}

To compute $m(\pi)$ for $\GSO_{2n}^{E/F}$ we rely on Theorem \ref{thm:A}, Arthur's multiplicity formula \cite{ArthurBook} and a result of Bin Xu \cite{BinXuLPackets} to show that $m(\pi) = m(\pi^\flat)$ for $\pi^\flat \subset \pi$ a well-chosen $\SO_{2n}^{E/F}(\A_F)$-subrepresentation. We remark that Arthur's multiplicity formula computes multiplicities up to an outer automorphism orbit, but $m(\pi)$ in the theorem is the honest multiplicity.

Our final application is meromorphic continuation of the (half) spin-$L$ functions. Let $\pi$ be a cuspidal automorphic representation of $\GSO^{E/F}_{2n}(\A_F)$ unramified away from a finite set of places $S$. To make uniform statements, define a set
\begin{equation}\label{eq:fke}
\fke :=\begin{cases}
  \{+,-\}, & \mbox{if}~n~\mbox{is even (thus}~E=F),\\
  \{\emptyset\}, & \mbox{if}~n~\mbox{is odd (thus}~[E:F]=2),\\
\end{cases}
\end{equation}
with the understanding that $\spin^\emptyset=\spin$. The partial (half-)spin $L$-function for $\pi$ away from $S$ is by definition
\begin{equation}\label{eq:PartialhalfSpinLfunction}
L^S(s,\pi,\spin^\varepsilon):=\prod_{\fkp\notin S} \frac{1}{\det(1-q_\fkp^{-s}\spin^\varepsilon(\phi_{\pi_\fkp}(\Frob_{\fkp})))},\qquad \varepsilon\in \fke,
\end{equation}
where $q_{\fkp} := \#(\cO_F/\fkp)$ and $\phi_{\pi_\fkp}$ is the unramified $L$-parameter of $\pi_\fkp$. Consider the following hypothesis for $L$-parameters $\phi_{\pi_y}$ at infinite places $y$.
\vspace{.05in}
\begin{itemize}[leftmargin=+1in]
\item[\textbf{(spin-reg)}] $\spin^\varepsilon(\phi_{\pi_y})$ is regular for every infinite place $y$ of $F$ and every $\varepsilon\in \fke$.\vspace{.05in}
\end{itemize} 
When $n\ge 3$, (spin-reg) implies (std-reg). This hypothesis ensures that $\spin^\varepsilon(\rho_\pi)$ has distinct Hodge--Tate weights. Our construction and Theorem \ref{thm:A} allow us to apply the potential automorphy theorem of Barnet-Lamb--Gee--Geraghty--Taylor \cite{BLGGT14-PA} to the weakly compatible system of $\spin^\varepsilon(\rho_\pi)$ (as $\ell$ and $\iota$ vary). Thereby we obtain the following.

\begin{thmx}\label{thm:meromorphic}
Assume $n\ge 3$. Let $\pi$ be a cuspidal automorphic representation of $\GSO_{2n}^{E/F}(\A_F)$ satisfying (L-coh), (St) and (spin-reg). Then there exists a finite totally real extension $F'/F$ (which can be chosen to be disjoint from any prescribed finite extension of $F$ in $\ol{F}$) such that $\spin^\varepsilon\circ \rho_\pi|_{\Gal(\li F/F')}$ is automorphic for each $\varepsilon\in \fke$. More precisely, there exists a cuspidal automorphic representation $\Pi^\varepsilon$ of $\GL_{2^n/|\fke|}(\A_{F'})$ such that
\begin{itemize}
\item for each finite place $\qq'$ of $F'$ not above $S_{\mathrm{bad}}\cup \{\ell\}$, the representation $\iota^{-1}\spin^\varepsilon\circ\rho_\pi|_{W_{F'_{\qq'}}}$ is unramified and its Frobenius semisimplification is the Langlands parameter for~$\Pi^\varepsilon_w$,
\item at each infinite place $y'$ of $F'$ above a place $y$ of $F$, we have $\phi_{\Pi^\varepsilon_{y'}}|_{W_\C}\simeq \spin^\varepsilon\circ \phi_{\pi_y}|_{W_\C}$.
\end{itemize}
In particular the partial spin $L$-function $L^S(s,\pi,\spin^\varepsilon)$ admits a meromorphic continuation and is holomorphic and nonzero in an explicit right half plane (e.g., in the region $\Re(s)\ge 1$ if $\pi$ has unitary central character).
\end{thmx}

We now give a sketch of the argument for Theorem A. For simplicity, we put ourselves in the \emph{split} case (when $n$ is even), and assume $F = \Q$ to simplify notation. We also ignore all character twists and duals in the following sketch and keep the isomorphism $\iota:\C\simeq \lql$ implicit. (See the main text for correct twists and duals.)

The basic idea is to construct $\rho_\pi$ and prove its expected properties by understanding what should be $\spin^+ \circ\rho_\pi$, $\spin^- \circ\rho_\pi$, $\std\circ \rho_\pi$, and $\cN \circ \rho_\pi$. One already has access to $\std\circ \rho_\pi$ via Arthur's endoscopic classification and known instances of the global Langlands correspondence. The seemingly innocuous $\cN \circ \rho_\pi$ is not so trivial to combine with the other representations, but refer to the proof of Proposition \ref{prop:FinalProp}. Most importantly, we realize $\spin^+ \circ\rho_\pi$ and $\spin^- \circ\rho_\pi$ in the cohomology of suitable Shimura varieties; this is the port of embarkation.

In fact $\rho_\pi$ would not be recovered from $\spin^+ \circ\rho_\pi$, $\spin^- \circ\rho_\pi$, $\std\circ \rho_\pi$, and $\cN \circ \rho_\pi$ in general due to essential group-theoretic difficulties (e.g., $\GSpin_{2n}$ is not acceptable in the sense of \cite{LarsenConjugacy, Larsen2}), but condition (St) mitigates the matter. Another important role of (St) is to remove complexity associated with endoscopy.

Our Shimura varieties are associated with an inner twist $G/\Q$ of the split group $\GSO_{2n}$ (unique up to isomorphism) which splits at all primes $p \neq p_{\St}$, and whose derived subgroup is isomorphic to the quaternionic orthogonal group $\SO^*(2n)$ over $\R$ (which is not isomorphic to $\SO(a,b)$ for any signature $a+b=2n$). Concretely $G(\R)$ is isomorphic to the group $\GSO_{2n}^J(\R)$ in \S\ref{sect:forms-of-GSO} below.

The group $G$ admits two abelian-type Shimura data $(G, X^\eps)$ with $\eps\in \{+,-\}$, corresponding to the two edges of the ``fork''  in the Dynkin diagram of type $D_n$ (see Section~\ref{sect:Shimura}). These two Shimura data are not isomorphic. (The analogous Shimura data \emph{are} isomorphic via an outer automorphism when $n$ is odd; see Lemma (ii) below. Even then, we distinguish the two data as the outer automorphism changes isomorphism classes of representations.)

Let $\pi$ be as in Theorem \ref{thm:A}. Using a trace formula argument, we transfer $\pi$ to a $\xi$-cohomological cuspidal automorphic representation $\pi^\natural$ of $G(\A)$ with isomorphic unramified local components as $\pi$ such that $\pi^\natural$ is Steinberg at a finite prime. Let $\rho_{\pi}^{\Sh,\eps}$ be the $\Gal(\li \Q/\Q)$-representation on the $\pi^{\natural,\infty}$-isotypical part of the (semisimplified) compact support cohomology of the $\ell$-adic local system $\cL_\xi/\textup{Sh}(G, X^\eps)$ attached to $\xi$. Conjecturally the two representations $\rho_\pi^{\Sh,\eps}$ should realize $\spin^{\eps} \circ \rho_{\pi}$ up to semi-simplification (and up to a twist and a multiplicity that we ignore in this introduction), in the non-endoscopic case. In particular, if $\phi_{\pi_p} \colon W_{\Q_p}\ra \GSpin_{2n}(\C)$ is the unramified $L$-parameter of $\pi_p$ at a prime $p\neq \ell$ where $\pi_p$ is unramified, then $\rho_\pi^{\Sh,\eps}|_{\Gal(\li \Q_p/\Q_p)}$ ought to be unramified and satisfy
\begin{equation}\label{eq:wrinkle}
\Tr  \rho_\pi^{\Sh,\eps}(\Frob_\p^j) = \Tr \spin^\eps (\phi_{\pi_p}(\Frob_p)^j) \in \lql,\qquad j\gg 1.
\end{equation}
Employing Kisin's results on the Langlands--Rapoport conjecture \cite{KisinPoints} and the Langlands--Kottwitz method for Shimura varieties of abelian type in the forthcoming work of Kisin--Shin--Zhu \cite{KSZ}, we prove \eqref{eq:wrinkle} for almost all $p$.

Let $\pi^\flat \subset \pi$ be an irreducible cuspidal automorphic $\SO_{2n}(\A)$-subrepresentation. From the aforementioned weaker version of Conjecture \ref{conj:BuzzardGee} for $\SO_{2n}$, we construct (see Theorem \ref{thm:ExistGaloisSO} below)
$$
\rho_{\pi^\flat} \colon \Gal(\li \Q/\Q) \to \SO_{2n}(\lql).
$$
such that
\begin{equation}\label{eq:wrinkle1}
\rho_{\pi^\flat}(\Frob_p)_{\tu{ss}} \osim  \pr^\circ(\phi_{\pi_p}(\Frob_p)) \in \SO_{2n}(\lql),
\end{equation}
for all primes $p \neq \ell$ where $\pi^\flat$ is unramified. Here $\osim$ indicates $\OO_{2n}(\lql)$-conjugacy, and $\pr^\circ:\GSpin_{2n} \surjects \SO_{2n}$ is the natural surjection.

We expect $\rho_\pi$ to lift $\rho_{\pi^\flat}$ (up to outer automorphism) and to sit inside $\rho^{\Sh}:=\rho^{\Sh,+}_{\pi}\oplus \rho^{\Sh,-}_{\pi}$ as illustrated below. By $\li\spin$ we mean the unique projective representation of $\SO_{2n}$ that the projectivization of $\spin$ factors through.
\begin{equation}\label{eq:BigDiagramIntro}
\xymatrix{
\Gal(\li \Q/\Q) \ar@/^1.5pc/[rrr]^{\rho_\pi^\Sh} \ar@/_1.5pc/[drr]_{\rho_{\pi^\flat}} \ar@{-->}[rr]_-{\rho_\pi} && \GSpin_{2n}(\lql) \ar[d]_{\pr^\circ}\ar@{^(->}[r]_-{\spin} &\GL_{2^{n}}(\lql)\ar[d] &\cr && \SO_{2n}(\lql) \ar@{^(->}[r] \ar@{^(->}[r]_{\overline{\spin}\quad}&\PGL_{ 2^{n}}(\lql). \cr }
\end{equation}
We deduce from \eqref{eq:wrinkle} and \eqref{eq:wrinkle1} that the outer diagram commutes, after a conjugation if necessary. In fact this is not straightforward because two $\PGL_{2^n}$-valued Galois representations need not be conjugate even if they map each $\Frob_p$ into the same conjugacy class for almost all $p$. We get around the difficulty by using a classification of reductive subgroups of $\SO_{2n}$ containing a regular unipotent element by Saxl--Seitz \cite{SaxlSeitz}. This is applicable since (St) tells us that the image of $\rho_{\pi^\flat}$ contains a regular unipotent element. As a consequence, the Zariski closure of the image of $\rho_{\pi^\flat}$ is connected mod center. If it is connected, we have the commutativity of \eqref{eq:BigDiagramIntro} after a conjugation, and it follows that there exists $\rho_\pi$ completing the diagram. If the Zariski closure is connected only mod center, then we need a variant of \eqref{eq:BigDiagramIntro} as explained in \S\ref{sect:Construction}. A similar group-theoretic consideration shows that $\rho_\pi$ is characterized up to isomorphism by the images of Frobenius elements at almost all primes, cf.~(A6) of Theorem \ref{thm:A}.

Having constructed $\rho_\pi$, we verify that $\rho_\pi$ enjoys the expected properties. Let us focus here on (A1). By construction,
$$
\spin(\rho_{\pi}(\Frob_p)_{\tu{ss}})\sim \spin(\phi_{\pi_p}(\Frob_p)),\qquad \mbox{for~almost~all}~p.
$$
The key point is to refine this, or break the symmetry, by showing the same relation with $\spin^+$ and $\spin^-$ in place of $\spin$ (cf.~proof of Proposition \ref{prop:FinalProp} below) with the help of \eqref{eq:wrinkle}. Roughly speaking, we are in a situation
$$
\rho^{\Sh,+}\oplus \rho^{\Sh,-} \simeq \spin^+ \rho_\pi\oplus \spin^- \rho_\pi$$
and want to match the $+$ and $-$ parts. The problem is easy enough if $\spin^+\rho_\pi\simeq \spin^-\rho_\pi$ as there is little to distinguish. If $\spin^+\rho_\pi\not\simeq\spin^-\rho_\pi$ then the idea is that the $+$ and $-$ parts do not overlap at sufficiently many places (by a Chebotarev type argument) to match the $+$ and $-$ parts unambiguously. If $\spin^+\rho_\pi$ or $\spin^-\rho_\pi$ is irreducible, it is quite doable to promote this idea to a robust argument. In general, the smaller image of $\rho_\pi$, the harder this problem becomes. On the other hand, in certain cases where the image is really small, such as contained in a principal $\PGL_2$, the conjugacy classes $\rho_\pi(\Frob_p)_{\tu{ss}}$ are stable under outer conjugation, and there is no distinction between inner and outer conjugacy. As we also have a classification of the (Zariski closure) of the possible images of $\rho_\pi$, we can deal with each case via explicit group-theoretic computation. This finishes the sketch of proof for Theorem \ref{thm:A}.

\subsection*{Structure of the paper}
The paper splits roughly into four parts consisting of Sections~1--8 (preparation), Sections~9--12 (the core argument), Sections~13--15 (applications), and the appendices. Let us go over these parts in more detail. In Sections~1--5 we define (variants of) orthogonal groups and spin groups along with subgroups containing regular unipotent elements and the outer automorphism. We define the spin groups and their spin representations through root data as well as Clifford algebras by fixing the underlying quadratic spaces, and clarify the relationship between them. The root-theoretic approach is natural in the context of Langlands correspondence whereas Clifford algebras have the advantage that various maps are determined and diagrams commute on the nose and not just up to conjugation. In Section~6 we construct Galois representations for certain cuspidal automorphic representations of quasi-split even orthogonal groups. This relies on Arthur's book~\cite{ArthurBook} and the known construction of automorphic Galois representations, but a few extra steps are taken to get the information that we need later on. In particular we study what happens to the Steinberg representation under Arthur's transfer from $\SO_{2n}^{E/F}$ to $\GL_{2n}$ (this relies on Appendix B). In Section 7 we list a number of basic results on comparing representations of $\SO_{2n}^{E/F}$ with those of $\GSO_{2n}^{E/F}$. Section~8 discusses properties of the real points of $\GSO_{2n}^{E/F}$ and introduces  certain global inner forms $G$ of $\GSO_{2n}^{E/F}$. The core argument starts in Section~9, where we take the cohomology of Shimura varieties associated with two Shimura data $(G, X^\pm)$ to find two Galois representations $\rho_{\pi}^{\Sh, \pm}$ attached to $\pi$ as in the main theorem. In Section 10 we construct a $\GSpin_{2n}$-valued Galois representation $\rho_\pi$ of $\Gal(\li F/E)$ from $\rho_\pi^{\Sh, \pm}$ and $\rho_{\pi^\flat}$. This representation is not quite the one of Theorem A: The image of Frobenius under $\rho_\pi$ is controlled only outside an unspecified finite set of primes, and moreover $\rho_\pi$ should be extended to a representation of $\Gal(\li F/F)$. The two problems are resolved in Sections~11 and 12 respectively. We emphasize that neither of these arguments is formal, the first one relies on Bin Xu's work \cite{BinXuLPackets} and the second on a subtle global argument. The proof of Theorem~A is also completed in Section~12. Sections~13--15 present applications of our main theorem to the construction of Galois representations for $\SO_{2n}^{E/F}$, automorphic multiplicity, and meromorphic continuation of (half)-spin $L$-functions.

\subsection*{Acknowledgments}

We are very grateful for an anonymous referee for his or her comments and suggestions. SWS is partially supported by NSF grant DMS-1802039 and NSF RTG grant DMS-1646385. AK is partially supported by an NWO VENI and an NWO VIDI grant. 

\section{Notation and preliminaries}\label{sect:Notation}

We fix the following notation.
\begin{itemize}
\item $n \geq 3$ is an integer.\footnote{We should mention that if $n \leq 3$, there are exceptional isomorphisms of $\GSO_{2n}$ (and its outer forms) to other simpler groups; for instance for $n = 3$ the Shimura varieties that we obtain are (closely related to) Shimura varieties for unitary similitude groups, in particular more general results are already known.}
\item If $k$ is a field, $\ol k$ denotes an algebraic closure of $k$.
\item When $X$ is a square matrix, $\mathscr{EV}(X)$ denotes the multi-set of eigenvalues of $X$.
\item When $A$ is a multi-set with elements in a ring $R$ with $r\in R$, write $r\cdot A$ for the multi-set formed by the elements $ra \in A$ as $a$ ranges over $A$. For $n\in \Z_{>0}$, write $A^{\oplus n}$ for the multi-set consisting of $a\in A$ whose multiplicity in $A^n$ is $n$ times that in $A$.
\item $F$ is a number field. (In the main text, $F$ is a totally real field with a distinguished embedding into $\C$.)
\item $\mathcal{O}_F$ is the ring of integers of $F$.
\item $\A_F$ is the ring of ad\`eles of $F$, $\A_F := (F \otimes \R) \times (F \otimes \Zhat)$.
\item If $S$ is a finite set of $F$-{pla\-ces}, then $\A_F^S \subset \A_F$ is the ring of ad\`eles with trivial components at the places in $S$, and $F_S := \prod_{v \in S} F_v$; $F_\infty:=F\otimes_\Q \R$.
\item If $\qq$ is a finite $F$-place, we write $q_{\qq}$ for the cardinality of the residue field of $\qq$.
\item $|\cdot|:\A_F^\times \ra \R^\times_{>0}$ is the norm character on $\A_F^\times$ that is trivial on $F^\times$. Denote by $|\cdot|_v:F_v^\times \ra \R^\times_{>0}$
the restriction of $|\cdot|$ to the $v$-component. Our normalization is that $|\cdot|_{\qq}$ sends a uniformizer of $F_\qq$ to $q_{\qq}^{-1}$, whereas $|\cdot|_v$ is the usual absolute value (resp.~squared absolute value) when $v$ is real (resp.~complex).
\item  If $S$ is a set of prime numbers we write $S^F$ for the set of $F$-places above $S$.
\item If $p$ is a prime number, then $F_p := F \otimes_{\Q} \qp$.
\item $\ell$ is a prime\textbf{} number (typically different from $p$).
\item $\lql$ is a fixed algebraic closure of $\ql$, and $\iota \colon \C \isomto \lql$ is an isomorphism.
\item For each prime number $p$ we fix the positive root $ p^{1/2} \in \R_{>0} \subset \C$. From $\iota$ we then obtain a choice for $p^{1/2} \in \lql$. If $q$ is a power of $p$, we obtain similarly a preferred choice $q^{1/2}$ in $\lql$ and in $\C$.
\item $\Gamma =\Gamma_F:= \Gal(\li F/F)$ is the absolute Galois group of $F$.
\item For a finite extension $E$ of $F$ in $\li F$, write $\Gamma_E:=\Gal(\li F/E)$ and $\Gamma_{E/F}:=\Gal(E/F)$.
\item $\Gamma_v = \Gamma_{F_v}:=\Gal(\li F_v/F_v)$ is (one of) the local Galois group(s) of $F$ at the place $v$, $W_{F_v} \subset \Gamma_v$ is the corresponding Weil group.
\item For each $F$-place $v$, choose an embedding $\iota_v:\li F \hra \li F_v$, which induces $\Gamma_v\hra \Gamma$ that is canonical up to conjugation.
\item $\cV_\infty:= \Hom_\Q(F, \R)$ is the set of infinite places of $F$.
\item $c_y\in \Gamma$ is the complex conjugation (well-defined as a conjugacy class) induced by any embedding $\li{F}\hookrightarrow \C$ extending $y\in \cV_\infty$.
\item If $S$ is a finite set of $F$-places, write $\Gamma_{F,S}$ for the Galois group $\Gal(F(S)/F)$ where $F(S) \subset \li F$ is the maximal extension of $F$ that is unramified away from $S$. If $S$ is a set of rational places we write
$\Gamma_{F, S} := \Gamma_{F, S^F}$.
\item $\Frob_\qq$ at a finite prime $\qq$ of $F$ means the \emph{geometric} Frobenius element in the quotient of $\Gamma_\qq$ by the inertia subgroup, or the image thereof in $\Gamma_{F,S}$. (The image in $\Gamma_{F,S}$ depends on the choice of $\iota_\qq$ but its conjugacy class is independent of the choice.)
\item When $G$ is a connected reductive group over $F$, write $\hat G$ and $^L G= \hat G\rtimes \Gamma_F$ for the Langlands dual group and the $L$-group, respectively (with coefficients in $\C$ or $\lql$, depending on the context). If $G$ splits over a finite extension $E/F$ in $\li F$ then $\hat G\rtimes \Gamma_{E/F}$ denotes the $L$-group with respect to $E/F$. (Namely such a semi-direct product is always understood with the $L$-action of $\Gamma_{E/F}$ on $\hat G$.) Often we use $^L G$ to mean $\hat G\rtimes \Gamma_{E/F}$.\footnote{This is harmless for us as the inflation map induces a bijection of isomorphism classes of $^L G$-valued Galois representations when $\Gamma_{E/F}$ is replaced with $\Gamma_F$ in the semi-direct product.}
\item When $H$ is a reductive group over $\lql$, we also use $H$ to mean the topological group $H(\lql)$ by abuse of notation. This should be clear from the context and not leading to confusion.
\item When $F$ is a $p$-adic field and $G$ is the set of $F$-points of a reductive group over $F$, we write $\St_G$ for the Steinberg representation of $G$ (defined in \cite[X.4.6]{BorelWallach} for instance). Moreover, we write $\one_G$ for the trivial representation of $G$. In certain cases, when $G$ is clear, we write $\St = \St_G$ or $\one = \one_G$. We also write sometimes $\St_n$ for $\St_{\GL_n(F)}$ (in case $F$ is clear from the context).
\item If $G$ is an algebraic group, we write $Z(G)$ for its center.
\item An inner twist of a reductive group $G$ over a perfect field $k$ means a reductive group $G'$ over $k$ together with
an isomorphism $i:G_{\ol k}\ra G'_{\ol k}$ such that the automorphism $i^{-1}\sigma(i)$ of $G_{\ol k}$ is inner for every $\sigma\in \Gal(\ol k/k)$.
There is an obvious notion of isomorphism for inner twists $(G',i)$, cf.~\cite[2.2]{KalethaLLC}.
We often say $G'$ is an inner twist of $G$, keeping $i$ implicit. If we forget $i$ and only remember the $k$-group $G'$ and the existence of $i$, we refer to it as an inner form of $G$.
\end{itemize}

Fix $G$ and $E/F$ as above. We introduce some notions on the Galois side.
By an ($\ell$-adic) \emph{Galois representation} of $\Gamma_F$ (with values in $\hat G\rtimes \Gamma_{E/F}$), we mean a continuous morphism 
$$
\rho \colon \Gamma_F \ra \hat{G}(\lql)\rtimes \Gamma_{E/F}
$$
which factors through $\Gamma_{F,S}$ for some finite set $S$ and commutes with the obvious projections onto $\Gamma_{E/F}$. Similarly we define a Galois representation with the source $\Gamma_\qq$ or with values in $^L G(\lql)$. Two Galois representations are considered isomorphic if they are conjugate by an element of $\hat{G}(\lql)$. We say that $\rho$ as above is \emph{(totally) odd} if for every real place $y$ of $F$, the following holds: writing $\tu{Ad}$ for the adjoint action of $^L G$ on $\Lie G(\lql)$, which preserves the Lie algebra of the derived subgroup $G_{\textup{der}}$, the image of $c_y$ under the composite
$$
\Gamma_y \hra \Gamma \stackrel{\rho}{\ra} {}^L G(\lql) \stackrel{\tu{Ad}}{\ra} \GL(\Lie G_{\textup{der}}(\lql))
$$
has trace equal to the rank of $\hat G_{\textup{der}}$. (Compare with \cite{GrossOdd}.)

An $^L G$-valued \emph{Weil--Deligne representation} of $W_{F_\qq}$ is a pair $(r,N)$ consisting of a morphism
$$
r \colon W_{F_\qq} \ra \hat{G}(\lql)\rtimes \Gamma_{E_\pp/F_\qq}
$$
which has open kernel on the inertia subgroup and commutes with the canonical projections onto $ \Gamma_{E_\pp/F_\qq}$, and a nilpotent operator $N\in \Lie \hat{G}(\lql)$ such that $\Ad (r(w)) N= |w| N$ for $w\in W_{F_\qq}$, where $|\cdot|:W_{F_\qq}\ra \|\qq\|^{\Z}$ is the homomorphism sending a geometric Frobenius element to $\|\qq\|^{-1}$; here $\|\qq\|\in \Z_{>0}$ denotes the norm of $\qq$. The Frobenius-semisimplification $(r^{\tu{ss}},N)$ is obtained by replacing $r$ with its semisimplification. We say $(r,N)$ is Frobenius-semisimple if $r=r^{\tu{ss}}$.

Let $\rho \colon \Gamma_F \ra \hat{G}(\lql)\rtimes \Gamma_{E/F}$ be a Galois representation. Write $\pp$ for the prime of $E$ induced by $\iota_\qq:\li F \hra \li F_\qq$. Then the restriction (via $\iota_\qq$)
$$
\rho|_{\Gamma_\qq} \colon \Gamma_{F_\qq} \ra \hat{G}(\lql)\rtimes \Gamma_{E_\pp/F_\qq}
$$
gives rise to an $^L G$-valued Weil--Deligne representation, to be denoted by $\tu{WD}(\rho|_{\Gamma_{F_\qq}})$. The construction follows from the case of $G=\GL_n$ by the Tannakian formalism via algebraic representations of $\hat{G}(\lql)\rtimes \Gamma_{E_\pp/F_\qq}$. (The case $\qq|\ell$ is more subtle than $\qq\nmid \ell$. In the former case, a detailed explanation is given in the proof of \cite[Lem.~3.2]{GSp}, where $\hat G$ is denoted by $H$. In \emph{loc.~cit.}~$\Gamma_{E_\pp/F_\qq}$ is trivial but the same argument extends.) When $\qq\nmid \ell$, one can alternatively appeal to Grothendieck's $\ell$-adic monodromy theorem to construct $\tu{WD}(\rho|_{\Gamma_{F_\qq}})$ directly (without going through general linear groups).

A local $L$-parameter $\phi:W_{F_\qq} \times \SL(2)\ra \hat{G}(\lql)\rtimes \Gamma_{E_\pp/F_\qq}$ is associated with a Frobenius-semisimple $^L G$-valued Weil--Deligne representation $(r,N)$ given by the following recipe:
$$
r(w) = \phi\left( w, \begin{pmatrix} |w|^{1/2} & 0 \\ 0 & |w|^{-1/2} \end{pmatrix} \right), \qquad \mbox{and}\qquad
N = \phi\left( 1,\begin{pmatrix} 0 & 1 \\ 0 & 0 \end{pmatrix} \right).
$$ 
This induces a bijection on the sets of equivalence classes of such objects \cite[Prop.~2.2]{GrossReederArithmeticInvariants}. In practice (where only equivalence classes matter), we will use them interchangeably.

We introduce some further notation and conventions in representation theory. If $\pi$ is a representation on a complex vector space then we set $\iota\pi:=\pi\otimes_{\C,\iota} \lql$. Similarly if $\phi$ is a local $L$-parameter of a connected reductive group $G$  over a nonarchimedean local field so that $\phi$ maps into $^L G(\C)$, then $\iota\phi$ is the parameter with values in $^L G(\lql)$ obtained from $\phi$ via $\iota$. If $G$ is a locally profinite group equipped with a Haar measure, then we write $\cH(G)$ for the \emph{Hecke algebra} of locally constant, complex valued functions with compact support. We write $\cH_{\lql}(G)$ for the same algebra, but now consisting of $\lql$-{va\-lued} functions. We normalize every parabolic induction by the half power of the modulus character as in \cite[1.8]{BZ77}, so that it preserves unitarity.

Let $G$ be a real reductive group, $K$ a maximal compact subgroup of $G(\R)$, and $\tilde K := K \cdot Z(G)(\R)$. Let $\xi$ be an irreducible algebraic representation of $G$ over $\C$. An irreducible admissible representation $\pi$ of $G(\R)$ is said to be \emph{$\xi$-cohomological} if $H^i(\Lie G(\C),\tilde K, \pi\otimes_\C \xi)\neq 0$ for some $i\ge 0$. If this is the case, we assign a Hodge cocharacter over $\C$ (well-defined up to $\hat G$-conjugacy) as in \cite[Def 1.14]{GSp}: 
$$
\mu_{\Hodge}(\xi) \colon \Gm \ra \hat G.
$$
Let $L$ be a finite extension of $\Q_{\ell}$. Let $H$ be a possibly disconnected reductive group over $\lql$ (e.g., an $L$-group relative to a finite Galois extension), and $\rho \colon \Gal(\li L/L)\ra H(\lql)$ a continuous morphism. If $\rho$ is Hodge--Tate with respect to each $\Q_{\ell}$-embedding $i \colon L \hra \lql$, we define a Hodge--Tate cocharacter over $\lql$ (well-defined up to $H$-conjugacy) as in \cite[\S2.4]{BuzzardGee} (cf.~\cite[Def 1.10]{GSp}): 
$$
\mu_{\HT}(\rho,i):\G_m \ra H.
$$

We recall the following lemma that can be easily deduced from the Chebotarev density theorem, as it will be needed in \S\ref{sect:Construction}. Let $F$ be a number field. The \emph{density} of a set $S$ consisting of primes of $F$ is defined to be the limit $d(S) = \lim_{n \to \infty} a_n(S) / a_n(F)$, where $a_n(F)$ is the number of primes $\qq$ with bounded norm $\|\qq\| < n$ and $a_n(S)$ is the number of $\qq \in S$ with $\|\qq\| < n$ \cite[Sect.~I.2.2]{SerreAbelianGalois}. Depending on $S$, the limit $d(S)$ may or may not exist~---~in the former case, we say $S$ has density $d(S)$, and otherwise we leave the density undefined.

\begin{lemma}\label{lem:Chebotarev}
Let $S$ be a finite set of places of a number field $F$. Let $G/\lql$ be a linear algebraic group and let $r \colon \Gamma_{F,S} \to G(\lql)$ be a Galois representation with Zariski dense image. Let $X \subset G$ be a closed  subvariety that is invariant by $G$-conjugation and such that $\dim(X) < \dim(G)$. Then the set of $F$-places $\qq \notin S$ with $r(\Frob_\qq) \in X(\lql)$ has density $0$.
\end{lemma}
\begin{proof}
Let $\mu$ be the Haar measure on $\Gamma_{S} = \Gamma_{F,S}$ with total volume $1$. We write $X$ to also mean $X(\lql)$ to simplify notation. Then $Y = r^{-1}(X)$ is a closed subset of $\Gamma_S$ (hence measurable) and stable under $\Gamma_S$-conjugation.  If we further have that $\mu(Y)=0$, then the Chebotarev density theorem \cite[I-8 Cor.~2b]{SerreAbelianGalois} implies that the set of places $\qq \notin S$ such that $\Frob_{\qq} \in Y$ has measure $0$, so we will be done.

So it suffices to prove that $\mu(Y) = 0$. We induct on $\dim(X)\in \{0,1,\ldots,\dim(G)-1\}$. We may assume that $X$ is irreducible by induction. When $\dim X=0$ then $X$ is a point and the preimage $Y$ of $X$ is a torsor under $\ker(r)$. We then have $\vol(Y) = \vol(\ker(r)) = 0$, since $\ker(r)\subset \Gamma_S$ is a closed subgroup of infinite index by hypothesis. Now assume that the assertion is known whenever $\dim(X) < d$ and consider the case $\dim(X)=d< \dim(G)$. There exists an infinite sequence $\gamma_1, \gamma_2, \ldots \in \Gamma_S$ such that the subset $r(\gamma_i) X$ are mutually distinct. (If the choice were impossible after $i=r$, then  multiplication by $r(g)$ preserves  $\bigcup_{i<=r}  r(\gamma_i) X$  for every $g\in\Gamma_S$. This can't happen because $r$ has Zariski dense image, and the union has dimension $d < \dim(G)$.) Consider
$$
\Gamma_S \supset \bigcup_{i=1}^{\infty} \gamma_i r^{-1}(X).
$$
The volume of $\Gamma_S$ is finite. Each term on the right hand side is closed (so measurable), and the volumes of $\gamma_i r^{-1}(X)$ are all equal. We claim that their pairwise intersections have volume $0$. If this is true, then we deduce that $\vol (\gamma_i r^{-1}(X)) = \vol( r^{-1}(X) ) =0$, completing the proof.

It remains to verify the claim. Observe that the intersection
$$
(*) \quad\quad \gamma_i r^{-1}(X) \cap \gamma_j r^{-1}(X),  \quad\quad  i \neq j,
$$
maps into the intersection $r(\gamma_i)X \cap r(\gamma_j)X$ in $G$, which has dimension less than $d$, so indeed (*) has measure $0$ by induction hypothesis. This completes the proof.
\end{proof}

We also record a lemma on projective algebraic representations, which will be usefull later on.

\begin{lemma}\label{lem:AlgebraicProjRepsConjugate}
Let $G$ be a connected simply-connected semi-simple group over $\C$. Let $T \subset G$ be a maximal torus.  
Let $r_1, r_2 \colon G \to \PGL_N$ be two projective representations whose restrictions to $T$ are conjugate. Then $r_1, r_2$ are conjugate. 
\end{lemma}
\begin{proof}
We claim that any $r \colon G \to \PGL_N$ can be lifted to a representation $\wt r \colon G \to \SL_N$. Let $H := (G \times_{\PGL_N} \SL_N)^0$, then $f \colon H \to G$ is a central isogeny, and hence is an isomorphism as $G$ is simply connected \cite[Prop.~18.8]{MilneAlgebraicGroupsBookCambridge}. The composition $G \to H \to \SL_N$ is the desired lift.

After conjugating, we may assume that $r_1|_T = r_2|_T$. By the preceding paragraph, we can choose lifts $\wt r_i$ of $r_i$ for $i=1,2$. Define a morphism of varieties $\chi:G \to \SL_N$ by $\chi(g):=\wt r_1(g) \wt r_2(g)\inv$. The image of $\chi|_T$ lies in $\mu_N$ since $r_1|_T = r_2|_T$. Hence the image is trivial as $T$ is connected, that is, $\wt r_1|_T=\wt r_2|_T$. Hence $\wt r_1$, $\wt r_2$ are $\GL_N$-conjugate because the trace functions coincide on semisimple elements. It follows that $r_1$, $r_2$ are $\PGL_N$-conjugate.
\end{proof}

\section{Root data of $\GSO_{2n}$ and $\GSpin_{2n}$}\label{sect:RootDatum}

Let $\GO_{2n}/\Q$ be the algebraic group such that for all $\Q$-algebras $R$ we have
\begin{equation}\label{eq:DefinitionGroupGo}
\GO_{2n}(R) = \lbr g \in \GL_{2n}(R)\ \left| \
\exists\, \simil(g) \in R^\times \ \colon\ g\tra \cdot \grootvierkant \ {1_n}{1_n}\ \cdot g = \simil(g) \cdot \grootvierkant \ {1_n}{1_n}\
 \right. \ \rbr.
\end{equation}
(in the above formula $1_n$ is the $n\times n$ identity matrix.) The group $\GO_{2n}$ is disconnected; its neutral component $\GSO_{2n} \subset \GO_{2n}$ is defined by the condition $\det(g) = \simil(g)^n$. The groups $\GO_{2n}$, $\GSO_{2n}$ are split and defined by a quadratic form of signature $(n,n)$. An element $t$ of the diagonal torus $\TGSO \subset \GSO_{2n}$ is of the form
\begin{equation}\label{eq:MaximalTorusGSO}
t = \diag(t_i)_{i=1}^{2n} =  \diag(t_1, t_2, \ldots, t_n, t_0 t_1\inv, t_0 t_2\inv, \ldots, t_0 t_n\inv), \quad t_0:= \simil(t)
\end{equation}
hence $\TGSO \simeq \Gm^{n+1}$ by sending $t$ to $(t_0,t_1,\ldots,t_n)$. We identify $X^*(\TGSO) = \bigoplus_{i=0}^n \Z \cdot e_i$ and $X_*(\TGSO) = \bigoplus_{i=0}^n \Z \cdot e^*_i$ accordingly. We let $\BGSO$ be the Borel subgroup of $\GSO_{2n}$ of matrices of the form
\begin{equation}\label{eq:BorelDesc}
g = \grootvierkant A{AB}0{c A^{\textup{t}, -1}}, \quad A \in B_{\GL_n}, \ B \in \textup{M}_n, \ B^\textup{t} = -B \quand c = \simil(g),
\end{equation}
where $B_{\GL_n} \subset \GL_n$ is the upper triangular Borel subgroup. (To see that $\BGSO$ is indeed a Borel subgroup, notice that any block matrix $g = \vierkant ABCD$ with $C = 0$ is of the above form if and only if $g \in \GSO_{2n}$, and moreover the displayed group is solvable of dimension $n^2 + 1$).

We realize the split forms of even (special) orthogonal groups in $\GO_{2n}/\Q$. Namely we write $\OO_{2n}$ (resp.~$\SO_{2n}$) for the subgroup of $\GO_{2n}$ (resp.~$\GSO_{2n}$) where $\simil$ is trivial.

\begin{lemma}\label{lem:GSORoots}
The root datum of $\GSO_{2n}$ with respect to $\BGSO$ is described as follows. 
\begin{enumerate}[label=(\roman*)]
\item The set of roots (resp.~coroots) consists of $\pm(e_i - e_j)$ and $\pm(e_i+e_j-e_0)$ (resp.~$\pm(e_i^*-e_j^*)$ and $\pm(e_i^*+e_j^*)$) with $1\leq i<j\leq n$.
\item The positive roots are $\{e_i + e_j - e_0\}_{1\le i < j \leq n} \cup \{e_i - e_j\}_{1\le i<j\leq n}$ and the positive coroots are $\{e^*_i \pm e^*_j\}_{1\le i < j \leq n}$.
\item The simple roots are $\alpha_1 = e_1 - e_2$, $\ldots$, $\alpha_{n-1} = e_{n-1} - e_n$, and $\alpha_n = e_{n-1} + e_n - e_0$.
\item The simple coroots $\Delta^\vee$ are $\alpha_1^\vee = e_1^* - e_2^*$, $\alpha_2^\vee = e_2^* - e_3^*$, $\ldots$, $\alpha_{n-1}^\vee = e_{n-1}^* - e_n^*$, and $\alpha_n^\vee = e_{n-1}^* + e_n^*$.
\end{enumerate}
\end{lemma}

\begin{remark}
The root datum of $\SO_{2n}$ is described similarly. Putting $\TSO := \TGSO\cap \SO_{2n}$ and $\BSO := \BGSO \cap \SO_{2n}$, we have $\TSO = \{t\in \TGSO : t_0=1\}$ as well as $X^*(\TSO)=\oplus_{i=1}^n e_i\cdot \Z$ and $X_*(\TSO)=\oplus_{i=1}^n e_i^* \cdot \Z$. To describe (positive or simple) roots and coroots, we only need to formally set $e_0=0$ in the lemma above.
\end{remark}

\begin{proof}
The standard computation for $\SO_{2n}$ as in \cite[18.1]{FultonHarris} can be easily adapted to $\GSO_{2n}$.
\end{proof}

We define the following element (over any $\Q$-algebra point of $\tu{O}_{2n}$)\footnote{The minus sign for $\vartheta^\circ$ makes it compatible with $\vartheta\in \GSpin_{2n}$ to be introduced above Lemma \ref{lem:conjugation-by-w}.}
\begin{equation}\label{eq:Elementw}
\vartheta^\circ := -
\lhk
 \begin{array}{cc;{1pt/1pt}cc}
1_{n-1} &   &           & \\
        & 0 &           & 1 \\ \hdashline[1pt/1pt]
        &   &  1_{n-1}  & \\
        & 1&           & 0
        \end{array}
\rhk\in \tu{O}_{2n}.
\end{equation}
Since $\det(\vartheta^\circ) = -1$
we have $\vartheta^\circ \notin \SO_{2n}$. We write $\theta^\circ \in \Aut(\GSO_{2n})$ for the automorphism given by $\vartheta^\circ$-conjugation.

\begin{lemma}\label{lem:Elem_w}
The automorphism $\theta^\circ$ stabilizes $\BGSO$ and $\TGSO$, and acts on $\TGSO$ by
$$
(t_0, t_1, \ldots, t_n) \mapsto (t_0, t_1, \ldots, t_{n-1}, t_0t_n\inv).
$$
Furthermore $\theta^\circ(\alpha_i)= \alpha_i$ for $i < n-2$, $\theta^\circ(\alpha_{n-1}) = \alpha_n$, and $\theta^\circ(\alpha_n) = \alpha_{n-1}$.
\end{lemma}
\begin{proof}
By a direct computation, $\theta^\circ( \TGSO)  = \TGSO$ and $\theta^\circ(\BGSO) = \BGSO$. Since $\theta^\circ$ only switches $t_n$ and $t_{2n}=t_0 t_n^{-1}$, its action on $\TGSO$ is explicitly described as in the lemma. Thus $\theta^\circ(e_i)=e_i$ for $1\le i\le n-1$ and $\theta^\circ(e_n)=e_0-e_n$, from which the last assertion follows.
\end{proof}

We define $\GSpin_{2n}$ to be the Langlands dual group $\hat{\GSO_{2n}}$ over $\C$ (or later over $\lql$ via $\iota:\C\simeq \lql$). That is, $\GSpin_{2n}$ is the connected reductive group over $\C$, equipped with a Borel subgroup $B_{\GSpin}$ and a maximal torus $\TGSpin$, whose based root datum is dual to the one of $\GSO_{2n}$ that we described above. In particular
$$
X_*(\TGSpin)=X^*(\TGSO) \quand X^*(\TGSpin)=X_*(\TGSO).
$$
Via the identification $X^*(\TGSO) = \Z^{n+1}$, we represent elements $s \in \TGSpin$ as $(s_0, s_1, \ldots, s_n)$. In Section \ref{sect:CliffordGroups} we will also define an explicit model of $\GSpin_{2n}$ over $\Q$ using Clifford algebras.

\begin{lemma}\label{lem:ThetaGSpin}
There is a unique $\theta \in \Aut(\GSpin_{2n})$ that fixes $\TGSpin$ and $\BGSpin$, switches $\alpha_{n-1}^\vee$ and $\alpha_n^\vee$, leaves the other $\alpha_i^\vee$ invariant, and induces the trivial automorphism of the cocenter of $\GSpin_{2n}$. We have $\theta^2 = 1$, and on the torus $\TGspin$ the involution $\theta$ is given by
\begin{equation}\label{eq:ThetaActionTGSpin}
(s_0, s_1, \ldots, s_n) \mapsto (s_0s_n, s_1, \ldots, s_{n-1}, s_n\inv).
\end{equation}
\end{lemma}
\begin{proof}
We have
$\theta(e_i^* - e_{i+1}^*) = e_i^* - e_{i+1}^*$ ($1 \leq i < n$) and $\theta(e_{n-1}^* - e_n^*) = e_{n-1}^* + e_n^*$.
Thus
 \begin{equation}\label{eq:thetaei}
 \theta(e_{i}^*) = e_{i}^* \ \  (1 \leq i < n) \quand \theta(e_n^*) = -e_n^*.
 \end{equation}
The center of $\GSO_{2n}$ is the image of
$
\Gm \owns z \mapsto (z^2, z, \ldots, z) \in \TGSO.
$
The dual map is
 \begin{equation}\label{eq:TGSpin-to-Gm}
\TGSpin \to \Gm, \quad (s_0, s_1, \ldots, s_n) \mapsto s_0^2 s_1 \cdots s_n.
 \end{equation}
Thus $\theta(2e_0^* + e_1^* + \cdots + e_n^*) = 2e_0^* + e_1^* + \cdots + e_n^*$, so $\theta(2e_0^*)  -e_n^* = 2e_0^* + e_n^*$ and $\theta(e_0^*) = e_0^* + e_n^*$.
\end{proof}

\begin{lemma}\label{lem:ComputeCenter}
We have $\ZGSpin=\{(s_0,\ldots,s_n):s_1=s_2=\cdots=s_n\in \{\pm1\}\}$, which is isomorphic to $\Gm \times \{\pm1\}$ via $(s_0,\ldots,s_n)\mapsto (s_0,s_1)$. In the latter coordinate, $\theta(s_0,s_1) = (s_0s_1,s_1)$. 
\end{lemma}
\begin{proof}
Let $s \in \TGSpin$. Then $s \in \ZGSpin$ if and only if $\alpha^\vee(t) = 1$ for all $\alpha^\vee \in \Delta^\vee$. From Lemma~\ref{lem:GSORoots}(\textit{iii}) we obtain
$s_i/s_{i+1} = 1$ ($i \leq n-1$), and $s_{n-1}s_n = 1$.
Hence $s \in \ZGSpin$ if and only if $s_1 = \cdots = s_n \in \{\pm 1\}$.
By \eqref{eq:ThetaActionTGSpin} we get $\theta(s_0,s_1) = (s_0s_1,s_1)$.
\end{proof}

The Weyl group of $\GSO_{2n}$ (and $\GSpin_{2n}$) is equal to $\{\pm 1\}^{n, \prime} \rtimes \iS_n$, where $\{\pm 1\}^{n,\prime}$ is the group of $a \in \{\pm 1\}^{n}$ such that $\prod_1^n a(i) = 1$. The action of $\WGSO$ on $\TGSO$ is determined by 
\begin{equation}\label{eq:WeylGroupAction}
\begin{cases}
\sigma \cdot (t_0, t_1, \ldots, t_n) = (t_0, t_{\sigma(1)}, \ldots t_{\sigma(n)}) & \sigma \in \iS_n \cr
a \cdot (t_0, t_1, \ldots, t_n) = (t_0, t_0 t_1^{-1}, t_0t_2\inv, t_3, \ldots, t_n) & a = (-1, -1, 1\ldots, 1) \in \{\pm 1\}^{n,\prime}.
\end{cases}
\end{equation}
We define, for $\eps \in \{\pm1\}$ the following cocharacter
\begin{equation}\label{eq:Spin-eps-def}
\mu_\eps := \begin{cases}
(1, 1, \ldots, 1, 1) & \tu{if } \eps = (-1)^n \cr
(1, 1, \ldots, 1, 0) & \tu{if } \eps =  (-1)^{n+1}
\end{cases}
\quad
\in \Z^{n+1} =  X_*(\TGSO) = X^*(\TGSpin).
\end{equation}
Then $\mu_\eps$ is a minuscule cocharacter of $\GSO_{2n}$ with $\langle \alpha_i, \mu_\eps \rangle = 1$ if and only if $i = n$ (for $\eps = (-1)^n$) and $i = n-1$ (for $\eps = (-1)^{n+1}$).

\begin{definition}\label{def:HalfSpinDef}
For $\eps\in \{+,-\}$, define the \emph{half-spin representation} $\spin^\eps = \spin^\eps_{2n}$ to be the irreducible representation of $\GSpin_{2n}$ whose highest weight is equal to $\mu_\eps$ in $X^*(\TGSpin)$. By the \emph{spin representation} of $\GSpin_{2n}$ we mean $\spin := \spin^+ \oplus \spin^-$.
\end{definition}

These representations will be realized explicitly via Clifford algebras. Our sign convention is natural in that $\spin^+$ (resp.~$\spin^-$) accounts for even (resp.~odd) degree elements. See \eqref{eq:CliffordHalfSpinDef} and Lemma \ref{lem:half-spin-highest-weight} below.

The minuscule  $\mu_\eps$ has $2^{n-1}$ translates under the Weyl group action. Thus each half-spin representation has dimension $2^{n-1}$. More precisely the weights of $\spin^{(-1)^n}_{2n}$ are
\begin{equation}\label{eq:spin-highest-weights}
\TGSpin \owns (s_0, s_1, \ldots, s_n) \mapsto \lhk s_0 \prod_{i \in U} s_i \rhk_{U \subset \{1, 2, \ldots, n\},\atop  2|\# U }
\in \Z^{n+1} =  X_*(\TGSO) = X^*(\TGSpin)
\end{equation}
and $\spin^{(-1)^{n+1}}_{2n}$ has similar weights, except that the cardinality of $U$ is now required to be odd. By computing the $\theta$-action on weights, we verify that (see Lemma \ref{lem:CliffordSpinThetaAction} for an explicit intertwiner)
$$
\spin^+ \circ \theta \simeq \spin^-\quad\mbox{and}\quad~\spin^- \circ \theta \simeq \spin^+.
$$

\begin{lemma}\label{lem:spin-kernel}
The kernel $Z^\eps$ of $\spin^\eps$ is central in $\GSpin_{2n}$ and finite of order $2$. The non-trivial element $z_\eps$ of $Z^\eps$ equals $(\eps, -1) \in \Gm \times \{\pm 1\}$. The spin representation of $\GSpin_{2n}$ is faithful.
\end{lemma}
\begin{proof}
Since $\GSpin_{2n}$ is simple modulo the center, the kernel $Z^\eps \subset \GSpin_{2n}$ must be central. The central character is the restriction of $\mu_\eps \colon \TGSpin \to \Gm$ to the center $\ZGSpin \subset \TGSpin$. Let $s = (s_0, s_1, \ldots, s_n) = (a,b) \in \ZGSpin \subset \TGSpin$. Then (see proof of Lemma~\ref{lem:ComputeCenter})
\begin{equation}\label{eq:CentralChar}
\mu_\eps(s) = \begin{cases}
s_0 s_1 \cdots s_n = ab^n& \tu{if } \eps = (-1)^n \cr
s_0 s_1 \cdots s_{n-1} = ab^{n-1}& \tu{if } \eps = (-1)^{n+1}.
\end{cases}
\end{equation}
The first assertion follows by considering the 4 different cases where $n$ even or odd and $\eps = \pm1$. For the second point, it suffices to observe that $Z^+ \cap Z^- = \{1\}$.
\end{proof}

Later on the following fact on $\SO_{2n-1}$ will be needed, so we record it here.

\begin{lemma}\label{lem:RepsOfDim2n}
Let $n \geq 3$. Up to isomorphism the group $\SO_{2n-1}$ has exactly one faithful representation of dimension $2n$, namely $\std_{2n-1} \oplus \one$.
\end{lemma}
\begin{proof}
We use the root system notation and conventions from \cite[Ch.~4, p.~253]{BourbakiLie}. Assume $V_\lambda$ is a non-trivial irreducible representation of $\SO_{2n-1}$ with highest weight $\lambda \neq \omega_1$. We show that $\dim(V_\lambda) > 2n$. Write $\lambda = \sum_{i=1}^{n-1}x_i \omega_i$ with $x_i \geq 0$. If $x_i \neq 0$ for some $i$ with $n - 1 > i > 1$, then $\dim(V_\lambda) \geq \dim(V_{\omega_i})$, and $\dim(V_{\omega_i}) = \dim(\wedge^i \std) > 2n$. We thus assume $\lambda = x_1 \omega_1 + x_{n-1} \omega_{n-1}$. If $x_{n-1} = 0$, then, as $\lambda \neq \omega_1$, we have $\dim(V_\lambda) \geq \dim(V_{2 \omega_1}) = (n-1)(2n+1) > 2n$ by the Weyl dimension formula. Assume $x_{n-1} \neq 0$. We can't have $x_{n-1} = 1$, because then $V_\lambda$  does not descend to $\SO_{2n-1}$. Thus  $\dim(V_\lambda) \geq \dim(V_{2 \omega_{n-1}})$ which equals $10$ if $n= 3$, $35$ if $n = 4$, and if $n > 4$ then $\dim(V_{2\omega_{n-1}}) \geq \dim(V_{\omega_{n-1}}) = 2^{n-1} > 2n$.
\end{proof}

\section{Clifford algebras and Clifford groups}\label{sect:CliffordGroups}

We recall how $\GSpin_{2n}$ is realized using the Clifford algebra, and define a number of fundamental maps such as $i_{\std} \colon \GSpin_{2n-1} \hookrightarrow \GSpin_{2n}$ and the projections from $\GSpin_{2n}$ to $\GSO_{2n}$ and $\SO_{2n}$. We also give a concrete definition of outer automorphisms $\theta$ of $\GSpin_{2n}$ and $\theta^\circ$ of $\GSO_{2n}$. Our main reference is \cite{BassCliffordAlgebra}, which introduces Clifford algebras over arbitrary commutative rings (with unity). Other useful references are \cite[\S9]{BourbakiAlgebreChapter9} and \cite[\S20]{FultonHarris}.

Let $V$ be a quadratic space over $\Q$ with quadratic form $Q$, giving rise to the groups $\uO(V)$, $\GO(V)$, $\SO(V)$ and $\GSO(V)$. The \emph{Clifford algebra} $C(V)$ is a universal map $V \to C(V)$ which is initial in the category of $\Q$-linear maps $f \colon V\to A$ into associative $\Q$-algebras $A$ with unity $1_A$ such that $f(v)^2=Q(v)\cdot 1_A$ for all $v\in V$. (See \cite[(2.3)]{BassCliffordAlgebra} or \cite[\S9.1]{BourbakiAlgebreChapter9}.)

We define $\langle x, y \rangle := Q(x+y) - Q(x) - Q(y)$ for $x,y \in V$, and similarly $\langle x, y \rangle = (x+y)^2 - x^2 - y^2$ for $x,y \in C(V)$. In particular $\langle x, y \rangle$ measures if $x$ and $y$ anti-commute in $C(V)$:
\begin{equation}\label{eq:betaInvolution}
\langle x, y \rangle = (x+y)^2 - x^2 - y^2 = xy + yx \in C(V).
\end{equation}

The map $V\to C(V)$ induces a map $V \to C(V)^{\tu{opp}}$ (sending each $v\in V$ to the same element), where $C(V)^{\tu{opp}}$ is the opposite algebra. The latter factors through a unique $\Q$-algebra map $\beta \colon C(V) \to C(V)^{\tu{opp}}$. It is readily checked that $\beta^2$ is the identity on $C(V)$. By the universal property $\beta$ is the unique involution of $C(V)$ that is the identity on $V$.

The universal property also yields a surjection from the tensor algebra
$$
\bigoplus_{d \in \Z_{\geq 0}} V^{\otimes d} \surjects C(V).
$$
Define $C^+ = C(V)^+$ (resp.~$C^- = C(V)^-$) to be the image of $\oplus_{d \in \Z_{\geq 0}} V^{\otimes 2d}$ (resp.~$\oplus_{d \in \Z_{\geq 0}} V^{\otimes 2d+1}$) so that $C(V) = C(V)^+ \oplus C(V)^-$. In fact the discussion of Clifford algebras so far works when $V$ is replaced with a quadratic space on a module over an arbitrary commutative ring, in a way compatible with base change: in particular if $R$ is a (commutative) $\Q$-algebra then $C(V\otimes_\Q R)=C(V)\otimes_\Q R$ \cite[\S9.1, Prop 2]{BourbakiAlgebreChapter9}. By scalars in $C(V\otimes_\Q R)$ we mean $R$ times the multiplicative unity. We keep using $\beta$ to denote the main involution of $C(V\otimes_\Q R)$.

The \emph{Clifford group} $\tu{GPin}(V)$ is the $\Q$-group such that for every $\Q$-algebra $R$,
$$
\tu{GPin}(V)(R)=\{ x\in C(V\otimes_\Q R)^\times : x(V\otimes_\Q R) x^{-1} = V\otimes_\Q R, \textup{ $x$ is homogeneous}\},
$$
where homogeneity of $x$ means that $x \in C(V\otimes_\Q R)^\eps$ for some sign $\eps$. The \emph{special Clifford group} $\GSpin(V)$ is defined similarly with $C^+$ in place of $C$. The embedding of invertible scalars in $C(V\otimes_\Q R)$ induces a central embedding
\begin{equation}\label{eq:central-embed-GSpin}
  \G_m\ra \GSpin(V).
\end{equation}

Since $x\beta(x)\in R$ for $x\in C(V\otimes_\Q R)$ by \cite[Prop 3.2.1 (a)]{BassCliffordAlgebra}, we have the \emph{spinor norm} morphism
$$
\cN \colon \tu{GPin}(V)\ra \Gm,\qquad x\mapsto x\beta(x)
$$
over $\Q$. (The involution in \emph{loc.~cit.}~differs from our $\beta$ by $C(-1_P)$ in their notation, so our $\cN$ does not coincide with their $N$, but $\cN$ and $N$ have the same kernel.) Evidently, composing $\cN$ with \eqref{eq:central-embed-GSpin} yields the squaring map.

Define $\Spin(V)$ by the following exact sequence of algebraic groups:
$$
1\ra \Spin(V)\ra \GSpin(V) \stackrel{\cN}{\ra} \G_m\ra 1.
$$

\begin{lemma}\label{lem:SurjectionOntoGSO} The following are true, where kernels and surjectivity are always meant in the category of algebraic groups over $\Q$.
\begin{enumerate}[label=(\roman*)]
\item The map
$\pr^\circ=\pr^\circ_V \colon \tu{GPin}(V)\to \tu{O}(V),~ x \mapsto ( v\mapsto x v x^{-1})$
is surjective for $\dim V$ even, and $\pr^\circ \colon \GPin(V) \to \SO(V)$ is surjective when $\dim V$ is odd. 
\item We have $\ker(\pr^\circ) = \Gm$ via \eqref{eq:central-embed-GSpin}.
\item $\pr \colon \GPin(V) \to \GO(V),~  x \mapsto (v \mapsto x v \beta(x))$ is a surjection, and $\tu{sim} \circ \pr = \cN^2$.
 \item The map $\tu{pr}$ factors as $\GPin(V)\stackrel{(\tu{pr}^\circ,\cN)}{\longrightarrow} \OO(V)\times \GL_1 \stackrel{\tu{mult.}}{\longrightarrow} \GO(V)$, where the latter is the multiplication map. The map $(\tu{pr}^\circ,\cN)$ has kernel $\mu_2$ (scalars $\{\pm 1\}$ in $C(V)$) and image $\OO(V)\times \GL_1$ (resp.~$\SO(V)\times \GL_1$) for $n$ even (resp.~odd).
\item The multiplication map $\Spin(V)\times \G_m\ra \GSpin(V)$ is a surjection with kernel $\{\pm(1,1)\}$ (diagonally embedded $\mu_2$), where $\{\pm1\}\hra \Spin(V)$ via \eqref{eq:central-embed-GSpin}.
\end{enumerate}
\end{lemma}
\begin{proof}
(\textit{i}) The surjectivity can be checked on field-valued points. This is proved in \cite[\S9.5, Thm.~4]{BourbakiAlgebreChapter9}.

(\textit{ii}) As $V\subset C(V)$ generates the Clifford algebra, the identity $xvx^{-1}=v$ implies $xyx^{-1}=y$ for all $y\in C(V)$, and the analogue holds for $C(V\otimes_\Q R)$ for $\Q$-algebras $R$. Thus $\ker(\pr^\circ)(R)$ consists of invertible elements in the center of $C(V\otimes_\Q R)$. Let $W\subset V$ be an isotropic subspace. Then $C(V\otimes_\Q R)\simeq \End(\bigwedge (W\otimes_\Q R))$ as super $R$-algebras by \cite[(2.4) Thm.]{BassCliffordAlgebra}, so the center of $C(V\otimes_\Q R)$ is $R$, implying that $\ker(\pr^\circ) = \Gm$.

(\textit{iii}) We observe that $\pr(x)$ preserves $V$: as $x (V\otimes_\Q R) x^{-1} = V\otimes_\Q R$ and $x\beta(x)\in R^\times$ imply that $x(V\otimes_\Q R)\beta(x)=V\otimes_\Q R$. Moreover $\pr(x)\in \GO(V)$ as
\begin{equation}\label{eq:SpinorNormFactorSimilitude}
Q(xv\beta(x))=xv\beta(x) xv\beta(x)=\cN(x)^2 Q(v).
\end{equation}
Moreover $\pr$ and $\pr^\circ$ coincide on $\tu{Pin}(V)$, so $\tu{(S)O}(V)$ is in the image of $\pr$. On the other hand, $\cN$ is seen to be surjective by considering scalar elements, telling us that the image of $\pr$ also contains $\Gm$ (scalar matrices in $\tu{GO}(V)$). Since $\Gm$ and $\tu{(S)O}(V)$ generate $\tu{G(S)O}(V)$, the surjectivity of $\pr$ follows.
The equality $\tu{sim} \circ \pr = \cN^2$ follows from
 \eqref{eq:SpinorNormFactorSimilitude}.

(\textit{iv}) The first part follows from $\tu{pr}(x)(v)=xv \beta(x)=xvx^{-1}x\beta(x)=\tu{pr}^\circ(x)(v) \cN(x)$ when $x\in \GPin(V)$ and $v\in V$. The second part is easily seen from (i) and (ii).

(\textit{v}) This readily follows from the preceding points.
\end{proof}

If $V$ is odd dimensional then $\SO(V) \times \{\pm 1\} = \uO(V)$, and the group $\GO(V)$ is connected. For convenience we define $\GSO(V) := \GO(V)$ in this case. If $\dim(V)$ is even, then $\OO(V)$ (resp.~$\GO(V)$) has two connected components but does not admit a direct product decomposition into $\uO(V)$ (resp.~$\GSO(V)$) and $\{\pm 1\}$.

Assume that we have an orthogonal sum decomposition $\varphi \colon W_1 \oplus W_2 \isomto V$ of non-degenerate quadratic spaces over $\Q$. As super algebras we have (\!\!\cite[(2.3)]{BassCliffordAlgebra} or \cite[\S9.3, Cor.~3, Cor.~4]{BourbakiAlgebreChapter9})
$$
C_\varphi \colon C(W_1) \wo C(W_2) \isomto C(V), \quad w_1 \wo w_2 \mapsto w_1 w_2.
$$
By definition, the algebra given by $\wo$ on the left side has underlying vector space $C(W_1) \otimes C(W_2)$ and product
$$
(a \wo b) \cdot (c \wo d) := (-1)^{k_b \cdot k_c} ac \wo bd,
$$
if $a, c \in C(W_1)$, $b, d \in C(W_2)$ are homogeneous elements of degree $k_a, k_b, k_c, k_d \in \Z/2\Z$.
The sign is there to make $C_\varphi$ compatible with products since $bc=(-1)^{k_b k_c}cb$ in $C(V)$.

In fact $C_\varphi$ intertwines the involution $\beta$ on $C(V)$ with the involution
$$
\beta' \colon C(W_1) \wotimes C(W_2) \to C(W_1) \wotimes C(W_2),\quad
\beta'(a \wo b) = (-1)^{k_a k_b} \beta_1(a) \wo \beta_2(b),
$$
for homogeneous elements $a\in C(W_1)$, $b\in C(W_2)$ of degree $k_a, k_b \in \Z/2\Z$, where $\beta_1,\beta_2$ are the involutions of $C(W_1)$ and $C(W_2)$ (see below \eqref{eq:betaInvolution}). To verify that $\beta$ is compatible with $\beta'$, observe that $\beta$ on $C(V)$ restricts to $\beta_1,\beta_2$ via the obvious inclusions $C(W_1)\hra C(V)$ and $C(W_2)\hra C(V)$ induced by $W_1\subset V$ and $W_2\subset V$ (since $\beta$ acts as the identity on both $W_1$ and $W_2$), and use the property that $\beta_1$, $\beta_2$, and $\beta$ are preserving degrees. It follows that
$$
\beta(ab)= \beta(b)\beta(a)= (-1)^{k_a k_b}\beta(a)\beta(b)=(-1)^{k_a k_b}\beta_1(a)\beta_2(b).
$$

\begin{lemma}\label{lem:CliffordMapping}
The mapping $C_\varphi$ induces a morphism $\GSpin(W_1) \times \GSpin(W_2) \to \GSpin(V)$.
\end{lemma}
\begin{proof}
We check that the image of $C_\varphi$ is in $\GSpin(V)$. Let $g \in \GSpin(W_1)$, $h \in \GSpin(W_2)$. Note that $C_\varphi(g \wh \otimes h)=gh \in C^+(V)$. Let $w_1 + w_2 \in V$ with $w_i \in W_i$, $i=1,2$. To verify that $gh \in \GSpin(V)$, since homogeneous elements of even degree commute with each other if they are perpendicular, we see that
$$
gh(w_1+w_2)h^{-1}g^{-1}
= gw_1g^{-1}+ hw_2h^{-1}\in V.
$$
\end{proof}

\begin{lemma}\label{lem:CliffordMapping2}
The diagram
$$
\xymatrix{
\GSpin(W_1) \times \GSpin(W_2)\quad \ar@{->>}[d]_-{\pr_{W_1}^\circ \times \pr_{W_2}^\circ} \ar[r]^-{C_\varphi} &\quad \GSpin(V)  \ar@{->>}[d]^-{\pr_V^\circ}
\cr
\SO(W_1) \times \SO(W_2)\ar[r]^-{i_{W_1, W_2}} & \SO(V)
}
$$
commutes, where $i_{W_1, W_2}$ is the block diagonal embedding. 
\end{lemma}
\begin{proof}
Immediate from the computation in the proof of the preceding lemma.
\end{proof}

In later chapters we will carry out explicit computations. It will then be convenient to work with fixed bases and quadratic forms. For this reason we now fix quadratic forms on the vector spaces $V_{2n}=\C^{2n}$ and $V_{2n-1}=\C^{2n-1}$. We take the following quadratic forms:
\begin{align}\label{eq:StdQuadSpace}
Q_{2n} \colon x_1 x_{n+1} + x_2 x_{n+2} + \ldots + x_n x_{2n} & \tu{ on $\C^{2n}$} \cr
Q_{2n-1} \colon y_1 y_{n+1} + \ldots + y_{n-2} y_{2n-2} + y_{2n-1}^2 & \tu{ on $\C^{2n-1}$.}
\end{align}
Using them, we write $\SO_{m}=\SO(V_{m})$, $\GSO_{m}=\GSO(V_{m})$, and likewise for $\OO_m$, $\GO_m$, for $m=2n$ and $m=2n-1$. This is identical to the convention of \S\ref{sect:RootDatum} for $m$ even. Similarly we write $\pr^\circ_{2n-1}=\pr^\circ_{V_{2n-1}}$ and $\pr^\circ_{2n}=\pr^\circ_{V_{2n}}$.

Now we claim that $\GSpin(V_{2n})$ is isomorphic to $\GSpin_{2n}$ of \S\ref{sect:forms-of-GSO} that is, the Clifford algebra definition is compatible with the root-theoretic definition as the Langlands dual of $\GSO_{2n}$. (An analogous argument shows that $\GSpin_{2n-1}$ is dual to $\GSp_{2n-2}$.) As this is a routine exercise, we only sketch the argument. First, $\pr^\circ$ restricts to a connected double covering $\Spin(V_{m})\ra \SO(V_{m})$ (\!\!\cite[Prop.~20.38]{FultonHarris}), which must then be the unique (up to isomorphism) simply connected covering. This determines the root datum of $\Spin(V_{m})$. From this, we compute the root datum of $\GSpin(V_m)$ via the central isogeny $\Spin(V_{m})\times \G_m\ra \GSpin(V_m)$ of Lemma \ref{lem:SurjectionOntoGSO}. Finally when $m=2n$, we deduce that the outcome is dual to the root datum of $\GSO_{2n}$ in Lemma \ref{lem:GSORoots}. Therefore $\GSpin(V_{2n})$ is isomorphic to $\GSpin_{2n}$ of  \S\ref{sect:forms-of-GSO}. Henceforth we identify
\begin{equation}\label{eq:GSpin=GSpin}
\GSpin(V_{2n})=\GSpin_{2n}.
\end{equation}
In fact we may and will choose $\BGSpin$ and $\TGSpin$ to be the preimages of $\BSO$ and $\TSO$ via $\pr^\circ: \GSpin_{2n}\ra \SO_{2n}$. We fix pinnings of $\GSpin_{2n}$, $\GSO_{2n}$, and $\SO_{2n}$ (which are $\Gamma_F$-equivariant if $(V_{2n},Q_{2n})$ is defined over $F$) compatibly via $\pr$ and $\pr^\circ$.

\begin{lemma}\label{lem:dualizing-spinor-norm-and-similitude}
Via \eqref{eq:GSpin=GSpin}, the central embedding of scalar matrices $\tu{cent}^\circ:\G_m\ra \GSO_{2n}$ and $\simil:\GSO_{2n}\ra \G_m$ are dual to $\cN:\GSpin_{2n}\ra \G_m$ and the central embedding $\tu{cent}:\G_m\ra \GSpin_{2n}$ of \eqref{eq:central-embed-GSpin}, respectively.
\end{lemma}

\begin{remark}\label{rem:explicit-spinor-norm}
The dual map of $\tu{cent}^\circ$ was made explicit in \eqref{eq:TGSpin-to-Gm}. According to the present lemma, \eqref{eq:TGSpin-to-Gm} gives an explicit formula for $\cN$ restricted to $T_{\GSpin}$.
\end{remark}

\begin{proof}
Write $Z^0$ for the identity component of the center of $\GSpin_{2n}$, consisting of $(s_0,1,\ldots,1)$ with $s_0\in \G_m$ in the notation of Lemma \ref{lem:ComputeCenter}. The dual of $\simil:\GSO_{2n}\ra \G_m$ is calculated as the central cocharacter $\G_m\ra Z^0\subset \GSpin_{2n}$, $z\mapsto (z,1,\ldots,1)$. The inclusion $\tu{cent}:\G_m\ra \GSpin_{2n}$ identifies $\G_m$ with $Z^0$. Thus \tu{cent} is dual to $\simil$.

Both $\cN\circ\tu{cent}$ and $\simil\circ \tu{cent}^\circ$ are the squaring map on $\G_m$. Using the hat symbol to denote a dual morphism, we see that 
$$
\cN\circ\tu{cent}=\hat{\tu{cent}^\circ}\circ \hat{\simil} = \hat{\tu{cent}^\circ}\circ \tu{cent}
$$
and that they are all equal to the squaring map. It follows that $\cN$ is dual to $\tu{cent}^\circ$.
\end{proof}

We have the morphism of quadratic spaces
$$
\varphi \colon (\C^{2n-1}, Q_{2n-1}) \to (\C^{2n}, Q_{2n}), \quad y \mapsto (y_1, y_2, \ldots, y_{n-1}, y_{2n-1}, y_n, y_{n+1}, \ldots, y_{2n-1}).
$$
Indeed, $Q_{2n} \varphi = Q_{2n-1}$ as readily checked. We have the complementary embedding:
$$
\varphi' \colon \C \to \C^{2n}, \quad u \mapsto  x, \quad \textup{where} \quad
\begin{cases} x_k = 0 & k \neq n, 2n \cr
x_k = (-1)^{k/n} u & \textup{ if $k = n$ or $k = 2n$.}
\end{cases}
$$
Write $U := \varphi'(\C)=(e_n-e_{2n})\cdot \C$ for the image. The induced quadratic form on $U$ is then $a\cdot (e_n-e_{2n})\mapsto -a^2$. This gives us an orthogonal decomposition of quadratic spaces $\C^{2n} = \C^{2n-1} \wh \oplus U$. Let $\PO_m$ denote the adjoint group of $\OO_m$. The decomposition induces morphisms (cf.~Lemmas \ref{lem:CliffordMapping}, \ref{lem:CliffordMapping2})
\begin{align}\label{eq:std_emb_def}
i_{\std} &:= C_{\varphi, \varphi'} \colon \GSpin_{2n-1} \times \GSpin_1 \to \GSpin_{2n}, \cr
i_{\std}^\circ &:= i_{\C^{2n-1}, \C} \colon \OO_{2n-1} \times \OO_1 \to \OO_{2n},~~\mbox{and} \cr
\li {i_{\std}} &:= \PO_{2n-1} \to \PO_{2n},
\end{align}
where $\li {i_{\std}}$ is induced from $i_{\std} \colon \GSpin_{2n-1} \times \GSpin_1 \to \GSpin_{2n} \surjects \PSO_{2n}\subset \PO_{2n}$. By Lemma \ref{lem:CliffordMapping2}, we have $\tu{pr}^\circ \circ i_{\std}=i_{\std}^\circ\circ(\tu{pr}^\circ_{2n-1}\times \tu{pr}^\circ_{U})$.

Let $1_{2n-1}, 1_U$ denote the identity map on $\C^{2n-1},U$. Then (cf. \eqref{eq:Elementw})
$$
i_{\std}^\circ(-1_{2n-1}, 1_U) =
-\lhk
 \begin{array}{cc;{1pt/1pt}cc}
1_{n-1} &   &           & \\
        & 0 &           & 1 \\ \hdashline[1pt/1pt]
        &   &  1_{n-1}  & \\
        & 1&           & 0
        \end{array}
\rhk = \vartheta^\circ \in \tu{O}_{2n}.
$$
Fix $\sqrt{-1}\in \G_m=Z(\GPin_{2n})$. Define
\begin{equation}\label{eq:Elementw_spin}
\vartheta := \sqrt{-1}\cdot i_{\std}(1_{C(\C^{2n-1})} \wo (e_n-e_{2n})) = \sqrt{-1}(e_n - e_{2n}) \in \GPin_{2n}\backslash \GSpin_{2n}.
\end{equation}

\begin{lemma}\label{lem:conjugation-by-w}
We have
\begin{enumerate}[label=(\roman*)]
\item
$\pr^\circ_{2n}(\vartheta) = \vartheta^\circ$ and $\vartheta^2 = 1$.
\item The conjugation action of $\vartheta$ (resp.~$\vartheta^\circ$) fixes the subgroup $i_{\std}(\GSpin_{2n-1}\times \GSpin_1) \subset \GSpin_{2n}$ via $i_{\std}$ (resp.~ $\SO_{2n-1}\times \SO_1 \subset \SO_{2n}$ via $i^\circ_{\std}$) and induces the identity automorphism on that subgroup.
\item The conjugation action of $\vartheta$ (resp.~$\vartheta^\circ$) defines the outer automorphism $\theta$ of $\GSpin_{2n}$ (resp.~$\theta^\circ$ of $\GSO_{2n}$) in Lemmas \ref{lem:Elem_w} and \ref{lem:ThetaGSpin}.
\end{enumerate}
\end{lemma}
\begin{proof}
(\textit{i}) Let $w_1\in \C^{2n-1}$ and $w_2:=e_n-e_{2n}\in U$. All of $w_1,w_2,\vartheta$ have degree 1 in $C(\C^{2n})$.
In either $C(\C^{2n})$ or $C(U)$, we have $w_2^2=Q_{2n}(w_2)=-1$ and $\vartheta^2=-w_2^2 = 1$.
Thus
$\vartheta w_1 \vartheta^{-1}=-w_1\vartheta \vartheta^{-1}=-w_1$ and $\vartheta w_2 \vartheta^{-1}= w_2$. Hence $\pr^\circ_{2n}(\vartheta) = \vartheta^\circ$.

(\textit{ii}) This is obvious for $\vartheta^\circ$. The conjugation by $\vartheta$ is the identity on $C^+(\C^{2n-1})$ and $C^+(U)$, since $\vartheta\perp \C^{2n-1}$ and $C^+(U)$ is commutative, respectively. The assertion for $\vartheta$ follows.

(\textit{iii}) This is true by definition for $\theta^\circ$. Since $\theta$ and the conjugation by $\vartheta$ act trivially on the center of $\GSpin_{2n}$, it suffices to check that their actions are identical on the adjoint group. This reduces to the fact that $\theta^\circ$ is given by the $\vartheta^\circ$-conjugation, as $\theta$ and $\theta^\circ$ (resp.~ $\vartheta$ and $\vartheta^\circ$) induce the same action on the adjoint group (thanks to part (i)).
\end{proof}

We have fixed pinnings of $\GSpin_{2n}$, $\GSO_{2n}$, and $\SO_{2n}$ compatibly via $\pr$. They are fixed by $\theta\in \Aut(\GSpin_{2n})$ and $\theta^\circ\in \Aut(\GSO_{2n})$. It is easy to see that $\theta$ and $\theta^\circ$ induce automorphisms of based root data, which correspond to each other via duality of the two based root data. Thus letting $E/F$ be a quadratic extension of fields of characteristic 0, and $\GSO_{2n}^{E/F}$ an outer form of $\GSO_{2n}$ over $F$ with respect to the Galois action $\Gamma_{E/F}=\{1,c\}\isom \{1,\theta\}$, we can identify
$$
^L (\GSO_{2n}^{E/F}) = \GSpin_{2n}\rtimes \{1,c\} = \GPin_{2n},
$$
where the semi-direct product is given by $c g c^{-1} = \theta(g)$. (Of course $c=c^{-1}$.) The second identification above is via $c\mapsto \vartheta$.
Similarly, for $\SO_{2n}^{E/F}$ an outer form of $\SO_{2n}$ with respect to $\Gamma_{E/F}=\{1,c\}\isom \{1,\theta^\circ\}$, we have
$$
^L (\SO_{2n}^{E/F}) = \SO_{2n}\rtimes \{1,c\} = \OO_{2n}\qquad\mbox{via}\quad c\mapsto \vartheta^\circ.
$$

Let us describe the center $\Zspin$ of $\Spin_{2n}=\Spin(V_{2n})$ explicitly as this is going to be useful for classifying inner twists of (quasi-split forms of) $\SO_{2n}$ and $\GSO_{2n}$ in \S\ref{sect:forms-of-GSO}. In what follows, we identify $\ZGSpin=\{(s_0,s_1)\,:\,s_0\in \G_m,~s_1\in \{\pm1\}\}$
as in Lemma \ref{lem:ComputeCenter} and write $1,-1$ for $(1,1),(1,-1)\in \ZGSpin$.

\begin{lemma}\label{lem:spin-center2}
Let $\zeta_4$ be a primitive fourth root of unity. Recall the elements $z_{\pm}$ defined in Lemma \ref{lem:spin-kernel}. Then we have $\Zspin \subset \ZGSpin$ via $\TSpin\subset T_{\GSpin}$ explained above, and the following are true.
\begin{enumerate}[label=(\textit{\roman*})]
\item If $n $ is even, $\Zspin=\{1,-1,z_+,z_-\}$ and is isomorphic to $(\Z/2\Z)^2$.
If $n$ is odd, $\Zspin=\{1,-1,\zeta,-\zeta=\zeta^{-1}\}$ and is isomorphic to $\Z/4\Z$, where $\zeta=(\zeta_4,-1)$.
\item The action of $\theta$ is trivial on $\{1,-1\}$ and permuting $\{z_+,z_-\}$ (resp.~$\{\zeta,-\zeta\}$).
\end{enumerate}
\end{lemma}
\begin{proof}
We have $\Zspin=\ZGSpin \cap \Spin_{2n}= \{z\in \ZGSpin: \cN(z)=1\}$, where $\cN$ is described by \eqref{eq:TGSpin-to-Gm} (Remark \ref{rem:explicit-spinor-norm}). It follows from Lemma \ref{lem:ComputeCenter} that
$$\Zspin=\{(s_0,s_1): s_0^2 = s_1^n\},$$
which is alternatively described as in (i). Assertion (ii) is also clear from that lemma.
\end{proof}

\section{The spin representations}\label{sect:HalfSpinReps}

We recollect how to construct the spin representations via Clifford algebras, and show that they coincide with the highest weight representations in Section \ref{sect:RootDatum}. We also check some compatibility of maps that will become handy.

Consider the quadratic space $V_{2n} := \C^{2n}$ from \eqref{eq:StdQuadSpace} with standard basis $\{e_1,..,e_{2n}\}$ and quadratic form $Q_{2n}$. Define $W_{2n} := \oplus_{i=1}^n \C e_i$ and $W'_{2n} := \oplus_{i=n+1}^{2n} \C e_i$. We often omit the subscript $2n$ to lighten notation, when there is no danger of confusion. Since $W$ is isotropic we obtain a morphism $\bigwedge W \isomto C(W) \hookrightarrow C(V)$. Through this injection we view $\bigwedge W$ as a subspace of $C(V)$.  The space $\bigwedge W$ carries an $C(V)$-module structure 
$$
\spin \colon C(V) \to \End(\bigwedge W)
$$
that is uniquely characterized by the following:
\begin{itemize}
\item $w \in W \subset V$ acts through left multiplication,
\item and $w' \in W' \subset V$ acts as
\begin{equation}\label{eq:daction}
w' (w_1 \wedge w_2 \wedge \cdots \wedge w_r) = \sum_{i=1}^r (-1)^{i+1} \langle w', w_i \rangle (w_1 \wedge w_2 \wedge \cdots \wedge \widehat {w_i} \wedge \cdots \wedge w_r),
\end{equation}
on $w_1 \wedge \cdots \wedge w_r \in \bigwedge^r W \subset \bigwedge W$.
\end{itemize}

The subspaces ${\bigwedge}^+ W := \bigwedge_{i \in 2\Z_{\geq 0}} W$ and ${\bigwedge}^- W := \bigwedge_{i \in 1+2\Z_{\geq 0}} W$ of $\bigwedge W$ are stable under $C^+(V)$. By restriction we obtain the spin representations
\begin{equation}\label{eq:CliffordHalfSpinDef}
\spin \colon \GPin_{2n} \to \GL \lhk \bigwedge W \rhk \quand \spin^\pm \colon \GSPin_{2n} \to \GL \lhk {\bigwedge}^\pm W \rhk.
\end{equation}
We recall that the representations $\spin^{\pm}$ are irreducible. In \eqref{eq:CliffordExplicitBasis} and \eqref{eq:CliffordExplicitBasis2} below, we will choose (ordered) bases for $ \bigwedge W$ and $ \bigwedge^\pm W$ coming from $\{e_1,\ldots,e_n\}$ to view $\spin$ and $\spin^\pm$ as $\GL_{2^n}$ and $\GL_{2^{n-1}}$-valued representations, respectively. We had another definition of $\spin^\eps$ as the representation with highest weight $\mu_\eps$ (Definition \ref{def:HalfSpinDef}), $\varepsilon\in \{+,-\}$. Let us check that the two definitions coincide via \eqref{eq:GSpin=GSpin}.

\begin{lemma}\label{lem:half-spin-highest-weight}
The highest weight of the half-spin representation $\spin^\eps$ of $\GSpin_{2n}$ on $\bigwedge^\eps W$ is equal to $\mu_{\eps}$.
\end{lemma}
\begin{proof}
We may compare $\mu_{\eps}$ and the highest weight of $\spin^\eps$ after pulling back along $\Spin_{2n}\times \G_m\twoheadrightarrow \GSpin_{2n}$. They coincide on $\Spin_{2n}$ by \cite[Prop.~20.15]{FultonHarris} and evidently restrict to the weight 1 character on $\G_m$. The lemma follows.
\end{proof}

Let us introduce a bilinear pairing on $\bigwedge W$ which is invariant under the spin representation up to scalars. Let $\tu{pr}_n: \bigwedge W\ra \C$ denote the projection onto $\bigwedge^n W$, identified with $\C$ via $e_1\wedge \cdots \wedge e_n \mapsto 1$. Write $\tau: \bigwedge W\isom \bigwedge W$ for the $\C$-linear anti-automorphism $w_1\wedge \cdots \wedge w_r \mapsto w_r \wedge \cdots \wedge w_1$ for $r\ge 1$ and $w_1,\ldots,w_r\in W$. Define
$$ 
(\!( \dot w_1,\dot w_2 )\!):=\tu{pr}_n(\tau(\dot w_1) \wedge \dot w_2),\qquad \dot w_1,\dot w_2\in \bigwedge W.
$$
We write $\spin^\vee$ and $\spin^{\eps,\vee}$ for the dual representations of $\spin$ and $\spin^\eps$. By the preceding lemma, the highest weight of $\spin^{\eps,\vee}$ is in the Weyl group orbit of $(\mu_\eps)^{-1}$.

\begin{lemma}\label{lem:spin-inv-pairing}\label{lem:Duality}
The pairing $(\!(~,~)\!)$ is nondegenerate; it is alternating if $n\equiv 2,3~(\tu{mod}~4)$ and symmetric if $n\equiv 0,1~(\tu{mod}~4)$. The restriction of $(\!(~,~)\!)$ to each of $\bigwedge^+ W$ and~$\bigwedge^- W$  is nondegenerate if $n$ is even, and identically zero if $n$ is odd. We have
\begin{equation}\label{lem:Duality-Eq1}
(\!( \spin(g)\dot w_1, \spin(g)\dot w_2 )\!)=\cN(g) (\!( \dot w_1,\dot w_2 )\!),\qquad g\in \GPin_{2n}(\C),~~\dot w_1,\dot w_2\in \bigwedge W.
\end{equation}
In particular, we have $\spin^{\eps} \simeq  \spin^{(-1)^n \eps,\vee} \otimes \cN$.
\end{lemma}
\begin{proof}
The first two assertions are elementary and left to the reader. The last assertion follows from the rest. For the equality \eqref{lem:Duality-Eq1}, we claim that
\begin{equation}\label{eq:C(V)-inv-pairing}
(\!( c \dot w_1, \dot w_2 )\!)= (\!( \dot w_1,\beta(c) \dot w_2 )\!),\qquad c\in C(V),~\dot w_1,\dot w_2\in \bigwedge W.
\end{equation}
Since $\GPin_{2n}\subset C(V)$, this implies \eqref{lem:Duality-Eq1} as
$$
(\!( \spin(g)\dot w_1, \spin(g)\dot w_2 )\!)=(\!( \dot w_1, \spin(\beta(g)g)\dot w_2 )\!)=\beta(g)g(\!(\dot w_1, \dot w_2 )\!).
$$

It remains to prove the claim. The proof of \eqref{eq:C(V)-inv-pairing} reduces to the case $c\in V$, then to the two cases $c\in W$ and $c\in W'$ by linearity. In both cases, \eqref{eq:C(V)-inv-pairing} follows from the explicit description of the $C(V)$-action as in \eqref{eq:daction}. Indeed, \eqref{eq:C(V)-inv-pairing} is obvious if $c\in W$. When $c\in W'$, it is enough to show that for $0\le r,s\le n$, $1\le i_1<\cdots <i_r\le n$, $1\le j_1<\cdots<j_s\le n$, and $1\le k\le n$,
$$ 
\tau(e_{n+k}(e_{i_1}\wedge \cdots \wedge e_{i_r})) \wedge (e_{j_1}\wedge \cdots \wedge e_{j_s}) =
  \tau(e_{i_1}\wedge \cdots \wedge e_{i_r})\wedge \left( e_{n+k}(e_{j_1}\wedge \cdots \wedge e_{j_s})\right).
$$
(This implies \eqref{eq:C(V)-inv-pairing} by taking $\tu{pr}_n$.) The equality is simply $0=0$ unless $k=r_0=s_0$ for some $1\le r_0\le r$ and $1\le s_0\le s$. In the latter case, the equality boils down to
$$ 
(-1)^{r_0+1} e_{i_r}\wedge \cdots \wedge \hat{e}_{i_{r_0}}\wedge \cdots \wedge e_{i_1}\wedge e_{j_1}\wedge \cdots \wedge e_{j_s}
 = (-1)^{s_0+1} e_{i_r}\wedge \cdots  \wedge e_{i_1}\wedge e_{j_1}\wedge \cdots \wedge \hat{e}_{j_{s_0}}  \wedge \cdots \wedge e_{j_s},$$
 which is clear. The proof is complete.
\end{proof}

We also discuss the odd case. Equip $V_{2n-1}=\C^{2n-1}$ with standard basis $\{f_1,\ldots,f_{2n-1}\}$ and quadratic form $Q_{2n-1}$ of \eqref{eq:StdQuadSpace}. As in \cite[p.306]{FultonHarris}, we decompose
$$
V_{2n-1} := \C^{2n-1} = W_{2n-1} \oplus W_{2n-1}' \oplus U_{2n-1},
$$
where $W_{2n-1}:=\oplus_{i=1}^{n-1} \C f_i$, $W_{2n-1}':=\oplus_{i=n}^{2n-2} \C f_i$, and $U_{2n-1}:=\C f_{2n-1}$. Again we omit the subscript $2n-1$ when it is clear from the context. Then $W$ and $W'$ are $(n-1)$-dimensional isotropic subspaces, and $U$ is a line perpendicular to them. As in the even case, each of $\bigwedge W $ and $\bigwedge^\pm W$ can be viewed as a subspace of $C(V)$ and has a unique structure of left $C(V)$-module where:
\begin{itemize}
\item $w \in W \subset V$ acts on $\bigwedge W$ through left multiplication,
\item $w' \in W' \subset V$ acts as in \eqref{eq:daction} (cf. \cite[20.16]{FultonHarris}),
\item $f_{2n-1}$ acts trivially on $\bigwedge^+ W$ and as $-1$ on $\bigwedge^- W$.
\end{itemize}

Consider the bijection
$$
\psi \colon \bigwedge W_{2n-1} \isomto {\bigwedge}^+ W_{2n}, \quad w_1 \wedge  \cdots \wedge  w_r \mapsto \begin{cases}
w_1 \wedge  \cdots \wedge  w_r \wedge e_n, & \tu{$r$ odd} \cr
w_1 \wedge  \cdots \wedge  w_r, & \tu{$r$ even}.
\end{cases}
$$
\begin{lemma}\label{lem:CliffordSpinRestrict}
For all $g \in \GSpin_{2n-1}$ and all $w \in \bigwedge W_{2n-1}$ we have $\istd(g) \psi(w) = \psi(g w)$,
where $\istd(g)$ and $g$ act by $\spin^+$ of $\GSpin_{2n}$ and $\spin$ of $\GSpin_{2n-1}$, respectively.
\end{lemma}
\begin{proof}
We keep writing $W=W_{2n-1}$, $W'=W'_{2n-1}$, $U=U_{2n-1}$. We identify $V_{2n}=(W\oplus U^1 \oplus W'\oplus U^2$ via $W_{2n}=W\oplus U^1$ and $W'_{2n}=W'\oplus U^2$ with $U^1=\C e_n$ and $U^2=\C e_{2n}$, mapping the basis of $W$ (resp.~$W'$) onto the first $n-1$ elements in the basis of $W_{2n}$ (resp.~$W'_{2n}$). This also gives the embedding $V_{2n-1}\subset V_{2n}$, with $U$ diagonally embedded in $U^1\oplus U^2$ (so $f_{2n-1}$ maps to $e_n+e_{2n}$), as in the formula below \eqref{eq:StdQuadSpace}.

There is an obvious embedding $\iota^+: \bigwedge W \hookrightarrow \bigwedge (W\oplus U^1)$. We also have $\iota^-:\bigwedge W \hookrightarrow \bigwedge (W\oplus U^1)$ by $(\cdot)\wedge e_n$. Both $\iota^+$ and $\iota^-$ are $C(W\oplus W')$-equivariant, by using that left and right multiplications commute and that $e_n$ is orthogonal to $W\oplus W'$. Furthermore, $\iota^-$ intertwines the $f_{2n-1}$-action on $\bigwedge^- W$, which is by multiplication by $-1$, and the $e_n+e_{2n}$-action on $\bigwedge^+ (W\oplus U^1)$, since $w\wedge e_n=-e_n\wedge w$ if $w\in \bigwedge^- W$ and since $W\perp e_{2n}$ with respect to $Q_{2n}$.

Now we claim that $\psi$ is $C^+(W\oplus W'\oplus U)$-equivariant, which implies the lemma by restricting from $C^+(W\oplus W'\oplus U)$ to $\GSpin_{2n-1}$. It suffices to verify equivariance of $\psi$ under $C^+(W\oplus W')$ and $C^-(W\oplus W')\otimes f_{2n-1}$. But $\psi$ is $\iota^+$ on $\bigwedge^+ W$ and $\iota^-$ on $\bigwedge^- W$. Thus the claim is deduced by putting together the equivariance in the preceding paragraph.
\end{proof}

\begin{lemma}\label{lem:CliffordSpinThetaAction}
Let $\vartheta \in \GPin_{2n}$ be the element from \eqref{eq:Elementw_spin}. We have ${\bigwedge}^+ W_{2n} \isomto {\bigwedge}^- W_{2n}$, $x \mapsto \vartheta x$. We have $\spin^+\circ \theta=\spin^-$ via this isomorphism, i.e., $\vartheta(\spin^+(g) x) = \spin^-(\theta(g)) \vartheta x$ for each $g\in \GSpin_{2n}$.
\end{lemma}
\begin{proof}
Henceforth we omit the symbol $\wedge$ for the wedge product in $W_{2n}$.
Consider $v = e_{k_1} \cdots e_{k_r} \in {\bigwedge}^+ W_{2n}$, with $k_1 < k_2 < \ldots < k_r$ and $r$ is even. Then
\begin{align*}
\vartheta v  = \sqrt{-1}(e_{n} e_{k_1} \cdots e_{k_r} - e_{2n}\cdot e_{k_1} \cdots e_{k_r}) \in \bigwedge W_{2n},
\end{align*}
where $e_{2n}$ acts by \eqref{eq:daction}. Thus the isomorphism follows from the following computations.
\begin{align*}
e_{n} e_{k_1} \cdots e_{k_r} & = \begin{cases} 0, & k_r = n, \cr e_{k_1}\cdots e_{k_r} e_n, & k_r \neq n, \end{cases} \cr
e_{2n} e_{k_1} \cdots e_{k_r} & = \sum_{i=1}^r (-1)^{i+1} \langle e_{2n}, e_{k_i} \rangle e_{k_1} \cdots \wh{e_{k_i}} \cdots e_{k_r}  = \begin{cases} - e_{k_1} \cdots e_{k_{r-1}}, & k_r = n, \cr 0, & k_r \neq n. \end{cases}
\end{align*}
The last assertion comes down to showing that $\vartheta g x= \theta(g) \vartheta x$, where $\vartheta g, \theta(g) \vartheta\in C(V)$ act through the $C(V)$-module structure on $x\in \bigwedge W_{2n}$. But this is clear since $\theta(g)=\vartheta g \vartheta^{-1}$.
\end{proof}

Consider the basis $\{b_U\}$ of $\bigwedge W_{2n}$, with
\begin{equation}\label{eq:CliffordExplicitBasis}
b_U =  (-1)^{\# U} e_{k_1} \cdot e_{k_2} \cdots e_{k_r} \in \bigwedge W_{2n},
\end{equation}
where $U = \{k_1 < k_2 < \cdots < k_r\}$ ranges over the subsets of $\{1, 2, \ldots, n\}$.
The $U$ of even size form a basis for ${\bigwedge}^+ W_{2n}$; and the $U$ with odd size form a basis for ${\bigwedge}^- W_{2n}$. Order the $b_U$ for $U$ odd, and the $b_U$ for $U$ even in such a way that the ordering of $\{b_U\}_{|U|:\tu{even}}$ corresponds to that of $\{b_U\}_{|U|:\tu{odd}}$ via $b_U\mapsto \vartheta b_U/\sqrt{-1}$. Then these orderings of the $b_U$ gives us two identifications
\begin{equation}\label{eq:CliffordExplicitBasis2}
\GL \lhk {\bigwedge}^+ W_{2n}\rhk  \isomto \GL_{2^{n-1}} \quand \GL \lhk {\bigwedge}^- W_{2n} \rhk \isomto \GL_{2^{n-1}},
\end{equation}
such that the following proposition holds.

\begin{proposition}\label{prop:res-of-spin}
The following diagram commutes
$$
\xymatrix{
& \GSpin_{2n} \ar[dr]^{\spin^+}\ar[dd]^(.25)\theta & \cr
\GSpin_{2n-1} \ar'[r][rr]^(-.4){\spin} \ar[ur]^{i_{\std}}\ar[dr]_{i_{\std}} & & \GL_{2^{n-1}} \cr
& \GSpin_{2n} \ar[ur]_{\spin^-} &
}
$$
\end{proposition}
\begin{proof}
This follows from Equation \eqref{eq:CliffordExplicitBasis}, Lemmas \ref{lem:CliffordSpinRestrict}, and (proof of) Lemma \ref{lem:CliffordSpinThetaAction}.
\end{proof}

\section{Some special subgroups of $\GSpin_{2n}$}\label{sec:subgroups-Gspin}

In this section, the base field of all algebraic groups is an algebraically closed field of characteristic 0 such as $\C$ or $\lql$. We begin with principal morphisms for $\GSpin_{2n-1}$ and $\GSpin_{2n}$. (See \cite[Sect.~7]{patrikis2016deformations} and \cite{GrossPrincipal, SerrePrincipal} for general discussions.) The following notation will be convenient for us.  Denote by
$$
j_\reg \colon \Gm \times \SL_2 \to \GSpin_{2n-1}
$$
the product of the central embedding $\G_m\hra \GSpin_{2n-1}$ and a fixed principal $\SL_2$-mapping. Note that $j_\reg$ has the following kernel\footnote{To see this, one can use Proposition 6.1 of \cite{GrossMinuscule}, where the $\SL_2$-representations appearing in the composition $\SL_2 \overset {\tu{pri}} \to \GSpin_{2n-1} \overset {\tu{spin}} \to \GL_{2^{n-1}}$ are computed.}
$$
\begin{cases}
\langle (-1, \vierkant {-1}00{-1})\rangle, & \tu{if}\,~n(n-1)/2 \,\tu{ is odd},\\
\langle (1, \vierkant {-1}00{-1})\rangle, &  \tu{if}\,~n(n-1)/2 \,\tu{ is even}.
\end{cases}
$$
We write $G_\pri \subset \GSpin_{2n-1}$ for the image of $j_\reg$. The group $G_\pri$ is isomorphic to $\GL_2$ if $n(n+1)/2$ is odd, and to $\Gm \times \PGL_2$ otherwise. Using $i_{\std}$ from \eqref{eq:std_emb_def}, we define
$$
i_{\reg} = i_{\std} \circ j_\reg \colon \Gm \times \SL_2 \to \GSpin_{2n}.
$$
The map $\pr^{\circ}\circ i_{\reg}: \Gm \times \SL_2 \to \SO_{2n}$ factors through $\PGL_2\ra \SO_{2n}$, to be denoted $i^\circ_{\reg}$, via the natural projection from $\Gm \times \SL_2 \ra \PGL_2$ (trivial on the $\Gm$-factor). We see that the preimage of $i^\circ_{\reg}(\PGL_2)$ in $\GSpin_{2n}$ is $i_{\std}(G_\pri)$. Denote by $\li{j_{\reg}}:\PGL_2\to \PSO_{2n-1}$ the map induced by $j_{\reg}$ on the adjoint groups.\footnote{When denoting the group standing alone, we prefer $\SO_{2n-1}$ to $\PSO_{2n-1}$. When thinking of a projective representation or a subgroup of $\PSO_{2n}$ via $\li{i_{std}}$, we usually write $\PSO_{2n-1}$.}
We also introduce the map
$$
 \li {i_{\reg}} = \li {i_{\std}} \circ \li{j_{\reg}} \colon \PGL_2 \to \PSO_{2n}.
$$

Recall that we have fixed earlier the group $\SO_8$ in \eqref{eq:DefinitionGroupGo} (cf. below \eqref{eq:BorelDesc}).
Let $T_{\GSpin_7} \subset \GSpin_7$ be as in \cite[\S Notation]{GSp} and put $T_{\Spin_7} = T_{\GSpin_7} \cap \Spin_7$. 

We will now fix a convenient basis for $X_*(T_{\Spin_7})$. We have $X_*(T_{\GSpin_7}) = X^*(T_{\GSp_6}) = \bigoplus_{i=0}^3 \Z e_i$, the center $Z_{\GSp_6} \subset T_{\GSp_6}$ equals $\{(t^2, t, t, t) | t \in \Gm\}$ (use the roots $\alpha_i$ listed in [\textit{loc. cit.}, p. 10]), and so $X^*(T_{\GSp_6}) \to X^*(Z_{\GSp_6})$ identifies with $\Z^4 \to \Z$, $(x_i) \mapsto 2a_0 + a_1 + a_2 + a_3$. Thus $X_*(T_{\Spin_7}) = \{ (a_i)\in \Z^4\, |\, 2a_0 + a_1 + a_2 + a_3 = 0\}$. By projecting $(a_i)\in \Z^4$ onto $(a_1, a_2, a_3)$ we obtain
\begin{equation}\label{eq:DesctorusSpin7}
X_*(T_{\Spin_7})=\{(a_1,a_2,a_3) \in \Z^3 : a_1 + a_2 + a_3 \equiv 0 \mod 2\}.
\end{equation}

We write $T_{\SO_8} \subset \SO_8$ for the maximal torus corresponding to \eqref{eq:MaximalTorusGSO} (so with $t_0 = 1$). The spin representation of $\Spin_7$ is orthogonal (\!\!\cite[Lem.~0.1]{GSp}), yielding an embedding $\spin^{\circ \prime} : \Spin_7\hra \SO(q)$, for some quadratic form $q$ in $8$ variables. We fix an isomorphism $u \colon \SO(q) \isomto \SO_8$, in such a way that the composition
$$
\spin^{\circ} := u \circ \spin^{\circ \prime} \colon \Spin_7 \hra \SO_8
$$
maps $T_{\Spin_7}$ into $T_{\SO_8}$ and such that
\begin{equation}\label{eq:Spin-circ-definition}
\spin^\circ(a) = (\tfrac 12 (\tau_1^j a_1 + \tau_2^j a_2 + \tau_3^j a_3)) \in X_*(T_{\SO_8}) \subset \Z^8,
\end{equation}
for some choice of numbering $\tau^j=(\tau^j_1,\tau^j_2,\tau^j_3)\in\{\pm 1\}^3$ for $j=1,...,8$, such that $\tau^j = -\tau^{j + 4}$ for $j=1,2,3, 4$. In \eqref{eq:Spin-circ-definition} the embedding $X_*(T_{\SO_8}) \subset \Z^8$ comes from \eqref{eq:MaximalTorusGSO}. 

We write  $\li{\spin^\circ}:\Spin_7\hra \PSO_8$ for the projectivization of $\spin^\circ$. Fixing a non-isotropic line in the underlying 8-dimensional space, the stabilizer of the line in $\Spin_7$ is isomorphic to a group of type $G_2$, cf.~\cite[p.169,~Prop.~2.2(4)]{GrossSavin}. Thereby we obtain an embedding $j_{\spin} \colon G_2\hra \Spin_7$. Alternatively, an embedding $G_2\hra \Spin_7$ can be constructed using the octonion algebra \cite[Sect.~2.5]{ChenevierG2}. The conjugacy class of $j_{\spin}$ is unique (thus independent of choices) by [\textit{loc. cit.}, Prop.~2.11].
Denote by
\begin{equation}\label{eq:i_spin-map}
i_{\spin}\colon G_2\hra \Spin_8
\end{equation}
the composite $i_{\std}\circ j_{\spin}$.
The restriction of $\spin^\eps:\Spin_8\ra \GL_8$ via $i_{\spin}$ is isomorphic to $\mathbf{1}\oplus \std$, where $\mathbf{1}$ and $\std$ are the trivial and the unique irreducible 7-dimensional representation of $G_2$, respectively. (This is easy to see by dimension counting, as the other irreducible representations have dimension\,$\ge 14$.)

\begin{lemma}\label{lem:spin7}
The representation $\spin^\circ:\Spin_7\hra \SO_8$ is $\OO_8$-conjugate to $\theta^\circ \spin^\circ$ but not locally conjugate (thus not conjugate) as an $\SO_8$-valued representation. In fact, there exists an open dense subset $U \subset \Spin_7$ such that $\spin^\circ t$ and $\theta^\circ \spin^\circ t$ are not conjugate for any $t \in U$. Moreover $\spin^\circ(\Spin_7)$ and $\theta^\circ\spin^\circ(\Spin_7)$ are not $\SO_8$-conjugate. The analogous assertion holds for $\li{\spin}:\SO_7\hra \PSO_8$.
\end{lemma}
\begin{proof}
Evidently $\spin^\circ$ and $\theta^\circ \spin^\circ$ are $\OO_8$-conjugate since $\theta^\circ = \Int(\vartheta^\circ)$ with $\vartheta^\circ\in \OO_8$. Let $T_{\GL_8} \subset \GL_8$ be the diagonal torus.  Let $\Omega_{\Spin_7},\Omega_{\SO_8},\Omega_{\GL_8}$ denote the Weyl groups corresponding to $T_{\Spin_7},T_{\SO_8},T_{\GL_8}$. 
In view of the weights of the spin representation~\cite[Prop.~20.20]{FultonHarris}, we know that
$$
\std(\spin^\circ(a_1,a_2,a_3))\in \Omega_{\GL_8}((\eps_1 a_1 + \eps_2 a_2 + \eps_3 a_3)/2 : \eps_i\in \{\pm1\}).
$$
(The $\Omega_{\GL_8}$-orbit of 8-tuples is simply an unordered 8-tuple.) When $\eps_1 a_1 + \eps_2 a_2 + \eps_3 a_3$ are all distinct, the right hand side breaks up into exactly two $\Omega_{\SO_8}$-orbits, which are permuted by $\theta^\circ$. Similarly, if $U_T$ is the open dense subset of $T_{\Spin_7}$ consisting of $t = (t_1, t_2, t_3) \in T_{\Spin_7}$ with $t^{\eps_1}_1,t_2^{\eps_2},t_3^{\eps_3}$ all distinct, then $\spin^\circ(t)$ and $\theta^\circ(\spin^\circ(t))$ are not $\SO_8$-conjugate. This implies the existence of $U$ as in the lemma by taking $U$ to be the set of regular semisimple elements whose conjugacy classes meet $U_T$.

Now assume that  $\spin^\circ(\Spin_7) = g \theta^\circ\spin^\circ(\Spin_7) g\inv$ for some $g \in \SO_8$. Then the composition $c_g = \spin^{\circ, -1} \circ \tu{Int}_g \circ \theta^\circ \spin^\circ$ is an automorphism of $\Spin_7$, which is hence inner and of the form $x \mapsto hxh\inv$ for some $h \in \Spin_7$. Thus $\spin^{\circ}$ and $\theta^\circ \spin^\circ$ are conjugate by $g\inv \spin^\circ(h)$, a contradiction. Thus $\spin^\circ(\Spin_7)$ and $\theta^\circ\spin^\circ(\Spin_7)$ are not $\SO_8$-conjugate. The projective analogue for $\li{\spin} \colon \SO_7\hra \PSO_8$ also follows.
\end{proof}

\begin{lemma}\label{lem:IrreducibilityOfSpin-}
Write $H := \pr^{-1}(\spin^\circ(\Spin_7)) \subset \Spin_8$. The restriction of $\spin^\eps$ to $H$ is irreducible if and only if $\eps = -$. More precisely, we have $\li {\spin^+} \circ \spin^\circ \simeq \li{\std \oplus \one}$ and
$\li {\spin^-} \circ \spin^\circ \simeq \li{\spin^\circ}$. 
\end{lemma}
\begin{proof}
We compute the composition $\li {\spin^\eps} \circ \spin^\circ \colon \Spin_7 \overset {\spin^\circ} \to \SO_8 \to \PSO_8 \overset { \li {\spin^\eps}} \to \PGL_8$ on $T_{\Spin_7}$. For $a \in X_*(T_{\Spin_7})$, $\spin^\circ(a)$ is given by \eqref{eq:Spin-circ-definition}, and for $b \in X_*(T_{\SO_8}) = \Z^4$ we have $\li {\spin^\eps}(b) = (\tfrac12 (\tau_1 b_1 + \tau_2 b_2 + \tau_3 b_3 + \tau_4 b_4))_{\tau \in \{\pm1\}^4, \prod_{i=1}^4 \tau_i = \eps}$ (both up to the Weyl group actions). From this it follows that  $\li {\spin^\eps} \circ \spin^\circ |_{T_{\Spin_7}}$ is conjugated to $\li {\std \oplus \one}|_{T_{\Spin_7}}$ and $\li {\spin^\circ}|_{T_{\Spin_7}}$ for $\eps=+$ and $\eps = -$, respectively. The lemma now follows from Lemma \ref{lem:AlgebraicProjRepsConjugate}.
\end{proof}

\begin{lemma}\label{lem:containingRegularUnipotent-1}
Let $n \geq 3$ and $H \subsetneq \SO_{2n}$ be a proper connected reductive subgroup containing a regular unipotent element. Then $H$ is isomorphic to a quotient of $\Spin_{2n-1}$, $\SL_2$ or $G_2$ (the last can occur only if $n = 4$). 
\end{lemma}
\begin{proof}
We begin with some preliminaries. When $G$ is a reductive group, write $\Sigma(G)$ for the set of maximal proper connected reductive subgroups $M$ of $G$ that contain a regular unipotent element of $G$. From \cite[Thms.~ A,B]{SaxlSeitz} we have the following\footnote{The statement of \cite[Thms.~ A,B]{SaxlSeitz} are not entirely clear on whether the list describes $H^0$ or $H$. We interpret it as the former since that is what their proof shows. For instance, regarding (i)(a) of their theorem, a maximal reductive subgroup of type $B_{n-1}$ in $\SO_{2n}$ is not $i^\circ_{\std}(\SO_{2n-1})$ but $Z(\SO_{2n})\times i^\circ_{\std}(\SO_{2n-1})$, which is disconnected.}:
\begin{itemize}
\item[(a)] Case $G \simeq \SO_{2n}$ ($n \geq 3$), then every $M \in \Sigma(G)$ is isomorphic to a quotient of $\Spin_{2n-1}$. 
\item[(b)] Case $G \simeq \SO_{2n-1}$ ($n \geq 3$, $n \neq 4$), then every $M \in \Sigma(G)$ is isomorphic to a quotient of $\SL_2$.
\item[(c)] Case $G \simeq \SO_7$, then every $M \in \Sigma(G)$ is isomorphic to $G_2$.
\item[(d)] Case $G \simeq G_2$, then every $M \in \Sigma(G)$ is isomorphic to a quotient of $\SL_2$.
\end{itemize}
We prove the following claim: If $H$ is a connected reductive subgroup of some connected reductive group $G$ (over $\C$ or $\lql$), such that $H$ contains a regular unipotent element $u$ of $G$, then $u$ is also regular unipotent in $H$. To see this, write $B_u \subset G$ for the unique Borel subgroup that contains $u$. Now let $B_0 \owns u$ be a Borel subgroup of $H$ that contains $u$. Then $B_0$ is a connected solvable subgroup of $G$, and hence is contained in a Borel subgroup $B_1$ of $G$. As $u \in B_1$, we must have $B_1 = B_u$. Hence $B_0 \subset B_u \cap H$. Since $(B_u \cap H)^0$ is connected solvable and contains $B_0$, we must have $(B_u \cap H)^0= B_0$ by maximality of $B_0$. This shows that in $H$, the element $u$ is contained in exactly one Borel subgroup. Therefore $u \in H$ is regular unipotent.

Now let $H$ be as in the statement of the lemma. Let $u \in H$ be regular unipotent in $\SO_{2n}$. Let $M \in \Sigma(\SO_{2n})$ such that $H \subset M$. If $H = M$, we are done by (a) above. So assume $H \neq M$. Again by (a), $M$ is a quotient of $\Spin_{2n-1}$, and hence $M_{\ad} \simeq \SO_{2n-1}$. Then $H_\ad$ maps to $\SO_{2n-1}$ (since the center of $H$ commutes with $u$, it is contained in $Z_{\SO_{2n}}$ by \cite[Thm.~4.11]{SpringerUnipotent} thus also in $Z_M$), and by the claim we can find an $M' \in \Sigma(\SO_{2n-1})$ that contains the image of $H_\ad$ in $\SO_{2n-1}$. If $H_{\ad} = M'$, we are done by (b) or (c) if $n = 4$. If $H_{\ad} \neq M'$, then again $M'$ is either $G_2$ or a quotient of $\SL_2$, and we can argue similarly.
\end{proof}

If $H$ is an algebraic group, we write $\Upsilon(H)$ for the set of $\SO_{2n}$-conjugacy classes of morphisms $H \to \SO_{2n}$ that have a regular unipotent element in their image. By abuse of notation, we often identify $\Upsilon(H)$ with a set of representatives for the conjugacy classes.

\begin{lemma}\label{lem:containingRegularUnipotent-2}
We have 
\begin{align*}
 \Upsilon(\Spin_{2n-1}) &= \lbr  {i^\circ_\std} \colon \Spin_{2n-1} \to \SO_{2n} \rbr & (n \geq 3, n \neq 4)  \cr
 \Upsilon(\Spin_7) &= \lbr  {i^\circ_\std}, \spin^\circ, \theta^\circ \spin^\circ \colon \Spin_7 \to \SO_8 \rbr & (n = 4) \cr
\Upsilon(\SL_2) &= \{ i_\reg \colon \SL_2 \to \SO_{2n}\} & (n \geq 3) \cr
\Upsilon(G_2) &= \{ {i_\spin} \colon G_2 \to \SO_8\} & (n=4)
\end{align*}
\end{lemma}
\begin{proof}
\noindent \textbf{Case} $H = \Spin_{2n-1}$. We have $\ker(f) \subset \{\pm 1\}$. Assume first $\ker(f) = \{\pm 1\}$. Then $r := \std \circ f$,  $r' := \std \circ  {i^\circ_{\std}}$ are two faithful representations of dimension $2n$. 
By Lemma \ref{lem:RepsOfDim2n}, $r$ and $r'$ are isomorphic. By acceptability of $\tu{O}_{2n}$ we find a $g \in \tu{O}_{2n}$ that conjugates $r$ to $r'$ \cite[Prop.~B.1]{GSp}. The element $g$ might have negative determinant. In this case we can replace $g$ by $g \cdot \vartheta^\circ$, as $\vartheta^\circ$ centralizes $i^\circ_{\std}(\SO_{2n-1})$ by Lemma~\ref{lem:conjugation-by-w}(ii).

Now assume $\ker(f) = 1$. Then $r = \std \circ f$ is a faithful $2n$-dimensional representation of $\Spin_{2n-1}$. The smallest such representation by dimension is the spin representation, and therefore $2^{n-1} \leq 2n$, and $n \leq 4$. We distinguish in subcases $n = 3$ or $n = 4$:

\tu{(When $n = 3$.)} We show that this subcase $(H = \Spin_{2n-1},\ker(f)=1,n=3)$ does not occur. Assume $f \colon \Spin_5 \to \SO_6$ is injective with a regular unipotent element in its image. Recall $\SO_6 \simeq \SL_4/\{\pm1\}$. Write $\wt H \subset \SL_4$ for the pre-image of $f(\Spin_5)$ in $\SL_4$. Let $M \subset \SL_4$ be a proper maximal connected reductive subgroup of $\SL_4$ that contains $H$. Then $M$ is isomorphic to $\Sp_4$ by \cite[Thm. B]{SaxlSeitz}. By dimension consideration we must have $\wt H = M$. In particular the image of $f$ in $\SL_4/\{\pm 1\}$ must be isomorphic to $\PSp_4$. Hence $\Spin_5 \isomto \PSp_4$, a contradiction.   

\tu{(When $n = 4$.)} We want to classify all conjugacy classes of injections $f \colon \Spin_7 \to \SO_8$ with a regular unipotent element in their image. In this case they do exist, as $\spin^\circ$ is an example.  The representations $\std \circ f$ and $\std \circ \spin^\circ$ are both faithful representations of $\Spin_7$ of dimension $8$. Hence they are isomorphic. By acceptability of $\uO_8$ there exists a $g \in \uO_8$ that conjugates $f$ to $\spin^\circ$. This implies that $f$ is $\SO_8$-conjugate to either $\spin^\circ$ or $\theta^\circ \spin^\circ$. By Lemma~\ref{lem:spin7}, the representations $\spin^\circ$ and $\theta^\circ \spin^\circ$ are not $\SO_8$-conjugate.
Observe finally that $\ker( i^\circ_{\std}) = \{\pm 1\}$ while $\spin^\circ$ and $\theta^\circ \spin^\circ$ are injective. Hence $i^\circ_{\std}$ is not conjugate to either $\spin^\circ$ or $\theta^\circ \spin^\circ$. 
This verifies the description of $ \Upsilon(\Spin_7)$ in the lemma.

\smallskip

\newcommand{\Rep}{{\textup{Rep}}}
\newcommand{\isomfrom}{\overset {\sim} \leftarrow}

\noindent \textbf{Case} $H = \SL_2$. We want to classify the morphisms $f \colon \SL_2 \to \SO_{2n}$ with a regular unipotent element in the image.
In fact, such a morphism is called the ``principal morphism" in the literature, and it is well-known that it is unique up to conjugacy. However, we could not find a precise reference, so we give some detail for a general connected reductive group $G$.

We first note that the natural map $\Hom(\SL_2, G) \to \Hom(\Lie(\SL_2), \Lie(G))$, equivariant for the adjoint action of $G$, is a bijection. To construct the inverse, let $g \in \Hom(\Lie(\SL_2), \Lie(G))$. The composition
$\Rep(G) \to \Rep(\Lie(G)) \to \Rep(\Lie(\SL_2)) \isomfrom \Rep(\SL_2)$  
is a $\otimes$-functor preserving the underlying vector spaces, where the last arrow is an equivalence (e.g., see \cite[VIII.1.5]{BourbakiLGLA7to9}). Thus the composition arises from a morphism of groups $f \colon \SL_2 \to G$ by \cite[Cor.~2.9]{DeligneMilne}, and one checks directly that $g \mapsto f$ and $f \mapsto \Lie(f)$ are inverse to each other.  By the Jacobson--Morozov lemma, $\Hom(\Lie(\SL_2), \Lie(G))$ is in bijection with the set of nilpotent elements in $\Lie(G)$ via $g \mapsto g\vierkant 0100 $. Above we consider $f$ such that $f\vierkant 1101$ is regular unipotent, thus $g \vierkant 0100$ is regular nilpotent in $\Lie(G)$, and hence unique up to conjugacy. The same statement follows for $f$ then as well. 
\smallskip

\noindent \textbf{Case} $H = G_2$ \textbf{and} $n = 4$. Let $f \colon G_2 \to \SO_8$ be a morphism with a regular unipotent element in its image. Recall $i_\spin \colon G_2 \to \Spin_8$ from \eqref{eq:i_spin-map}, it induces a morphism $\ol{i_\spin} \colon G_2 \to \SO_8$. The representations $\std \circ f$ and $\std \circ \ol{i_\spin}$ are both faithful and of dimension $8$, hence isomorphic (they are both isomorphic to $r_7 \oplus \one$, where $r_7$ is the unique representation of $G_2$ of dimension $7$). By $O_8$-acceptability, we can find a $g \in O_8$ such that $f = g \ol{i_\spin} g^{-1}$. If $\det(g) = 1$ we are done. The element $\vartheta^\circ \in O_8$ centralizes the subgroup $\SO_{2n-1}$. The map $i_\spin$ factors over the map $j_\spin \colon G_2 \to \Spin_7$ (see above \eqref{eq:i_spin-map}). In particular $\vartheta^\circ{i_\spin} \vartheta^{\circ, -1} = \li{i_\spin}$, and we can replace $g$ by $g \vartheta^{\circ}$.
\end{proof}

\begin{proposition}\label{prop:containingRegularUnipotent}
Let $n\ge 3$. Let $\ol H \subset \PSO_{2n}$ be a (possibly disconnected) reductive subgroup (over $\C$ or $\lql$) containing a regular unipotent element. Up to conjugation by an element of $\PSO_{2n}$, the following holds (in particular  $\ol H$ is connected in all cases):
\begin{enumerate}[label=(\roman*)]
  \item if $n\neq 4$, then $\ol H = \PSO_{2n}$, $\ol H=\ol{i_{\std}}(\PSO_{2n-1})$, or $\ol H=\ol{i_{\reg}}(\PGL_2)$;
  \item if $n=4$, then $\ol H$ is either as in (1), $\ol H=\li{\spin^\circ}(\SO_7)$, $\ol H=\li{\theta^\circ\spin^\circ}(\SO_7)$, or $\ol H = \li{i_{\spin}}(G_2)$.
\end{enumerate}
If $H \subset \SO_{2n}$ is a (possibly disconnected) reductive subgroup containing a regular unipotent element, then $H^0\subset H\subset H^0\cdot Z(\SO_{2n})$ and $H^0$ surjects onto $\ol H\subset \PSO_{2n}$ as in the list above. (See the proof for the list of possible $H^0$.)
\end{proposition}

\begin{proof}
We first focus on the classification of reductive subgroups $H \subset \SO_{2n}$ containing a regular unipotent element. If $H = \SO_{2n}$ there is nothing to do, so we assume $H$ is proper. By Lemma~\ref{lem:containingRegularUnipotent-1} the group $H^0$ is isomorphic to a quotient of $\Spin_{2n-1}, G_2$ or $\SL_2$. Using Lemma~\ref{lem:containingRegularUnipotent-2} we conjugate so that $f \colon H^0 \hookrightarrow \SO_8$ is one of the maps listed in that lemma. So
\begin{itemize}
\item[(a)] When $n \neq 4$, $f$ equals $\li {i^\circ_\std}$ or $i_\reg$.
\item[(b)] When $n = 4$, $f$ equals $\li {i^\circ_\std}$, $i_\reg$, $\li {i_\spin}$, $\spin^\circ$ or $\theta^\circ \spin^\circ$. 
\end{itemize} 
From this list we see that if $\std(H^0)$ is reducible (so $f \neq \spin^\circ, \theta^\circ \spin^\circ$) then $\std(H)$ is contained in a parabolic subgroup of $\GL_{2n}$ with Levi component $\GL_{2n-1}\times\GL_1$. By reductivity $\std(H)$ is contained in $\GL_{2n-1}\times\GL_1$, and it is an irreducible subgroup. We see that $H^0\subset H^+:=i^\circ_{\std}(\SO_{2n-1})\times Z(\SO_{2n})$, and by Schur's lemma, the centralizer of $H^0$ in $H^+$ is $Z(\SO_{2n})$. Since $H^0$ has no nontrivial outer automorphism, the conjugation by each $h\in H$ on $H^0$ is an inner automorphism. Thus there exists $h'\in H^0$ such that $h' h^{-1}$ centralizes $H^0$. It follows that $H\subset H^0\times Z(\SO_{2n})$. If $\std(H^0) \subset \GL_{2n}$ is irreducible, the centralizer of $H^0$ in $\SO_{2n}$ is $ Z(\SO_{2n})$ again by Schur's lemma, with no nontrivial outer automorphism for $H^0$. As in the reducible case, we deduce $H^0\subset H\subset H^0\times Z(\SO_{2n})$.

Finally, if a reductive subgroup $\li H \subset \PSO_{2n}$ contains a regular unipotent element, then so does its preimage $H$ in $\SO_{2n}$. By the previous argument we may conjugate so that $H^0$ is of type (a) or (b), and moreover we find $H^0\subset H\subset H^0\times Z(\SO_{2n})$. In particular $\li H$ is connected and of the type listed in (\textit{i}) and (\textit{ii}).
\end{proof}

In the next lemma, and also in the later sections, the following group will play a role 
\begin{equation}\label{eq:H-subgroup}
H_{2n-1} := \GSpin_{2n-1} Z(\GSpin_{2n}) \subset \GSpin_{2n}.
\end{equation}
Recall $z^\eps \in Z(\GSpin_{2n})$ is such that $\langle z^\eps \rangle = \ker(\spin^\eps)$.
We have $H_{2n-1} = \GSpin_{2n-1} \times \langle z^+ \rangle$. By projecting we obtain a quadratic character
\begin{equation}\label{eq:kappa-def}
\kappa \colon H_{2n-1} \to \langle z^+ \rangle \subset Z(\GSpin_{2n-1})
\end{equation}
such that  
the composition $\spin^{\eps} \circ \kappa$ is trivial if $\eps = +$ and otherwise equal to the composition
\begin{equation}\label{eq:kappa_0-def}
\kappa_0 \colon H_{2n-1} \overset {\kappa} \to \langle z^+ \rangle \simeq \{\pm 1\}. 
\end{equation}
In the definition of $\kappa$ we could have also used $z^-$, in that case the convention would be slightly different. But notice that $\kappa_0$ does not depend on this choice as this character is simply the canonical map of $H_{2n-1}$ onto its component group. Observe also that $\theta$ acts trivially on $\GSpin_{2n-1}$ and on $z^+$ via $z^+ \mapsto -z^+$. This gives the simple formula 
\begin{equation}\label{eq:kappa_0-and-theta}
\theta(g) = \kappa_0(g) g \quad\quad \textup{ for all $g \in H_{2n-1}$.}
\end{equation}

\begin{lemma}\label{lem:for-mult-one}
Let $r \colon \Gamma \to \GSpin_{2n}(\lql)$ be a semisimple representation containing a regular unipotent element in its image. Let $\chi \colon \Gamma \to \lql^\times$ be a character and $\varepsilon\in\{+,-\}$. 
\begin{enumerate}[label=(\roman*)]
\item If $\chi \otimes \spin^\varepsilon (r)  \simeq \spin^{\varepsilon} (r)$ then $\chi = 1$.
\item If $\chi \otimes \spin^+ (r) \simeq \spin^- (r)$ then $r$ has image in the group $H_{2n-1} \subset \GSpin_{2n}$ up to conjugation, and $\chi$ is equal to $\kappa_0 \circ r$. 
\item  If $\chi^+ \spin^+(r) \oplus \chi^- \spin^-(r) \simeq \spin(r)$ for two characters $\chi^\pm \colon \Gamma \to \lql^\times$. Then $\chi^\pm$ are both trivial,  or $r$ has image in $H_{2n-1}$ up to conjugation and $\chi^+ = \chi^- = \kappa_0 \circ r$. 
\end{enumerate}
\end{lemma}
\begin{proof}
(\textit{i}) Write $\li r \colon \Gamma \to \PSO_{2n}(\lql)$ for the projectivization of $r$. By Proposition~\ref{prop:containingRegularUnipotent} we can distinguish between two cases for the Zariski closure of the image of $\li r$ in $\PSO_{2n}(\lql)$. If the Zariski closure of $\li r$ is either $\PSO_{2n}$ or $\ol{i_{\std}}(\PSO_{2n-1})$ then $\spin^\eps r$ is strongly irreducible, and the statement follows from \cite[Lem.~4.8(i)]{GSp}. In the remaining cases the Zariski closure of $\tu{Im}(\li r)$ is,  $\li{\spin^\circ}(\SO_7)$, $\li{\theta^\circ \spin^\circ}(\SO_7)$, $\ol{i_{\reg}}(\PGL_2)$ or $i_{\spin}(G_2)$. In the last two of these, $\tu{Im}(\li r)\subset \ol{i_{\std}}(\SO_{2n-1}(\lql))$ so $\tu{Im}(r)$ is contained in $\lql$-points of $H_{2n-1}$ (which is the preimage of $i_{\std}(\SO_{2n-1})$ in $\GSpin_{2n}$). Then we show $\chi=1$ by the argument exactly as in Cases (i), (ii), (iv) in the proof of \cite[Lem.~5.1, Prop.~5.2]{GSp}, noting that $\spin^\varepsilon$ restricts to $\spin$ on $\GSpin_{2n-1}$ by Proposition~\ref{prop:res-of-spin}. Finally, if the Zariski closure of the image is $\li{\spin^\circ}(\SO_7)$ (resp. $\li{\theta^\circ \spin^\circ}(\SO_7)$), then $\li {\spin^\eps} \circ \li r$ is irreducible if $\eps = -$ (resp. $\eps = +$) and isomorphic to $\std \oplus \one$ if $\eps = +$ (resp. $\eps = -$) by Lemma~\ref{lem:IrreducibilityOfSpin-}. In particular the representation $\spin^\eps r$ satisfies the conditions of \cite[Prop.~4.9]{GSp}, and so $\chi = 1$ in this case as well.

(\textit{ii}) If the Zariski closure of the image of $r$ contains $\Spin_{2n}$, then (ii) cannot occur. Thus $r$ has either image in $H_{2n-1}$, or it has image in $\pr^{-1}(\spin^\circ(\Spin_7))$. In the latter case, $\spin^- r$ is strongly irreducible while $\spin^+ r$ is not by Lemma~\ref{lem:IrreducibilityOfSpin-}, which is a contradiction. Thus $\tu{Im}(r) \subset H_{2n-1}$. For $g \in H_{2n-1}$ we have
$$
\spin^+(g) = \spin^-(\theta g) = \spin^-(\kappa_0(g) \cdot g) = \kappa_0(g) \spin^-(g). 
$$
Put $t = \kappa_0 \circ r$. Then $t \otimes \spin^+r \simeq \spin^- r$ and $t \chi\inv \otimes \spin^+ r \simeq \spin^+ r$, and $t = \chi$ by (\textit{i}). 

(\textit{iii}) Write $H$ for the Zariski closure of the image of $r$. By the proof of (\textit{i}) we see that either $\Spin_{2n} \subset H$, or $H \subset H_{2n-1}$ or $H \subset \pr^{-1}(\spin^\circ(\Spin_7))$ up to conjugation. Assume that $H \not\subset H_{2n-1}$ (even after conjugation), so that $\Spin_{2n} \subset H$ or $H \subset \pr^{-1}(\spin^\circ(\Spin_7))$. In this case, we need to show that $\chi^+=\chi^-=1$. Suppose $\chi^-\neq 1$ to the contrary. Since $\spin^-(r)$ is strongly irreducible by assumption on $H$ (cf.~Lemma \ref{lem:IrreducibilityOfSpin-}), we have
\begin{equation}\label{eq:Hom-is-zero}
\Hom(\chi^- \otimes \spin^-(r), \spin^-(r)) = 0
\end{equation}
by \cite[Lem. 4.8(i)]{GSp}. In particular $\chi^+ \spin^+(r) \oplus \chi^- \spin^-(r)\isomto \spin(r)$ induces an isomorphism $\chi^- \otimes \spin^-(r) \isomto \Ker(\spin(r) \surjects \spin^-(r)) = \spin^+(r)$. As $H \not\subset H_{2n-1}$, this contradicts (\textit{ii}). Therefore $\chi^- = 1$. From $\chi^+ \spin^+ r \oplus \chi^- \spin^- r \cong \spin(r)$ we then obtain $\chi^+ \spin^+ r \cong \spin^+r$, which implies $\chi^+ = 1$ by item (i).

Now assume $H \subset H_{2n-1}$ up to conjugation. We obtain a character $t = \kappa_0 \circ r \colon \Gamma \to \{\pm1\}$. Write $\Gamma_0 := \ker(t)$. We have $\chi^+ \spin^+(r) \oplus \chi^- \spin^-(r) \simeq \spin^+(r) \oplus \spin^-(r)$. By Proposition~\ref{prop:res-of-spin} we have $\spin^+(r|_{\Gamma_0}) \simeq \spin_{2n-1}(r|_{\Gamma_0})$, where $\spin_{2n-1}$ denotes the spin representation of $\GSpin_{2n-1}$. By the proof of \cite[Prop. 5.1]{GSp}, $\spin_{2n-1}(r|_{\Gamma_0})$ decomposes as a direct sum $r_1^{e_1} \oplus \cdots \oplus r_k^{e_k}$, such that $k, e_i \in \Z_{\geq 1}$, $r_i$ is irreducible, and $\dim (r_i) \neq \dim (r_j)$ for $i \neq j$. Moreover, the projective image of $\spin_{2n-1}(r|_{\Gamma_0})$ in $\PGL_{2^{n-1}}$ is Zariski connected. This implies that the $r_i$ are strongly irreducible $\Gamma_0$-representations [Prop.~4.8(ii), \textit{loc. cit.}].

Suppose $\chi^-|_{\Gamma_0} \neq 1$. We again claim that \eqref{eq:Hom-is-zero} holds. To see this, assume $f \colon \chi^- \otimes \spin^-(r) \to \spin^-(r)$ is a non-trivial $\Gamma_0$-morphism. Since the $\dim(r_i)$ are distinct, the morphism $f$ induces an isomorphism from $\chi^- \otimes r_1$ to one of the copies of $r_1$ in $\spin^-(r)$. Since $r_1$ is strongly irreducible by [Prop. 4.8(i), \textit{loc. cit.}], this implies $\chi^-|_{\Gamma_0} = 1$. Thus $\chi^- \in \{t, 1\}$.

Arguing with $+$ instead of $-$ we find similarly $\chi^+ \in \{t, 1\}$. If $\chi^+ = 1$, then we have
$$
\spin^+(r) \oplus \chi^- \spin^-(r) \simeq \spin^+(r) \oplus \spin^-(r),
$$
which implies $\chi^- \spin^-(r) \simeq \spin^-(r)$, and thus $\chi^- = 1$ by (\textit{i}). By the same argument, if $\chi^- = 1$, then $\chi^+ = 1$. Thus $\chi^+ = \chi^-$. The statement follows.
\end{proof}

\begin{lemma}\label{lem:GPin-conjugacy}
Let $\gamma, \gamma' \in \GSpin_{2n}$ be two semi-simple elements. Then $\gamma, \gamma'$ are $\GPin_{2n}$-conjugate if and only if they are conjugate in the representations $\cN, \std$ and $\spin$.
\end{lemma}
\begin{proof}
This follows from \cite[\S 1]{GSp} and the fact that $\GPin_{2n}$ has $\{\std,\cN,\spin\}$ as a fundamental set in the sense thereof, which follows from the fact that $\{\std,\cN,\spin^+,\spin^-\}$ is a fundamental set for $\GSpin_{2n}$ as checked therein. 
\end{proof}

\begin{proposition}\label{prop:local-conjugacy-global-conjugacy}
Let $E/F$ be a quadratic extension of characteristic zero fields. Let $H$ be one of the following algebraic groups 
$$
\SO_{2n}, \GSpin_{2n}, \SO_{2n} \rtimes \Gamma_{E/F}, \GSpin_{2n} \rtimes \Gamma_{E/F},
$$
where $\Gamma_{E/F}$ acts through $\theta^\circ$ or $\theta$ in the semi-direct products. We write $H^0$ for the neutral component of $H$. Let 
$$
r_1,r_2:\Gamma_F \ra H(\lql)
$$
be semisimple Galois representations such that
\begin{itemize}
\item $r_1$ and $r_2$ are locally conjugate and
\item the Zariski closure of $r_1(\Gamma)$ contains a regular unipotent element.
\end{itemize}
Then $r_1$ and $r_2$ are $H^0$-conjugate.
\end{proposition}
\begin{proof}
For simplicity we abbreviate $H(\lql)$ as $H$ if there is no danger of confusion.

\smallskip

\noindent\textbf{The case $H = \SO_{2n}$.} Write $\li r_1,\li r_2:\Gamma\ra \PSO_{2n}$ for the projectivizations of $r_1,r_2$. Write $I_i$ for the Zariski closure of $\li r_i(\Gamma)$ in $\PSO_{2n}$, for $i=1,2$. Since $\OO_{2n}$ is acceptable~\cite[Prop.~B.1]{GSp}, $r_1$ and $r_2$ are conjugate by an element of $w\in \OO_{2n}$, i.e., $r_2=w r_1 w^{-1}$. In particular the Zariski closure of $r_2(\Gamma)$ also contains a regular unipotent element. We are done if $w\in \SO_{2n}$, so we may assume that $w\notin \SO_{2n}$ henceforth.

There are now three cases by Proposition~\ref{prop:containingRegularUnipotent}: either (A) $I_1=\PSO_{2n}$, (B) $I_1$ is $\SO_{2n-1}$, $\PGL_2$, or $G_2$ (the last case when $n=4$) or (C) $n=4$ and $I_1$ is $\ol{\spin}(\SO_7)$ or $\theta^\circ\ol{\spin}(\SO_7)$.

Case (A). Since $r_1$ has Zariski dense image in $\SO_{2n}$, there exists $\qq$ such that $r_1(\Frob_\qq)$ and $w r_1(\Frob_\qq) w^{-1}$ are not outer conjugate by Lemma \ref{lem:Chebotarev}. This contradicts $r_2(\Frob_\qq)=w r_1(\Frob_\qq) w^{-1}$ since $w\in \OO_{2n}\backslash \SO_{2n}$.

Case (B). The image of $r_1$ is contained in $i^\circ_{\tu{std}}(\SO_{2n-1})\times Z_{\SO_{2n}}$, which is centralized by $\vartheta^\circ\in \OO_{2n}\backslash \SO_{2n}$. Since $\vartheta^\circ$ and $w$ belong to the same $\SO_{2n}$-coset, it follows that $r_1$ and $r_2$ are $\SO_{2n}$-conjugate.

Case (C). Without loss of generality, we may assume $I_1=\ol{\spin}(\SO_7)$. Since $r_1=\Int(w)\circ r_2$ for $w\in \OO_{8}\backslash \SO_{8}$, we see that $I_2$ is $\PSO_8$-conjugate to $\theta^\circ\ol{\spin}(\SO_7)$. 

We claim that this case does not arise. By assumption $r_1$ and $r_2$ are locally conjugate representations with values in $\SO_8$. As $\textup{O}_8$ is acceptable, we may assume, after replacing $r_2$ with a suitable $\SO_8$-conjugate, that either (a) $r_1(\gamma) = r_2(\gamma)$ for all $\gamma \in \Gamma$, or (b) that $r_1(\gamma) = \theta^\circ r_2(\gamma)$ for all $\gamma \in \Gamma$.  Case (a) implies that $I_1=I_2$, so $\ol {\spin}(\SO_7)$ is $\PSO_{8}$-conjugate to $\theta^\circ \ol {\spin}(\SO_7)$, which contradicts Lemma \ref{lem:spin7}. In case (b), there exists a representation $\wt r \colon \Gamma \to \Spin_7$ such that $r_1(\gamma) = \spin(\wt r(\gamma))$, $r_2(\gamma) = \theta^\circ \spin(\wt r(\gamma))$. On the other hand, by Lemma \ref{lem:Chebotarev} there exists an $F$-place $\qq$ where $\wt r$ is unramified and $\wt r(\Frob_{\qq}) \in U$ with $U$ is as in Lemma \ref{lem:spin7}. Then $\spin(\wt r(\Frob_{\qq}))$ and $\theta^\circ \spin(\wt r(\Frob_{\qq}))$ are not $\SO_8$-conjugate. This contradicts the assumption that $r_1$ and $r_2$ are locally conjugate. The claim is proved.

\smallskip

\noindent\textbf{The case $H = \SO_{2n} \rtimes \Gamma_{E/F}$.} By the preceding case, we may assume that $r_1|_{\Gamma_E}=r_2|_{\Gamma_E}$. (Strictly speaking, we proved the $\SO_{2n}$-case for $\Gamma=\Gamma_F$, but the proof goes through without change for $\Gamma_E$.) Since $\SO_{2n} \rtimes \Gamma_{E/F}\simeq \OO_{2n}$ via $g\rtimes c\mapsto g\vartheta^\circ$, we identify the two groups. In particular $H$ is acceptable, so there exists $w\in \OO_{2n}$ such that $r_2=w r_1 w^{-1}$. We are done if $w\in \SO_{2n}$, so assume that $w\notin \SO_{2n}$. Depending on the projective image of $r_1|_{\Gamma_E}$, we have Cases (A), (B), (C) as above. The arguments there tell us that Cases (A) and (C) are impossible when $w\notin \SO_{2n}$. In Case (B), we know $r_1(\Gamma_E)$ is contained in $i^\circ_{\tu{std}}(\SO_{2n-1})\times Z_{\SO_{2n}}$. The normalizer of the latter in $\OO_{2n}$ is $\OO_{2n-1}\times \OO_1$ (embedded in $\OO_{2n}$ via $i^\circ_{\tu{std}}$), which is centralized by $\vartheta^\circ$. Hence if we write $w=w_0\vartheta^\circ$ with $w_0 \in \SO_{2n}$, then $r_2=w_0 r_1 w_0^{-1}$. Namely $r_1$ and $r_2$ are $H^0$-conjugate.

\smallskip

\noindent\textbf{The $\GSpin_{2n}$-case.} Write $r_1^\circ,r_2^\circ$ for the composition of $r_1,r_2$ with $\tu{pr}^\circ :\GSpin_{2n}\ra \SO_{2n}$. Then $r_1^\circ$ and $r_2^\circ$ are conjugate by the $\SO_{2n}$-case treated above. Hence we may assume that $r_2=\chi r_1$ with a continuous character $\chi:\Gamma\ra \lql^\times$, where $\lql^\times=\ker(\GSpin_{2n}\ra \SO_{2n})$ via Lemma \ref{lem:SurjectionOntoGSO} (ii). Since $r_1$ and $\chi\otimes r_1$ are locally conjugate by the initial assumption, we have
$$
\spin^\varepsilon(r_1)\simeq \spin^\varepsilon(\chi\otimes r_1)\simeq \chi\otimes \spin^\varepsilon(r_1),\qquad \varepsilon\in \{\pm1\}.
$$
It follows from Lemma \ref{lem:for-mult-one} that $\chi=1$.

\smallskip

\noindent\textbf{The $\GSpin_{2n} \rtimes \Gamma_{E/F}$-case.} By the $\GSpin_{2n}$-case above, we may assume that $r_1|_{\Gamma_E}=r_2|_{\Gamma_E}$.  Writing $r_i^\circ:=\tu{pr}^\circ \circ r_i$ for $i=1,2$, we have $r^\circ_1|_{\Gamma_E}=r^\circ_2|_{\Gamma_E}$. By the preceding argument, we deduce that $r^\circ_1=r^\circ_2$. On the other hand, $r_1|_{\Gamma_E}=r_2|_{\Gamma_E}$ implies that $r_1 = r_2$ or $r_1 = r_2\otimes \chi$ by Example \ref{ex:GSpin2n-extend}, with $\chi$ as in that example. If $r_1=r_2$ then we are done so suppose $r_1= r_2\otimes \chi$. Then $r^\circ_1 = r^\circ_2\otimes \chi_{E/F}$ for $\chi_{E/F}:\Gamma_F\twoheadrightarrow \Gamma_{E/F}=\{\pm 1\}$. Set $R_i:=\std\circ r^\circ_i$ for $i=1,2$, so that $R_1 = R_2\otimes \chi_{E/F}$. Since $r_1$ and $r_2$ are locally conjugate, the $\GL_{2n}$-valued representations $R_1$ and $R_2$ are locally conjugate and thus conjugate. So $R_1\simeq R_1\otimes \chi_{E/F}$. By \cite[Lem.~4.8]{GSp}, $R_1$ is not strongly irreducible. Considering the projective image of $r_1|_{\Gamma_E}$ as in the $\SO_{2n}$-case above, we see that Case (A) is excluded and only Case (B) or (C) occurs. In either case, again because $R_1$ is not strongly irreducible, the only possibility is that $R_1|_{\Gamma_E}$ decomposes into two strongly irreducible representations of dimensions $(2n-1)$ and 1. Then it is easy to see that $R_1=R'_1\oplus R''_1$ already on $\Gamma_F$, with strongly irreducible $R'_1$ and $R''_1$ of dimensions $(2n-1)$ and 1. It follows from $R_1\simeq R_1\otimes \chi_{E/F}$ that $R'_1\simeq R'_1\otimes \chi_{E/F}$ (and similarly for $R''_1$), but this contradicts strong irreducibility of $R'_1$ \cite[Lem.~4.8]{GSp}.
\end{proof}

\section{On $\SO_{2n}$-valued Galois representations}\label{sect:GSO-valued-Galois}

In this section we construct Galois representations associated with automorphic representations of even orthogonal groups over a totally real field $F$. More precisely, we will derive a weaker version of Conjecture \ref{conj:BuzzardGee} for such groups from the literature. Let either
\begin{itemize}
\item $E=F$, or
\item $E$ be a CM quadratic extension of $F$.
\end{itemize}
 In the latter case write $c$ for the nontrivial element of $\Gamma_{E/F}:=\Gal(E/F)$.
Write $\SO_{2n}^{E/F}$ for the split group $\SO_{2n}$ if $E=F$, and the quasi-split outer form of $\SO_{2n}$ over $F$ relative to $E/F$ otherwise. To be precise, in the latter case,
 \begin{equation}\label{eq:QuasiSplitSO}
 \OO^{E/F}_{2n}(R) := \{g \in \GL_{2n}(E \otimes_F R)\ |\ c(g) = \vartheta^\circ g \vartheta^\circ, g^t \vierkant 0{1_n}{1_n}0 g = \vierkant 0{1_n}{1_n}0\}
\end{equation}
for $F$-algebras $R$, and $\SO^{E/F}_{2n}$ is the connected component where $\det(g)=1$.
We can extend the standard embedding $\std:\SO_{2n}(\C)\hra \GL_{2n}(\C)$ to a map (still denoted $\std$)
\begin{equation}\label{eq:std-rep-L-gp}
\std:{} ^L( \SO^{E/F}_{2n})=\SO_{2n}(\C)\rtimes \Gamma_{E/F} \hra \GL_{2n}(\C),
\end{equation}
whose image is $\SO_{2n}(\C)$ if $E=F$ and $\OO_{2n}(\C)$ if $E\neq F$. More precisely, when $E\neq F$, we fix the extended map $\std$ by requiring $c\mapsto \vartheta^\circ$. (We defined $\OO_{2n}$ explicitly in the last section, and $\vartheta^\circ$ was given in \eqref{eq:Elementw}.)

Let $\pi^\flat$ be a cuspidal automorphic representation of $\SO^{E/F}_{2n}(\A_F)$. The following will be key assumptions on $\pi^\flat$. (Recall from \S\ref{sect:Notation} that $ \St_{\SO,\qst}$ denotes the Steinberg representation.)
\begin{itemize}
\item[\textbf{(coh$^\circ$)}] $\pi^\flat_\infty$ is cohomological for an irreducible algebraic representation $\xi^\flat = \otimes_{y \in \cV_\infty} \xi_y^\flat$ of $\SO^{E/F}_{2n, F \otimes \C}$.
\item[\textbf{(St$^\circ$)}] There exists a prime $\qst$ of $F$ such that $\pi_{\qst}^\flat \simeq \St_{\SO,\qst}$ up to a character twist.
\end{itemize}
Condition (coh$^\circ$) implies that the infinitesimal character of $\xi^\flat_y$ is given by $\rho_{\SO}+\lambda(\xi_y^\flat)$ at each $y\in \cV_\infty$; see \cite[Thm.~I.5.3]{BorelWallach}. In particular $\pi^\flat$ is C-algebraic in the sense of Buzzard--Gee \cite[Lem.~7.2.2]{BuzzardGee}, thus also L-algebraic as the half sum of positive (co)roots is integral for $\SO^{E/F}_{2n}$. In (St$^\circ$), characters of $\SO^{E/F}_{2n}(F_{\qst})$ are exactly the characters factoring through the cokernel of $\Spin^{E/F}_{2n}(F_\qst)\ra \SO^{E/F}_{2n}(F_\qst)$. (This is a special case of the general fact \cite[Cor.~2.3.3]{H0main}.) Such characters are in a natural bijection with characters of $F_\qst^\times/(F_\qst^\times)^2 \simeq H^1(F_{\qst}, \{\pm 1\})$. 

Write $\TSO := \TGSO \cap \SO_{2n}$ over $\C$ and choose the Borel subgroup containing $T_{\SO}$ in $\SO^{E/F}_{2n}$ as in the preceding section. For each $y\in \cV_\infty$, the highest weight of $\xi_y^\flat$ gives rise to a dominant cocharacter $\lambda(\xi_y^\flat)\in X_*(\TSO)$. Let $\phi_{\pi^\flat_y}:W_{F_y}\ra {}^L \SO_{2n}^{E/F}$ denote the $L$-parameter of $\pi^\flat_y$ assigned by \cite{LanglandsRealClassification}. Recall $\std:\SO_{2n}\hra \GL_{2n}$ denotes the standard embedding. We also consider the following conditions:
\begin{itemize}[leftmargin=+1in]
\item[\textbf{(std-reg$^\circ$)}] $\std\circ \phi_{\pi^\flat_y}|_{W_{\li{F}_y}}$ is regular (i.e., the centralizer group in $\GL_{2n}(\C)$ is a torus) for every $y\in \cV_\infty$.
\item[\textbf{(disc-$\infty$)}] If $n$ is odd then $[E:F]=2$. If $n$ is even then $E=F$.
\end{itemize}
Since $E$ is either $F$ or a CM quadratic extension of $F$, condition (disc-$\infty$) is equivalent to requiring $\SO^{E/F}_{2n}(F_y)$ to admit discrete series at all infinite places $y$ of $F$ (or equivalently, to admit compact maximal tori). Condition (std-reg$^\circ$), when (coh$^\circ$) is imposed, amounts to requiring that $\std \circ (\rho_{\SO} +\lambda(\xi^\flat_y))$ should be a regular cocharacter of $\GL_{2n}$ for every $y\in \cV_\infty$, since $\phi_{\pi^\flat_y}|_{W_{\li{F}_y}}$ encodes the infinitesimal character of $\pi_y$ according to \cite[Prop.~7.4]{VoganLLC}.

When $v$ is a prime of $F$, write $\phi_{\pi^\flat_v} \colon W_{F_\qq} \ra {}^L \SO^{E/F}_{2n}$ for the $L$-parameter of $\pi^\flat_v$ as given by \cite[Thm 1.5.1]{ArthurBook}. (By the Langlands quotient theorem, $\pi^\flat_v$ is the unique quotient of an induced representation from a character twist of a tempered representation on a Levi subgroup. Apply Arthur's theorem to this tempered representation.) Note that $\phi_{\pi^\flat_v}$ is well-defined only up to $\OO_{2n}(\C)$-conjugacy in \emph{loc.~cit.} (This does not matter for the statement of (SO-i) in Theorem \ref{thm:ExistGaloisSO} below.)

Let $\tu{Unr}(\pi^\flat)$ denote the set of finite primes $\qq$ of $F$ such that $\qq$ is unramified in $E$ and $\pi^\flat_\qq$ is unramified. In this case, the unramified $L$-parameter $\phi_{\pi^\flat_\qq}$ is determined (up to $\mathrm{SO}_{2n}(\C)$-conjugacy, not just up to outer automorphism) by the Satake isomorphism.

Thanks to Arthur, we can lift $\pi^\flat$ to an automorphic representation of $\GL_{2n}$ as follows. The Hecke character $F^\times\backslash \A_F^\times \ra \{\pm1\}$ corresponding to the Galois character $\chi_{E/F}:\Gamma_F\twoheadrightarrow \Gamma_{E/F}=\{\pm1\}$ is still denoted $\chi_{E/F}$. Let $\chi_{E/F,\qq}$ denote its local component at $\qq$.

\begin{proposition}[Arthur]\label{prop:SO-to-GL}
Assume that $\pi^\flat$ satisfies (St$^\circ$). Then there exists a self-dual automorphic representation $\pi^\sharp$ of $\GL_{2n}(\A_F)$, which is either cuspidal or the isobaric sum of two cuspidal self-dual representations of $\GL_{2n-1}(\A_F)$ and $\GL_1(\A_F)$, such that
\begin{enumerate}[label=(Ar\arabic*)]
\item $\pi^\sharp_\qq$ is unramified at every $\qq \in \tu{Unr}(\pi^\flat)$.
\item $\pi^\sharp_{\qst} \simeq \tu{St}_{2n-1} \boxplus \chi_{E/F,\qst}$ up to a quadratic character of $\GL_{2n}(\A_F)$.
\item $\phi_{\pi^\sharp_v} \simeq \std \circ \phi_{\pi^\flat_v}$ at every $F$-place $v$.
\end{enumerate}
If $\pi^\flat$ satisfies both (St$^\circ$) and (coh$^\circ$) then we furthermore have
\begin{enumerate}
\item[(Ar4)] $\pi^\sharp_y$ and $\pi^\flat_y$ are tempered for all infinite $F$-places $y$.
\end{enumerate}
If $\pi^\flat$ has properties (coh$^\circ$), (St$^\circ$), and (std-reg$^\circ$), then the following strengthening holds:
\begin{enumerate}
\item[(Ar4)+] $\pi^\sharp_v$ and $\pi^\flat_v$ are tempered for all $F$-places $v$.
\end{enumerate}
\end{proposition}

\begin{remark}
In fact (Ar1) is implied by (Ar3) since $\phi_{\pi^\flat_\qq}$ is an unramified parameter at every $\qq \in \tu{Unr}(\pi^\flat)$, but we state (Ar1) to make (SO-ii) below more transparent.
\end{remark}

\begin{proof}
Consider $\pi^\flat$ satisfying (St$^\circ$). For notational convenience, we assume $\pi_{\qst}^\flat \simeq \St_{\SO,\qst}$ (not just up to a quadratic character twist) as the general case works in the same way. By \cite[Thm.~1.5.2]{ArthurBook} (using the notation there),\footnote{E.g., $\tilde\Phi(\SO^{E/F}_{2n,F_\qq})$ means the set of isomorphism classes of $L$-parameters for $G=\SO^{E/F}_{2n,F_\qq}$ \emph{modulo the action of the outer automorphism group} $\widetilde{\tu{Out}}_{2n}(G)$ as defined in \cite[1.2]{ArthurBook}. Similarly the packet $\tilde \Pi(\psi_{\qq})$ of \cite[1.5]{ArthurBook} consists of finitely many $\widetilde{\tu{Out}}_{2n}(G)$-orbits of isomorphism classes of representations of $G(F_{\qq})$. By abuse of terminology, a representation will often mean the outer automorphism orbit of representations in this proof.} 
we have a formal global parameter $\psi$ (as in \cite[1.4]{ArthurBook}) such that $\pi^\flat_v$ appears as a subquotient of a member of the packet $\tilde\Pi(\psi_v)$ at every place $v$ of $F$, where $\psi_v$ denotes the localization of $\psi$ at $v$ as in \emph{loc.~cit.} (A priori, members of $\tilde\Pi(\psi_v)$ may be reducible due to possible failure of the generalized Ramanujan conjecture. Only from the argument below it follows that $\psi$ is a generic parameter, i.e., its $\SU(2)$-part is trivial. Then $\tilde\Pi(\psi_v)$ consists of irreducible representations by \cite[Appendix A]{BinXuLPackets}.)

Since $\St_{\SO,\qst}$ appears as a subquotient of a member of $\tilde\Pi(\psi_{\qst})$, Proposition \ref{prop:A-parameter-1-St} implies that $\psi_{\qst}\simeq \psi_{\St,\qst}$, where $\psi_{\St}$ is defined above that proposition. Thus
\begin{equation}\label{eq:psi-at-vst}
\psi_{\St,\qst}\simeq \psi_{\St_{2n-1},\qst}\oplus \chi_{E/F,\qst},
\end{equation}
where $\psi_{\St_{2n-1},\qst}$ (resp.~$\psi_{\mathbf 1,\qst}$) denotes the $A$-parameter for the Steinberg (resp.~trivial) representation $\St_{2n-1}$ of $\GL_{2n-1}(F_{\qst})$ (resp.~ $\GL_{1}(F_{\qst})$). It follows that either $\psi=\pi^\#$ or $\psi=\pi^\#_1\boxplus \pi^\#_2$, where $\pi^\#$, $\pi^\#_1$, and $\pi^\#_2$ are cuspidal self-dual automorphic representations of $\GL_{2n}(\A_F)$, $\GL_{2n-1}(\A_F)$, and $\GL_1(\A_F)$, respectively. In the second case, we take $\pi^\#$ to be the isobaric sum of $\pi^\#_1$ and $\pi^\#_2$. Now (Ar2) follows from \eqref{eq:psi-at-vst}. We define $\phi_{\qq}\in \tilde\Phi(\SO^{E/F}_{2n,F_\qq})$ as the restriction of $\psi_{\qq}\in \tilde\Psi(\SO^{E/F}_{2n,F_\qq})$ from $\cL_{F_\qq}\times \SU(2)$ to $\cL_{F_\qq}$. Then Properties (Ar1) and (Ar3) with $\phi_\qq$ in place of $\phi_{\pi_\qq^\flat}$ are part of Arthur's result already cited.

To complete the proof of (Ar1) and (Ar3), it suffices to verify that $\phi_v = \phi_{\pi_v^\flat}$ in $\tilde\Phi(G_{F_v})$. In the notation of \cite{ArthurBook} (between Theorems 1.5.1 and 1.5.2), $\phi_v$ gives rise to 
\begin{itemize}
    \item a $F_v$-rational parabolic subgroup $P_v\subset G_{F_v}$ with a Levi factor $M_v$,
    \item a bounded parameter $\phi_{M_v}\in \tilde\Phi(M_v)$,
    \item a point $\lambda$ in the open chamber for $P_v$ in $X_*(M_v)_{F_v}\otimes_\Z \R$,
\end{itemize}
such that $\phi_v$ comes from the $\lambda$-twist $\phi_{M_v,\lambda}$ of $\phi_{M_v}$. (This is the counterpart of the Langlands quotient construction for $L$-parameters.) The statement of \cite[Thm.~1.5.2]{ArthurBook} tells us that $\pi^\flat_v$ is a subrepresentation of the normalized induction $\Ind^{G(F_v)}_{P_v(F_v)}(\sigma_{v,\lambda})$ for some $\sigma_v\in \tilde\Pi(M_v)$, where $\sigma_{v,\lambda}$ denotes the $\lambda$-twist of $\sigma_v$, since $\pi^\flat_v$ appears in the packet of $\psi_v$ in \emph{loc.~cit.} According to the same theorem, $\Ind^{G(F_v)}_{P_v(F_v)}(\sigma_{v,\lambda})$ must be completely reducible since it appears in the $L^2$-discrete spectrum. This means that  $\pi^\flat_v$ is irreducible and the Langlands quotient of $\Ind^{G(F_v)}_{P_v(F_v)}(\sigma_{v,\lambda})$ (thus $\pi^\flat_v$ is isomorphic to the latter). Since the formation of Langlands parametrization is compatible with the Langlands quotient, it follows that $\phi_v$ is the $L$-parameter of $\pi_v^\flat$, namely that $\phi_v=\phi_{\pi_v^\flat}$.

It remains to check (Ar4) and (Ar4)+. Assume (coh$^\circ$) in addition to (St$^\circ$). Thanks to (Ar3), $\pi^\#$ is $L$-algebraic since L-algebraicity is preserved by $\std$. Applying  \cite[Lem.~4.9]{ClozelAnnArbor} to $\pi^\#\otimes|\det|^{1/2}$ if $\pi^\#$ is cuspidal, and $\pi^\#_1$ and $\pi^\#_2$ otherwise, to deduce that $\pi^\#_v$ is essentially tempered at all $y|\infty$. Since $\pi^\#$ is self-dual, $\pi^\#_y$ are a fortiori tempered. Now suppose furthermore that $\pi^\flat_y$ has property (std-reg$^\circ$). Then $\pi^\#$ is regular $L$-algebraic. Arguing as above but applying \cite[Thm.~1.2]{Caraiani-Ramanujan} to $\pi^\#$ at finite places, in place of  \cite[Lem.~4.9]{ClozelAnnArbor} at infinite places, we deduce (Ar4)+. Finally, whenever $\pi^\#_v$ is tempered (for finite or infinite $v$), this implies that $\psi_v$ is bounded, hence that $\pi^\flat_v$ is tempered by \cite[Thm.~1.5.1]{ArthurBook}.
\end{proof}

\begin{corollary}\label{cor:disc-series-at-infty}
Assume (disc-$\infty$). If $\pi^\flat$ satisfies (St$^\circ$) and (coh$^\circ$) then $\pi^\flat_y$ is a discrete series representation for every infinite place $y$.
\end{corollary}

\begin{proof}
The condition (disc-$\infty$) guarantees that $\SO_{2n}^{E/F}(F_y)$ contains an elliptic maximal torus at infinite places $y$, so that it admits discrete series. In this case, a tempered $\xi$-cohomological representation is a discrete series representation by \cite[Thm.~III.5.1]{BorelWallach}. Thus the corollary follows from (Ar4) of the preceding proposition.
\end{proof}

Under the assumptions of the corollary, let us describe $\phi_{\pi^\flat_y}|_{W_{\ol F_y}}$ explicitly. Fix an $\R$-isomorphism $\ol F_y\simeq \C$ once and for all, so that we can identify $W_{\ol F_y}=\C^\times$. We noted that the infinitesimal character of $\xi^\flat_y$ is $\rho_{\SO}+\lambda(\xi_y^\flat)$. The half sum of positive coroots $\rho_{\SO}\in X_*(\TSO)$ is equal to $(n-1)e_1+(n-2)e_2+\cdots+e_{n-1}$. It follows from the construction of discrete series $L$-packets in \cite[p.134]{LanglandsRealClassification} that possibly after $\SO^{E/F}_{2n}(\C)$-conjugation, we have
\begin{equation}\label{eq:L-parameter-SO-infty}
\phi_{\pi^\flat_y}(z) = (z/\ol z)^{\rho_{\SO}+\lambda(\xi_y^\flat)},\quad z\in W_{\ol F_y}.
\end{equation}

Continue to assume (St$^\circ$) and (coh$^\circ$) for $\pi^\flat$ as well as (disc-$\infty$). We noted that $\pi^\flat$ is L-algebraic thanks to (coh$^\circ$). Then Conjecture \ref{conj:BuzzardGee} predicts the existence of an $^L \SO^{E/F}_{2n}$-valued Galois representation attached to $\pi^\flat$. When (std-reg$^\circ$) is also assumed, Theorem \ref{thm:ExistGaloisSO} below proves the conjecture modulo outer automorphisms in that (SO-i) is weaker than what is predicted. (This is to be upgraded by (SO-i+) in \S\ref{sect:refinement}; also see Remark \ref{rem:SO-i-vs-SO-i+}.) The proof is carried out by reducing to the known results for $\pi^\#$ on $\GL_{2n}$. We will get to the theorem after observing that (disc-$\infty$) is automatically satisfied under the additional hypothesis (std-reg$^\circ$); this observation is related to (SO-v) of the theorem.

\begin{lemma}\label{lem:disc-infty-automatic}
Suppose there exists a cuspidal automorphic representation $\pi^\flat$ of $\SO^{E/F}_{2n}(\A_F)$ such that (St$^\circ$),(coh$^\circ$), and (std-reg$^\circ$) hold. Then (disc-$\infty$) is satisfied.\footnote{We heartily thank the referee for pointing out this lemma and explaining its proof.} 
\end{lemma}

\begin{proof} 
For each $y\in \cV_\infty$, Proposition~\ref{prop:SO-to-GL} tells us that $\phi_{\pi^\flat_y}$ is tempered. From this and (coh$^\circ$), we obtain a decomposition of the form 
$$
\std\circ \phi_{\pi^\flat_y}=\bigoplus_{i\in I} \Ind^{W_{F_y}}_{W_{\ol F_y}} \chi_{a_i} \oplus \bigoplus_{i'\in I'} \omega_{i'},\quad a_i\in \Z_{\ge 0}
$$
where $\chi_{a_i}: \C^\times \ra \C^\times$ is given by $z\mapsto (z/\ol z)^{a_i} $ using the identification $W_{\ol F_y}=\C^\times$ above, and $\omega_{i'}$ is a quadratic character of $W_{F_y}$. (In fact, $a_i$ are mutually distinct.) By the dimension reason $|I'|$ is even. On the other hand, (std-reg$^\circ$) implies that $|I'|\le 1$. Hence $I'$ is empty and $|I|=n$. Now the image of $j\in W_{F_y}$ in $\GL_{2n}(\C)$ under $\std\circ \phi_{\pi^\flat_y}$ has determinant $(-1)^n$ since the determinant of $j$ is $-1$ in each induced representation. In view of \eqref{eq:std-rep-L-gp}, we deduce that $E=F$ if $n$ is even and $E\neq F$ otherwise. That is, (disc-$\infty$) holds true.
\end{proof}

\begin{theorem}\label{thm:ExistGaloisSO}
Let $\pi^\flat$ be a cuspidal automorphic representation of $\SO^{E/F}_{2n}(\A_F)$ satisfying (coh$^\circ$), (St$^\circ$), and (std-reg$^\circ$). Then there exists a semisimple Galois representation (depending on $\iota$) 
$$
\rho_{\pi^\flat} = \rho_{\pi^\flat, \iota} \colon \Gamma_F \to \SO_{2n}(\lql)\rtimes \Gamma_{E/F},
$$
whose restriction to $\Gamma_{F_\qq}$ at every $F$-place $\qq|\ell$ is potentially semistable, such that the following hold. Here $\osim$ means $\OO_{2n}(\lql)$-conjugacy.
\begin{enumerate}[label=(SO-\roman*)]
\item For every finite $F$-place $\qq$ (including $\qq|\ell$), in the convention of \S\ref{sect:Notation}, we have
$$
\iota \phi_{\pi^\flat_\qq} \osim \textup{WD}(\rho_{\pi^\flat}|_{\Gamma_{F_\qq}})^{\textup{F-ss}}.
$$
\item Let $\qq \in \tu{Unr}(\pi^\flat)$. If $\qq \nmid \ell$ then $\rho_{\pi^\flat, \qq}$ is unramified at $\qq$, and for all eigenvalues $\alpha$
of $\std(\rho_{\pi^\flat}(\Frob_\qq))_\tu{ss}$ and all embeddings $\lql \hookrightarrow \C$ we have $|\alpha | = 1$.
\item For each $\qq | \ell$, and for each $y \colon F \hra \C$ such that $\iota y$ induces $\qq$, we have
$\mu_{\HT}(\rho_{\pi^\flat,\qq},\iota y) \osim \iota\mu_{\Hodge}(\xi^\flat, y)$.
\item If $\pi^\flat_\qq$ is unramified at $\qq | \ell$, then $\rho_{\pi^\flat, \qq}$ is crystalline. If $\pi^\flat_\qq$ has a non-zero Iwahori fixed vector at $\qq | \ell$, then $\rho_{\pi^\flat, \qq}$ is semistable.
\item $\rho_{\pi^\flat}$ is totally odd. More explicitly, for each real place $y$ of $F$ and the corresponding complex conjugation $c_y\in \Gamma_F$ (well-defined up to conjugacy),
$$
\rho_{\pi^\flat}(c_y)\sim \begin{cases}
  \diag(\underbrace{1,\ldots,1}_{n/2},\underbrace{-1,\ldots,-1}_{n/2},\underbrace{1,\ldots,1}_{n/2},\underbrace{-1,\ldots,-1}_{n/2}), & n:\mbox{even},\\
  \diag(\underbrace{1,\ldots,1}_{(n-1)/2},\underbrace{-1,\ldots,-1}_{(n-1)/2},1,\underbrace{1,\ldots,1}_{(n-1)/2},\underbrace{-1,\ldots,-1}_{(n-1)/2},1) \rtimes c,
  & n:\mbox{odd}.
\end{cases}.
$$
\end{enumerate}
Condition (SO-i) characterizes $\rho_{\pi^\flat}$ uniquely up to $\OO_{2n}(\lql)$-conjugation.
\end{theorem}

\begin{remark}\label{rem:complex-conj-SO}
Since $\pi^\flat_\infty$ is a discrete series representation, the conjugation by $\phi_{\pi^\flat_y}(j)$ on $\TSO$ is the inverse map, where $j$ denotes the usual element of the real Weil group. Thus (SO-v) and \eqref{eq:L-parameter-SO-infty} imply Buzzard--Gee's prediction on the image of complex conjugation in \cite[Conj. 3.2.1, 3.2.2]{BuzzardGee}. When $n$ is odd, we also observe that (SO-v) is equivalent to
$$
\rho_{\pi^\flat}(c_y)\sim \diag(\underbrace{1,\ldots,1}_{(n-1)/2},\underbrace{-1,\ldots,-1}_{(n-1)/2},a,\underbrace{1,\ldots,1}_{(n-1)/2},\underbrace{-1,\ldots,-1}_{(n-1)/2},a^{-1}) \rtimes c,\quad \forall a\in \lql^\times.
$$
\end{remark}

\begin{remark}
Without (St$^\circ$), an analogous theorem can be proved only under (coh$^\circ$) and (std-reg), but in a weaker and less precise form. The strategy is similar: transfer $\pi^\flat$ to a regular algebraic automorphic representation of $\GL_{2n}(\A_F)$, which is an isobaric sum of cuspidal self-dual automorphic representations, and apply the known results on associating Galois representations.
\end{remark}

\begin{proof}[Proof of Theorem \ref{thm:ExistGaloisSO}]
Let $\pi^\#$ be as in Proposition \ref{prop:SO-to-GL} so that
\begin{enumerate}
\item[Case 1:] $\pi^\#$ is cuspidal, or
\item[Case 2:] $\pi^\#=\pi^\#_1\boxplus \pi^\#_2$, with $\pi^\#_1$ (resp.~$\pi^\#_2$) a cuspidal automorphic representation of $\GL_{2n-1}(\A_F)$ (resp.~$\GL_1(\A_F)$).
\end{enumerate}
As in the proof there, we know that $\pi^\#$ is $L$-algebraic.

In Case 1, consider the C-algebraic twist $\Pi := \pi^\sharp \otimes |\det|^{(1 - 2n)/2}$, which is regular by (std-reg), and essentially self-dual (``essentially'' means up to a character twist). Applying the well-known construction of Galois representations (see \cite[Thm.~2.1.1]{BLGGT14-PA} for a summary and further references) to $\Pi$, we obtain a semisimple Galois representation (recall $\Gamma=\Gamma_F$ by convention)
$$
\rho_{\Pi} \colon \Gamma \to \GL_{2n}(\lql),
$$
satisfying the obvious analogues of properties (SO-i) through (SO-iv) for $\GL_{2n}$, with $\rho_{\Pi}$ and $\GL_{2n}$ in place of $\rho_{\pi^\flat}$ and $\OO_{2n}$; call these analogues (GL-i), \ldots, (GL-iv). By `obvious', we mean for instance that (GL-ii) is about the eigenvalues of $\rho_{\Pi}(\Frob_\qq)$ having absolute value 1. We also spell out (GL-i), which states that
\begin{equation}\label{eq:rhoPi-LGC}
\iota \phi_{\Pi_{\qq}\otimes |\det|_{\qq}^{(1-2n)/2}} \sim \textup{WD}(\rho_{\Pi}|_{\Gamma_{F_\qq}})^{\textup{F-ss}},\quad \qq \nmid \ell.
\end{equation}
In particular, for all $\qq\in \tu{Unr}(\pi^\flat)$, since $\Pi_\qq$ is unramified by (Ar1), we see that $\rho_\Pi$ is unramified at $\qq$ as well and that
\begin{equation}\label{eq:rhoPi-LGC-unram}
\rho_{\Pi}(\Frob_\qq)_{\tu{ss}} \sim \iota \phi_{\Pi_\qq \otimes |\det|_{\qq}^{(1-2n)/2}} (\Frob_\qq)  \sim
\iota \phi_{\pi^\#_\qq} (\Frob_\qq)\sim \iota \std (\phi_{\pi^\flat_\qq}(\Frob_\qq)).
\end{equation}
Since each $\pi^\#_\qq$ is self-dual, we see that $\rho_{\Pi}$ is self-dual. By (Ar2) and \eqref{eq:rhoPi-LGC} at $\qq =\qst$ as well as semisimplicity of $\rho_\Pi$, we see that either
\begin{itemize}
\item $\rho_{\Pi}$ is irreducible, or
\item $\rho_{\Pi}=\rho_1\oplus \rho_2$ for self-dual irreducible subrepresentations $\rho_1$ and $\rho_2$ with $\dim \rho_1 = n-1$ and $\dim \rho_2 = 1$.
\end{itemize}
Either way, it follows from \cite[Cor.~1.3]{BellaicheChenevierSign} that every irreducible constituent of $\rho_{\Pi}$ is orthogonal in the sense of \emph{loc. cit.} (As we are in Case 1, apply their corollary with $\eta = |\cdot|^{2n-1}$, in which case $\eta_\lambda(c) = -1$ in their notation.)

Now we turn to Case 2. Take $\Pi_1 := \pi^\#_1|\det|^{1-n}$ and $\Pi_2 := \pi^\#_2$. Each of $\Pi_1$ and $\Pi_2$ is cuspidal, regular $C$-algebraic, and essentially self-dual, so the same construction yields $\rho_{\Pi_1}$ and $\rho_{\Pi_2}$, which are $2n-1$ and $1$-dimensional, respectively. Then put $\rho_{\Pi} := \rho_{\Pi_1} \oplus \rho_{\Pi_2}$. As before, (GL-i), \ldots, (GL-iv) hold true for $\rho_\Pi$. Moreover an argument as in Case 1 shows that $\rho_{\Pi_1}$ and $\rho_{\Pi_2}$ are self-dual and orthogonal.  It follows from (Ar2) and \eqref{eq:rhoPi-LGC} at $v=\qst$ that $\rho_{\Pi_1}$ and $\rho_{\Pi_2}$ are irreducible. 

From here on, we treat the two cases together. Since $\rho_\Pi$ is self-dual and orthogonal, after conjugating $\rho_{\Pi}$ by an element of $\GL_{2n}(\lql)$, we can ensure that $\rho_\Pi(\Gamma)\subset \OO_{2n}(\lql)$. Write
$$
\rho_{\pi^\flat} \colon \Gamma \ra \OO_{2n}(\lql)
$$
for the $\OO_{2n}(\lql)$-valued representation that $\rho_{\Pi}$ factors through. (In case $\rho_{\Pi}$ is reducible, we even have $\rho_\Pi(\Gamma)\subset (\OO_{2n-1}\times \OO_1)(\lql)$.) Let us check that this is the desired Galois representation and deduce properties (SO-i) through (SO-v) from (GL-i) through (GL-iv).

We start with the case $E=F$. Then $\phi_{\pi^\flat_v}(\Frob_\qq)\in \SO_{2n}(\C)$ in \eqref{eq:rhoPi-LGC-unram}, so we deduce via the Chebotarev density theorem that $\rho_{\pi^\flat}$ has image in $\SO_{2n}(\lql)$. Note that (GL-ii) is the same statement as (SO-ii). The Hodge-theoretic properties at $\ell$ in (SO-iii) and (SO-iv) may be checked after composing with a faithful representation, so these properties hold. One sees from \cite[Appendix B]{GSp} (for $\OO_{2n}$) that (GL-i) implies (SO-i). (Alternatively, one can appeal to \cite[Thm.~8.1]{GGP12}.) The assertion on the cocharacters in (SO-iii) also follows (GL-iii) that the two cocharacters become conjugate in $\GL_{2n}$. 

We now prove (SO-v), namely that $\rho_{\pi^\flat}$ is totally odd. The following claim 
\begin{itemize}
\item[(E)] The element $\std(\rho_{\pi^\flat}(c_y))\in \GL_{2n}(\lql)$ has eigenvalues $1$ and $-1$ with multiplicity $n$ each, for every $y\in \cV_\infty$. 
\end{itemize}
follows from \cite{CaraianiLeHung} (in fact it can also be deduced from Ta\"\i bi's theorem \cite[Thm.~6.3.4]{TaibiEigenvariety} when $\pi^\sharp$ is cuspidal, and from Taylor \cite[Prop.~A]{TaylorOdd} when $\pi^\sharp$ is not cuspidal).

As $\rho_{\pi^\flat}(c_y)\in \SO_{2n}(\lql)$ has order 2, we have
$$
\rho_{\pi^\flat}(c_y)\sim \diag(\underbrace{1,\ldots,1}_{a_y},\underbrace{-1,\ldots,-1}_{b_y},\underbrace{1,\ldots,1}_{a_y},\underbrace{-1,\ldots,-1}_{b_y}),\quad a_y+b_y=n, \quad a_y,b_y\in \Z_{\ge0}.
$$
So (E) implies that $a_y=b_y$; this is possible as $n$ is even, which follows from Lemma \ref{lem:disc-infty-automatic} and the running assumption that $E=F$. Now one computes the adjoint action of $\rho_{\pi^\flat}(c_y)$ on $\Lie \SO_{2n}(\lql)$ to be $-n$. (A similar computation is done in the proof of \cite[Lem.~1.9]{GSp} for $\GSp_{2n}$.) Thus $\rho_{\pi^\flat}$ is totally odd.

It remains to treat the case $E\neq F$. In this case, the standard embedding $\SO_{2n}(\lql)\rtimes \Gamma_{E/F} \hra \GL_{2n}$ identifies $\SO_{2n}(\lql)\rtimes \Gamma_{E/F} \isom \OO_{2n}(\lql)$. The composition of $\rho_{\pi^\flat}$ with this isomorphism is still to be denoted by $\rho_{\pi^\flat}$. Since $\phi_{\pi^\flat_\qq}(\Frob_\qq) \in \OO_{2n}(\C)\backslash \SO_{2n}(\C)$ (resp.~$\phi_{\pi^\flat_y}(\Frob_\qq)\in \SO_{2n}(\C)$) in \eqref{eq:rhoPi-LGC-unram} when $\qq$ is inert (resp.~split) in $E$ by the unramified Langlands correspondence, we see that
$$
\rho_{\pi^\flat} \colon \Gamma \ra \SO_{2n}(\lql)\rtimes \Gamma_{E/F}
$$
commutes with the natural projections onto $\Gamma_{E/F}$. (By continuity it suffices to check the commutativity on Frobenius conjugacy classes.) Thus $\rho_{\pi^\flat}$ is a Galois representation valued in ${}^L (\SO^{E/F}_{2n})$. Properties (SO-i) through (SO-iv) follow from (GL-i) through (GL-iv) in the same way as for the $E=F$ case. 
 
We now prove (SO-v). The argument for claim (E) still applies, and since $n$ is odd by Lemma \ref{lem:disc-infty-automatic}, we have
\begin{eqnarray}
\std\rho_{\pi^\flat}(c_y) &\sim& \diag(\underbrace{1,\ldots,1}_{n},\underbrace{-1,\ldots,-1}_{n})  \\
&\sim& \diag(\underbrace{1,\ldots,1}_{\frac{n-1}{2}},\underbrace{-1,\ldots,-1}_{\frac{n+1}{2}},\underbrace{1,\ldots,1}_{\frac{n-1}{2}},\underbrace{-1,\ldots,-1}_{\frac{n+1}{2}})\cdot \std(c) \quad\mbox{in}~~\GL_{2n}(\lql).
\end{eqnarray}
(Recall that $\std(c)=\vartheta^\circ$ is the $2n\times 2n$ permutation matrix switching $n$ and $2n$.) Therefore 
$$
\rho_{\pi^\flat}(c_y)~\sim~\diag(\underbrace{1,\ldots,1}_{\frac{n-1}{2}},\underbrace{-1,\ldots,-1}_{\frac{n+1}{2}},\underbrace{1,\ldots,1}_{\frac{n-1}{2}},\underbrace{-1,\ldots,-1}_{\frac{n+1}{2}}) \rtimes c\quad\mbox{in}~~{}^L \SO_{2n}(\lql).
$$
From this, it follows that the adjoint action of $\rho_{\pi^\flat}(c_y)$ on $\Lie \SO_{2n}(\lql)$ has trace equal to $-n$. Hence $\rho_{\pi^\flat}$ is totally odd.
\end{proof}

The following corollary allows us to apply Proposition~\ref{prop:containingRegularUnipotent} to identify the Zariski closure of the image of $\rho_{\pi^\flat}$.

\begin{corollary}\label{cor:image-has-reg-unip}
In the setup of Theorem \ref{thm:ExistGaloisSO}, the image of $\rho_{\pi^\flat}$ (thus also $\rho_{\pi^\flat}(\Gamma_E)$ contains a regular unipotent element of $\SO_{2n}(\lql)$.
\end{corollary}

\begin{proof}
Suppose that $\qst\nmid \ell$. Then $\iota \phi_{\pi^\flat_\qst}\osim \tu{WD}(\rho_{\pi^\flat}|_{\Gamma_{F_{\qq}}})^{\tu{F-ss}}$ by (SO-i), where the two sides are compared through \cite[Prop.~2.2]{GrossReederArithmeticInvariants} by the convention of \S\ref{sect:Notation}.  Since $\phi_{\pi^\flat_\qst}$ contains a regular unipotent element in the image, so does $\tu{WD}(\rho_{\pi^\flat}|_{\Gamma_{F_{\qq}}})$. Therefore $\rho_{\pi^\flat}|_{\Gamma_{F_{\qq}}}$ has a regular unipotent in the image. If $\qst|\ell$ then the same is shown following the argument of \cite[Lem.~3.2]{GSp}.
\end{proof}

The next corollary is solely about automorphic representations, but proved by means of Galois representations. Interestingly we do not know how to derive it within the theory of automorphic forms. The corollary is not needed in this paper as (disc-$\infty$) will be imposed in the main case of interest.

\begin{corollary}\label{cor:non-cuspidal}
Let $\pi^\flat$ be a cuspidal automorphic representation of $\SO^{E/F}_{2n}(\A_F)$ satisfying (coh$^\circ$), (St$^\circ$), and (std-reg$^\circ$). If (disc-$\infty$) is false (i.e., $n$ is odd and $E=F$, or $n$ is even and $[E:F]=2$), then $\pi^\#$ in Proposition~\ref{prop:SO-to-GL} (the functorial lift of $\pi^\flat$ to $\GL_{2n}$) is the isobaric sum of cuspidal self-dual automorphic representations of $\GL_{2n-1}(\A_F)$ and $\GL_1(\A_F)$.
\end{corollary}

\begin{proof}
Fix a real place $y$ of $F$. Up to conjugation, we may assume that
$$
\rho_{\pi^\flat}(c_y)=\diag(t_1,\ldots,t_n,t_1^{-1},\ldots,t_n^{-1})\rtimes c_y,
$$
where the latter $c_y$ means its image in $\Gamma_{E/F}$; so $\std(c_y)=1$ if $E=F$ and $\std(c_y)=\vartheta^\circ$ if $[E:F] = 2$. The proof of Theorem~\ref{thm:ExistGaloisSO} shows that $\std(\rho_{\pi^\flat}(c_y))\in \GL_{2n}(\lql)$ is odd for every real place $y$. That is, $\std(\rho_{\pi^\flat}(c_y))$ has each of the eigenvalues 1 and $-1$ with multiplicity $n$. It is elementary to see that this is impossible when (disc-$\infty$) is false. Indeed, if $n$ is odd and $E=F$, then the number of $1$'s on the diagonal of $\rho_{\pi^\flat}(c_y)$ is obviously even (so cannot equal $n$). If $n$ is even and $[E:F]=2$, this is elementary linear algebra.
\end{proof}

\begin{remark}
The corollary suggests that in that setup, $\pi^\flat$ should come from an automorphic representation on $\Sp_{2n-2}(\A_F)$, where $\Sp_{2n-2}$ is viewed as a twisted endoscopic group for $\SO^{E/F}_{2n}$ (see the paragraph containing (1.2.5) in \cite{ArthurBook}).
\end{remark}

If we assume (coh$^\circ$) and (St$^\circ$) but not (std-reg$^\circ$), then some expected properties to be needed in our arguments are not known. We formulate them as a hypothesis so that our results become unconditional once the hypothesis is verified. (In the preceding arguments in this section, (std-reg$^\circ$) allowed us to apply the results on the Ramanujan conjecture and construction of automorphic Galois representations for \emph{regular} algebraic cuspidal automorphic representations of $\GL_n$ which are self-dual.)

\begin{hypothesis}\label{hypo:non-std-reg}
Assume (disc-$\infty$). When $\pi^\flat$ satisfies (coh$^\circ$) and (St$^\circ$) but not (std-reg$^\circ$), the following hold true.
\begin{enumerate}
\item $\pi^\flat_\qq$ is tempered at every finite prime $\qq$ where $\pi^\flat_\qq$ is unramified.
\item There exists a semisimple Galois representation $\rho_{\pi^\flat}:\Gamma_F \ra \SO_{2n}(\lql)\rtimes \Gamma_{E/F}$ satisfying (SO-i) at every $\qq$ where $\pi^\flat_\qq$ is unramified as well as (SO-iii), (SO-iv), and (SO-v). Moreover $\rho_{\pi^\flat}(\Gamma_F)$ contains a regular unipotent element.
\end{enumerate}
\end{hypothesis}
  The hypothesis readily implies (SO-ii) for $\rho_{\pi^\flat}$. We expect that this hypothesis is accessible via suitable orthogonal Shimura varieties.
  If one is only interested in constructing the $\GSpin_{2n}$-valued representation $\rho_\pi$ without proving its $\ell$-adic Hodge-theoretic properties, then (SO-iii) and (SO-iv) may be dropped from the hypothesis.

\begin{remark}
Corollary \ref{cor:image-has-reg-unip} (or the above hypothesis, if (std-reg$^\circ$) fails) tells us that the Zariski closure of $\rho_{\pi^\flat}(\Gamma_F)$ belongs to the list of subgroups of $\SO_{2n}$ in Proposition~\ref{prop:containingRegularUnipotent}. In the list, the $\PGL_2$, $G_2$, and $\PSO_{2n-1}$ cases can only occur when (std-reg$^\circ$) is not satisfied. Since $\PGL_2$ and $G_2$ are contained in $\PSO_{2n-1}$ (up to conjugation), we only need to observe this for $\PSO_{2n-1}$. In this case, $\mu_{\HT}(\rho_{\pi^\flat,\qq},\iota y)$ of Theorem \ref{thm:ExistGaloisSO} must factor through $i^\circ_{std}:\SO_{2n-1}\hra \SO_{2n}$, thus cannot be regular as a cocharacter of $\GL_{2n}$. By (SO-iii) of the theorem, $\std(\mu_{\Hodge}(\xi^\flat,y))$ is not regular either, contradicting (std-reg$^\circ$). 
\end{remark}

\section{Extension and restriction}\label{sec:Ext-Res}

In this section we study how the local conditions (St), (coh) on a cuspidal automorphic representation of $\GSO_{2n}^{E/F}(\A_F)$ (introduced in the introduction and \S\ref{sect:Construction} respectively) compare to conditions ($\tu{St}^\circ$), ($\tu{Coh}^\circ$) on an irreducible $\SO_{2n}^{E/F}(\A_F)$-subrepresentation
(given in \S\ref{sect:GSO-valued-Galois}).

\begin{lemma}\label{lem:SteinbergCheck}
Let $\qq$ be a finite place of $F$.
Let $\pi$ be an irreducible admissible representation of $\GSO_{2n}^{E/F}(F_\qq)$, and let $\pi^\flat \subset \pi$ be an irreducible $\SO_{2n}^{E/F}(F_\qq)$-subrepresentation. Then $\pi$ is a character twist of the Steinberg representation of $\GSO_{2n}^{E/F}(F_\qq)$  if and only if $\pi^\flat$ is
a character twist of the Steinberg representation of $\SO_{2n}^{E/F}(F_\qq)$.
\end{lemma}
\begin{proof}
Write $G = \GSO_{2n}^{E/F}(F_\qq)$ and $G_0 = \SO_{2n}^{E/F}(F_\qq)$. To lighten notation, when $H$ is an algebraic group over $F_{\qq}$, we still write $H$ for $H(F_{\qq})$ in this proof when there is no danger of confusion.

($\Rightarrow$) Let $P \subset G$ be a parabolic subgroup, and write $C_P$ for the space of smooth functions on $P \backslash G$. Fixing a Borel subgroup $B$ and taking $P \supset B$, we may view $C_P \subset C_B$ as those functions on $G/B$ that are $P$-invariant. These spaces $C_P$ define a (non-linear) filtration on $C_B$, and the Steinberg representation $\St_G$ is the quotient of $C_B$ generated by all subrepresentations $C_P$ with $P \subset G$ proper \cite[X.4.6]{BorelWallach}. There is a natural bijection between the parabolic subgroups of $G$ with those of $G_0$ by $P\mapsto P_0 := P\cap G_0$. Applying \textit{loc.~cit.}~now to $G_0$, we take $B_0 := G_0 \cap B$ and consider the spaces $C_{B_0} \supset C_{P_0}$ for $B_0 \subset P_0 \subsetneq G_0$ and $\St_{G_0}$ as before. The inclusion $G_0\hookrightarrow G$ induces an isomorphism $P_0 \backslash G_0 \isomto P \backslash G$ for each parabolic $P$ (injectivity is clear; surjectivity can be seen by using the Bruhat decomposition, for instance). Thereby we have a $G_0$-equivariant filtration-preserving isomorphism $C_B\isomto C_{B_0}$ restricting to $C_P\isomto C_{P_0}$ for each $P$. Therefore $\St_G|_{G_0} \simeq \St_{G_0}$.

($\Leftarrow$) Write $G' = \GSpin_{2n}^{E/F}(F_\qq)$ and $G_0' = \Spin_{2n}^{E/F}(F_\qq)$. By abuse of notation, write $G_0/G'_0:=\tu{coker}(\tu{pr}:G'_0\ra G_0)$ and likewise for $G/G'$. These are finite abelian groups. We claim that every smooth character $G_0 \to \C^\times$ extends to a smooth character $G \to \C^\times$. Since such characters factor through $G_0/G'_0$ and $G/G'$, respectively (see e.g., \cite[Cor.~2.6]{H0main}) the claim would follow once we verify that $G_0/G'_0\ra G/G'$ is injective. So let $g_0\in G_0$ and suppose that $g_0=\pr(g)$ for $g\in G'$. Then $1=\simil(g_0)=\simil(\pr(g))=\cN(g)^2$ by Lemma~\ref{lem:SurjectionOntoGSO} (iii). If $\cN(g)=-1$ then we replace $g$ with $zg$ using $z\in Z_{\GSpin}$ such that $\cN(z)=-1$ and $\pr(z)=1$ (in the coordinates of Lemma~\ref{lem:ComputeCenter}, choose $z=(1,-1)$ if $n$ is odd, and $z=(\zeta_4,-1)$ if $n$ is even); so we may assume that $\cN(g)=1$. But this means that $g_0$ is trivial in $G_0/G'_0$. The claim has been proved.

Thanks to the claim, we may assume $\pi^\flat = \St_{G_0}$ after twisting by a character. Since $\pi|_{G_0}$ contains $\St_{G_0}$, we can twist $\pi$ such that the central character of $\pi$ is trivial. (The central character is a character $\chi$ of $F^\times/\{\pm1\}$, so there exists a smooth character $\chi^{1/2}:F^\times\ra\C^\times$ whose square is $\chi$. Then we twist by $\simil \circ \chi^{-1/2}$.) By assumption 
$$0\neq \Hom_{G_0}(\St_{G_0},\pi) = \Hom_{G_0}(\St_G,\pi)=\Hom_{Z(G)G_0}(\St_G,\pi),$$
where the first equality is from the implication ($\Rightarrow$), and the second from the triviality of central characters. By Frobenius reciprocity, this realizes $\pi$ as a constituent of $\Ind^G_{Z(G)G_0} \St_G$, which is the direct sum of twists of $\St_G$ by characters of the finite abelian group $G/Z(G)G_0$.
\end{proof}

Let $y$ be a real place of $F$ so that $E_y/F_y=\C/\R$ if $n$ is odd and $E_y=F_y=\R$ if $n$ is even.

\begin{lemma}\label{lem:coh-GSO-coh-SO}
Let $\pi$ be an irreducible admissible representation of $\GSO_{2n}^{E_y/F_y}(F_y)$ with central character $\omega_\pi$. Let $\pi^\flat$ be an irreducible $\SO_{2n}^{E_y/F_y}(F_y)$-subrepresentation.
Let $\xi$ be an irreducible algebraic representation of $\GSO_{2n}^{E_y/F_y}$, and $\xi^\flat$ its pullback to $\SO_{2n}^{E_y/F_y}$.
Then:
\begin{enumerate}
\item The representation $\pi$ is essentially unitary if and only if $\pi^\flat$ is unitary.
\item The representation $\pi$ is a discrete series representation if and only if $\pi^\flat$ is a discrete series representation.
\item Assume $\pi$ is essentially unitary.
Then $\pi$ is $\xi$-cohomological if and only if $\pi^\flat$ is $\xi^\flat$-cohomological and $\omega_\pi=\omega_{\xi}^{-1}$, where $\omega_\xi$ is the central character of $\xi$ on $Z(\GSO_{2n}^{E_y/F_y})(F_y)$
\end{enumerate}
\end{lemma}
\begin{proof}
Write $G = \GSO_{2n}^{E_y/F_y}(F_y)$, $G_0 = \SO_{2n}^{E_y/F_y}(F_y)$, and $F_y^\times\subset G$ for the image of $\G_m(F_y)$.

(1) The ``only if'' direction is obvious. For the ``if'' direction, assume $\pi^\flat$ is unitary. We may assume $\omega_\pi = 1$. Choose a Hermitian form $h( \cdot,\cdot)$ on $\pi$, extending the $G_0$-equivariant one on $\pi^\flat$. Choose representatives $\{g_1, \ldots, g_r\}$ for the quotient $G/F_y^\times G_0$ and define $h'(\cdot, \cdot) = \sum_{i=1}^r h(g_i \cdot, g_i \cdot)$. Then $h'(\cdot, \cdot)$ is a $G$-equivariant Hermitian form on $\pi$.

(2) This follows directly from the characterization of discrete series representations through square-integrability (modulo center) of their matrix coefficients.

(3) This is implied by the fact that a unitary representation is cohomological if and only if its central character and infinitesimal character coincide with those of an algebraic representation.
The ``only if'' direction is true without the unitarity condition by \cite[Thm.~I.5.3.(ii)]{BorelWallach}. We explain the ``if'' direction in the case of interest. (This argument adapts to the general case.) For $G_0$, this follows from \cite[Thm.~1.8]{Salamanca-Riba}, which applies to connected semisimple real Lie groups. The case of $G$ follows from that of $G_0$, by applying \cite[Cor.~I.6.6]{BorelWallach} to $\pi\otimes\xi$ by taking $H$ there to be $Z(G)$, and similarly to $(\pi|_{G_0})\otimes\xi^\flat$ with $H=Z(G_0)$. 
\end{proof}

\section{Certain forms of $\GSO_{2n}$ and outer automorphisms }\label{sect:forms-of-GSO}

In this section we introduce a certain form of the split group $\GSO_{2n}$ over a totally real field $F$, to be used to construct Shimura varieties. We start by considering real groups. Let $\GO_{2n}^{\cmpt}, \uO_{2n}^{\cmpt}, \SO_{2n}^{\cmpt}, \PSO_{2n}^{\cmpt}$, and $\GSO_{2n}^{\cmpt}$ be the various versions of the orthogonal group defined by the quadratic form $x_1^2 + x_2^2 + \cdots + x_{2n}^2$ on $\R^{2n}$. Consider the matrix $J = \vierkant 0{-1_n}{1_n}0 \in \SO_{2n}^{\cmpt}(\R)$. We define the group $\GSO_{2n}^J$ over $\R$ to be the inner form of $\GSO_{2n,\R}^{\cmpt}$ obtained by conjugating the $\Gal(\C/\R)$-action by $J$ (using that $J^2$ is central). Namely, for all $\R$-algebras $R$ we have
\begin{equation}\label{eq:GSO-J}
\GSO_{2n}^J(R) = \{g \in \GSO_{2n}^{\cmpt}(\C \otimes_{\R} R)\ |\
J \li g J^{-1} = g \}.
\end{equation}
For $g \in \GSO_{2n}^{\cmpt}(\C \otimes_{\R} R)$ we have $g^{\tu{t}} J \li g = \simil(g) J$ if and only if $J \li g J^{-1} = g$, and thus $\GSO_{2n}^J(\R)$ is the group of matrices $g \in \GL_{2n}(\C)$ preserving the forms
\begin{equation}\label{eq:forms-GSO-J}
\begin{cases}
x_1^2 + x_2^2 + \cdots + x_{2n}^2 \cr
-x_1 \li x_{n+1} + x_{n+1} \li x_1 - x_2 \li x_{n+2} + x_{n+2} \li x_2 - \cdots - x_n \li x_{2n} + x_{2n} \li x_n
\end{cases}
\end{equation}
up to the scalar $\simil(g) \in \R^{\times}$ (the scalar is required to be the same for both forms), and such that $g$ satisfies the condition $\det(g) = \simil(g)^n$.

Similarly we define the inner forms $\GO_{2n}^J, \SO_{2n}^J, \uO_{2n}^J, \PSO_{2n}^J$ of $\GO_{2n}^\cmpt, \SO_{2n}^\cmpt, \uO_{2n}^\cmpt, \PSO_{2n}^\cmpt$. Then $\SO_{2n}^J(\R)$ is the real Lie group which is often denoted $\SO^*(2n)$ in the literature (e.g., \cite[Sect.~X.2, p.445]{Helgason}). Note that $\SO_{2n}^J(\R)$ is not isomorphic to any of the classical groups $\SO(p, q)$, where $2n = p + q$ (see \cite[thm 6.105(c)]{KnappLieGroupsBeyond}). The group $\SO(p, q)$ with $2n = p + q$ is quasi-split if and only if $|n-p|\le 1$, giving rise to two classes of inner twists (recall that $\SO(p, q)$ and $\SO(p', q')$ lie in the same class if and only if $p \equiv p' \mod 2$). The group $\SO^J_{2n}$, and hence the group $\GSO^J_{2n}$, is not quasi-split since $\SO^J_{2n}$ is not isomorphic to any group of the form $\SO(p, q)$.

We pin down the isomorphisms
\begin{align}\label{eq:PinDownIsoms}
C_X \colon \GSO_{2n}^{\cmpt}(\C) &
  \isomto \GSO_{2n}(\C),  \quad g \mapsto X\inv g X, \quad X = \vierkant 11i{-i}, \cr
\GSO_{2n}^{\cmpt}(\C) &
\isomto \GSO_{2n}^J(\C), \quad g \mapsto (g, J\inv g J) \in \GSO_{2n}^{\cpt}(\C)^2 =\GSO_{2n}^{\cpt}(\C \otimes_{\R} \C).
\end{align}

\begin{lemma}\label{lem:Out(GSO)overR}
\begin{enumerate}[label=(\roman*)]
\item The group $\GSO^J_{2n}$ is an inner form of the split group $\GSO(n,n)$ over $\R$ if $n$ is even, and an outer form otherwise.
\item Explicitly, 
$$
\GO_{2n}^J(\R) =
\lbr \vierkant AB{-\li B}{\li A} \in \GL_{2n}(\C) \left| \begin{matrix} A, B \in \uM_n(\C) \tu{ such that} \cr
A^t A + \li B^t \li B = \lambda \cdot 1_n \tu{ (where $\lambda = \tu{sim}(g) \in \R^\times$)} \cr
A^t B = \li B^t \li A
\end{matrix}
\right.
\rbr.
$$
\item The following are true:
\begin{enumerate}
\item The groups $\SO^J_{2n}(\R)$ and $\textup{O}^J_{2n}(\R)$ are connected and equal to each other. 
\item The map $\tu{sim} \colon \GO_{2n}^J(\R) \to \R^\times$ is surjective; $\tu{sim} \colon \GSO_{2n}^J(\R) \to \R^\times$ is surjective if and only if $n$ is even.
\item We have $\GSO_{2n}^J(\R) = \textup{GO}^J_{2n}(\R)$ if and only if $n$ is even. 
\item If $n$ is even (resp.~odd) then $|\pi_0(\GSO^J_{2n}(\R))|$ equals 2 (resp.~1).
\end{enumerate}
\item The mapping 
$$
\theta^J \colon \GSO_{2n}^J(\R) \to \GSO_{2n}^J(\R), \quad g = \vierkant AB{-\li B}{\li A} \mapsto  T g T\inv = \vierkant A{-B}{\li B}{\li A}
$$
for $T = i \cdot \vierkant 0110 \in \GO_{2n}^J(\R)$ is an automorphism of $\GSO_{2n}^J$ over $\R$. It is outer if and only if $n$ is odd.
\item The group $\SO_{2n}^J$ (resp.~$\GSO_{2n}^J$) has a nontrivial outer automorphism defined over $\R$ that acts trivially on the center if and only if $n$ is odd.
\item The groups $\SO_{2n}^\cmpt(\R)$ and $\GSO_{2n}^\cmpt(\R)$ are connected.
\end{enumerate}
\end{lemma}
\begin{proof}
(\textit{i}). The group $\SO_{2n}^J$ is an inner form of $\SO_{2n,\R}^{\cmpt}$ and the compact form lies in the split inner class if and only if $n$ is even.

(\textit{ii}).
Let $g = \vierkant ABCD \in \GL_{2n}(\C)$. Write $\lambda = \textup{sim}(g)$. We compute
\begin{align*}
& J \li g = g J \Leftrightarrow \vierkant 0{-1}10 \vierkant {\li A}{\li B}{\li C}{\li D} = \vierkant ABCD \vierkant 0{-1}10 \Leftrightarrow
\vierkant {-\li C}{-\li D}{\li A}{\li B} = \vierkant B{-A}D{-C}
\Leftrightarrow  g = \vierkant AB{-\li B}{\li A} \cr
&\vierkant AB{-\li B}{\li A}^t \vierkant AB{-\li B}{\li A} = \lambda \vierkant 1001  \Leftrightarrow
 \lambda \vierkant 1001 = \vierkant {A^t}{-\li B^t}{B^t}{\li A^t} \vierkant AB{-\li B}{\li A} = \vierkant {A^tA + \li B^t \li B}{A^tB - \li B^t \li A}{B^tA-\li A^t \li B}{B^tB + \li A^t \li A}.
\end{align*}
These identities are equivalent to the stated conditions on $g$.

(\textit{iii.a}) By \cite[Cor.~6.3]{zhang1997quaternions}, $\det \vierkant AB{-\li B}{\li A} \geq 0$ for all $A, B \in \uM_n(\C)$. By Lemma~8.1(ii) any $g \in \tu{O}_{2n}^J(\R)$ has $\det(g) \geq 0$ and thus $\det(g) = 1$.  Thus $\tu{O}_{2n}^J(\R) = \tu{SO}_{2n}^J(\R)$. By \cite[prop I.1.145]{KnappLieGroupsBeyond} the group $\SO_{2n}^J(\R)$ (and hence $\tu{O}_{2n}^J(\R)$) is connected.

(\textit{iii.b}) By restricting to the center we see that the image of the similitudes factor contains $\R^{\times}_{>0}$ in all stated cases. The element $g = \vierkant A00{\li A}$ with $A = i 1_n$ lies in $\GO_{2n}^J(\R)$ and has $\tu{sim}(g) = -1$,  proving the first part. Since $\det(g) = 1$ we have $g \in \GSO_{2n}^J(\R)$ if $n$ is even, proving the second part in that case. Assume for a contradiction that $\GSO_{2n}^J(\R) \to \R^\times$ is surjective when $n$ is odd. Take some $g' \in \GSO_{2n}^J(\R)$ with $\tu{sim}(g') = -1$. Then $\tu{sim}(gg') = 1$ thus $gg' \in \tu{O}_{2n}^J(\R) = \SO_{2n}^J(\R)$ and hence $g \in \GSO_{2n}^J(\R)$: Contradiction.

(\textit{iii.c}) For $n$ odd the element $g = \vierkant A00{\li A}$ from (\textit{iii.b}) shows 
$\GSO_{2n}^J(\R) \neq \GO_{2n}^J(\R)$. Assume $n$ even. If $h \in \GO_{2n}^J(\R)$, choose $g \in \GSO_{2n}^J(\R)$ with $\tu{sim}(g) = \tu{sim}(h)\inv$ using (\textit{iii.b}). Then $hg \in \tu{O}_{2n}^J(\R) = \SO_{2n}^J(\R)$, hence also $h \in \GSO_{2n}^J(\R)$. Thus $\GO_{2n}^J(\R) = \GSO_{2n}^J(\R)$.

(\textit{iii.d}) Write $c := \# \pi_0(\GSO_{2n}^J(\R)) $. As $\tu{H}^1(\R, \mu_2)$ has $2$ elements, we have $c \leq 2$. If $n$ is even, then $\tu{sim}$ is surjective, hence $c \geq 2$ and $c = 2$. If $n$ is odd, we have $\R^{\times}_{>0} \times \SO_{2n}^J(\R) \isomto \GSO_{2n}^J(\R)$, hence $c =1$.

(\textit{iv}) We have $T^t T = -1$ and $J \li T J\inv = J$, so indeed $T \in \GO_{2n}^J(\R)$. As $\tu{sim}(T) = -1$ and $\det(T) = i^{2n}(-1)^n = 1$, we have $\tu{sim}(T)^n \neq \det(T)$ if and only if $n$ is odd.

(\textit{v}) By the example in (\textit{iv}) we may assume $n$ even. Any $\R$-automorphism $\theta \in \Aut(\GSO_{2n}^J)$ acting trivially on the center is given by $\theta \colon g \mapsto Y g Y\inv$ for some $Y \in \GO_{2n}^J(\C)$. Replacing $Y$ with $tY$ for some $t \in \C^\times$ we may assume that $\tu{sim}(Y) = 1$ (as $\theta$ does not change, it is still defined over $\R$). Write $\sigma \colon \GO_{2n}^\cmpt(\C) \to \GO_{2n}^\cmpt(\C)$ for the automorphism $g \mapsto J \li g J\inv$, so that $\GO_{2n}^J(\R) = \GO_{2n}^\cmpt(\C)^{\sigma = \tu{id}}$. As $\theta$ is defined over $\R$,
$$
\theta(\sigma g) = \sigma {\theta(g)} \quad\quad \forall g \in \GO_{2n}^J(\C),
$$
and therefore $YJ \cdot \li g \cdot J\inv Y\inv = J \li Y \cdot \li g \cdot \li Y\inv J\inv$, so $\li Y\inv J\inv YJ \cdot \li g = \li g \cdot  \li Y\inv J\inv Y J$. Thus
$$
  \lambda \cdot Y J =  J \li Y \quad\quad \tu{ for some } \lambda \in Z(\GSO_{2n}^J(\C)) = \C^\times.
$$
We have $Y^t Y = 1$, so we compute as follows using $J^t J=1$: $$ 1=\li Y^t \li Y = (\lambda J^{-1} Y J)^t (\lambda J^{-1} Y J) = \lambda^2 (J^t Y^t Y J) = \lambda^2.$$ Therefore $\lambda\in \{\pm1\}$. If $\lambda = 1$ then $Y \in \tu{O}_{2n}^J(\R) = \SO_{2n}^J(\R)$, and $\theta$ is inner. If $\lambda = -1$ then $\sigma(Y) = -Y$ so $iY \in \tu{GO}_{2n}^J(\R) = \GSO_{2n}^J(\R)$ ($n$ is even). Thus $\theta = (g \mapsto (iy)g(iY)^{-1})$  is inner.

(\textit{vi}) It is standard that $\SO^{\cmpt}_{2n}(\R)$ is connected. Let us show that $\GSO^{\cmpt}_{2n}(\R)$ is connected from this. The multiplication map $\SO^{\cmpt}_{2n}(\R)\times \R^\times \ra \GSO^{\cmpt}_{2n}(\R)$ has connected image since $\SO^{\cmpt}_{2n}(\R)$ meets both connected components of $\R^\times$. So we will be done if we check the surjectivity. This is equivalent to the injectivity of $H^1(\R,\{\pm1\})\ra H^1(\R,\SO^{\cmpt}_{2n}\times \GL_1)$, which follows from the fact that there is no $g\in \SO^{\cmpt}_{2n}(\C)$ with $g^{-1} \ol{g}=-1$. (Via $h=\sqrt{-1} g$, the latter is equivalent to non-existence of $h\in \GL_{2n}(\R)$ with $h ^t h=-1$, which is clear.)
\end{proof}

Now we turn to the global setup. Let $n$ and $E/F$ be as in \S\ref{sect:GSO-valued-Galois} and impose condition (disc-$\infty$) from now on. In analogy with the $\SO_{2n}$-case, we introduce a quasi-split form $G^*$ of $\GSO_{2n}$ over $F$. If $n$ is even, we have $E=F$ and take the split form $G^* := \GSO_{2n,F}$ (or simply written as $\GSO_{2n}$). If $n$ is odd then $E/F$ is a totally imaginary quadratic extension. In this case, let $G^*$ be the quasi-split form $\GSO^{E/F}_{2n,F}$ of $\GSO_{2n,F}$ (up to $F$-automorphism) given by the 1-cocycle $\Gal(E/F)\ra \Aut(\GSO_{2n,E})$ sending the nontrivial element to $\theta^\circ=\Int(\vartheta^\circ)$. Since $\vartheta^\circ\in \OO_{2n}(E)$, this cocycle comes from the $\Aut(\SO_{2n,E})$-valued cocycle determining $\SO^{E/F}_{2n}$ as an outer form of $\SO_{2n}$, thus we have $\SO^{E/F}_{2n}\hookrightarrow \GSO^{E/F}_{2n}$. Concretely, in analogy with \eqref{eq:QuasiSplitSO},
 \begin{equation}\label{eq:QuasiSplitGSO}
  \GSO^{E/F}_{2n}(R) = \left\{g \in \GL_{2n}(E \otimes_F R)\ \left|\ \begin{array}{c} \textup{there exists}~\lambda\in R^\times~\textup{such that} \\
  c(g) = \vartheta^\circ g \vartheta^\circ, ~ g^t \vierkant 0{1_n}{1_n}0 g = \lambda \vierkant 0{1_n}{1_n}0, ~\det(g) = \lambda^n \end{array} \right.\right\},
\end{equation}
and $\GO^{E/F}_{2n}(R)$ is defined by removing the condition $\det(g) = \lambda^n$.

We write $G^*=\GSO^{E/F}_{2n}$ for both parities of $n$, understanding that $E=F$ if $n$ is even, for a streamlined exposition. In both cases, we have an exact sequence
\begin{equation}\label{eq:SO-GSO-exact-seq}
1\ra \SO^{E/F}_{2n} \ra \GSO^{E/F}_{2n} \ra \G_m \ra 1,
\end{equation}
where the similitude map $\GSO^{E/F}_{2n} \ra \G_m$ is the usual one if $E=F$, and $g\mapsto \lambda$ in \eqref{eq:QuasiSplitGSO} if $E\neq F$. Note that $\hat{G^*_{\tu{ad}}}$ is isomorphic to $\Spin_{2n}(\C)$, on which $\Gamma$ acts trivially (resp.~non-trivially via $\Gal(E/F)$ as $\{1,\theta\}$) if $n$ is even (resp.~odd).

Write $(\cdot)^D$ for the Pontryagin dual of a locally compact abelian group. Let $v$ be a place of $F$.
By \cite[Thm~1.2]{KottwitzEllipticSingular} we have a map\footnote{This map has been computed explicitly by Arthur \cite[Section 9.1]{ArthurBook} for all inner forms of classical groups of type $B$, $C$, and $D$.}
$$
\alpha_v \colon \uH^1(F_v, G^*_{\tu{ad}}) \to \pi_0(Z(\hat{G^*_{\tu{ad}}})^{\Gamma_v})^D,
$$
which is an isomorphism if $v$ is a finite place (but not if $v$ is infinite).

\begin{lemma}\label{lem:Invariants}
We have
$$
Z(\hat{G^*_{\tu{ad}}})^{\Gamma_v} \simeq
\begin{cases}
(\mu_2)^2, & \tu{$n$ is even}, \cr
\mu_2, & \tu{$n$ is odd, $v$ is non-split in ${E}/F$}, \cr
\mu_4, & \tu{$n$ is odd, $v$ is split in ${E}/F$.}
\end{cases}
$$
\end{lemma}
\begin{proof}
This follows from Lemma~\ref{lem:spin-center2}.
\end{proof}

By \cite[Prop.~2.6]{KottwitzEllipticSingular} and the Hasse principle from \cite[Thm.~6.22]{PlatonovRapinchuk} we have an exact sequence of pointed sets
\begin{align}\label{eq:ComputeForms}
1 \to \uH^1(F, G^*_{\ad}) \to \bigoplus_v \uH^1(F_v, G^*_{\ad}) & \overset {\Sigma_v \alpha_v} \longrightarrow  \pi_0(Z(\hat{G^*_{\tu{ad}}})^{\Gamma})^D \to 1.
\end{align}
Since $Z(\hat{G^*_{\tu{ad}}})$ is finite, we may forget $\pi_0(\cdot)$ in~\eqref{eq:ComputeForms} and the proof of the lemma below. From now until the end of \S\ref{sect:Shimura}, we fix a finite place $\qst$ and an infinite place $y_\infty$ of $F$.

\begin{lemma}\label{lem:existence-of-G}
Let $\qst$ (resp.~$y_\infty$) be a fixed finite (resp.~infinite) place of $F$. There exists an inner twist $G$ of $G^*$ such that for all $F$-places $v \neq \qst$, we have
\begin{equation}\label{eq:ShimuraInnerForm}
G_v \simeq \begin{cases}
\GSO^J_{2n,F_v} & v = y_\infty \cr
\GSO^{\cmpt}_{2n, F_v}  & v \in \cV_{\infty} \backslash \{y_\infty\} \cr
G^*_{F_v} & v \notin \cV_\infty\cup \{\qst\}.
\end{cases}
\end{equation}
This inner twist $G$ is unique up to isomorphism if either $n$ is even or $\qst$ is non-split in $E/F$; otherwise there are two choices for $G$. (Recall the notion of inner twist from \S\ref{sect:Notation}.)
\end{lemma}
\begin{proof}
Put 
\begin{equation}\label{eq:vst_inv}
a_{\qst}:= -\alpha_{y_\infty}(\GSO^J_{2n, F_{y_\infty}}) - \sum_{v \neq y_\infty} \alpha_v(\GSO^\cmpt_{2n, F_{v}}) \in (Z(\hat{G^*_{\tu{ad}}})^\Gamma)^D.
\end{equation}
By duality, the inclusion $Z(\hat{G^*_{\tu{ad}}})^{\Gamma} \subset Z(\hat{G^*_{\tu{ad}}})^{\Gamma_v}$ induces a surjection $(Z(\hat{G^*_{\tu{ad}}})^{\Gamma_v})^D \surjects (Z(\hat{G^*_{\tu{ad}}})^\Gamma)^D$. Hence we can choose some invariant $\wt {a_{\qst}} \in (Z(\hat{G^*_{\tu{ad}}})^{\Gamma_v})^D$ mapping to the expression on the right hand side of~\eqref{eq:vst_inv}. Let $G_{\qst}$ be the inner twist of $G^*$ over $F_{\qst}$ corresponding to $\wt {a_{\qst}}$. Then, by~\eqref{eq:ComputeForms} the collection of local inner twists $\{G_v\}_{\tu{places $v$}}$ comes from a global inner twist $G/F$, unique up to isomorphism. Conversely, any $G$ as in the lemma satisfies $\alpha_{\qst}(G)=a_{\qst}$ by~\eqref{eq:ComputeForms}. Therefore the number of choices for $G$ equals the number of choices for $\wt {a_{\qst}}$, which can be computed using Lemma~\ref{lem:Invariants}.
\end{proof}

\begin{remark}
The group $G_{\qst}$ in the lemma is never quasi-split, regardless of the parity of $[F:\Q]$. It is always a unitary group for a Hermitian form over a quaternion algebra. This corresponds to the ``$d=2$ case'' in \cite[\S9.1]{ArthurBook}. In this case the rank of $G_{\qst}$ is roughly $n/2$ (see \cite{ArthurBook} for precise information).
\end{remark}

\section{Shimura varieties of type $D$ corresponding to $\spin^{\pm}$}\label{sect:Shimura}

We continue in the same global setup, with an inner form $G$ of a quasi-split form $G^*$ of $\GSO_{2n}$ over a totally real field $F$, depending on the fixed places $\qst$ and $y_\infty$ of $F$. We are going to construct Shimura data associated with $\Res_{F/\Q}G$ by giving an $\R$-morphism $\SS:=\Res_{\C/\R}\G_m \ra (\Res_{F/\Q}G)\otimes_\Q \R$. Our running assumption (disc-$\infty$) is clearly a necessary condition for the existence of such Shimura data. We define
\begin{align*}
h_{(-1)^n} \colon \SS \to \GSO_{2n}^J, &\quad x + yi \mapsto \grootvierkant {x1_{n}}{y1_{n}}{-y1_{n}}{x1_{n}} \cr
h_{(-1)^{n+1}} \colon \SS \to \GSO_{2n}^J, &\quad x + y i \mapsto \grootvierkant{x 1_n}{\diag(y 1_{n-1}, -y)}{\diag(-y1_{n-1}, y)}{x 1_n}.
\end{align*}
We will often omit $1_n$ in the cases similar to the above if a matrix is clearly $2n\times 2n$ in the context.

By slight abuse of notation, we write $\Ad$ for either the natural map from $\GSO_{2n}^J\ra\GSO_{2n,\ad}^J$ or the adjoint representation of $\GSO_{2n}^J$ on $\Lie \GSO_{2n}^J$.

\begin{lemma}\label{lem:IsShimuraDatum}
Let $\eps\in \{+,-\}$ and put $K_\eps := \tu{Cent}_{\GSO_{2n}^J(\R)}(h_\eps)$. The following hold.
\begin{enumerate}[label=(\roman*)]
\item In the representation of $\C^\times$ on $\Lie \GSO_{2n}^J(\C)$ via $\Ad \circ h_\eps$, only the characters $z \mapsto z\inv \li z$, $z \mapsto 1$, and $z \mapsto z \li z\inv$ appear.
\item The involution on $\GSO^J_{2n,\ad}$ given by $\Ad\ h_\eps(i)$ is a Cartan involution.
\item $K_+$ and $K_-$ are $\GSO_{2n}^J(\R)$-conjugate.
\end{enumerate}
\end{lemma}
\begin{proof}
For (i) and (ii), we only treat the case of $\eps=(-1)^n$ as the argument for $-\eps$ is the same. Let $z = x+yi \in \C^\times$ and consider the left-multiplication action of the matrix $h_\eps(x + iy)= \vierkant xy{-y}x$ on $\uM_{2n}(\C)$. The matrix $\vierkant xy{-y}x$ is conjugate to $\vierkant {x+yi}{}{}{x-yi}$ via $\vierkant 11i{-i}$. Hence only the characters $z \li z^{-1}$, $\li z z^{-1}$ and $1$ appear in the representation of $\SS$ on $\uM_{2n}(\C)$ via conjugation by $h_\eps(x + iy)$. Since $\Lie \GSO_{2n}^J(\R)$ is contained in $\uM_{2n}(\C)$ via the standard representation, (i) is true for $h_\eps(z)$. Since $J^{-1} = h_\eps(i)$, the inner form of $\GSO_{2n}^J$ defined by $h_\eps(i)$ is the compact-modulo-center form $\GSO_{2n, \R}^\cmpt$, so part (ii) follows.

Let us prove (iii). Write $\ol h_\eps:=\tu{Ad}\circ h_\eps$. Clearly $\tu{Ad}(K_\eps)\subset \tu{Cent}_{\GSO_{2n,\ad}^J(\R)}(\li h_\eps)$.   The Lie algebra $\tu{Lie}(K_\eps)$ (resp.~the Lie algebra of $\tu{Cent}_{\GSO_{2n,\ad}^J(\R)}(\ol h_\eps)$)  is the $(0,0)$ part of $\tu{Lie}(\GSO_{2n}^J)$ (resp.~$\tu{Lie}(\GSO_{2n,\ad}^J)$) via $h_\eps$, in the sense of \cite{DeligneCorvallis}. In particular
$$
\ad \colon \tu{Lie}(K_\eps) \to \Lie (\tu{Cent}_{\GSO_{2n,\ad}^J(\R)}(\ol h_\eps))
$$
is surjective. Therefore $\Ad(K_\eps)\supset \tu{Cent}_{\GSO_{2n,\ad}^J(\R)}(\ol h_\eps)^0$. Since $\tu{Cent}_{\GSO_{2n,\ad}^J(\R)}(\li h_\eps)$ is connected by \cite[proof of Prop.~1.2.7]{DeligneCorvallis}, we have
$\tu{Ad}(K_\eps)=\tu{Cent}_{\GSO_{2n,\ad}^J(\R)}(\li h_\eps)$. The latter is the identity component of a maximal compact subgroup of $\GSO_{2n,\ad}^J(\R)$ by \emph{loc.~cit.}~so $\tu{Ad}(K_-)$ and $\tu{Ad}(K_+)$ are conjugate in $\GSO_{2n, \ad}^J(\R)$. Since $K_\eps=\tu{Ad}^{-1}(\tu{Ad}(K_\eps))$ and since $\tu{Ad}\colon \GSO_{2n}^J \to \GSO_{2n,\ad}^J$ is surjective on real points by Hilbert 90, we lift a conjugating element to see that $K_+$ and $K_-$ are conjugate in $\GSO_{2n}^J(\R)$.
\end{proof}

Recall the cocharacters $\mu_+, \mu_-$ from \eqref{eq:Spin-eps-def}, which are outer conjugate as $\mu_+ = \vartheta^\circ \mu_- (\vartheta^\circ)\inv$ (but not inner, cf.~\eqref{eq:WeylGroupAction}).

\begin{lemma}\label{lem:h+h-} Let $\eps\in \{+,-\}$.
\begin{enumerate}[label=(\roman*)]
\item Consider the inclusion of $\C^\times$ in $(\C \otimes_{\R} \C)^\times = (\C^{\times})^{\Gal(\C/\R)}$ indexed by the identity $\Id_{\C/\R} \in \Gal(\C/\R)$. Then $C_X h_{\eps,\C}|_{\C^\times} = \mu_\eps$.
\item The complex conjugate morphism $z \mapsto h_\eps(\li z)$ is $\GSO_{2n}^J(\R)$-conjugate to $h_{(-1)^{n}\eps}$.
\end{enumerate}
\end{lemma}
\begin{proof}  In the proof, put $\eps=(-1)^n$.
(\textit{i}). Recall $C_X$ from \eqref{eq:PinDownIsoms}, which induces $\GSO_{2n}^J(\R)\hra \GSO_{2n}^{\cmpt}(\C) \stackrel{C_X}{\simeq} \GSO_{2n}(\C)$. The morphism $C_X h_\eps$ equals $x+yi \mapsto \vierkant {x+yi}00{x-yi}$. The holomorphic part of this morphism is $z \mapsto \vierkant  z 0 0 1$, which is $\mu_\eps$. Then $h_\eps = \vartheta^\circ h_{-\eps} \vartheta^\circ$, where $\vartheta^\circ$ is as in \eqref{eq:Elementw}.  Write $\vartheta^c = \vierkant {-1_{2n-1}}00{1}$. Note that $C_X(\vartheta^c) = \vartheta^\circ$, so $\vartheta^c$-conjugation becomes $\vartheta^\circ$-conjugation under $C_X$. As $\vartheta^\circ$ swaps $\mu_\eps$ and $\mu_{-\eps}$, we obtain $C_X h_{-\eps}|_{\C^\times} = \mu_{-\eps}$.

(\textit{ii}). Write $z = x + yi \in \C$. We compute 
$$
h_\eps(\li z) = \grootvierkant x{-y}yx  = T h_\eps(z) T\inv, 
$$
where $T = i \vierkant 0110 \in \GO_{2n}^J(\R)$. By Lemma~\ref{lem:Out(GSO)overR}(iv) $T \in \GSO_{2n}^J(\R)$ if $n$ is even, which proves (\textit{ii}) in that case. For $n$ odd the above identity shows $z\mapsto h_\eps(\li z)$ and $h_\eps$ are conjugate under an outer automorphism. Thus, by (i), the cocharacter attached to $h_\eps(\li z)$ is conjugate to $\mu_{-\eps}$.

By Lemma~\ref{lem:IsShimuraDatum} the conditions of \cite[Prop.~1.2.2]{DeligneCorvallis} on $(\PSO_{2n}^J, h_{\pm}(z))$ and $(\PSO_{2n}^J, h_{\pm}(\li {z}))$ are satisfied. By this proposition there exists a $g \in \PSO_{2n}^J(\C)$ that conjugates $(\PSO_{2n, \R}^J, h_\eps(\li z))$ to $(\PSO_{2n, \R}^J, h_{-\eps}(z))$. This implies that for all $x \in \PSO_{2n}^J(\C)$ we have $g \li x g\inv = \li {g x g\inv}$. Thus also $\li g \inv g \li x = \li x \li g\inv g$ and $\li g\inv g \in Z(\PSO_{2n}^J(\C)) = \{1\}$, and thus $\li g = g$, which means $g \in \PSO_{2n}^J(\R)$. By Hilbert 90 we may lift $g$ to an element $\wt g \in \GSO_{2n}^J(\R)$. Then the map $\C^\times \owns z \mapsto \chi(z) := h^{-1}_\eps(\li z) \wt g h_{-\eps}(z) \wt g\inv$ is a continuous homomorphism to $Z(\GSO_{2n}^J(\R)) = \R^\times$. It suffices to show that $\chi$ is trivial. The subgroup $\chi(U(1)) \subset \R^\times$ is connected and compact, hence trivial. Since $h_\eps$ and $h_{-\eps}$ agree and are central on $\R^\times \subset \C^\times$, we have $\chi(\R^\times)=\{1\}$ as well. So $\chi$ is trivial as desired. The proof for $h_{-\eps}$ is similar.
\end{proof}

Let $G$ be as in Lemma \ref{lem:existence-of-G}. Let $X^\eps$ be the $(\Res_{F/\Q}G)(\R)$-conjugacy class of the morphism
\begin{equation}\label{eq:Shimurah}
h^\eps \colon \SS \to (\Res_{F/\Q} G)_\R, \quad z \mapsto (h_\eps(z), 1, \ldots, 1) \in \prod_{y \in \cV_\infty} G_{F_y},
\end{equation}
where the non-trivial component corresponds to the place $y_\infty$. Then $\mu^\eps = (\mu_\eps, 1, \ldots, 1) \in X_*((\Res_{F/\Q} G)_\C) = X_*(\GSO_{2n,\C})^{\cV_\infty}$ is the cocharacter attached to $h_\eps$, in the same way as in Lemma \ref{lem:h+h-} (i). The reflex field of $(\Res_{F/\Q} G, X^\eps)$ means the field of definition for the conjugacy class of $\mu^\eps$, as a subfield of $\C$.

\begin{lemma}\label{lem:Shimura-datum} Let $\eps \in \{\pm 1\}$. Then
\begin{enumerate}[label=(\roman*)]
\item The pair $(\Res_{F/\Q} G, X^\eps)$ is a Shimura datum of abelian type.
\item The Shimura data $(\Res_{F/\Q} G, X^+)$ and $(\Res_{F/\Q} G, X^-)$ are isomorphic only if $n$ is odd.
\item If $F \neq \Q$, the Shimura varieties attached to $(\Res_{F/\Q} G, X^\eps)$ are projective.
\item The reflex field of the datum $(\Res_{F/\Q} G, X^\eps)$ is equal to $E$, equipped with an embedding $x_\infty: E\hra \C$ extending $y_\infty:F\hra \C$.
\end{enumerate}
\end{lemma}

\begin{remark}
About (\textit{i}): When $F=\Q$, the Shimura datum  $(G, X^\eps)$ can be shown to be of Hodge type but we do not need this fact. About (\textit{ii}): If $n$ is odd and $\qst$ is inert in $E/F$, then one can show that $(\Res_{F/\Q} G, X^+) \simeq (\Res_{F/\Q} G, X^-)$.
\end{remark}

\begin{proof}
(\textit{i}) Clearly, $(\Res_{F/\Q} G)_{\tu{ad}}$ has no compact factor defined over $\Q$, which is one of Deligne's axioms of Shimura datum \cite[2.1]{DeligneCorvallis}. The remaining two axioms follow from Lemma~\ref{lem:IsShimuraDatum}, and hence $(\Res_{F/\Q} G, X^\eps)$ is a Shimura datum. In the terminology of \emph{loc.~cit.}, $(\Res_{F/\Q} G, X^\eps)$ is of type $D^{\mathbb H}$. By \cite[Prop.~2.3.10]{DeligneCorvallis}, a datum $(G', X')$ of type $D^{\mathbb H}$ is of abelian type if the derived group of $G'_{\C}$ is (a product of) $\SO_{2n,\C}$. (Not all Shimura data of type $D^{\mathbb H}$ are of abelian type.)

(\textit{ii}) If $n$ is even then every automorphism of $(G_{F_{y_\infty}})_{\ad}$ (isomorphic to $\GSO_{2n,\ad}^J$) is inner by Lemma~\ref{lem:Out(GSO)overR} (v). On the other hand, it follows from Lemma \ref{lem:h+h-} (i) that no inner automorphism of $\GSO_{2n,\ad}^J$ takes $\Ad\circ h_+$ to $\Ad\circ h_-$, since $\Ad\circ \mu_+$ to $\Ad\circ \mu_-$ are not conjugate by $\GSO_{2n,\ad}(\C)$. Hence no automorphism of $(\Res_{F/\Q} G)_{\R}$ (thus also of $\Res_{F/\Q} G$) carries $X^+$ onto $X^-$.

(\textit{iii}) If $F \neq \Q$ there exists some real place $y_\infty' \in \cV_\infty$ of $F$ different from $y_\infty$. Since $G_{y_\infty'}$ is compact modulo center, $\Res_{F/\Q}G$ is anisotropic modulo center over $\Q$. Hence the associated Shimura varieties are projective by Bailey-Borel \cite[Thm.~1]{BailyBorel}.

(\textit{iv}) Assume that $n$ is odd (thus $[E:F]=2$). Suppose that $\sigma \in \Aut(\C/\Q)$ stabilizes the conjugacy class of $\mu^\eps$. Since $\sigma(\mu^\eps) \sim \mu^\eps$ we have $\sigma(y_\infty) = y_\infty$, so $\sigma \in \Aut(\C/F)$ with respect to $y_\infty:F\hra \C$. If $\sigma$ has non-trivial image in $\Gal(E/F)$, then Lemma~\ref{lem:h+h-} (ii) tells us that $\sigma(\mu^\eps) \sim (\mu_{-\eps}, 1, \ldots, 1)$, which is not $\GSO_{2n}(\C)$-conjugate to $\mu^\eps$. Thus $\sigma$ is trivial on $E$ (embedded in $\C$ extending $y_\infty$). Conversely, if $\sigma \in \Aut(\C/E)$, then $\sigma(\mu^\eps) = \mu^\eps$. Hence the reflex field is $E$.  When $n$ is even (thus $E=F$), the preceding argument shows that the reflex field is $F$.
\end{proof}

We introduce the following notation. Let $\eps\in \{+,-\}$.
\begin{itemize}
\item Taking an algebraic closure of $E$ in $\C$ via $x_\infty:E\hra \C$, we fix $\ol F=\ol E\hra \C$.
\item We fix an isomorphism $G \otimes_F \A_F^{\infty, \qst} \simeq G^* \otimes_F \A_F^{\infty, \qst}$.
\item $Z$ is the center of $G$.
\item $\xi=\otimes_{y|\infty} \xi_y$ is an irreducible algebraic representation of $(\Res_{F/\Q}G)\times_{\Q} \C=\prod_{y|\infty}G\times_{F,y} \C$.
\item $\Pi_{\xi}^{G(F_\infty)}$ is the set of isomorphism classes of (irreducible) discrete series representations of $G(F_\infty)$ which have the same infinitesimal and central characters as $\xi^\vee$.
\item $K_\infty^\eps$ is the centralizer of $h^\eps$ in $(\Res_{F/\Q}G)(\R)=G(F_\infty)$.
\item For irreducible admissible representations $\tau_\infty$ of $G(F_\infty)$, put
\begin{equation}\label{eq:ep_pm}
\ep^\eps(\tau_\infty \otimes \xi): = \sum_{i=1}^{n(n-1)} (-1)^i \dim \uH^i(\Lie G(F_\infty), K_\infty^\eps; \tau_\infty \otimes \xi)
\end{equation}
\end{itemize}

Let $\pi^\natural$ be a cuspidal automorphic representation of $G(\A_F)$ such that
\begin{itemize}
  \item $\pi^\natural_{\qst}$ is a Steinberg representation up to a character twist,
  \item $\pi^\natural$ is $\xi$-cohomological.
\end{itemize}
The latter condition that implies via (the proof of) \cite[Lem.~7.1]{GSp} the following condition:
\begin{itemize}
\item[\textbf{(cent)}] There exists an integer $w \in \Z$, called the \emph{central weight} of $\xi$, such that for every infinite $F$-place $y|\infty$ the central character of $\xi_y$ is of the form $x \mapsto x^w$.
\end{itemize}
We also make the folowing assumption: 
\begin{itemize}
\item[\textbf{(temp)}] $\pi^\natural_{\qq}$ is essentially tempered at every finite $F$-place $\qq$ where $\pi_{\qq}$ is unramified.
\end{itemize}
This may seem strong, but (temp) will be satisfied in practice; see the paragraph above \eqref{eq:LGC-Shimura}. Let $A(\pi^\natural)$ be the set of (isomorphism classes of) cuspidal automorphic representations $\tau$ of $G(\A_F)$ such that
\begin{enumerate}[label=(\roman*)]
  \item $\tau_{\qst}\simeq \pi^\natural_{\qst}\otimes \delta$ for an unramified character $\delta$ of the group $G(F_{\qst})$,
  \item  $\tau^{\infty,\qst}\simeq \pi^{\natural, \infty, \qst}$, and
  \item $\tau_\infty$ is $\xi$-cohomological.
\end{enumerate}
By (temp), $\tau_{\qq}$ is essentially tempered at every $\qq$ where $\pi_{\qq}$ is unramified. Define
\begin{equation}\label{eq:Multiplicity}
a^\eps(\pi^{\natural}) := (-1)^{n(n-1)/2} N^{-1}_\infty \sum_{\tau \in A(\pi^\natural)} m(\tau) \cdot \tu{ep}^\eps(\tau_\infty \otimes \xi) \in \Q,
\end{equation}
where $m(\tau)$ is the multiplicity of $\tau$ in the discrete automorphic spectrum of $G$, and
\begin{equation}\label{eq:N_infty}
N_\infty := |\Pi_\xi^{G(F_\infty)}| \cdot |\pi_0(G(F_\infty)/Z(F_\infty))| = 
\begin{cases}
2^{n-1} \cdot 2, & \mbox{if}~n~\mbox{is~even},\\
2^{n-1} , & \mbox{if}~n~\mbox{is~odd}.
\end{cases}
\end{equation}
Here $|\pi_0(G(F_\infty)/Z(F_\infty))|\in \{1,2\}$ depending on the parity of $n$ from Lemma~\ref{lem:Out(GSO)overR} (iii), (vi).

\begin{lemma}\label{lem:a-=a+}
The groups $K_\infty^+$ and $K_\infty^-$ are $G(F_\infty)$-conjugate.  In particular $a^-(\pi^\natural) = a^+(\pi^\natural)$.
\end{lemma}

Henceforth we will write $a(\pi^\natural)\in \Q$ for the common value of $a^\eps(\pi^\natural)$.

\begin{proof} 
The $y_\infty$-components of $K_\infty^\eps$ is $K_\eps$, which are conjugate to each other by Lemma~\ref{lem:IsShimuraDatum}. The components of $K_\infty^\eps$ at the other real places $y$ equal $G(F_y)\simeq\GSO_{2n}^{\cmpt}(\R)$, which is connected. Therefore $K_\infty^+$ and $K_\infty^-$ are connected and $G(F_\infty)$-conjugate. It then follows that $\ep^+(\tau_\infty \otimes \xi) = \ep^-(\tau_\infty \otimes \xi)$ for all $\tau_\infty$. Thus $a^+(\pi^\natural) = a^-(\pi^\natural)$.
\end{proof}

Since condition (cent) holds, we can attach to $\xi$ a lisse $\lql$-sheaf $\cL_{\iota \xi}$ on $\Sh^\eps_K$ as in \cite[below Lem.~7.1]{GSp} and \cite[Sect.~2.1, 2.1.4]{Car86}. We have a canonical model $\Sh^\eps_{K}$ over $E$ for each neat open compact subgroup $K\subset G(\A_F^\infty)$ (see \cite[\S0]{PinkThesis} for the definition of neat subgroups) and a distinguished embedding $E\subset \ol F$ (compatible with $E\subset \C$ and the fixed embedding $\li F\hra \C$). We take the limit over $K$ of the \'etale cohomology of with compact support
$$
\uH_\uc^i(\Sh^\eps, \cL_{\iota\xi}):=\varinjlim\limits_{K} \uH_\uc^i(\Sh_{K}^\eps\times_{E} \ol F, \cL_{\iota\xi}),
$$
equipped with commuting linear actions of $\Gamma_E=\Gal(\ol F/E)$ and $G(\A_F^\infty)$. The two groups act continuously and admissibly, respectively. Write $\uH_\uc^i(\Sh_K^\eps, \cL_{\iota\xi})^{\tu{ss}}$ for the semisimplification as a $\Gamma_E\times G(\A_F^\infty)$-module. (No semisimplification is necessary for the $ G(\A_F^\infty)$-action if $F\neq \Q$, in which case $\Sh_K^\eps$ is projective. This can be seen from the semisimplicity of the discrete $L^2$-automorphic spectrum via Matsushima's formula.)

We construct Galois representations of $\Gamma_E$ by taking the $\iota\tau^\infty$-isotypic part in the cohomology as follows. We consider $\tau_1,\tau_2\in A(\pi^\natural)$ are equivalent and write $\tau_1\sim \tau_2$ if $\tau_1^\infty\simeq \tau_2^\infty$. Let $A(\pi^\natural)/{\sim}$ denote the set of (representatives for) equivalence classes. Let $\tau\in A(\pi^\natural)$. Define
\begin{equation}\label{eq:tau-isotypic}
 \uH_c^i(\Sh^\eps, \cL_\xi)[\iota\tau^\infty]:=\Hom_{G(\A^\infty_F)}\big(\iota\tau^{\infty}, \uH_\uc^i(\Sh^\eps, \cL_{\iota\xi})^{\tu{ss}}\big),
\end{equation}
\begin{equation}\label{eq:Rho2Plus}
\rho_{\pi^\natural}^{\Sh, \eps}:=(-1)^{n(n-1)/2} \sum_{\tau \in A(\pi^\natural)/\sim} \sum_{i=0}^{n(n-1)} (-1)^i  \uH_c^i(\Sh^\eps, \cL_\xi)[\iota\tau^\infty].
\end{equation}
A priori $\rho_{\pi^\natural}^{\Sh, \eps}$ is an alternating sum of semisimple representations of $\Gamma_E$, thus a virtual representation (but see Theorem~\ref{thm:PointCounting} below). Fix a neat open compact subgroup
$$
K = \prod_{\qq \nmid \infty} K_\qq \subset G(\A_F^\infty)\quad\mbox{such~that}\quad(\pi^{\natural, \infty})^ K \neq 0,
$$
and also such that $K_{\qq}$ is hyperspecial whenever $\pi^\natural_{\qq}$ (or equivalently $\pi_{\qq}$) is unramified. Let $\Sbad$ be the set of rational primes $p$ for which either
\begin{itemize}
  \item  $p=2$,
  \item  $\Res_{F/\Q} G$ is ramified over $\Q_p$, or
  \item $K_p = \prod_{\qq|p} K_\qq$ is not hyperspecial.
\end{itemize}
We write $\SFbad$ (resp.~$\SEbad$) for the $F$-places (resp.~${E}$-places) above $\Sbad$.
We apply the Langlands--Kottwitz method at level $K$ to compute the image of Frobenius elements under $\rho_{\pi^\natural}^{\Sh, \eps}$ at almost all primes.

\begin{theorem}\label{thm:PointCounting}
Consider $\pi^\natural$ satisfying condition (temp); see below \eqref{eq:ep_pm}. There exists a finite set of rational primes $S$ containing $\Sbad$, such that for all $\pp$ not above $S$ and all sufficiently large integers $j$ (with the lower bound for $j$ depending on $\pp$), writing $\qq:=\pp\cap F$, we have
\begin{equation}\label{eq:PointCounting}
\Tr\rho_{\pi^\natural}^{\Sh,\eps}(\Frob_\pp^j) = \iota a(\pi^\natural) q_\pp^{j n(n-1)/4} \cdot \Tr (\spin^{\eps,\vee}(\phi_{\pi_\qq^\natural}))(\Frob_\pp^j),\qquad \varepsilon\in\{+,-\}.
\end{equation}
Moreover the summand of \eqref{eq:Rho2Plus} is nonzero only if $i=n(n-1)/2$. In particular the virtual representation $\rho_{\pi^\natural}^{\Sh,\eps}$ is a true semisimple representation.
\end{theorem}

\begin{proof} 
We mimic the proof of \cite[Prop.~8.2]{GSp} closely. Note that our $\rho_{\pi^\natural}^{\Sh,\eps}$ corresponds to $\rho_2^\coh$ there. Another difference is that we use $S$ to denote a set of primes of $\Q$ (not $F$ or $E$). It is enough to find $S$ as in the theorem for each $\eps$ separately, as we can take the union of the set for each of $+$ and $-$ (and take the maximum of lower bounds for $j$). We suppose that $F\neq \Q$ so that our Shimura varieties are proper. The case $F=\Q$ will be addressed at the end of proof.

Let $f_\infty= N_\infty^{-1} f_{\xi}$, where $f_{\xi}$ is the Euler-Poincar\'e (a.k.a.~Lefschetz) function for $\xi$ on $G(F_\infty)$ as recalled in \cite[Appendix~A]{GSp}. Then
$$
\Tr \tau_\infty(f_\infty)=N_\infty^{-1} \ep^\eps(\tau_\infty\otimes\xi)=N_\infty^{-1}\sum_{i = 0}^\infty (-1)^i \dim \uH^i(\ig, K^\eps_\infty; \tau_\infty \otimes \xi).
$$
Choose a decomposable Hecke operator $f^{\infty,\qst}=\prod_{\qq\neq \qst} f_\qq \in \cH(G(\A^{\infty,\qst}_F)\sslash K^{\qst})$ such that for all automorphic representations $\tau$ of $G(\A_F)$ with $\tau^{\infty,K} \neq 0$ and $\Tr \tau_\infty(f_\infty) \neq 0$ we have
$$
\Tr\tau^{\infty,\qst}(f^{\infty,\qst}) = \begin{cases}
1 & \tu{if $\tau^{\infty,\qst} \simeq \pi^{\natural,\infty,\qst}$} \cr
0 & \tu{otherwise.}
\end{cases}
$$
This is possible since there are only finitely many such $\tau$ (one of which is $\pi^\natural$). Let $f_{\qst}$ be a Lefschetz function from \cite[Eq.~(A.4)]{GSp}. There exists a finite set of primes $\Sigma \supset \Sbad\cup \{\ell\}$ such that $f_{\pp'}$ is the characteristic function of $K_{\pp'}$ (which is hyperspecial) for every $\pp'$ not above $\Sigma$. We fix $\Sigma$ and $f^\infty=\prod_{v\nmid \infty} f_v$ as above.

In the rest of the proof we fix an $E$-prime $\pp$ not above $\Sigma\cup \{\ell\}$. Write $\qq:=\pp\cap F$, and $p$ for the rational prime below $\pp$. To apply the Langlands--Kottwitz method, we need an integral model for $\Sh^\varepsilon_K$ over $\cO_{E_{\pp}}$. Thus we choose an isomorphism $\iota_p:\C\isom \lqp$ such that the valuation on $\lqp$ restricts to the $\pp$-adic valuation via $\iota_p x_\infty: E\hra \lqp$. (Recall $x_\infty$ from Lemma~\ref{lem:Shimura-datum} (iv).) The $(\Res_{F/\Q}G)(\lqp)$-conjugacy class of $\iota_p \mu:\G_m\ra (\Res_{F/\Q}G)_{\lqp}$ is defined over $E_\pp$.

For $j\in \Z_{\ge 1}$, let $f^{(j)}_p$ denote the function in the unramified Hecke algebra of $G(F_p)$ constructed in \cite[\S7]{KottwitzAnnArbor} for the endoscopic group $H=G^*$, which is isomorphic to $G$ over $F_p=F\otimes_\Q \Q_p$. (This is the function $h_p$ in \emph{loc.~cit.} We take $s$ and $t_i$'s on p.179 there to be trivial, so that $h_p$ is the image of $\phi_j$ under the standard base change map on p.180.)
 The $L$-group for $(\Res_{F/\Q}G)_{E_{\pp}}$ (with coefficients in $\C$) can be identified as
$$^L (\Res_{F/\Q}G)_{E_{\pp}} = \Big(\prod_{\sigma\in \Hom(F,\lqp)} \hat G \Big) \rtimes \Gamma_{E_{\pp}},$$
where $\Gamma_{E_{\pp}}$ acts trivially on the factor for $\sigma=\iota_p y_\infty$. (The Galois action may permute the other factors via its natural action on $\Hom(F,\lqp)$ but this does not matter to us.) The representation of $^L (\Res_{F/\Q}G)_{E_{\pp}}$ of highest weight $\iota_p \mu$ is the representation
$(\spin^\varepsilon,1,\ldots,1)$. Here $\spin^\eps$ is on the factor for $\sigma=\iota_p y_\infty$, where we identify 
$$
G\times_{F,\sigma} \lqp = \GSO^{E/F}_{2n}\times_{F,\sigma} \lqp \stackrel{\textup{via}~\iota_p x_\infty}{=\joinrel=\joinrel=} \GSO_{2n,\lqp},
$$
(in the ambient group $\GL_{2n}(E\otimes_F \lqp)\simeq \GL_{2n}(\lqp)\times \GL_{2n}(\lqp)$ of the left hand side, we project onto the $\iota_p x_\infty$-component) thus identify $\hat G=\GSpin_{2n}$ on the $\iota_p y_\infty$-component. Now let $\tau_p=\prod_{\qq'|p} \tau_{\qq'} $ be an unramified representation of $G(F_p)=(\Res_{F/\Q}G)(\Q_p)=\prod_{\qq'|p} G(F_{\qq'})$, and denote by $\phi_{\tau_p}:W_{\Q_p}\ra {}^L (\Res_{F/\Q} G)_{\Q_p}$ its $L$-parameter. Then the $\iota_p y_\infty$-component of $\phi_{\tau_p}|_{W_{E_\pp}}$ is given by $\phi_{\tau_{\qq}}|_{W_{E_\pp}}$. All in all, we can explicate \cite[(2.2.1)]{KottwitzTwistedOrbital} in our setup as\footnote{A word on the sign convention is appropriate here. The sign of \cite[(2.2.1)]{KottwitzTwistedOrbital} was flipped on \cite[p.193]{KottwitzAnnArbor}, meaning that the highest weight $-\iota_p \mu_{\eps}$ (up to the Weyl group action) should be used in \eqref{eq:trace-identity-at-p}. This was caused by the arithmetic vs geometric convention for Frobenius, and explains why $\spin^\eps$ is dualized, cf.~the paragraph above Lemma \ref{lem:Duality}. (It may appear that the sign has to be changed once again when going from \cite{KottwitzAnnArbor} to \cite{KottwitzPoints}, since the latter paper asserts that $(G,h^{-1})$ in its notation, not $(G,h)$, corresponds to the canonical model of \cite{DeligneCorvallis}. However we think the sign change is unnecessary; it should be $(G,h)$ as long as we fix the sign errors in \cite{DeligneCorvallis} as pointed out at the end of \S12 in \cite{MilneIntro}.)
}
\begin{equation}\label{eq:trace-identity-at-p}
\Tr \tau_p(f^{(j)}_p) = q_{\pp}^{j n(n-1)/4} \Tr (\spin^{\eps,\vee}(\phi_{\tau_{\qq}})(\Frob_{\pp}^j)).
\end{equation}

As in the proof of \cite[Prop.~8.2]{GSp} (where our $f^{(j)}_p$ is denoted by $h^{G^*}_p$), the Lefschetz functions $f_\infty$ and $f_{\qst}$ allow us to simplify the stabilized Langlands-Kottwitz formula \cite[Thm.~8.3.11]{KSZ} (recalled in \cite[Thm.~7.3]{GSp})  and obtain a simple stabilization of the trace formula for $G$; the outcomes are formulas (8.8) and (8.9) of \cite{GSp}. Combining them, we obtain
\begin{equation}\label{eq:Stabilizing}
\iota^{-1}\Tr(\iota f^{\infty,p}f_p \times \Frob_\pp^j,\, \uH_c(\Sh^\eps, \cL_{\iota\xi})) = T^G_{\cusp,\chi}(f^{\infty,p} f^{(j)}_p f_\infty),\qquad j\gg 1.
\end{equation}
Note that $f_p$ is the characteristic function of the hyperspecial subgroup $K_p=\prod_{\pp|p} K_\pp$. Following the argument from Equation (8.10) to (8.13) in \cite{GSp}, we compute
\begin{equation}\label{eq:LKPCF4}
\iota^{-1}\Tr \rho_{\pi^\natural}^{\Sh,\eps}(\Frob_\pp^j) = a(\pi^\natural) \Tr\pi^\natural_p(f^{(j)}_p)\stackrel{\eqref{eq:trace-identity-at-p}}{=\joinrel=} a(\pi^\natural) q_\pp^{j n(n-1)/4} \cdot \Tr (\spin^{\eps,\vee}(\phi_{\pi_\qq^\natural}))(\Frob_\pp^j).
\end{equation}

Let us show that $\rho_{\pi^\natural}^{\Sh,\eps}$ is a true representation by showing that only the middle degree cohomology contributes to $\rho_{\pi^\natural}^{\Sh,\eps}$. Since the canonical smooth integral model of $\Sh_K^\eps$ constructed by Kisin is proper as shown in \cite[Thm.~2.1.29]{Youcis} (extending the analogous result for Hodge-type Shimura varieties by Madapusi Pera \cite[Cor.~4.1.7]{MadapusiPeraToroidal}), the action of $\Frob_\pp$ on $\uH_c^i(\Sh_K, \cL_\xi)$ is pure of weight $-w+i$ by \cite[Cor.~3.3.6]{WeilII} since $\cL_\xi$ is pure of weight $-w$ \cite[\S5.4,~Prop.~5.6.2]{Pinkelladic}. (To obtain purity from Pink's result, we enlarge the set $S$ if necessary; cf.~\cite[1.3]{MorelBook}, especially the proof of (7) in Proposition 1.3.4 there.) The argument for Part (2) of \cite[Lem.~8.1]{GSp} (replacing Lemma 2.7 in the proof therein with our condition (temp)) implies that $\tau_\qq|\simil|^{w/2} = \pi^\natural_\qq|\simil|^{w/2}$ is tempered and unitary. Combining with~\eqref{eq:LKPCF4} we conclude that $\uH_c^i(\Sh^\eps, \cL_\xi)[\iota\tau^\infty]=0$ unless $i=n(n-1)/2$.

Finally, the case $F=\Q$ is handled via intersection cohomology as in the proof of \cite[Prop.~8.2]{GSp}. Thus we content ourselves with giving a sketch. For each $\tau\in A(\pi^\natural)$, one observes as in \cite[Lem.~8.1]{GSp} that $\uH_c^i(\Sh^\eps, \cL_\xi)[\iota\tau^\infty]$ is isomorphic to the $\iota\tau^\infty$-isotypic part of the intersection cohomology as $\Gamma_E$-representations. The point is that $\tau^\infty$ does not appear in any parabolic induction of an automorphic representation on a proper Levi subgroup of $G(\A)$. (If it does appear, then restricting $\tau$ from $G(\A)$ to its derived subgroup $G^{\der}(\A)$ and transferring to the quasi-split inner form $\SO_{2n}^{E/F}(\A)$ via \cite[Prop.~6.3]{GSp}, we would have a cohomological automorphic representation $\tau^{\flat}$ of $\SO_{2n}^{E/F}(\A)$ with a Steinberg component up to a twist that appears as a constituent in a parabolically induced representation. Then the Arthur parameter for $\tau^{\flat}$ cannot have the shape described in Proposition~\ref{prop:SO-to-GL}, leading to a contradiction.) The rest of the proof of \cite[Prop.~8.2]{GSp} carries over, via the analogue of part 2 of \cite[Lem.~8.1]{GSp} (the latter is justified using condition (temp) in our case), bearing in mind that the middle degree is $n(n-1)/2$ for us (which was $n(n+1)/2$ for the group $\GSp_{2n}$).
\end{proof}

\begin{corollary}\label{cor:Pi-is-in-Disc-series-Packet-2}
Let $\pi^\natural$ be as above. If $\tau\in A(\pi^\natural)$ then
\begin{enumerate}
  \item $\tau_\infty$ belongs to the discrete series $L$-packet $\Pi^{G(F_\infty)}_\xi$,
  \item  $\tau^\infty\tau'_\infty\in A(\pi^\natural)$ and $m(\tau)=m(\tau^\infty\tau'_\infty)$ for all $\tau'_\infty\in \Pi^{G(F_\infty)}_\xi$.
\end{enumerate}
Moreover $a(\pi^\natural)=\sum_{\tau\in A(\pi^\natural)/\sim} m(\tau)\in \Z_{>0}$.
\end{corollary}

\begin{proof}
This is the exact analogue of \cite[Cor.~8.4, Cor.~8.5]{GSp} and the same proof applies. (Since $a(\pi^\natural)=a^+(\pi^\natural)=a^-(\pi^\natural)$, we adapt the argument there to either $\eps\in \{+,-\}$ to compute.)
\end{proof}

\begin{proposition}\label{prop:HT-weights}
Assume that $F\neq \Q$.  Let $x_\infty:E\hra \C$ and $y_\infty:F\hra \C$ be as in Lemma \ref{lem:Shimura-datum}. Then   
$$ 
\mu_{\HT}( \rho_{\pi^\natural}^{\Sh,\varepsilon} ,\iota x_\infty) \sim i_{a(\pi^\natural)}\circ  \spin^{\varepsilon,\vee}\circ \Big(\mu_{\Hodge}(\xi_{y_\infty}) - \tfrac{n(n-1)}{4} \simil \Big), \quad \varepsilon\in \{\pm 1\}.
$$
\end{proposition}

\begin{proof}
We start by setting up some notation. Let $\pp$ be a prime of $E$ above $\ell$, and $\sigma: E\hra \lql$ an embedding inducing the $\pp$-adic valuation on $E$. Let $r$ be a de Rham Galois representation of $\Gamma_E$ on a $\lql$-vector space. Write $D_{\tu{dR},\sigma}(r)$ for the filtered $\lql$-vector space associated with $r|_{\Gamma_{E_{\pp}}}$ with respect to $\sigma$ (as on p.99 of \cite{HarrisTaylor}). Define $\tu{HT}_\sigma(r)$ to be the multi-set containing each $j\in \Z$ with multiplicity $\dim \tu{gr}^j (D_{\tu{dR},\sigma}(r))$. (So the cardinality of $\tu{HT}_\sigma(r)$ equals $\dim r$.) When $a\in \Z_{>0}$ and $A$ is a multi-set, we write $ A^{\oplus a}$ to denote the multi-set such that the multiplicity of each element in $A^{\oplus a}$ is $a$ times that in $A$.

Write $\lambda(\xi)=\{\lambda(\xi_y)\}_{y|\infty}$ for the highest weight of $\xi=\otimes_{y|\infty} \xi_y$. In the basis of \S\ref{sect:RootDatum} for $X^*(\uT_{\GSO})=X_*(\uT_{\GSpin})=\Z^{n+1}$, we write $\xi_{y_\infty}$ and the half sum of positive roots $\rho$ for $\GSO_{2n}$ as
\begin{eqnarray*}
  \lambda(\xi_{y_\infty}) &=& (a_0,a_1,\ldots,a_n) ,\qquad a_1\ge a_2\ge \cdots \ge |a_n|\ge 0, \\
  \rho &=& (-n(n-1)/4,n-1,n-2,\ldots,1,0).
\end{eqnarray*}
Let $\mathscr P^\varepsilon(n)$ denote the collection of subsets of $\{1,2,\ldots,n\}$ whose cardinality has the same parity as $n$ if $\varepsilon=(-1)^n$ and different parity if $\varepsilon=(-1)^{n+1}$. Put
\begin{eqnarray*}
 (b_0,b_1,\ldots,b_n)&:=&(a_0-\tfrac{n(n-1)}{2},a_1+n-1,a_2+n-2,\ldots,a_{n-1}+1,a_n)\\
 &=&\lambda(\xi_{y_\infty})+\rho - (n(n-1)/4,0,0,\ldots,0),
\end{eqnarray*}
which equals $\mu_{\Hodge}(\xi_{x_\infty}) - \frac{n(n-1)}{4} \simil$. Via the description of weights in the representation $\spin^\varepsilon$ in \eqref{eq:spin-highest-weights} (which gives the weights in $\spin^{\eps,\vee}$), the proposition amounts to the assertion that
\begin{equation}\label{eq:HTweights}
\HT_{\iota x_\infty}(\rho_{\pi^\natural}^{\Sh,\varepsilon}) =  \Big\{ -b_0 - \sum_{i\in I} b_i \,\Big|\, I\in   \mathscr P^\varepsilon(n)   \Big\}^{\oplus a(\pi^\natural)}
=  \Big\{ -a_0 - \sum_{i\in I} a_i +\sum_{i\notin I}  (n-i)\,\Big|\, I\in   \mathscr P^\varepsilon(n)   \Big\}^{\oplus a(\pi^\natural)}.
\end{equation}

We prove this assertion using a result from \cite{LanLiuZhu}. Let us introduce some more notation. Write $\Sh^\varepsilon_{K}(\C)$ for the complex manifold obtained from $\Sh^\varepsilon_{K}$ by base change along $x_\infty:E\hra \C$, and $\cL_{\xi}^{\tu{top}}$ for the topological local system on $\Sh^\varepsilon_{K}(\C)$ coming from $\xi$. Writing $K^\varepsilon$ (Lemma~\ref{lem:IsShimuraDatum}) as $K^\eps=\prod_{y} K^\eps_{y}$, we have $K^\eps_{y_\infty}=K_\varepsilon$ and $K^\eps_{y}=G(F_y)\simeq\GSO^{\cmpt}_{2n}(F_y)$ for $y\neq y_\infty$. Restricting $h^\eps_{\C}$ to the first factor of $\SS_\C = \G_{m,\C}\times \G_{m,\C}$ (labeled by the identity $\C\ra\C$, not the complex conjugation), we obtain a cocharacter $\G_{m,\C}\ra K^\eps_{\C}$, which we denote by $\mu^\eps$; this is consistent with the definition of $\mu^\eps$ below \eqref{eq:Shimurah}.
We also have a parabolic subgroup $Q\subset (\Res_{F/\Q}G)_{\C}$ with Levi component $K^\eps_{\C}$ as \cite[p.57]{FaltingsLocallySymmetric} (such that the Borel embedding goes into $(\Res_{F/\Q}G)(\C)/Q$).
Fix an elliptic maximal torus $T_\infty\subset K^\varepsilon$ and a Borel subgroup $B\subset (\Res_{F/\Q}G)_{\C}$ contained in $Q$ such that $B$ contains $T_\infty$. Let $R^+$ denote the set of positive roots of $T_\infty$ in $B$. By $R^-$ we denote the set of roots of $T_\infty$ in the opposite Borel subgroup. Write $\Omega$ for the Weyl group of $T_{\infty,\C}$ in $(\Res_{F/\Q}G)_\C$, and $\Omega_{\tu{nc}}$ for the subset of $\omega\in \Omega$ such that $\omega\lambda$ is $B\cap K^\eps_{\C}$-dominant whenever $\lambda\in X^*(T_\infty)$ is $B$-dominant. Let $\Omega_{\tu{c}}$ denote the Weyl group of $T_{\infty,\C}$ in $K^\eps_\C$. The inclusion $\Omega_{\tu{nc}}\subset \Omega$ induces a canonical bijection $\Omega_{\tu{nc}}\simeq \Omega/\Omega_{\tu{c}}$. We parametrize members of the discrete series $L$-packet $\Pi^{G(F_\infty)}_\xi$ as $\{ \pi(\omega) | \omega\in \Omega_{\tu{nc}}\}$ following \cite[3.3]{HarrisAnnArbor}. (Our $\pi(\omega)$ is $\pi(\omega\lambda,\omega R^+)$ in their notation.)
Even though $Q,T_\infty,B,\Omega_c$ depend on $\eps$ (since they do on $K^\eps$), we suppress it from the notation for simplicity.

Write $\rho_G\in X^*(T)$ for the half sum of all roots in $R^+$, and define $\omega\star \lambda_0:=\omega(\lambda_0+\rho_G)-\rho_G$ for $\lambda_0\in X^*(T_\infty)$. Every irreducible representation $V_{\lambda_0}$ of $K^\varepsilon_\C$ of highest weight $\lambda_0\in X^*(T_\infty)$ gives rise to an automorphic vector bundle, to be denoted by $\cE_{\lambda_0}$. Write $\lambda=\lambda(\xi)\in X^*(T_\infty)$ for the $B$-dominant highest weight of $\xi$. 

For a finite multi-set $A$, write $\tu{mult}(a|A)\in \Z_{\ge 0}$ for the multiplicity of $a$ in $A$. For each $j\in \Z$, define $\Omega_{\tu{nc}}(j)$ to be the set of $\omega\in \Omega_{\tu{nc}}$ such that the composition $\G_m \stackrel{(\mu^\eps)^{-1}}{\ra} T_{\infty,\C}\stackrel{\omega\star\lambda}{\ra} \G_m$ equals $z\mapsto z^j$. Then
\begin{eqnarray*}
& &\tu{mult}\big(j \,|\, \HT_{\iota x_\infty}(\rho_{\pi^\natural}^{\Sh,\varepsilon})\big) \nonumber\\
&=& \sum_{\tau\in A(\pi^\natural)/\sim} \tu{mult}\big(j \,|\, \HT_{\iota x_\infty}(H^{n(n-1)/2}(\Sh^\varepsilon,\cL_\xi)[\iota \tau^\infty]\big) \quad \mbox{by~Thm.~\ref{thm:PointCounting}~and~\eqref{eq:Rho2Plus}},\nonumber\\
&=& \sum_{\tau\in A(\pi^\natural)/\sim} \sum_{\omega\in \Omega_{\tu{nc}}(j)}
\dim H^{\frac{n(n-1)}{2}-l(\omega)}(\Sh^\epsilon(\C),\cE_{\omega\star \lambda})[\tau^\infty], \quad \mbox{by~\cite[Thm.~6.2.9]{LanLiuZhu}}. 
\end{eqnarray*}
From \cite[\S3]{HarrisAnnArbor} we have an isomorphism of $G(\A_F^\infty)$-modules:
$$
H^{k}(\Sh(\C),\cE_{\omega\star \lambda})\simeq \bigoplus_{\tau}
m(\tau)\tau^\infty\otimes H^{k} (\Lie Q, K^\eps,\tau_\infty\otimes V_{\omega\star \lambda}).
$$
For each $\tau\in A(\pi^\natural)$, we pass to the $\tau^\infty$-isotypic parts (with notation as in \eqref{eq:tau-isotypic}) to obtain
$$ H^k(\Sh(\C),\cE_{\omega\star \lambda})[\tau^\infty]
\simeq \bigoplus_{\tau'_\infty} m(\tau^\infty\otimes \tau'_\infty) H^{k} (\Lie Q, K^\eps,\tau'_\infty\otimes V_{\omega\star \lambda}),$$
where the sum runs over irreducible unitary representations of $G(F_\infty)$. 
By \cite[Prop.~4.4.12]{HarrisAnnArbor}, if the cohomology on the right hand side is nonzero then $\tau'_\infty$ is $\xi$-cohomological, so $\tau^\infty\otimes\tau'_\infty\in A(\pi^\infty)$. It follows from Corollary \ref{cor:Pi-is-in-Disc-series-Packet-2} that $\tau'_\infty\in \Pi^{G(F_\infty)}_\infty$ and that $ m(\tau^\infty\otimes \tau'_\infty)=m(\tau)$. Moreover, \cite[Thm.~3.4]{HarrisAnnArbor} implies that $H^{k} (\Lie Q, K^\eps,\tau'_\infty\otimes V_{\omega\star \lambda})$ is nonzero for a unique $\tau'_\infty$, in which case the cohomology is one-dimensional. We use this to resume the computation of $ \tu{mult}(j \,|\, \HT_{\iota x_\infty}(\rho_{\pi^\natural}^{\Sh,\varepsilon}))$ and obtain
$$
\tu{mult}\big(j \,|\, \HT_{\iota x_\infty}(\rho_{\pi^\natural}^{\Sh,\varepsilon})\big)
=  \sum_{\tau\in A(\pi^\natural)/\sim} \sum_{\omega\in \Omega_{\tu{nc}}(j)}
m(\tau) = a(\pi^\natural) \cdot |\Omega_{\tu{nc}}(j)|.
$$

To conclude \eqref{eq:HTweights}, it remains to prove the following claim: that $| \Omega_{\tu{nc}}(j)|$ is precisely the number of ways $j$ can be written as $-a_0 - \sum_{i\in I} a_i +\sum_{i\notin I}  (n-i)$ with $I\in   \mathscr P^\varepsilon(n)$.

As a preparation, we fix an isomorphism between the pairs $(T_{\infty,\C},B)$ and $(\TGSO,B_{\GSO})$ induced by an inner twist $(\Res_{F/\Q}G)_{\C}\simeq (\Res_{F/\Q}G^*)_{\C}$. So the Weyl action of $\Omega$ is identified with the $W_{\GSO}$-action in \eqref{eq:WeylGroupAction}, while $\Omega_{\tu{c}}$ is identified with $\mathfrak{S}_n$ therein if $\eps=(-1)^n$. (If $\eps=(-1)^{n+1}$ then $\Omega_{\tu{c}}$ is the $\theta^\circ$-conjugate of $\mathfrak{S}_n$.) For a subset $I\subset \{1,...,n\}$, let $\omega'_I$ denote the action on $(t_0,t_1,\ldots,t_n)\in X^*(\TGSO)$ by $t_i\mapsto t_i$ for $i\in I$, $t_i\mapsto -t_i$ for $i\notin I$, and $t_0\mapsto t_0+\sum_{i\in I} t_i$. Then $\omega'_I\in\Omega$ if and only if $n-|I|$ is even.

Let us prove the claim, starting with the case $\eps=(-1)^n$. Then $n-|I|$ is even for each $I\in   \mathscr P^\varepsilon(n)$. Write $\omega_I\in \Omega_{\tu{nc}}$ for the unique intersection of the $\Omega_{\tu{c}}$-orbit of $\omega_I$ with $\Omega_{\tu{nc}}$. We have bijections
$$ 
\mathscr P^\varepsilon(n) \stackrel{\sim}{\ra}   \Omega/\Omega_{\tu{c}}   \stackrel{\sim}{\leftarrow} \Omega_{\tu{nc}}, \qquad I\mapsto \omega'_I \mapsto \omega_I.
$$
Since $\eps=(-1)^n$, we have from \eqref{eq:Spin-eps-def}
$$
\mu^\eps =(z\mapsto (z,z,,\ldots,z,z))\in  X_*(\TGSO) \simeq X_*(T_\infty),
$$ 
a priori up to the $\Omega_{\tu{c}}$-action, but $\mu^\eps$ is $\Omega_{\tu{c}}$-invariant. From this, we compute for $\lambda=(a_0,a_1,\ldots,a_n)\in X^*(T_\infty)\simeq X^*(\TGSO)$:
$$
(\omega_I \star \lambda)\circ (\mu^\eps)^{-1} = (\omega'_I \star \lambda)\circ (\mu^\eps)^{-1} =
  -a_0 - \sum_{i\in I}(a_i+n-i) + \sum_{\tu{all}~i} (n-i)= -a_0-\sum_{i\in I}a_i + \sum_{i\notin I} (n-i).
$$
Thus the claim for $\eps=(-1)^n$ follows. 

Keep $\eps=(-1)^n$ and let us prove the claim for $P^{-\varepsilon}(n)$. Since $n-|I|$ is odd, we no longer have $\omega'_I\in \Omega$ but instead have $\omega''_I:=\theta^\circ\omega'_I = \omega'_I \theta^\circ \in \Omega$. Replacing $\omega'_I$ with $\omega''_I$ in the previous paragraph, we obtain $\omega_I$ and analogous bijections
$$ 
 \mathscr P^{-\varepsilon}(n) \stackrel{\sim}{\ra}   \Omega/\Omega_{\tu{c}}   \stackrel{\sim}{\leftarrow} \Omega_{\tu{nc}},
\qquad I\mapsto \omega''_I \mapsto \omega_I.
$$
It follows from $\mu^{-\eps}=\theta^\circ \mu^{\eps}$ and $\Omega_{\tu{c}}$-invariance of $\mu^\eps$ that
$$
(\omega_I \star \lambda)\circ (\mu^{-\eps})= (\theta^\circ(\omega_I \star \lambda))\circ (\mu^\eps)= (\omega'_I \star \lambda)\circ (\mu^\eps).
$$
The proof is now done since the computation of $(\omega'_I \star \lambda)\circ(\mu^\eps)^{-1}$ in the preceding case goes through verbatim.
\end{proof}

\section{Construction of $\GSpin_{2n}$-valued Galois representations}\label{sect:Construction}

We continue in the setting of \S\ref{sect:forms-of-GSO} and \S\ref{sect:Shimura}. The goal of this section is to attach $\GSpin_{2n}$-valued Galois representations of $\Gamma_{E}$ to the automorphic representations of $G^*=\GSO^{E/F}_{2n}$ under consideration. The main input comes from the cohomology of Shimura varieties studied in the last section. Write $\std \colon \GSpin_{2n} \hra \GL_{2n}$ for the composite of $\tu{pr} \colon \GSpin_{2n}\ra \GSO_{2n}$ and the inclusion $\GSO_{2n}\subset \GL_{2n}$.

Let $\pi$ be a cuspidal automorphic representation of $G^*(\A_F)$. Let $\phi_{\pi_y}$ denote the $L$-parameter of $\pi_y$ for $y\in \cV_\infty$. Throughout this section, we assume that
\begin{itemize}[leftmargin=+.8in]
\item [\textbf{(St)}] for some finite $F$-place $\qst$ the local representation $\pi_{\qst}$ is isomorphic to the Steinberg representation up to a character twist,
\item [\textbf{(coh)}] the representation $\pi_\infty$ is cohomological for some representation $\xi$ of $(\Res_{F/\Q}G^*)\otimes_{\Q} \C$
(then $\xi$ satisfies condition (cent) by \cite[Lem.~7.1]{GSp} as before).
\end{itemize}

Choose $\pi^\flat$ a cuspidal automorphic representation of $\SO^{E/F}_{2n}(\A_F)$ contained in $\pi|_{\SO^{E/F}_{2n}(\A_F)}$ (see \cite{LabesseSchwermer2018}).
We observe that $\pi^\flat$ satisfies conditions (St$^\circ$) and (coh$^\circ$) of \S\ref{sect:GSO-valued-Galois} thanks to Lemma~\ref{lem:SteinbergCheck} and~\ref{lem:coh-GSO-coh-SO}.
Consider the following analogue of (std-reg$^\circ$) for $\pi$
\begin{itemize}[leftmargin=+.8in]
\item [\textbf{(std-reg)}] $\std\circ \phi_{\pi_y}|_{W_{\ol F_y}}$ is regular at every $y\in \cV_\infty$.
\end{itemize}
In addition to (St) and (coh), the following is also assumed throughout:
\begin{itemize}
\item Either (std-reg) holds for $\pi$, or Hypothesis \ref{hypo:non-std-reg} is true for $\pi^\flat$.
\end{itemize}
So Hypothesis \ref{hypo:non-std-reg} comes into play only when (std-reg) does not hold.

Condition (std-reg) is equivalent to the one given in the introduction via local Langlands for real groups, e.g., see \cite[\S2.3]{BuzzardGee}. If (std-reg) is imposed on $\pi$, then (std-reg$^\circ$) follows from (coh$^\circ$). By \cite[\S3, (iv)]{LanglandsRealClassification}, we have that $\phi_{\pi^\flat,y}=\tu{pr}^\circ\circ \phi_{\pi,y}$ at each $y\in \cV_\infty$. We can also see (std-reg$^\circ$) from this and (std-reg).

\begin{lemma}\label{lem:tempered}
In addition to (St) and (coh) for $\pi$, assume either (std-reg) for $\pi$ or Hypothesis \ref{hypo:non-std-reg} for $\pi^\flat$. Then $\pi^\flat$ is tempered at all places, and $\pi$ is essentially tempered at all places.
\end{lemma}

\begin{proof}
This follows from Proposition \ref{prop:SO-to-GL} if (std-reg) holds. Otherwise, the same proposition implies $\pi^\flat$ is tempered at infinite places, and Hypothesis \ref{hypo:non-std-reg} asserts that $\pi^\flat$ is tempered at finite places. The last assertion easily follows from the temperedness of $\pi^\flat$. (See the proof of \cite[Lem.~2.7]{GSp}.)
\end{proof}

The right hand side of \eqref{eq:QuasiSplitGSO} is easily extended to a model of $\GSO_{2n}^{E/F}$ over $\cO_{F}$ (by replacing $E,F$ with $\cO_E,\cO_F$). Similarly we have a model of $\SO_{2n}^{E/F}$ closed in the model of $\GSO_{2n}^{E/F}$, defined by the condition $\lambda=1$. At each $F$-prime $\qq$ not above $2$ and unramified in $E$, we have the hyperspecial group $H_{\qq}:=\GSO_{2n}^{E/F}(\cO_{F_\qq})$, whose intersection with $\SO_{2n}^{E/F}(F_{\qq})$ is the hyperspecial subgroup $H_{0,\qq}:=\SO_{2n}^{E/F}(\cO_{F_{\qq}})$ in the latter. We will fix these choices of hyperspecial subgroups for $\GSO_{2n}^{E/F}$ and $\SO_{2n}^{E/F}$.
At each $\qq\in \tu{Unr}(\pi)$ (so that $\pi_v^{H_{\qq}}$ is nontrivial), we can thus find an irreducible $\SO_{2n}^{E/F}(F_{\qq})$-subrepresentation in $\pi_v$ with nonzero $H_{0,\qq}$-fixed vectors.
Consequently, after translating $\pi^\flat$ inside of $\pi$ by a suitable $g \in \GSO_{2n}^{E/F}(\A_F)$, we may assume that $\pi^\flat_{\qq}$ is unramified at every $\qq$ not above $\Sbad$ (with respect to the hyperspecial subgroups above).

Thanks to Theorem \ref{thm:ExistGaloisSO} if (std-reg) is assumed, or instead by Hypothesis \ref{hypo:non-std-reg}, we have a Galois representation
$$
\rho_{\pi^\flat} \colon \Gamma_{F, \Sbad} \to \SO_{2n}(\lql)\rtimes \Gal(E/F),
$$
whose restriction to $\Gamma_{E,\Sbad}$ satisfies, writing $\qq:=\pp\cap F$ for each $\pp$,
\begin{equation}\label{eq:SmallAmbiguity}
\rho_{\pi^\flat}(\Frob_\pp)_{\tu{ss}} \osim  \iota \phi_{\pi^\flat_\qq}(\Frob_\pp) \in \SO_{2n}(\lql),
\end{equation}
for all ${E}$-places $\pp \notin \SEbad$. Here $\osim$ indicates $\OO_{2n}$-conjugacy (instead of $\SO_{2n}$-conjugacy).

Let $H\subset \SO_{2n}$ denote the Zariski closure of the image of $\rho_{\pi^\flat} \colon \Gamma_{E,\Sbad}\ra \SO_{2n}(\lql)$. By Proposition~\ref{prop:containingRegularUnipotent}, either $H$ is connected or $H=H^0\times Z(\SO_{2n})$.
Therefore, via $\{\pm1\}=Z(\SO_{2n})$, we can find a Galois character
\begin{equation}\label{eq:eta-def}
\eta \colon \Gamma_{E,\Sbad} \ra \{\pm1\}
\end{equation}
such that the product morphism $\eta \rho_{\pi^\flat}$ has Zariski dense image in $H^0$. In particular if $H^0 = H$ we take $\eta = 1$. We define the character
\begin{equation}\label{eq:eta-tilde-def}
\wt \eta \colon \Gamma_{E, \Sbad} \ra \langle z^+ \rangle \subset \GSpin_{2n}
\end{equation}
to be the character so that the composition $\wt \eta \colon \Gamma_{E, \Sbad} \to \langle z^+ \rangle \isomto_{\pr^\circ} \{\pm 1\}$ is equal to $\eta$. 

Recall that $G$ is an inner form of $G^*=\GSO^{E/F}_{2n}$ giving rise to the Shimura data $(\Res_{F/\Q} G,X^{\pm})$ studied earlier. By \cite[Prop.~6.3]{GSp}, there exists a cuspidal automorphic representation $\pi^\natural$ of $G(\A_F)$ such that
\begin{itemize}
\item $\pi^\natural_{\qq'} \simeq \pi_{\qq'}$ at every finite prime $\qq'$ where $\pi_{\qq'}$ is unramified (we have $G_{\qq'}\simeq G^*_{\qq'}$ at such $\qq'$),
\item $\pi^\natural_{\qst}$ is a character twist of the Steinberg representation,
\item $\pi^\natural_\infty$ is $\xi$-cohomological.
\end{itemize}
The first condition and Lemma \ref{lem:tempered} imply that $\pi^\natural$ satisfies condition (temp) of \S\ref{sect:Shimura}. Theorem~\ref{thm:PointCounting} yields semisimple representations $\rho_{\pi^\natural}^{\Sh,\eps}$ of $\Gamma_{E,S}$ for $\eps \in \{\pm 1\}$ such that its dual $\rho_{\pi^\natural}^{\Sh,\eps,\vee}$ has the following property:
\begin{equation}\label{eq:LGC-Shimura}
\rho_{\pi^\natural}^{\Sh,\eps,\vee}(\Frob_\pp)_{\tu{ss}} \sim \iota q_\pp^{-n(n-1)/4}\big( i_{a_\pi}\circ \spin^{\eps}(\phi_{\pi_{\qq}}(\Frob_\pp))\big)\in \GL_{a_\pi 2^{n-1}}(\lql),
 \qquad \pp\notin S^E,
\end{equation}
where $S$ is a finite set of rational primes containing $\Sbad$, large enough, so that Theorem \ref{thm:PointCounting} holds for both $\eps= +$ and $\eps = -$. We define $\rho_{\pi}^{\Sh, \eps} := \rho_{\pi^\natural}^{\Sh, \eps}$ for $\eps\in \{\pm\}$ (which depends not on the choice of $\pi^\natural$ but only on $\pi$ by \eqref{eq:LGC-Shimura}), and
$$
\wt{\rho}^{\Sh,\vee}_\pi := \rho_\pi^{\Sh, +,\vee} \oplus \big( \eta\otimes \rho_\pi^{\Sh, -,\vee}\big).
$$
Then $\wt \rho^{\Sh,\vee}_\pi$ is a $\Gamma_{E,S}$-representation of dimension $a_\pi 2^n$, where $a_\pi := a(\pi^\natural)$. We set
$$
\tspin(\cdot):=\spin^+(\cdot) \oplus ( \eta\otimes \spin^-(\cdot))
$$
when the input is a $\GSpin_{2n}$-valued Galois representation or a local $L$-parameter, and write $\tspin^{a}(\cdot)$ for the $a$-fold self-direct sum of $\tspin(\cdot)$. (So $\tspin=\spin$ if $\eta=1$.)
We have
\begin{equation}\label{eq:ShimuraTrace}
\wt{\rho}_\pi^{\Sh,\vee}(\Frob_\pp)_{\tu{ss}} \sim \iota q_\pp^{-n(n-1)/4} \tspin^{a_\pi}(\phi_{\pi_{\qq}}(\Frob_\pp))\in \GL_{a_\pi 2^{n}}(\lql),\qquad \pp\notin S^E.
\end{equation}
 Then $\rho^{\Sh,\vee}_\pi$, $\wt \rho^{\Sh,\vee}_\pi$, are a $\Gamma_{E,S}$-representation of dimension $a_\pi 2^n$, where $a_\pi := a(\pi^\natural)$. 

When $*$ is a map (resp.~an element), we use $\ol{*}$ to denote the composition with the adjoint map (resp.~the image under the adjoint map) that is clear from the context.

\begin{proposition}\label{prop:UpToOuter}
There exists a continuous semisimple representation
$$
\rho_\pi^C \colon \Gamma_{{E},S} \to \GSpin_{2n}(\lql)
$$
(with $C$ standing for a cohomological normalization) such that we have
\begin{eqnarray}\label{eq:ConsequencesOfDiagram0}
\forall \pp \notin S^E: & & \quad \tspin \big(\rho_\pi^C(\Frob_\pp)_{\tu{ss}}\big) \sim \iota q_\pp^{-n(n-1)/4}  \tspin (\phi_{\pi_\qq}(\Frob_\pp)) \in \GL_{2^{n}}(\lql), \\
\forall \pp \notin \Sbad^E: && \quad  \pr^\circ \rho_\pi^C(\Frob_\pp)_{\tu{ss}} \osim \iota  \pr^\circ \phi_{\pi_\qq}(\Frob_\pp) \in \SO_{2n}(\lql).\label{eq:ConsequencesOfDiagram}
\end{eqnarray}
\end{proposition}
\begin{proof}
Consider the diagram
\begin{equation}\label{eq:BigDiagram}
\xymatrix{
\Gamma_{E,S} \ar@/^1.5pc/[rrr]^{\tilde\rho_\pi^{\Sh,\vee}} \ar@/_1.5pc/[drr]_{\eta\rho_{\pi^\flat}}  && \GSpin_{2n}(\lql) \ar[d]_{\pr^\circ}\ar@{^(->}[r]_-{\spin^{a_\pi}} &\GL_{a_\pi 2^{n}}(\lql)\ar[d] &\cr && \SO_{2n}(\lql) \ar@{^(->}[r] \ar@{^(->}[r]_{\overline{\spin^{a_\pi}}\quad}&\PGL_{a_\pi 2^{n}}(\lql). \cr }
\end{equation}
At each prime $\pp$ of $E$ not above $S$, we obtain from~\eqref{eq:SmallAmbiguity} that
\begin{eqnarray}
&& \li {\spin^{a_\pi}} ( (\eta \rho_{\pi^\flat})(\Frob_\pp)_{\tu{ss}}) \sim \iota \li {\spin^{a_\pi}}((\eta \phi_{\pi_\qq^\flat})(\Frob_\pp)) \nonumber\\
& = &\iota \li { \spin^{a_\pi}((\tilde\eta\phi_{\pi_\qq^\flat})(\Frob_\pp)) } \sim \li{\wt{\rho}_\pi^{\Sh,\vee}(\Frob_\pp)_{\tu{ss}}} \in \PGL_{a_\pi 2^n}(\lql).
\label{eq:FrobeniusConjugacy}
\end{eqnarray}
Recall that $\eta \rho_{\pi^\flat}$ has connected image. So \eqref{eq:FrobeniusConjugacy} implies, via \cite[Prop.~4.6, Ex.~4.7]{GSp}, the existence of $g \in \GL_{a_\pi 2^{n}}(\lql)$ such that
$$
\overline {\wt{\rho}_\pi^{\Sh,\vee}} = g \big( \li {\spin^{a_\pi}}(\eta \rho_{\pi^\flat}) \big) g^{-1} \colon \Gamma_{{E},S} \to \PGL_{a_\pi 2^{n}}(\lql).
$$
Replace $\wt{\rho}_\pi^{\Sh,\vee}$ by $g^{-1} \wt{\rho}_\pi^{\Sh,\vee} g$ so that $\li{\wt{\rho}_\pi^{\Sh,\vee}} = \li {\spin^{a_\pi}}( \eta\rho_{\pi^\flat})$. From Diagram~\eqref{eq:BigDiagram} we deduce that
$$
\wt{\rho}_{\pi}^{\Sh,\vee}(\Gamma_{{E},S}) \subset \pr^{\circ,-1}\big( (\eta \rho_{\pi^\flat})(\Gamma_{{E},S})\big) \subset \GSpin_{2n}(\lql),
$$
where $\GSpin_{2n}$ is viewed as a subgroup of $\GL_{a_\pi 2^n}$ via $\spin^{a_\pi}$. That is, there exists a representation $\wt{\rho}_\pi^C \colon \Gamma_{{E},S} \to \GSpin_{2n}(\lql)$ such that 
$$
\spin^{a_\pi}(\wt{\rho}^C_\pi) = \wt{\rho}_\pi^{\Sh,\vee}\quad\mbox{and}\quad \pr^\circ \wt{\rho}_{\pi}^C = \eta\otimes \rho_{\pi^\flat}.
$$
Define $\rho_\pi^C:=\tilde\eta \wt{\rho}_\pi^C$. Then it follows that
$$
\tspin^{a_\pi}(\rho^C_\pi) = \wt{\rho}_\pi^{\Sh,\vee}\quad\mbox{and}\quad \pr^\circ \rho_{\pi}^C = \rho_{\pi^\flat}.
$$
Thanks to \eqref{eq:SmallAmbiguity} and \eqref{eq:ShimuraTrace}, $\rho_\pi^C$ satisfies \eqref{eq:ConsequencesOfDiagram} and \eqref{eq:ConsequencesOfDiagram0}. The proof is complete.
\end{proof}

\begin{remark}
The bottom row in \eqref{eq:BigDiagram} cannot be replaced with $\PSO_{2n}$. (If it did, since $\ol \rho_{\pi^\flat}$ has connected image in $\PSO_{2n}$ by Proposition~\ref{prop:containingRegularUnipotent}, the argument above would work without introducing the $\eta$-twist.) For instance, observe that $\GSpin_{2n} \stackrel{\spin}{\longrightarrow} \GL_{2^{n}} \ra \PGL_{2^{n}}$ does not factor through $\PSO_{2n}$ since $\spin^+$ and $\spin^-$ have different central characters.
\end{remark}

We can refine \eqref{eq:ConsequencesOfDiagram0} by separating $\spin^+$ and $\spin^-$, which is a key intermediate step towards the main theorem. Our argument is quite delicate and sensitive to the underlying group-theoretic structures.

\begin{proposition}\label{prop:separation}
Up to replacing $\rho_\pi^C$ by $\eta \theta(\rho_\pi^C)$ if necessary, we have the following.
For every $\pp \notin S^E$ and $\varepsilon\in \{+,-\}$
\begin{eqnarray}\label{eq:separation0}
\forall \pp \notin S^E: && \quad \spin^\varepsilon \rho_\pi^C(\Frob_\pp)_{\tu{ss}} \sim \iota q_\pp^{-n(n-1)/4}  \spin^\varepsilon \phi_{\pi_\qq}(\Frob_\pp) \in \GL_{2^{n-1}}(\lql), \\
\forall \pp \notin \Sbad^E: && \quad   \pr^\circ \rho_\pi^C(\Frob_\pp)_{\tu{ss}} \osim \iota  \pr^\circ\phi_{\pi_\qq}(\Frob_\pp) \in \SO_{2n}(\lql).\label{eq:separation}
\end{eqnarray}
where we write $\qq$ for the prime of $F$ below $\pp$.
\end{proposition}

\begin{proof}
Recall from \S\ref{sect:Notation} that we often write $G_0$ to mean $G_0(\lql)$ when $G_0$ is a reductive group over $\lql$. Moreover we assume $\pp \notin S^E$ throughout, without repeating this condition.

The assertion \eqref{eq:separation} follows from \eqref{eq:ConsequencesOfDiagram} (and it is invariant under conjugation by an element of $\GPin_{2n}$). The main thing to prove is \eqref{eq:separation0}. For simplicity, write $\rho:=\rho_\pi^C$, $\check\rho^{\Sh, \eps} := \rho_\pi^{\Sh, \eps,\vee}$,
$\rho^\circ:=\pr^\circ\rho_\pi^C$, and $a:=a_\pi$. 
 From \eqref{eq:ShimuraTrace} and \eqref{eq:ConsequencesOfDiagram0} we have
\begin{equation}\label{eq:Sh+-Sh-}
\check\rho^{\Sh,+}\oplus (\eta\otimes \check\rho^{\Sh,-})\simeq \left(\spin^+ \rho \oplus (\eta\otimes\spin^- \rho)\right)^{\oplus a}.
\end{equation}

Write $Z:=Z(\GSpin_{2n})$ and $H$ for the Zariski closure of $\textup{im}(\rho^\circ)$ in $\SO_{2n}$. Then $H$ contains a regular unipotent element by Corollary~\ref{cor:image-has-reg-unip}. We divide into two cases based on Proposition~\ref{prop:containingRegularUnipotent}.

\smallskip

\paragraph{\textbf{Case 1}} Assume $\spin^\varepsilon\rho_\pi^C$ is irreducible for $\eps = -$. This happens when $H^0$ is $\SO_{2n}$, $i^\circ_{\std}(\SO_{2n-1})$, or $n = 4$ and $H^0 = \spin^\circ(\Spin_7)$ (possibly after conjugation in $\GSpin_{2n}$). In the first two subcases $\spin^+ \rho$ is also irreducible; for irreducibility in the third case, see Lemma~\ref{lem:IrreducibilityOfSpin-}. 
 
If $\spin^+ \rho \simeq \eta\otimes\spin^- \rho$ then it is clear from \eqref{eq:Sh+-Sh-} that $\check\rho^{\Sh,+}\simeq \eta\otimes \check\rho^{\Sh,-}\simeq (\spin^+\rho)^{\oplus a} \simeq (\eta\otimes \spin^-\rho)^{\oplus a}$. So the proposition follows from Theorem~\ref{thm:PointCounting}. 

Henceforth assume that  $\spin^+ \rho \not\simeq \eta\otimes\spin^- \rho$.

We claim that $\spin^-\rho(\gamma)_{\textup{ss}}$ is regular in $\GL_{2^{n-1}}$ on a density 1 set of $\gamma\in \Gamma$.  Define $X^-$ to be the subset of $h\in H(\lql)$ such that the semisimple part of $\li{\spin}^-(h)$ is non-regular in $\PGL_{2^{n-1}}$. Then $X^-$ is Zariski-closed and conjugation-invariant in $H$. 
To show $H\neq X^-$, let $\tilde H\subset \GSpin_{2n}$ be the preimage of $H$, so that $\tilde H^0$ equals  $\Spin_{2n}$, $i_{\std}(\Spin_{2n-1})$, or $\li\spin(\Spin_7)$ in the three cases, respectively.
 Then the restriction of $\spin^-$ via $\tilde H^0 \hra \GSpin_{2n}$ is an irreducible representation with distinct weight vectors. (When $\tilde H^0=i_{\std}(\Spin_{2n-1})$, the restriction is the spin representation of $\Spin_{2n-1}$ by Proposition~\ref{prop:res-of-spin}.) So some element $h_0$ of $\tilde H^0$ maps to a regular element of $\GL_{2^{n-1}}$ under $\spin^-$. It follows that some element of $H^0$ maps to a regular element of $\PGL_{2^{n-1}}$. 
We know that $\wt H$ is a subgroup of $Z(\GSpin_{2n})\wt H \subset \GSpin_{2n}$, 
thus by multiplying $h_0$ by elements in the center, we obtain in each connected component of $H$ an element with regular image in $\GL_{2^{n-1}}$.
 In particular for each connected component $C$ of $H$ we have $X^- \cap C \neq C$  and thus $\dim X^- <\dim H$. Therefore the set of $\gamma$ such that $\rho^\circ(\gamma)\notin X^-$ has density 1 according to  Lemma~\ref{lem:Chebotarev}, and in this case $\li{\spin^-\rho(\gamma)_{\textup{ss}}}=\li{\spin^-}(\rho^\circ(\gamma)_{\tu{ss}})$ is regular. The claim is verified.

Given a square matrix $g$, let $\mathscr{EV}(g)$ for the multi-set of its eigenvalues. Since  $\spin^+ \rho \not\simeq \eta\otimes \spin^- \rho$, there exists $\gamma\in \Gamma$ such that 
\begin{itemize}
  \item $\spin^- \rho(\gamma)$ has distinct eigenvalues,
  \item $\mathscr{EV}(\eta(\gamma) \spin^- \rho(\gamma))\neq \mathscr{EV}(\spin^+ \rho(\gamma))$.
\end{itemize}
In particular there exists an eigenvalue $\alpha $ of $\eta(\gamma) \spin^- \rho(\gamma)$ which is not an eigenvalue of $\spin^+ \rho(\gamma)$. Then $\alpha$ appears as an eigenvalue with multiplicity $a$ on the right hand side of \eqref{eq:Sh+-Sh-}. We know from Theorem \ref{thm:PointCounting} that each eigenvalue of $\check\rho^{\Sh,+}(\gamma)$ and $\eta(\gamma)\check\rho^{\Sh,-}(\gamma)$ appears with multiplicity divisible by $a$. Thus $\alpha$ is an eigenvalue of either $\check\rho^{\Sh,+}(\gamma)$ or $\eta(\gamma) \check\rho^{\Sh,-}(\gamma)$ but not both. This implies, together with Theorem \ref{thm:PointCounting} and the irreducibility of $\spin^-\rho$, that (i) $(\eta \otimes \spin^-\rho)^{\oplus a} \simeq \eta \otimes \check \rho^{\Sh,-}$ and $(\spin^+ \rho)^{\oplus a} \simeq \check\rho^{\Sh,+}$, or (ii) $(\eta \otimes \spin^-\rho)^{\oplus a} \simeq \check\rho^{\Sh,+}$ and $(\spin^+\rho)^{\oplus a}\simeq \eta \otimes \check\rho^{\Sh,-}$. In case (i), Equation~\eqref{eq:separation0} follows from Theorem~\ref{thm:PointCounting}. If (ii) occurs, replace $\rho$ with $\eta\otimes (\vartheta \rho \vartheta^{-1})$, where $\vartheta\in \GPin_{2n}$ is as in \eqref{eq:Elementw_spin}. (Here $\tu{im}(\eta)=\{\pm1\}$ is viewed as the subgroup of $\ker(\pr^\circ)=\G_m$.) Then equations \eqref{eq:ConsequencesOfDiagram0} and \eqref{eq:ConsequencesOfDiagram} are still true (as $\pr^\circ(\eta) = 1$). Hence if we run the current proof again, we will be in Case 1(i). We are done in Case 1.

\smallskip

\paragraph{\textbf{Case 2}} We now assume $H^0 \subset i^{\circ}_{\std}(\SO_{2n-1})$, which covers the  cases $H^0 = i^\circ_{\reg}(\PGL_2)$, $H^0 = i^\circ_\std(G_2)$ and $n = 4$. By Proposition \ref{prop:containingRegularUnipotent} we have $H \subset H^0 Z(\SO_{2n})$, and $\rho$ has image in the group $H_{2n-1}$ from \eqref{eq:H-subgroup}. By \eqref{eq:eta-def}, $\eta \rho^\circ$ has dense image in $H^0$, and by \eqref{eq:eta-tilde-def}, $\wt \rho$ has image in $\GSpin_{2n-1} \subset H_{2n-1}$. In particular $\eta$ is equal to $\kappa_0 \circ \rho$, with $\kappa_0$ from \eqref{eq:kappa_0-def}. From \eqref{eq:kappa_0-and-theta} we obtain
\begin{equation}\label{eq:Case2-thetaStability-1}
\theta \rho_\p \sim \eta_\p \rho_\p  \in \GSpin_{2n},
\end{equation}
where we write $\rho_\p := \rho(\Frob_\p)_\tu{ss}$ and $\eta_\p := \eta(\Frob_\p)$. Similarly we write $\wt \eta_\p := \wt \eta(\Frob_\p)$ and $\phi_\p := \iota q_\p^{-n(n-1)/4} \phi_{\pi_\qq}(\Frob_\p)$. We claim
\begin{equation}\label{eq:Case2-thetaStability-2}
\theta \phi_\p \sim \eta_\p \phi_\p \in \GSpin_{2n}.
\end{equation}
By Equation \eqref{eq:SmallAmbiguity} we have $\pr^\circ \rho_\p \osim \pr^\circ \phi_\p \in \SO_{2n}$. Multiplying $\eta$, we obtain $\pr^\circ(\wt \eta_\p \phi_p) \osim \pr^\circ(\wt \eta_p \rho_\p)$. By assumption, $\pr^\circ(\wt \eta_p \rho_\p) \in i^{\circ}_{\std}(\SO_{2n-1})$. Hence $\wt \eta_\p \phi_\p = g x g\inv$ for some $g\in \GPin_{2n}$ and $x \in i_{\std}(\GSpin_{2n-1})$. We have $\theta(x) = x$ and thus
$$
\theta (\wt \eta_\p \phi_\p) = \theta(g) x \theta(g)\inv = (\theta(g) g\inv) \wt \eta_\p \phi_\p (g \theta(g)\inv)
$$
As $\theta(g) g\inv \in \GSpin_{2n}$, this implies that $\theta (\wt \eta_\p \phi_\p) \sim \wt \eta_\p \phi_\p$. Since $\theta(\wt \eta_\p) = \eta_\p \wt \eta_\p$, \eqref{eq:Case2-thetaStability-2} follows. 

In \eqref{eq:ConsequencesOfDiagram0} we established 
$$
\spin^+(\rho_\p) \oplus \eta_\p \spin^-(\rho_\p) \sim \spin^+(\phi_\p) \oplus \eta_\p \spin^-(\phi_\p),
$$
which implies by \eqref{eq:Case2-thetaStability-1} and \eqref{eq:Case2-thetaStability-2} $\spin^{+, \oplus 2}(\rho_\p) \sim \spin^{+, \oplus 2}(\phi_\p)$. It follows that $\spin^{+}(\rho_\p) \sim \spin^{+}(\phi_\p)$, Similarly we deduce $ \spin^{-}(\rho_\p) \sim \spin^{-}(\phi_\p)$.
\end{proof}

From now on, we replace, if necessary, $\rho_\pi^C$ by $\eta \theta(\rho_\pi^C)$ so that  the conclusion of Proposition \ref{prop:separation} holds for $\rho_\pi^C$. 

\begin{proposition}\label{prop:FinalProp}
We have that (writing $\qq:=\pp\cap F$)
\begin{equation}\label{eq:FinalCompatib}
\forall \pp \notin S^E:  \quad \rho_\pi^C(\Frob_\pp)_{\tu{ss}} \sim \iota q_\pp^{-n(n-1)/4} \phi_{\pi_\qq}(\Frob_\pp) \in \GSpin_{2n}(\lql).
\end{equation}
\end{proposition}
\begin{proof}
We first establish the claim that $\chi_\ell^{n(n-1)/2} \iota\omega_\pi = \cN\rho_\pi^C$, where $\chi_\ell$ is the cyclotomic character and we view $\omega_\pi$ as a Galois character via class field theory. In view of Lemma~\ref{lem:for-mult-one}(i), it suffices to check that
\begin{equation}\label{eq:SpinorNormCheckA}
\chi_\ell^{n(n-1)/2}\iota \omega_\pi \cdot \spin^\eps(\rho^C_{\pi})   \simeq \cN\rho_\pi^C \cdot \spin^\eps(\rho^C_\pi),\quad \eps \in \{\pm 1\}.
\end{equation}
By Lemma~\ref{lem:Duality} we have
\begin{equation}\label{eq:SpinorNormCheckB}
\spin^\eps(\rho_\pi^C) \simeq ( \spin^{(-1)^n \eps})^\vee(\rho_\pi^C) \otimes \cN\rho_\pi^C.
\end{equation}
Let $\pp \notin S^E$, and write shorthand $\rho_\p := \rho^C_\pi(\Frob_{\pp})_{\tu{ss}}$ and $\phi_\p := \iota q_\p^{-n(n-1)/4}\phi_{\pi_\qq}(\Frob_\pp)$. 
We apply \eqref{eq:separation0} and compute using Lemma~\ref{lem:Duality} again (but now locally)
\begin{align}\label{eq:SpinorNormCheckC}
\spin^\eps(& \rho_\p) \simeq \spin^\eps(\phi_\p) \simeq (\spin^{(-1)^n \eps})^\vee(\phi_\p) \otimes \cN(\phi_\p) \simeq (\spin^{(-1)^n \eps})^\vee(\rho_\p) \otimes \cN(\phi_\p)
\end{align}
We now appeal to functoriality of the Satake isomorphism (unramified local Langlands correspondence) with respect to $\Gm \hra \GSO_{2n}$ (dual to $\cN:\GSpin_{2n}\ra \Gm$), to get $\cN(\phi_\p) = \chi_\ell^{n(n-1)/2}(\Frob_\p) \iota \omega_\pi(\Frob_p)$. Therefore 
$$
\spin^\eps(\rho_\pi^C) \simeq ( \spin^{(-1)^n \eps})^\vee(\rho_\pi^C) \otimes \chi_\ell^{n(n-1)/2}\iota\omega_{\pi}.
$$ 
Comparing with \eqref{eq:SpinorNormCheckB}, we obtain \eqref{eq:SpinorNormCheckA}. At this point we have established that 
\begin{align}\label{eq:GatherEverything}
\spin^\varepsilon \rho_\pp  &\sim \spin^\varepsilon \phi_{\pp} \in \GL_{2^{n-1}}(\lql) & \textup{(Prop.~\ref{prop:separation}),} \cr
\pr^\circ \rho_\pp  & \osim \pr^\circ \phi_{\pp}  \in \SO_{2n}(\lql) & \textup{(Prop.~\ref{prop:separation}),} \cr
\cN \rho_\pp  & = \cN\phi_{\pp} \in \Gm(\lql) & \textup{(claim~above)}
\end{align}
By \cite[Lem.~1.1, table]{GSp} a semi-simple element $\gamma$ of $\GSpin_{2n}(\lql)$ is determined up to conjugacy by the conjugacy classes of $\spin^+\gamma, \spin^-\gamma \in \GL_{2^{n-1}}$, $\std \gamma \in \GL_{2n}$ and $\cN\gamma \in \Gm$. We complete the proof by noting that the two sides of \eqref{eq:FinalCompatib} become conjugate under $\spin^+$, $\spin^-$, $\std$, and $\cN$ by \eqref{eq:GatherEverything}.
\end{proof}

\section{Compatibility at unramified places}\label{sect:CompatUnrPlaces}

We continue in the setup of \S\ref{sect:Construction} with the same running assumptions. We determined the image of Frobenius under $\rho_{\pi}^C$ at each prime away from some finite set $S$. Now we compute the image at the finite places $\pp \nmid \ell$ above $S \backslash \Sbad$. The argument follows that of \cite[\S10]{GSp} but there are new technical difficulties due to half-spin representations and the automorphism $\theta$.

\begin{proposition}\label{prop:AllUnramifiedPlaces}
Let $\pp$ be a prime of $E$ not lying above $S_{\tu{bad}} \cup \{\ell\}$. Then $\rho^C_{\pi}$ is unramified at $\pp$. Moreover writing $\qq:=\pp\cap F$,
$$
\rho^C_{\pi}(\Frob_\pp)_{\ssimple} \sim \iota q_\pp^{-n(n-1)/4} \phi_{\pi_\qq} (\Frob_\pp) \in \GSpin_{2n}(\lql).
$$
\end{proposition}
\begin{proof}
Fix $\pp$ as in the statement. Let $p$ denote the prime of $\Q$ below $\pp$. Let $\pi^\natural$ be a transfer of $\pi$ from $G^*(\A_F)$  to $G(\A_F)$ as in the paragraph above \eqref{eq:LGC-Shimura}. Let $B(\pi^\natural)$ be the set of cuspidal automorphic representations $\tau$ of $G(\A_F)$ such that
\begin{itemize}
\item $\tau_{\qst}$ and $\pi^\natural_{\qst}$ are isomorphic up to a twist by an unramified character,
\item $\tau^{\infty, \qst,p}$ and $\pi^{\natural, \infty, \qst,p}$ are isomorphic,
\item $\tau_p$ is unramified,
\item $\tau_\infty$ is $\xi$-cohomological
\end{itemize}
We define an equivalence relation $\approx$ on the set $B(\pi^\natural)$ by declaring that $\tau_1\approx \tau_2$ if and only if $\tau_2 \in A(\tau_1)$. (Recall the definition of $A(\tau_1)$ from above \eqref{eq:Multiplicity}; notice that $\tau_1 \approx \tau_2$ if and only if $\tau_{1, \qq} \simeq \tau_{2,\qq}$.) 
To simplify notation, we will write $B$ for a set of representatives for $B(\pi^\natural)/{\approx}$.

For $\eps\in \{+,-\}$, define (true) representations of $\Gamma_{E}$ by $\rho^{\Sh, \eps}_B:=\sum_{\tau \in B} \rho^{\Sh,\eps}_\tau$ (see Theorem~\ref{thm:PointCounting}). Put $b(\pi^\natural) := \sum_{\tau\in B} a(\tau)\in \Z_{>0}$.
Since $\rho_{\tau}^{\Sh,\eps,\vee}$ satisfies \eqref{eq:LGC-Shimura} for each $\tau\in B$,
we deduce the following on the dual of $\rho^{\Sh, \eps}_B$ by comparing the images of Frobenius conjugacy classes at all but finitely many places 
via \eqref{eq:LGC-Shimura} and \eqref{eq:separation0}:
\begin{equation}\label{eq:rhoB}
  \rho_B^{\Sh, \eps,\vee} \simeq i_{b(\pi^\natural)}\circ \spin^\eps\circ\rho^C_{\pi}.
\end{equation}

We adapt the argument of Theorem~\ref{thm:PointCounting}. Consider the function $f$ on $G(\A_F)$ of the form $f = f_\infty  f_{\qst}  \one_{K_p} f^{\infty, \qst, p}$, where $f_\infty$ and $f_{\qst}$ are as in that argument, and $f^{\infty,\qst, p}$ is such that, for all automorphic representations $\tau$ of $G(\A_F)$ with $(\tau^{\infty})^K \neq 0$ and $\Tr \tau_\infty(f_\infty) \neq 0$, we have:
\begin{equation}\label{eq:AuxFunction2}
\Tr \tau^{\infty, \qst, p
}(f^{\infty, \qst, p
}) = \begin{cases}
1 & \tu{if $\tau^{\infty,\qst, p} \simeq \pi^{\natural,\infty,\qst,p}$}, \cr
0 & \tu{otherwise.}
\end{cases}
\end{equation}
Arguing as in Theorem~\ref{thm:PointCounting} we obtain 
\begin{equation}\label{eq:Bpi-natural-sum}
\iota^{-1}\Tr (\Frob_\pp^j, \rho_B^{\Sh, \eps}) = \sum_{\tau\in B} a(\tau) \Tr\tau_p(f_p^{(j)}) = \sum_{\tau\in B}  a(\tau)  q_\pp^{jn(n-1)/4} \Tr(\spin^{\eps,\vee}(\phi_{\tau_\qq}))(\Frob_\pp^j).
\end{equation}
Define $\rho_B^{\Sh} := \rho_B^{\Sh, +} \oplus \rho_B^{\Sh,-}$. Applying \eqref{eq:Bpi-natural-sum} for both $\eps = \pm$ and taking the sum, we obtain
\begin{equation}\label{eq:Bpi-natural-sum-outer}
\iota^{-1} \Tr (\Frob_\pp^j, \rho_B^{\Sh}) = \sum_{\tau\in B}  a(\tau) q_\pp^{jn(n-1)/4} \Tr(\spin^{\vee}(\phi_{\tau_\qq}))(\Frob_\pp^j).
\end{equation}
Since Xu \cite[Thm.~1.8]{BinXuLPackets} describes global $L$-packets for (not only $\GSp_{2n}$ but) quasi-split forms of $\GSO_{2n}$, the argument for \cite[Lem.~10.2]{GSp} goes through unchanged, except Corollary~\ref{cor:Pi-is-in-Disc-series-Packet-2} replaces  \cite[Cor.~8.4]{GSp}.
This argument shows that $\pi^\natural$ and $\tau\otimes \omega$ belong to the same global packet in Xu's paper for an automorphic quadratic character $\omega \colon \GSO_{2n}^{E/F}(\A_F) \to \C^\times$. Since each member of the packet in \cite{BinXuLPackets} is a $\theta$-orbit of representations, this tells us that either $\pi^\natural_x \simeq \tau_x \otimes \omega_x$ or  $\theta(\pi_x^\natural) \simeq \tau_x \otimes \omega_x$ at almost all places $x$ (where both $\pi^\natural_x$, $\tau_x$, and $\omega_x$ are unramified). Since $\pi^\natural_x\simeq\tau_x$ at almost all $x$,
\begin{equation}\label{eq:TwoOptionsForLocalFactor}
\pi^\natural_x \simeq \pi^\natural_x \otimes \omega_x \textup{ or } \theta(\pi_x^\natural) \simeq \pi^\natural_x \otimes \omega_x.
\end{equation}
Let us define characters $\chi^\eps:\Gamma_E\ra \{\pm1\}$ from $\omega$ via $\spin^\eps$ as follows. Via class field theory and Galois cohomology (applying \cite[Lem.~A.1]{LapidMao} to $\GSO_{2n}^{E/F}$; see also \cite{WaldspurgerCharacter}) we assign to $\omega$ the continuous character
$$
 W_F \to Z(\GSpin_{2n}) \rtimes \Gamma_{E/F},
$$
whose restriction to $W_E$ factors through a character $c \colon \Gamma_E \to  Z(\GSpin_{2n})$. We then define $\chi^\eps := \spin^\eps(c)$. We deduce
\begin{equation}\label{eq:chi+chi-spin}
\chi^+\spin^+(\rho_\pi^C) \oplus \chi^- \spin^-(\rho_\pi^C) \simeq \spin(\rho_\pi^C), 
\end{equation}
by using \eqref{eq:TwoOptionsForLocalFactor} to verify that the semisimplfication of the image of Frobenius matches at almost all places.

By Lemma~\ref{lem:for-mult-one}(\textit{iii}) we have $\chi^+ = \chi^-$. Set $\chi:=\chi^+$. The same lemma tells us that $\chi = 1$ or that $\rho_\pi^C$ has image in the group $H_{2n-1}$ from \eqref{eq:H-subgroup} and $\chi = \kappa_0 \circ \rho_\pi^C$.

\vspace{.05in}

\noindent\textbf{First case.} Suppose that $\chi = 1$ for every $\tau \in B$. Then, for each $\tau \in B$ there is some $l \in \Z/2\Z$ we have
$$
\spin^{\eps, \vee}(\phi_{\tau_\qq}) \simeq \spin^{\eps, \vee}(\theta^l(\omega_{\qq} \phi_{\pi_\qq})) \simeq \chi^{\eps (-1)^l}  \spin^{\eps(-1)^l,\vee}(\phi_{\pi_\qq}) = \spin^{\eps(-1)^l,\vee}(\phi_{\pi_\qq}).
$$
As $l$ does not depend on $\eps$, we obtain from \eqref{eq:Bpi-natural-sum-outer} that
$$
\iota\inv \Tr \rho_B^{\Sh}(\Frob_\pp^j) = b(\pi^\natural) q_\pp^{jn(n-1)/4} \Tr(\spin^\vee(\phi_{\pi_\qq}))(\Frob_\pp^j),\qquad j\gg1.
$$
Thus $\rho^{\Sh}_B(\Frob_\pp)_{\tu{ss}}\sim \iota q_\pp^{n(n-1)/4} i_{b(\pi^\natural)}\circ \spin^{\vee}(\phi_{\pi_\qq})(\Frob_\pp)$. Comparing the dual of this with \eqref{eq:rhoB}, we deduce that
$$ 
\spin \rho^C_{\pi}(\Frob_\pp)_{\tu{ss}} \sim \iota q_\pp^{-n(n-1)/4}\spin(\phi_{\pi_\qq}) (\Frob_\pp).
$$
Since we also know the conjugacy relation with $\std$ and $\cN$ in place of $\spin$  from \eqref{eq:separation} and Proposition~\ref{prop:FinalProp} (and the argument at~\eqref{eq:SpinorNormCheckA} in its proof), we use Lemma \ref{lem:GPin-conjugacy} to conclude that
\begin{equation}\label{eq:Bpi-only-outer}
\rho_\pi^C(\Frob_\pp)_{\tu{ss}} \sim \iota q_\pp^{-n(n-1)/4}  \theta^k \phi_{\pi_\qq}(\Frob_\pp)\in \GSpin_{2n}(\lql), \quad \tu{ for some } k \in \Z/2\Z.
\end{equation}
If $k = 0$, then \eqref{eq:Bpi-only-outer} implies the theorem. So we assume $k = 1$ in the rest of the argument.

We now distinguish between those $\tau \in B$ according to whether or not their Satake parameter at $\qq$ becomes conjugate to that of $\pi$ under $\spin^+$ and $\spin^-$: Write
\begin{align}\label{eq:GoodAndBadSet}
\Bgood := \{\tau \in B\,|\, \spin^{\eps}(\phi_{\tau_\qq}) \simeq \spin^{\eps}(\phi_{\pi_{\qq}}) ,~\eps\in \{+,-\} \}  
\end{align}
and $\Bbad := B - \Bgood$.
Thus \eqref{eq:Bpi-natural-sum} implies 
\begin{align}\label{eq:Bpi-only-outer2}
q_\pp^{-jn(n-1)/4}  \iota^{-1} b(\pi^\natural) \Tr \spin^\eps(\rho_\pi^C)(\Frob_\pp^j) & = \sum_{\tau \in \Bgood}  a(\tau)  \Tr(\spin^{\eps}(\phi_{\pi_\qq}))(\Frob_\pp^j) \cr 
& + \sum_{\tau \in \Bbad}  a(\tau)  \Tr(\spin^{-\eps}(\phi_{\pi_\qq}))(\Frob_\pp^j).
\end{align} 
Equation \eqref{eq:Bpi-only-outer} and \eqref{eq:Bpi-only-outer2} imply that $\spin^{ - \eps}(\phi_{\pi_\qq})^{b(\pi^\natural)} \simeq   \spin^{\eps}(\phi_{\pi_\qq})^{b_0} \oplus   \spin^{-\eps}(\phi_{\pi_\qq})^{b_1}$, as $W_{E_\pp}$-representations, where $b_0 = \sum_{\tau \in \Bgood}  a(\tau)$, $b_1 = \sum_{\tau \in \Bbad}  a(\tau)$ and $b(\pi^\natural) = b_0 + b_1$. Thus
\begin{equation}\label{eq:Bpi-only-outer3}
\spin^{-\eps}(\phi_{\pi_\qq})^{b(\pi^\natural) - b_1} \simeq  \spin^{\eps}(\phi_{\pi_\qq})^{b_0}.
\end{equation}
As $\pi^\natural$ contributes to $\Bgood$, we have $b_0 = b(\pi^\natural) - b_1 > 0$. Thus $\spin^{-\eps}(\phi_{\pi_\qq}) \simeq  \spin^{\eps}(\phi_{\pi_\qq})$, and $\phi_{\pi_\qq} \sim \theta \phi_{\pi_\qq}$, in which case the proposition follows from \eqref{eq:Bpi-only-outer}. 
Here we applied Lemma 1.1 of \cite{GSp} and the fact that $\spin^\pm$, $\std$, $\cN$ are fundamental representations (see table above Lemma 1.1 in [\textit{loc. cit}]).

\vspace{.05in}

\noindent\textbf{Second case.} Suppose that $\chi \neq 1$ for some $\tau \in B$. As explained, then $\rho_{\pi}^C$ has image in the group $H_{2n-1}$ from \eqref{eq:H-subgroup}. We obtain from \eqref{eq:Bpi-natural-sum-outer} and \eqref{eq:rhoB} that 
\begin{align}\label{eq:ClaimTrueIn4cases}
\iota\inv b(\pi^\natural) \Tr & \lhk \spin(\rho_\pi^C) \oplus \chi \spin(\rho_\pi^C) \rhk(\Frob_\pp^j) = \cr
&= \sum_{\tau\in B}  a(\tau) (1+ \chi(\Frob_\pp)^j)q_\pp^{jn(n-1)/4}\Tr(\spin(\phi_{\tau_\qq}(\Frob_\pp)^j)). 
\end{align}
For each $\tau \in B$ there exist $a, b \in \Z/2\Z$ such that for both $\eps \in \{\pm 1\}$ we have
$
\spin^\eps(\phi_{\tau_\pp}) \simeq \chi^{b} \spin^{(-1)^{a}\eps}(\phi_{\pi_\pp}).
$
Thus we have
\begin{equation*}
(1 \oplus \chi) \otimes \spin(\phi_{\tau_\pp}) \simeq 
(\chi^{b} \oplus \chi^{b + 1}) \otimes \left[\spin^{(-1)^{a}} (\phi_{\pi_\pp}) \oplus \spin^{-(-1)^{a}}(\phi_{\pi_\pp}) \right] \simeq
(1 \oplus \chi) \otimes \spin(\phi_{\pi_\pp}). 
\end{equation*}
as $W_{E_{\pp}}$-representations. In particular, on the right hand side of \eqref{eq:ClaimTrueIn4cases} we may replace $\phi_{\tau_\pp}$ by $\phi_{\pi_\pp}$. 
We have $\chi \spin^+(\rho_\pi^C) \simeq \spin^-(\rho_\pi^C)$ by Lemma \ref{lem:for-mult-one}(\textit{ii}) and so $ \chi \spin(\rho_\pi^C) \simeq \spin(\rho_\pi^C)$. By removing the multiplicity $b(\pi^\natural)$ on both sides of \eqref{eq:ClaimTrueIn4cases} we thus find that
\begin{equation}\label{eq:StillNeedToGetRidOfChi}
\spin(\rho_\pi^C)^{\oplus 2}|_{W_{E_\pp}} \simeq [1 \oplus \chi] \otimes \spin(\phi_{\pi_\qq}) \otimes |\cdot|^{-jn(n-1)/4}_{\qq}.
\end{equation}
We claim that in fact also 
$\chi \spin^+(\phi_{\pi_\qq}) \simeq \spin^-(\phi_{\pi_\qq})$. If true, \eqref{eq:StillNeedToGetRidOfChi} would imply that
\begin{equation}\label{eq:GotRidOfChi}
\spin(\rho_\pi^C)|_{W_{E_\pp}} \simeq \spin(\phi_{\pi_\qq}) \otimes |\cdot|^{-jn(n-1)/4}_{\qq}.
\end{equation}
We check the claim. As $\pr^\circ$ surjects $Z(\GSpin_{2n})$ onto $Z(\SO_{2n})$, we see that $\pr^\circ$ induces an isomorphism from the component group of $H_{2n-1}$ to the component group of $H_{2n-1}^\circ = \SO_{2n-1} Z(\SO_{2n})$. Consequently, $\chi$ (which equals $\kappa_0 \circ \rho_{\pi}^C$ by Lemma~\ref{lem:for-mult-one}(\textit{iii}))
is  equal to the composition
$$
\Gamma_E \overset {\rho_{\pi^\flat}^C} \to H_{2n-1}^\circ \surjects \{\pm 1\}.
$$
We know that $\pr^\circ(\phi_{\pi_{\qq}}(\Frob_\pp)) \osim \pr^\circ(\rho_\pi^C(\Frob_\pp)_\tu{ss}) \in \SO_{2n}$ since they become conjugate after applying $\std$. Therefore $\chi|_{W_{E_\pp}}$ equals 
$$
W_{E_\pp} \overset {\phi_{\pi_{\qq}^\flat}} \to H_{2n-1}^\circ \surjects \{\pm1\} 
$$
and hence equals $W_{E_\pp} \overset {\phi_{\pi_{\qq}}} \to H_{2n-1} \surjects \{\pm1\}$ (the argument is similar to the one below \eqref{eq:Case2-thetaStability-2}), which in turn implies that $\chi \spin^+(\phi_{\pi_\qq}) \simeq \spin^-(\phi_{\pi_\qq})$. Hence the claim is proved, and \eqref{eq:GotRidOfChi} holds true as observed above. 

We thus find \eqref{eq:Bpi-only-outer} again. If $k = 0$ in that equation, we are done. Now assume $k = 1$. Define $\Bgood, \Bbad$ as in \eqref{eq:GoodAndBadSet}. As $\chi \spin^+(\phi_{\pi_\qq}) \simeq \spin^-(\phi_{\pi_\qq})$,  it follows from \eqref{eq:chi+chi-spin} that for each $\tau \in \Bbad$, we have $\spin^{\eps}(\phi_{\tau_\qq}) \simeq \spin^{-\eps}(\phi_{\pi_{\qq}})$ for both signs $\eps \in\{\pm1\}$. Thus we obtain \eqref{eq:Bpi-only-outer2} with $\# \Bgood > 0$ again. By the same argument as in \eqref{eq:Bpi-only-outer3} we deduce that $\phi_{\pi_\qq} \sim \theta \phi_{\pi_\qq}$, in which case the proposition follows from \eqref{eq:Bpi-only-outer}.
\end{proof}

\section{The main theorem}\label{s:main-thm}

In this section we prove Theorem~\ref{thm:A} (Theorem~\ref{thm:ThmAisTrue}), the main result of this paper. Before doing this, we switch the normalization for $\pi$ from (coh) to (L-coh), and extend the Galois action from $\Gamma_E$ to $\Gamma_F$.

As in Theorem~\ref{thm:A}, let $\pi$ be a cuspidal automorphic representation of $G^*(\A_F)$ satisfying (St) and (L-coh). Fix a cuspidal automorphic representation $\pi^\flat$ of $\SO^{E/F}_{2n}(\A_F)$ which embeds in $\pi|_{\SO^{E/F}_{2n}(\A_F)}$ as it is possible by \cite{LabesseSchwermer2018}. Assume either (std-reg) for $\pi$ or Hypothesis \ref{hypo:non-std-reg} for an $\SO_{2n}(\A_F)$-subrepresentation $\pi^\flat$ of $\pi$. Define $\tilde \pi:=\pi |\simil|^{-n(n-1)/4}$. Then $\tilde \pi$ is $\xi$-cohomological and will play the role of $\pi$ in Sections \ref{sect:Construction} and \ref{sect:CompatUnrPlaces}. Naturally $\pi^\flat$ is a subrepresentation of $\tilde\pi|_{\SO_{2n}(\A_F)}$ since $|\simil|$ is trivial when restricted to $\SO_{2n}(\A_F)$.

Let $S^F$ (resp.~$S^E$) be the finite set of places of $F$ (resp.~$E$) above $S:=S_{\tu{bad}}\cup \{\ell\}$. Fix an infinite place $y_\infty\in \cV_\infty$ and also fix a finite place $\qq$ as in (St). (Recall that the group $G$, Shimura varieties, and the resulting $\GSpin_{2n}$-valued Galois representations in earlier sections depend on the choice of $y_\infty$ and $\qq$.) From Propositions \ref{prop:UpToOuter} and \ref{prop:AllUnramifiedPlaces}, we obtain
$$
\rho^C_{\tilde \pi}\colon
\Gamma_{E,S}\ra \GSpin_{2n}(\lql)
$$
such that for every $\pp\notin S^E$, writing $\qq := \pp|_F$, we have
\begin{equation}\label{eq:compatibility-over-E}
\rho^C_{\tilde \pi}(\Frob_\pp)_{\tu{ss}}\sim \iota q_\pp^{-n(n-1)/4} \phi_{\tilde \pi_\qq}(\Frob_\pp) = \iota\phi_{\pi_\qq}(\Frob_\pp).
\end{equation}

Let us explain the definition of $\rho_\pi$ on $\Gamma_{F,S}$. If $n$ is even (thus $E=F$) then we simply take $\rho_\pi:=\rho^C_{\tilde \pi}$. In case $n$ is odd (so $[E:F]=2$), write $c_{y_\infty} \in \Gamma$ for the complex conjugation corresponding to $y_\infty$ (canonical up to conjugacy). In order to apply Lemma \ref{lem:extend}, we check
\begin{lemma}\label{lem:condition-to-extend}
When $n$ is odd, we have ${}^{c_{y_\infty}} \rho^C_{\tilde \pi}\simeq \theta\circ \rho^C_{\tilde \pi}$.
\end{lemma}

\begin{proof}
In light of Proposition \ref{prop:local-conjugacy-global-conjugacy}, it is enough to check this locally, namely that
$$
\rho^C_{\tilde \pi}(c_{y_\infty} \Frob_\pp c_{y_\infty}^{-1})_{\tu{ss}} \sim \theta\circ \rho^C_{\tilde \pi}(\Frob_\pp)_{\tu{ss}} \qquad \mbox{in}~~~\GSpin_{2n}(\lql)
$$
for almost all primes $\pp$ of $E$. For each $\pp$, write $\qq:=\pp\cap F$.
Firstly if $\qq$ splits in $E$ as $\pp c(\pp)$ then we use \eqref{eq:compatibility-over-E} to deduce that
$$
\rho^C_{\tilde \pi}(c_{y_\infty} \Frob_\pp c_{y_\infty}^{-1})_{\tu{ss}} \sim \rho^C_{\tilde \pi}(\Frob_{c(\pp)})_{\tu{ss}} \sim \iota\phi_{\pi_\qq}(\Frob_{c(\pp)})
\sim \iota \theta(\phi_{\pi_\qq}(\Frob_{\pp})) \sim \theta(\rho^C_{\tilde \pi}(\Frob_\pp)).
$$
(To see the third conjugacy relation, we argue as follows. From \eqref{eq:QuasiSplitGSO} we see that an element of $\GSO^{E/F}_{2n,F_\qq}$ has the form $(g,\theta(g))$ with $g\in \GSO_{2n,E_{\pp}}$ and that $\GSO^{E/F}_{2n,F_\qq}$ is isomorphic to $\GSO_{2n,E_{\pp}}$ and $\GSO_{2n,E_{c(\pp)}}$ by the projection map onto the first and second components, respectively. Likewise the dual group of $\GSO^{E/F}_{2n,F_\qq}$ is naturally the subgroup of $\GSpin_{2n}\times \GSpin_{2n}$ consisting of elements of the form $(g,\theta(g))$, the two components corresponding to $\pp$ and $c(\pp)$. It follows that $\phi_{\pi_\qq}(\Frob_{c(\pp)}) \sim \theta(\phi_{\pi_\qq}(\Frob_{\pp}))$.)

Secondly if $\qq$ is inert in $E$ then $c_{y_\infty} \Frob_\pp c_{y_\infty}^{-1}\sim \Frob_\pp$. Thus we need to check that the conjugacy class of $\rho^C_{\tilde \pi}(\Frob_\pp)_{\tu{ss}}$ is $\theta$-invariant. Writing $\theta(\phi_{\pi_\qq}(\Frob_{\qq}))=s\rtimes c\in \GSpin_{2n}(\lql)\rtimes \Gamma_{E/F}$, 
$$
\theta(\phi_{\pi_\qq}(\Frob_{\pp}))\sim\theta(\phi_{\pi_\qq}(\Frob_{\qq}^2))=s \theta(s)\sim \theta(s)s\qquad \mbox{in}~~~\GSpin_{2n}(\lql).
$$
This implies the desired $\theta$-invariance via \eqref{eq:compatibility-over-E}. The proof is complete.
\end{proof}

We are assuming that $n$ is odd. By Lemmas \ref{lem:condition-to-extend} and \ref{lem:extend}, we extend $\rho^C_{\tilde \pi}$ to a Galois representation to be denoted $\rho_\pi$:
\begin{equation}\label{eq:rho-pi}
\rho_\pi \colon \Gamma_{F,S}\ra \GSpin_{2n}(\lql)\rtimes \Gamma_{E/F}.
\end{equation}
There are two choices up to conjugacy (Example \ref{ex:GSpin2n-extend}). We choose one arbitrarily and possibly modify the choice below.

We return to treating both parities of $n$. We fixed $\pi^\flat$ above. Theorem \ref{thm:ExistGaloisSO}, or Hypothesis \ref{hypo:non-std-reg} if (std-reg) is not assumed, supplies us with
$$
\rho_{\pi^\flat}:\Gamma_{F,S}\ra \SO_{2n}(\lql)\rtimes \Gamma_{E/F}
$$
such that $\rho_{\pi^\flat}(\Frob_\qq)_{\tu{ss}}\osim  \iota\phi_{\pi^\flat_\qq}(\Frob_\qq)$ for $\qq,\pp$ as above. Thanks to \eqref{eq:compatibility-over-E} and the unramified Langlands functoriality with respect to $\SO_{2n}\ra \GSO^{E/F}_{2n}$ (whose dual morphism is $\tu{pr}^\circ$),
$$
\rho_{\pi^\flat}(\Frob_\pp)_{\tu{ss}} \osim \iota\phi_{\pi^\flat_\qq}(\Frob_\pp) \sim \iota\,\tu{pr}^\circ(\phi_{\pi_\qq}(\Frob_\pp))
\sim \tu{pr}^\circ(\rho_{\pi}(\Frob_\pp)_{\tu{ss}}).
$$
Thus the conjugacy classes at the left and right ends are $\OO_{2n}(\lql)$-conjugate, under the identification $\SO_{2n}(\lql)\rtimes \Gamma_{E/F}= \OO_{2n}(\lql)$. Since $\OO_{2n}$ is acceptable, $\rho_{\pi^\flat}|_{\Gamma_{E,S}}$ and $\tu{pr}^\circ\circ\rho_{\pi}|_{\Gamma_{E,S}}$ are $\OO_{2n}(\lql)$-conjugate. Replacing $\rho_{\pi^\flat}$ by an $\OO_{2n}(\lql)$-conjugate, we may and will assume that
$$
\rho_{\pi^\flat}|_{\Gamma_{E,S}}=\pr^\circ\circ \rho_{\pi}|_{\Gamma_{E,S}}
$$
without disturbing the validity of (SO-i) through (SO-v) in Theorem \ref{thm:ExistGaloisSO}. When $n$ is odd, we take an extra step as follows. Observe that $\rho_{\pi^\flat}$ and $\pr^\circ\circ \rho_\pi$ are two $\SO_{2n}(\lql)\rtimes \Gamma_{E/F}$-valued representations of $\Gamma_{F,S}$ extending \eqref{eq:equality-over-E}. If they are not equal then $\pr^\circ\circ \rho_\pi=\rho_{\pi^\flat}\otimes \chi_{E/F}$ by Example \ref{ex:SO2n-extend} with $\chi_{E/F}:\Gamma_F\twoheadrightarrow \Gamma_{E/F}\isom \{\pm1\}$. Then we go back to \eqref{eq:rho-pi} and replace $\rho_\pi$ with $\rho_\pi\otimes \chi$, where $\chi$ is as in Example \ref{ex:GSpin2n-extend}; this does not affect the discussion between \eqref{eq:rho-pi} and here. Since $\pr^\circ\circ\chi=\chi_{E/F}$, this ensures that
\begin{equation}\label{eq:equality-over-E}
\rho_{\pi^\flat}=\pr^\circ\circ \rho_{\pi}.
\end{equation}
As in \S\ref{sect:RootDatum}, let $(s_0,s_1,\ldots,s_n)\in (\lql^\times)^{n+1}$ denote an element of $T_{\GSpin}(\lql)\subset \GSpin_{2n}(\lql)$. This element maps to $\diag(s_1,\ldots,s_n,s_1^{-1},\ldots,s_n^{-1})\in \SO_{2n}(\lql)$ under $\pr^\circ$, and maps to $s_0^2s_1s_2\cdots s_n$ under the spinor norm $\cN$.

\begin{lemma}\label{lem:image-of-complex-conjugation}
At every infinite place $y$ of $F$, the following are $\GSpin_{2n}(\lql)$-conjugate:
\begin{equation}\label{eq:image-of-complex-conjugation}
\rho_\pi(c_y)\sim
\begin{cases}
(a,\underbrace{1,\ldots,1}_{n/2},\underbrace{-1,\ldots,-1}_{n/2}), ~~a\in \{\pm 1\},& n:\tu{even},\\
(1,\underbrace{1,\ldots,1}_{(n-1)/2},\underbrace{-1,\ldots,-1}_{(n-1)/2},1)\rtimes c, & n:\tu{odd}.
\end{cases}
\end{equation}
where the right hand side lies in $\TGspin(\lql) \rtimes \Gal(E/F)$.
\end{lemma}

\begin{proof}
In light of \eqref{eq:equality-over-E} (which is valid for both odd and even $n$ as discussed above) and Theorem \ref{thm:ExistGaloisSO} (SO-v) (or Hypothesis \ref{hypo:non-std-reg}) which describes $\rho_{\pi^\flat}(c_y)$, the following are $\GSpin_{2n}(\lql)$-conjugate:
$$\pr^\circ(\rho_\pi(c_y)) \sim \begin{cases}
  \diag(\underbrace{1,\ldots,1}_{n/2},\underbrace{-1,\ldots,-1}_{n/2},\underbrace{1,\ldots,1}_{n/2},\underbrace{-1,\ldots,-1}_{n/2}), & n:\mbox{even},\\
  \diag(\underbrace{1,\ldots,1}_{(n-1)/2},\underbrace{-1,\ldots,-1}_{(n-1)/2},1,\underbrace{1,\ldots,1}_{(n-1)/2},\underbrace{-1,\ldots,-1}_{(n-1)/2},1) \rtimes \theta,
  & n:\mbox{odd}.
\end{cases}
 $$
 Therefore $\rho_\pi(c_y)$ is a lift of the right hand side (up to $\GSpin_{2n}(\lql)$-conjugacy) via $\pr^\circ$. Moreover $\rho_\pi(c_y)^2=\rho_\pi(c_y^2)=1$. We claim that these two conditions imply \eqref{eq:image-of-complex-conjugation}.

This is straightforward when $n$ is even. Now suppose that $n$ is odd. Evidently the right hand side of \eqref{eq:image-of-complex-conjugation} satisfies the two conditions. Any other lift of order 2 can only differ (possibly after conjugation) from the right hand side of \eqref{eq:image-of-complex-conjugation} by scalars $\{\pm1\}$. (Use Lemma \ref{lem:SurjectionOntoGSO} (ii) and the order two condition.) This implies \eqref{eq:image-of-complex-conjugation} since every $g\in \GSpin_{2n}(\lql)\rtimes c$ is conjugate to $-g$; indeed, $-g= \zeta g \zeta^{-1}$ if $\zeta\in Z_{\Spin}(\lql)$ is an element of order 4, noting that $\theta(\zeta)=\zeta^{-1}$.
\end{proof}

Let $\omega_\pi: F^\times \backslash \A_F^\times \ra \C^\times$ denote the central character of $\pi$. By abuse of notation, we still write $\omega_\pi$ (depending on the choice of $\iota$) for the $\ell$-adic character of $\Gamma_F$ corresponding to $\omega_\pi$ via class field theory (as in \cite[pp.20--21]{HarrisTaylor}). To make $\omega_\pi$ explicit, recall that $\tilde \pi=\pi|\simil|^{-n(n-1)/4}$ is $\xi$-cohomological. By condition (cent), the central character of $\xi$ is $z\mapsto z^w$ on $F_y^\times$ at every real place $y$ of $F$, for an integer $w$ independent of $y$. Therefore (recalling $\simil$ is the squaring map on the center)
$$
\omega_{\pi,y}(z)=z^{-w} |z|^{n(n-1)/2}=\tu{sgn}(z)^w |z|^{-w+n(n-1)/2},\qquad z\in F_y^\times.
$$
Then $\omega_\pi |\cdot|^{w-n(n-1)/4}$ is a finite-order Hecke character which is $\tu{sgn}^w$ at every real place. Hence $\omega_\pi=\chi_{\tu{cyc}}^{-w+n(n-1)/2} \chi_0$, where $\chi_{\tu{cyc}}$ is the $\ell$-adic cyclotomic character, and $\chi_0$ a finite-order character with $\chi_0(c_y)=(-1)^w$ at each real place $y$. The upshot is that 
\begin{equation}\label{eq:omega(c)}
\omega_{\pi}(c_y)=(-1)^{-w+n(n-1)/2} (-1)^w = (-1)^{n(n-1)/2},\qquad y:\tu{real~place~of}~F.
\end{equation}
We are ready to upgrade \eqref{eq:compatibility-over-E} to a compatibility at places of $F$ for odd $n$ (thus $[E:F]=2$).

\begin{corollary}\label{cor:Norm-and-compatibility-over-F}
We have $\cN\circ \rho_\pi = \omega_\pi$. Moreover, at every finite place $\qq$ of $F$ not above $S_{\tu{bad}}\cup \{\ell\}$, 
$$\rho_\pi(\Frob_\qq)_{\tu{ss}}\sim \iota\phi_{\pi_\qq}(\Frob_\qq).$$
\end{corollary}

\begin{remark}
The corollary is certainly not automatic from \eqref{eq:compatibility-over-E} since the unramified base change from $G^*(F_\qq)$ to $G^*(E_\pp)$ is not injective when $\qq$ does not split in $E$. Curiously our proof crucially relies on the image of complex conjugation. We have not found a local or global proof only using properties at finite places.
\end{remark}

\begin{proof}
Via the unramified Langlands functoriality with respect to the central embedding $\G_m\hra \GSO_{2n}^{E/F}$, \eqref{eq:compatibility-over-E} implies that $\cN\circ \rho_\pi|_{\Gamma_E}=\omega_\pi|_{\Gamma_E}$. If $n$ is even then $E=F$ so there is no more to prove as the latter assertion is already true by \eqref{eq:compatibility-over-E}.

Henceforth assume that $n$ is odd (so $[E:F]=2$). Then either $\cN\circ \rho_\pi = \omega_\pi$ or $\cN\circ \rho_\pi = \omega_\pi\otimes \chi_{E/F}$, where $\chi_{E/F}\colon\Gamma_F \twoheadrightarrow \Gamma_{E/F}\isom \{\pm1\}$. To exclude the latter case, let $y$ be a real place of $F$. We have $\cN(\rho_\pi(c_y))=(-1)^{(n-1)/2}$ from Lemma \ref{lem:image-of-complex-conjugation}, and $\omega_\pi(c_y)=(-1)^{n(n-1)/2}$ from \eqref{eq:omega(c)}, but clearly $\chi_{E/F}(c_y)=-1$. Then the only possibility is that $\cN\circ \rho_\pi = \omega_\pi$.

We prove the second assertion. If $\qq$ splits in $E$, this follows immediately from \eqref{eq:compatibility-over-E} for $\rho_\pi|_{\Gamma_{E,S}}$. Henceforth assume that $\qq$ is inert in $E$. We have seen that $\pr^\circ\circ\rho_\pi|_{\Gamma_{E,S}}=\rho_{\pi^\flat}|_{\Gamma_{E,S}}$. Theorem \ref{thm:ExistGaloisSO} (SO-i) (or Hypothesis \ref{hypo:non-std-reg}) tells us that
$$
\rho_{\pi^\flat_\qq}(\Frob_\qq)_{\tu{ss}} \sim \iota\phi_{\pi^\flat_\qq}(\Frob_\qq) = \iota \pr^\circ(\phi_{\pi_\qq}(\Frob_\qq)).
$$
(Note that the outer automorphism ambiguity disappears as it is absorbed by the $\SO_{2n}$-conjugacy on the nontrivial coset of $\SO_{2n}\rtimes \Gamma_{E/F}$; since $\qq$ is inert in $E$, the image of $\Frob_\qq$ in $\Gamma_{E/F}$ is nontrivial.) Therefore $\rho_\pi(\Frob_\qq)_{\tu{ss}}\sim z \iota\phi_{\pi_\qq}(\Frob_\qq)$ for some $z\in \lql^\times$. Taking the spinor norm,
$$
\cN(z)=(\cN\circ\rho_\pi(\Frob_\qq)_{\tu{ss}})\cN(\iota\phi_{ \pi_\qq}(\Frob_\qq))^{-1}
= \omega_\pi(\Frob_\qq) \omega_\pi(\Frob_\qq)^{-1}=1.
$$
It follows that $z\in \{\pm 1\}$. Since every $g\in \GSpin_{2n}(\lql)\rtimes c$ is conjugate to $-g$ (proof of Lemma \ref{lem:image-of-complex-conjugation}), we conclude that
$\rho_\pi(\Frob_\qq)_{\tu{ss}}$ is conjugate to $\iota\phi_{\pi_\qq}(\Frob_\qq)$.
\end{proof}

\begin{theorem}\label{thm:ThmAisTrue}
Theorem~\ref{thm:A} is true.
\end{theorem}
\begin{proof}

Let $\pi$ be as in the theorem. We fix an automorphic representation $\pi^\flat$ of $\SO^{E/F}_{2n}(\A_F)$ in $\pi|_{\SO_{2n}^{E/F}(\A_F)}$, take $\rho_{\pi^\flat}:\Gamma_F\ra \SO_{2n}(\lql)\rtimes \Gamma_{E/F}$ to be as in Theorem \ref{thm:ExistGaloisSO}, or Hypothesis \ref{hypo:non-std-reg} if (std-reg) is false, and define
\begin{equation}\label{eq:rhopi}
\rho_\pi:\Gamma_F \ra \GSpin_{2n}(\lql)\rtimes \Gamma_{E/F}
\end{equation}
such that $\rho_{\pi^\flat}=\pr^\circ\circ \rho_\pi$ as explained at the start of this section. We can inflate $\rho_\pi$ to a representation $\Gamma_F \ra \GSpin_{2n}(\lql)\rtimes \Gamma_{F}$ of Theorem~\ref{thm:A}, but we work with $\rho_\pi$ in the form of \eqref{eq:rhopi} as this is harmless for verifying the theorem.

The equality $\rho_{\pi^\flat}=\pr^\circ\circ \rho_\pi$ and Corollary \ref{cor:Norm-and-compatibility-over-F} imply (A2). Corollary \ref{cor:Norm-and-compatibility-over-F} exactly gives (A1). Item (A4) is straightforward from Lemma \ref{lem:image-of-complex-conjugation}. To see (A5), note that the image of $\rho_\pi$ in $\PSO_{2n}(\lql)$ is the same as the image of $\rho_{\pi^\flat}$ in the same group. The Zariski closure of the image is (possibly disconnected and) reductive since $\rho_{\pi^\flat}$ is semisimple and contains a regular unipotent element by Corollary \ref{cor:image-has-reg-unip}. Hence (A5) is implied by Proposition~\ref{prop:containingRegularUnipotent}. Now $\rho_{\pi}$ also contains a regular unipotent in the image, so (A6) and the uniqueness of $\rho_\pi$ up to conjugacy are consequences of Proposition \ref{prop:local-conjugacy-global-conjugacy}.

It remains to verify (A3). We begin with part (b). If $\pi_\qq$ has nonzero invariants under a hyperspecial (resp.~Iwahori) subgroup, then $\pi^\flat_\qq$ and $\omega_{\pi,\qq}$ enjoy the same property. Therefore (b) follows from (A2) and Theorem \ref{thm:ExistGaloisSO} (SO-iv). To prove part (c), write $\pp$ for a place of $E$ above $\qq$. Since $\pp$ is unramified over $E$, it suffices to check that $\rho_{\pi}|_{\Gamma_E}$ is crystalline at $\pp$. Moreover we may assume that $F\neq \Q$ by the automorphic base change of \cite[Prop.~6.6]{GSp} and (A6). (If $F=\Q$ then replace $F$ with a real quadratic field $F'$ in which $\ell$ is split, and $E$ with $EF'$. By (A6), $\rho_\pi|_{\Gamma_{F'}}\simeq \rho_{\pi_{F'}}$, where $\pi_{F'}$ is the base change of $\pi$ to $F'$ constructed in \emph{loc.~cit.} Thus the question is now about $\rho_{\pi_{F'}}$.) 
Now that $F\neq \Q$, the Shimura varieties in \S\ref{sect:Shimura} are proper, and $\rho_{\pi^\natural}^{\Sh,\varepsilon}$ is crystalline at all places above $\ell$ by \cite{Lovering}. Since $\spin\circ \rho_{\pi}|_{\Gamma_E}$ embeds in $\rho_{\pi^\natural}^{\Sh,+}\oplus \rho_{\pi^\natural}^{\Sh,-}$ (which is isomorphic to the $a(\pi^\natural)$-fold direct sum of $\spin\circ \rho_{\pi}$), and since $\spin$ is faithful, we deduce that $\rho_{\pi}|_{\Gamma_E}$ is crystalline at $\pp$ as desired.

Finally we prove (A3), part (a).  We first claim that if two cocharacters $\mu_1,\mu_2\in X_*(T_{\GSpin})$ become conjugate after composition with each of $\spin^{+,\vee}$, $\spin^{-,\vee}$, $\std^\circ$, and $\cN$, then $\mu_1$ and $\mu_2$ are $\GSpin_{2n}$-conjugate. To see this, note that a semi-simple conjugacy class $\gamma$ in $\GSpin_{2n}(\C)$ is determined by the conjugacy classes $\spin^\pm(\gamma)$, $\cN(\gamma)$ and $\std(\gamma)$ by \cite[Lem.~1.3]{GSp} (thus also determined by $\spin^{\pm,\vee}(\gamma)$, $\cN(\gamma)$ and $\std(\gamma)$) and the table above Lemma 1.1 therein. The same statement holds for the cocharacters via the Weyl group-equivariant isomorphism $X_*(T_{\GSpin}) \otimes_{\Z} \C^\times \to T_{\GSpin}(\C)$, proving the claim.

Our second claim is that for every $y:F\hra \C$,
\begin{equation}\label{eq:HT-weights}
\spin^{\varepsilon}(\mu_{\tu{HT}}(\rho_{\pi, \qq}, \iota y))\sim \spin^{\varepsilon}(\iota \mu_{\tu{Hodge}}(\xi_y) - \tfrac{n(n-1)}{4}\simil),\qquad \varepsilon\in \{\pm\}.
\end{equation}
Accept this for now.  
The representations $\std^\circ$ and $\cN$ factor over the isogeny $(\cN,\pr^\circ):\GSpin_{2n}\ra \Gm\times \SO_{2n}$,
so it follows easily from (A2) and Theorem \ref{thm:ExistGaloisSO} (SO-iii) that \eqref{eq:HT-weights} holds with $\std^\circ$ and $\cN$ in place of $\spin^{\varepsilon}$. Thus we can conclude by the first claim. 

To complete the proof of (A3)(a), we check the second claim \eqref{eq:HT-weights}. A base-change argument as in the preceding paragraph allows us to assume that $F\neq \Q$. Recall that we chose an embedding $y_\infty:F\hra \C$ in the definition of $G$ and the Shimura data $(\Res_{F/\Q} G,X^\varepsilon)$. It follows from Proposition \ref{prop:HT-weights} (applicable as $F\neq \Q$) that 
\begin{equation}\label{eq:HT-weights2}
\mu_{\HT}(\spin^{\varepsilon,\vee}\circ  \rho_{\pi} ,\iota y_\infty) ~\sim ~\spin^{\varepsilon,\vee}\circ \left( 
\iota \mu_{\Hodge}(\xi_{y_\infty}) -\tfrac{n(n-1)}{4} \simil \right),
\end{equation}
We can repeat the construction from the beginning of \S\ref{s:main-thm} up to now, with $y:F\hra \C$ in place of $y_\infty$. Write $\rho_\pi(y):\Gamma_F\ra \GSpin_{2n}(\lql)\rtimes \Gamma_{E/F}$ for the resulting Galois representation. From (A1) and (A6) of Theorem~\ref{thm:A} (which have already been verified) to $\rho_\pi$ and $\rho_\pi(y)$, we deduce that $\rho_\pi\simeq \rho_\pi(y)$. Applying Proposition \ref{prop:HT-weights} to $\rho_\pi(y)$, we see that \eqref{eq:HT-weights2} holds with $y$ in place of $y_\infty$.

Now the left hand side of \eqref{eq:HT-weights2} equals $\spin^{\varepsilon,\vee}\circ\mu_{\HT}(  \rho_{\pi} ,\iota y)$ by construction of Hodge--Tate cocharacters, so we are done with proving \eqref{eq:HT-weights} as desired.
\end{proof}

\begin{remark}\label{rem:complex-conj-GSO}
Lemma \ref{lem:image-of-complex-conjugation} tells us that $\rho_{\pi}$ is totally odd. Our result also shows that $\rho_{\pi}(c_y)$ is as predicted by \cite[Conj. 3.2.1, 3.2.2]{BuzzardGee} for every infinite place $y$ of $F$. Indeed, as explained in \S6 of their paper, their conjectures are compatible with the functoriality. Considering the $L$-morphism $^L \GSO_{2n}^{E/F} \ra {}^L \SO_{2n}^{E/F}$ dual to the inclusion $\SO_{2n}^{E/F} \hra \GSO_{2n}^{E/F}$, we reduce the question to the case of $\SO_{2n}^{E/F}$ in view of the characterization of $\rho_{\pi}(c_y)$ in terms of $\pr^\circ(\rho_{\pi_\pi}(c_y))$. The latter is conjugate to $\rho_{\pi^\flat}$, which is as conjectured by \emph{loc.~cit.}~by Remark \ref{rem:complex-conj-SO}.
\end{remark}

\begin{remark}
It was easier to determine the Hodge--Tate cocharacter in the $\GSp$-case \cite{GSp} , thanks to the absence of nontrivial outer automorphisms. In particular we did not need to prove the analogue of Proposition \ref{prop:HT-weights}. Compare with the proof of Theorem 9.1 (iii.a') of \emph{loc.~cit}.
\end{remark}

\section{Refinement for $\SO_{2n}$-valued Galois representations}\label{sect:refinement}

As an application of our results we improve upon Theorem \ref{thm:ExistGaloisSO} in this section by removing the outer ambiguity in the images of Frobenius conjugacy classes.

\bigskip

Let $E/F$ be a quadratic CM extension of $F$ in case $n$ is odd, and $E := F$ for $n$ even. Let $\SO_{2n}^{E/F}$ be the corresponding group defined above \eqref{eq:std-rep-L-gp}. If $\pi^\flat$ (resp.~$\pi$) is an automorphic representation of $\SO_{2n}^{E/F}(\A_F)$ (resp.~$\GSO_{2n}^{E/F}(\A_F)$), we write $\Sbad(\pi^\flat)$ (resp $\Sbad(\pi)$) for the set of rational prime numbers $p$, such that $p = 2$, $p$ ramifies in $E$, or $\pi^\flat_p$ (resp.~$\pi_p$) is a  ramified representation. For other notation, we refer to Section \ref{sect:Notation}.

In order to extend a given cohomological representation $\pi^\flat$ of $\SO_{2n}^{E/F}(\A_F)$ to a \emph{cohomological} representation $\pi$ of $\GSO_{2n}^{E/F}(\A_F)$, the following condition on the central character $\omega_{\pi^\flat}\colon \mu_2(F) \backslash \mu_2(\A_F) \to \C^\times$ is necessary in view of condition (cent) of \S\ref{sect:Shimura}. (If $\pi$ is $\xi$-cohomological with $w\in \Z$ as in (cent) then all $\omega_{\pi^\flat,y}$ are trivial, resp.~nontrivial, according as $w$ is even, resp.~odd.)

\begin{itemize}
\item[(\textbf{cent}$^\circ$)] The sign character $\omega_{\pi^\flat,y}: \mu_2(F_y)=\{\pm1\} \to \C^\times$ does not depend on $y|\infty$.
\end{itemize}

\begin{theorem}\label{thm:ExistGaloisSO-Plus}
Let $\pi^\flat$ be a cuspidal automorphic representation of $\SO^{E/F}_{2n}(\A_F)$ satisfying \tu{(cent$^\circ$), (coh$^\circ$), (St$^\circ$), and (std-reg$^\circ$)} of \S\ref{sect:GSO-valued-Galois}. Then there exists a semisimple Galois representation (depending on~$\iota$)
$$
\rho_{\pi^\flat} = \rho_{\pi^\flat, \iota} \colon \Gamma_F \to \SO_{2n}(\lql)\rtimes \Gamma_{E/F}
$$
satisfying \tu{(SO-i)--(SO-v)} as in Theorem \ref{thm:ExistGaloisSO} as well as the following.
\begin{enumerate}[leftmargin=+0.7in]
\item[\tu{(SO-i+)}]
For every finite prime $\qq$ of $F$ not above $\Sbad(\pi^\flat) \cup \{\ell\}$,
$$
\iota \phi_{\pi^\flat_\qq} \sim \textup{WD}(\rho_{\pi^\flat}|_{\Gamma_{F_\qq}})^{\textup{F-ss}},
$$
as $\SO_{2n}(\lql)$-parameters.
\item[\tu{(SO-iii+)}]
 For every $\qq | \ell$, the representation $ \rho_{\pi^\flat, \qq}$ is potentially semistable. For each $y \colon F \hra \C$ such that $\iota y$ induces $\qq$, we have
$\mu_{\HT}(\rho_{\pi^\flat,\qq},\iota y) \sim \iota\mu_{\Hodge}(\xi^\flat, y)$.
\end{enumerate}
Condition (SO-i+) characterizes $\rho_{\pi^\flat}$ uniquely up to $\SO_{2n}(\lql)$-conjugation.
\end{theorem}

\begin{remark}\label{rem:SO-i-vs-SO-i+}
Statement (SO-\textit{i+}) is stronger than (SO-\textit{i}) in that the statement is up to $\SO_{2n}(\lql)$-conjugacy, but also weaker as it excludes the  places above $\Sbad(\pi^\flat) \cup \{\ell\}$. Clearly (SO-\textit{iii+}) strengthens (SO-\textit{iii}). If we drop (std-reg$^\circ$) from the assumption, then the theorem can be proved by the same argument but conditionally on Hypothesis \ref{hypo:non-std-reg}.
\end{remark}

\begin{proof}[Proof of Theorem \ref{thm:ExistGaloisSO-Plus}]
We have $\mu_2 = Z(\SO_{2n}^{E/F})$. We claim that the central character $\omega_{\pi^\flat}$ extends (via $\mu_2(\A_F)\subset \A_F^\times$) to a Hecke character 
$$
\chi \colon F^\times \backslash \A_F^\times \to \C^\times
$$ 
such that $\chi_y(z)=z^w$ at every infinite place $y$, where $w=0$ (resp.~$w=1$) if $\omega_{\pi^\flat,y}$ is trivial (resp.~nontrivial) at every $y|\infty$.

To prove the claim, let $E'$ be a quadratic CM extension of $F$. We start by extending $\omega_{\pi^\flat}$ to a (unitary) Hecke character $\chi':F^\times \backslash \A_F^\times \to \C^\times$ whose infinite components are trivial if $w=0$ and the sign character if $w=1$. If $w=0$, then such a $\chi'$ exists since  $\mu_2(F)\mu_2(F_\infty)\backslash \mu_2(\A_F)$ is a closed subgroup of $\ol{F^\times F_{\infty}^\times}\backslash \A_F^\times $, where the bar means the closure in $\A_F^\times$. If $w=1$, consider the quadratic Hecke character $\chi_{E'/F}$ associated with $E'/F$ via class field theory. Then $\omega_{\pi^\flat}(\chi_{E'/F}|_{\mu_2(\A_F)})$ extends to a Hecke character with trivial components at $\infty$ by the $w=0$ case. Multiplying $\chi_{E'/F}$, we obtain a desired  choice of $\chi'$. Whether $w=0$ or $w=1$, we now see that $\chi:=\chi' |\cdot|^w$ has desired components at $\infty$, where $|\cdot|$ is the absolute value character on $\A_F^\times$. The claim is proved.

Consider the multiplication map $f \colon \GL_1\times \SO_{2n}^{E/F} \to \GSO_{2n}^{E/F}$. Let $\xi^\flat$ be such that $\pi^\flat$ is $\xi^\flat$-cohomological. Write $\varsigma$ for the algebraic character $z \mapsto z^w$ of $\GL_1$ over $F$. Then $(\varsigma, \xi^\flat)$ descends to an algebraic representation $\xi$ of $\GSO_{2n}^{E/F}$ via $f$.

Let us extend $\pi^\flat$ to an irreducible admissible $\GSO_{2n}^{E/F}(\A_F)$-representation, by decomposing $\pi^\flat = \otimes_v' \pi_v^\flat$ and taking an irreducible subrepresentation $\pi_v$ of
$$
\tu{Ind}_{\GL_1(F_v) \SO_{2n}^{E/F}(F_v)}^{\GSO_{2n}^{E/F}(F_v)}
\chi_v \pi_v^\flat,
$$
which is semisimple \cite[pp.1832--1833]{Xu-lifting}. Take $\pi_v$ to be unramified for almost all $v$, and define $\pi := \otimes_v' \pi_v$. Lemma 5.4 of \cite{BinXuLPackets} states that
$$
\sum_{\omega \in X/YX(\pi)} m(\pi \otimes \omega) =
\sum_{g \in \GSO_{2n}^{E/F}(\A_F)/\wt G(\pi) \GSO_{2n}^{E/F}(F)} m\big((\pi^\flat)^g\big),
$$
where $X$ is the set of characters of $\GSO_{2n}^{E/F}(\A_F)/\SO_{2n}^{E/F}(\A_F) Z(\GSO_{2n}^{E/F})(\A_F)$, and $\omega$ in the formula is represented by such characters. We refer to \textit{loc.~cit.} for some undefined notation that is not important for us but content ourselves with pointing out that both sides are finite sums. Since $m(\pi^\flat)>0$, the right hand side is positive. Thus the left hand side is positive, and thus we may (and do) twist $\pi$ so that it is discrete automorphic.

We now check that $\pi$ satisfies the conditions of Theorem~\ref{thm:A}. Since $\pi^\flat_\infty$ is $\xi^\flat$-cohomological, by construction $\pi_\infty$ is cohomological according to Lemma \ref{lem:coh-GSO-coh-SO}. By Lemma~\ref{lem:SteinbergCheck}, $\pi$ satisfies (St) thus also cuspidal. Condition (std-reg) is implied by (std-reg$^\circ$) on $\pi^\flat$. Hence we have a Galois representation
$$
\rho_\pi \colon \Gamma \to \GSpin_{2n}(\lql) \rtimes \Gamma_{E/F}.
$$
such that for every finite $F$-place $\qq$ not above $\Sbad(\pi) \cup \{\ell\}$,
\begin{equation}\label{eq:Thm33-plus-1}
\rho_\pi(\Frob_\qq)_{\tu{ss}} \sim \iota \phi_{\pi_\qq}(\Frob_\qq) \in \GSpin_{2n}(\lql) \rtimes \Gamma_{E/F}.
\end{equation}
As in the preceding section, we can arrange that $\rho_{\pi^\flat} = \pr^\circ \circ \rho_\pi$ (not just up to outer automorphism). The Satake parameter of $\pi^\flat_\qq$ is equal to the composition of the Satake parameter of $\pi_\qq$ with the natural surjection (cf. \cite[Lem.~5.2]{BinXuLPackets})
$$
(\pr^\circ,\tu{id}):\GSpin_{2n}(\C) \rtimes \Gamma \to \SO_{2n}(\C) \rtimes \Gamma.
$$
In particular (SO-\textit{i+}) follows from \eqref{eq:Thm33-plus-1} for the places not above $\Sbad(\pi) \cup \{\ell\}$. Similarly (SO-\textit{iii+}) follows from Theorem \ref{thm:A} (A3)(a).

At this point we have not yet completely proved (SO-\textit{i+}), as the inclusion $\Sbad(\pi^\flat) \subset \Sbad(\pi)$ is strict in general. Thus it remains to treat $\qq$ above a prime $p \in \Sbad(\pi) \backslash \Sbad(\pi^\flat)$. Consider for $n$ odd (resp.~even) the obvious hyperspecial subgroup (recall $\qq \nmid 2$)
\begin{align*}
K_{\qq} := \begin{cases} \{(g, \lambda) \in \GL_{2n}(\cO_E \otimes_{\cO_F} \cO_{F_\qq}) \times \cO_{F_\qq}^\times \ |\  \li g = \vartheta^\circ g \vartheta^\circ, g^t \vierkant 0{1_n}{1_n}0 g = \lambda \vierkant 0{1_n}{1_n}0, \det(g) = \lambda^n\} \cr
\tu{resp.} \cr
 \lbr (g, \lambda) \in \GL_{2n}(\cO_{F_\qq}) \times \cO_{F_\qq}^\times \ \left| \ g\tra \cdot \vierkant \ {1_n}{1_n}\ \cdot g = \lambda \cdot \vierkant \ {1_n}{1_n}\ , \det(g) = \lambda^n \right. \ \rbr
\end{cases}
\end{align*}
of $\GSO_{2n}^{E/F}(F_\qq)$. Define $K_{0\qq}$ to be the kernel of the similitudes mapping $K_\qq \to \cO_{F_\qq}^\times$, $(g, \lambda) \mapsto \lambda$. Then $\pi_\qq$ is a ramified representation of $\GSO_{2n}^{E/F}(F_\qq)$, but has nonzero $K_{0\qq}$-fixed vectors, on which $K_\qq$ acts through nontrivial characters of $K_\qq / K_{0\qq}\simeq \cO_{F_\qq}^\times$. We fix one such character $\chi^0_\qq$ of $K_\qq$, and do this at every $\qq$ above $p$. Now we globalize $\{\chi^0_\qq\}_{\qq|p}$ to a Hecke character $\chi:F^\times\backslash \A_F^\times\ra \C^\times$ whose restriction to each $\cO_{F_\qq}^\times$ is $\chi^0_{\qq}$ and whose archimedean components are trivial. (This is possible by \cite[Lem.~4.1.1]{CHT08}.) Define $\pi' := \pi \otimes \chi\inv$. Then $\pi'$ also satisfies the conditions of Theorem \ref{thm:A}. Moreover, $p \notin \Sbad(\pi')$ by construction. Therefore \eqref{eq:Thm33-plus-1} is true at each $\qq|p$, with $\pi'$ in place of $\pi$. Then (SO-\textit{i+}) for $\qq$ follows as before.
\end{proof}

\section{Automorphic multiplicity one}

Let $E/F$ be a quadratic CM extension of $F$ in case $n$ is odd, and $E := F$ for $n$ even. Let $\SO_{2n}^{E/F}$ and $\GSO_{2n}^{E/F}$ be as before. If $\pi$ (resp.~ $\pi^\flat$) is an automorphic representation of $\GSO_{2n}^{E/F}(\A_F)$ (resp.~$\SO_{2n}^{E/F}(\A_F)$), we write $m(\pi)$ (resp.~$m(\pi^\flat)$) for its automorphic multiplicity. In this section we will show that $m(\pi^\flat)$ and $m(\pi)$ are $1$ for certain classes of automorphic representations of $\SO_{2n}^{E/F}(\A_F)$ and $\GSO_{2n}^{E/F}(\A_F)$ (and some inner forms of those groups). To do this we combine our results with Arthur's result on multiplicities for $\SO_{2n}^{E/F}$, and Xu's result on multiplicities for $\GSO_{2n}^{E/F}$.

Let $\pi^\flat$ be a discrete automorphic representation of $\SO_{2n}^{E/F}(\A_F)$. Arthur gives in the discussion below \cite[Thm.~1.5.2]{ArthurBook} the following result towards the computation of $m(\pi^\flat)$. Let $\psi = \psi_1 \boxplus \cdots \boxplus \psi_r$ be the global (formal) parameter of $\pi^\flat$ \cite[\S1.4]{ArthurBook} (cf.~Section~\ref{sect:GSO-valued-Galois}). Technically, $\psi$ is an automorphic representation $\pi^\sharp$ of $\GL_{2n}(\A_F)$ given as an isobaric sum of discrete automorphic representations $\pi_i^\sharp$ of $\GL_{n_i}(\A_F)$, with $\pi_i^\sharp$ representing the formal parameter $\psi_i$. In terms of these parameters  Arthur proves a decomposition of the form
$$
\tu{L}^2_{\tu{disc}}\big(\SO_{2n}^{E/F}(F)\backslash \SO_{2n}^{E/F}(\A_F)\big) \isomto \bigoplus_{\psi \in \wt \Psi_2(\SO_{2n}^{E/F}) } \bigoplus_{\tau \in \wt \Pi_\psi(\eps_\psi)} m_\psi \tau
$$
as an $\wt \cH(\SO_{2n}^{E/F})$-Hecke module. It takes us too far afield to recall all the notation here, but we emphasize that $\wt \cH(\SO_{2n}^{E/F})$ is the restricted tensor product of the local algebras $\wt \cH(\SO_{2n}^{E/F}(F_v))$ consisting of $\theta^\circ$-invariant functions \cite[before (1.5.3)]{ArthurBook}. Similarly, the local packet $\wt \Pi_{\psi_v}(\eps_\psi)$ consists of $\theta^\circ$-orbits of representations. 

Assume $\pi^\flat \not\simeq \pi^\flat \circ \theta^\circ$ for the moment. Both $\pi^\flat$ and $\pi^\flat \circ \theta^\circ$ map to the same global parameter $\psi$, and are isomorphic as $\wt \cH(\SO_{2n}^{E/F})$-modules. Arthur proves $m_\psi \leq 2$ for all $\psi$. Thus
\begin{equation}\label{eq:m_psi}
m(\pi^\flat) + m(\pi^\flat \circ \theta^\circ) \leq m_\psi \leq 2.
\end{equation}
On the other hand, $\theta^\circ$ acts on $\tu{L}^2_{\tu{disc}}(\SO_{2n}^{E/F}(F)\backslash \SO_{2n}^{E/F}(\A_F))$, so if $\pi^\flat$ appears, then $\pi^\flat \circ \theta^\circ$ also appears. Hence $m(\pi^\flat), m(\pi^\flat \circ \theta^\circ) \geq 1$, forcing  $m(\pi^\flat) = 1$ and $m(\pi^\flat \circ \theta^\circ) = 1$. 

From now on we impose the assumption (std-reg${}^\circ$) on $\pi^\flat$. At the infinite $F$-places $v$ the infinitesimal character of $\pi^\flat_v$ is then not fixed by $\theta^\circ$. In particular $\pi^\flat \not\simeq \pi^\flat \circ \theta^\circ$. By the preceding paragraph, we have $m(\pi^\flat) = 1$, $m(\pi^\flat \circ \theta^\circ) = 1$, and $m_\psi = 2$.

\begin{proposition}\label{prop:AutomMult2}
Let $\pi$ be a cuspidal automorphic representation of $\GSO_{2n}^{E/F}(\A_F)$ satisfying (L-coh), (St) and (std-reg). Then $m(\pi) = 1$.
\end{proposition}
\begin{proof}
(cf.~\cite[Thm.~12.1]{GSp} ). Let $\pi^\flat$ be a cuspidal automorphic representation of $\SO_{2n}^{E/F}(\A_F)$ contained in $\pi$. Then $\pi^\flat$ satisfies (coh$^\circ$) and (St$^\circ$) as explained at the start of \S\ref{sect:Construction}. Let $Y(\pi)$ be the set of continuous characters $\omega \colon \GSO_{2n}^{E/F} (\A_F) \to \C^\times$ which are trivial on the subgroup $\GSO_{2n}^{E/F}(F)\A_F^\times\SO_{2n}^{E/F}(\A_F)$ of $\GSO_{2n}^{E/F}(\A_F)$ and such that $\pi \simeq \pi \otimes \omega$. Xu \cite[Prop.~1.7]{BinXuLPackets} proves that
\begin{equation}\label{eq:XuFormula}
m(\pi) = m_{\wt \psi} \, |Y(\pi)/\alpha(\cS_\phi)|,
\end{equation}
where $\wt \psi$ is the global parameter of $\pi$ as defined in~\cite[Sect.~3]{BinXuLPackets} ($\wt \psi$ is denoted $\wt \phi$ there), and $\alpha(\cS_\phi)$ will not be important for us.

We claim that $Y(\pi) = \{\mathbf 1\}$ in \eqref{eq:XuFormula}. Let $\omega \in Y(\pi)$ and let $\chi \colon \Gamma \to \lql^\times$ be the corresponding character via class field theory. As $\chi \rho_\pi$ and $\rho_\pi$ have conjugate Frobenius images at almost all places, we obtain $\chi \rho_\pi \simeq \rho_\pi$ by Proposition \ref{prop:local-conjugacy-global-conjugacy}, and thus $\chi = 1$ by Lemma \ref{lem:for-mult-one}. Hence $Y(\pi) =\{\mathbf 1\}$.
 
Let $\psi$ denote the Arthur parameter of $\pi^\flat$. In \cite[Cor.~5.10]{BinXuLPackets}, Xu states that $m_{\wt \psi} = m_{\psi} / \# \Sigma_Y(\pi)$, where $\Sigma_Y(\pi) := \Sigma_0 / \Sigma_0(\pi, Y)$, where $\Sigma_0$ is the $2$-group $\{1, \theta\}$,  and $\Sigma_0(\pi, Y)$ is the group of $\theta' \in \Sigma_0$ such that $\pi \otimes \omega \simeq \pi^{\theta'}$ for some $\omega \in Y(\pi)$. We saw below \eqref{eq:m_psi} that $m_\psi = 2$. It is enough to check that $\# \Sigma_Y(\pi) = 2$, which would imply $m_{\wt \psi} =1$. As $Y(\pi) =\{\mathbf 1\}$ this reduces to $\pi \not \simeq \pi^{\theta}$. Again by (std-reg) the infinitesimal character of $\pi_v$ for $v | \infty$ is not fixed by $\theta$, so this is indeed true.
\end{proof}

Let $G$ be the inner form of $\GSO_{2n}^{E/F}$ which was constructed in \eqref{eq:ShimuraInnerForm} and used in our Shimura data. We close this section with computing automorphic multiplicities for this $G$. In particular we prove that the multiplicities $a(\cdot)$ appearing in Section~\ref{sect:Shimura} are in fact equal to $1$.

\begin{proposition}\label{prop:AutomMult3}
Let $\pi$ be a cuspidal automorphic representation of $G(\A_F)$, satisfying (coh), (St) and (str-reg). Then $m(\pi) = 1$.
\end{proposition}
\begin{proof}
The proof is the same as the argument for \cite[Thm.~12.2]{GSp}. The main point is that automorphic representations $\tau^*$ of $G^*(\A_F)= \GSO_{2n}^{E/F}(\A_F)$ contributing to the analogue of [\textit{loc.~cit.},~Eq.~(12.2)] have automorphic multiplicity $1$. Notice that \cite[Thm.~1.8]{BinXuLPackets} may be used again, together with the existence of Galois representations (our Theorem A) to prove that for all $\pi^*$ and $\tau^*$ contributing to \cite[Eq.~(12.2)]{GSp} we have $\tau^*_{\qst} \simeq \pi^*_{\qst}$.
\end{proof}

\section{Meromorphic continuation of spin $L$-functions}

Let $n\in \Z_{\ge 3}$, and $\fke$ be as in \eqref{eq:fke}. Let $\pi$ be a cuspidal automorphic representation of $\GSO^{E/F}_{2n}(\A_F)$ unramified away from a finite set of places $S$ satisfying (St), (L-coh), and (spin-reg). This implies (std-reg) for $\pi$. Indeed, if the image of $(s_0,s_1,\ldots,s_n)\in T_{\GSpin}$ under $\spin^\varepsilon$ is regular for some $\varepsilon\in \fke$ then $s_1,\ldots,s_n$ must be mutually distinct, as the weights in $\spin^\varepsilon$ are described as the Weyl orbit(s) of \eqref{eq:Spin-eps-def}.

\begin{proposition}\label{prop:weakly-compatible}
Assume that $\pi$ satisfies (St), (L-coh), and (spin-reg). Let $n\in \Z_{\ge 3}$. There exist a number field $M_\pi$ and a semisimple representation
$$
R^\varepsilon_{\pi,\lambda}: \Gamma\ra \GL_{2^{n}/|\fke|}(\ol{M}_{\pi,\lambda})
$$
for each finite place $\lambda$ of $M_{\pi}$ such that the following hold for every $\varepsilon\in \fke$. (Write $\ell$ for the rational prime below $\lambda$.)
\begin{enumerate}
\item At each place $\qq$ of $F$ not above $S_{\bad}\cup \{\ell\}$, we have $$\mathrm{char}(R^\varepsilon_{\pi,\lambda}(\Frob_\qq))=\mathrm{char}(\spin^\varepsilon(\iota\phi_{\pi_\qq}(\Frob_\qq)))\in M_{\pi}[X].$$
\item $R^\varepsilon_{\pi,\lambda}|_{\Gamma_\qq}$ is de Rham for every $\qq|\ell$. Moreover it is crystalline if $\pi_v$ is unramified and $\qq\notin S^F_{\bad}$.
\item For each $\qq|\ell$ and each $y:F\hra \C$ such that $\iota y$ induces $\qq$, we have $\mu_{\HT}(R^\varepsilon_{\pi,\lambda}|_{\Gamma_\qq},\iota y)=\iota(\spin^\varepsilon\circ \mu_{\Hodge}(\phi_{\pi_y}))$. In particular $\mu_{\HT}(R^\varepsilon_{\pi,\lambda}|_{\Gamma_\qq},\iota y)$ is a regular cocharacter for each $y$.
\item $R^\varepsilon_{\pi,\lambda}$ is pure.
\item $R^\varepsilon_{\pi,\lambda}$ maps into $\GSp_{2^n/|\fke|}(\ol{M}_{\pi,\lambda})$ if $n\equiv 2,3~(\tu{mod}~4)$ (resp.~$\GO_{2^n/|\fke|}(\ol{M}_{\pi,\lambda})$ if $n\equiv 0,1~(\tu{mod}~4)$) for a nondegenerate alternating (resp.~symmetric) pairing on the underlying $2^{n}/|\fke|$-dimensional space over $\ol{M}_{\pi,\lambda}$. The multiplier character $\mu^\varepsilon_\lambda:\Gamma\ra\GL_{1}(\ol{M}_{\pi,\lambda})$ (so that $R^\varepsilon_{\pi,\lambda}\simeq (R^\varepsilon_{\pi,\lambda})^\vee\otimes \mu^\varepsilon_\lambda$) is totally of sign $(-1)^{n(n-1)/2}$ and associated with $\omega_\pi$ via class field theory and $\iota_\lambda$.
\end{enumerate}
\end{proposition}

\begin{proof}
Let $M$ be the field of definition of $\xi$, which is a finite extension of $\Q$ in $\C$. We can choose $M_\pi$ to be the field of definition for the $\pi^\infty$-isotypic part in the (compact support) Betti cohomology of $H^\bullet(\Sh^+(\C),\cL_\xi)\oplus H^\bullet(\Sh^-(\C),\cL_\xi)$ with $M$-coefficient. Then $M_\pi$ is a finite extension of $M$ in $\C$. For each prime $\ell$ and a finite place $\lambda$ of $M_\pi$ above $\ell$, extend $M\hra \C$ to an isomorphism $\ol M_{\pi,\lambda}\simeq \C$. Identifying $\ol M_{\pi,\lambda}\simeq \lql$, we have $\iota_\lambda:\C\isom \lql$. Take
$$
R^\varepsilon_{\pi,\lambda}:= \spin^\varepsilon\circ \rho_{\pi,\iota_\lambda}.
$$
Then (1), (2), and (3) follow from (A2) and (A3) of Theorem \ref{thm:A}, respectively. Part (4) follows from (SO-ii) of Theorem \ref{thm:ExistGaloisSO} via (A2). The first part of (5) holds true since $\spin^\varepsilon: \GSpin_{2n}\ra \GL_{2^{n-1}}$ is an irreducible representation preserving a nondegenerate symplectic (resp.~symmetric) pairing up to scalar if $n$ is 2 (resp.~0) mod 4, and since $\spin: \GPin_{2n}\ra \GL_{2^{n}}$ is irreducible and preserves a nondegenerate symplectic (resp.~symmetric) pairing up to scalar  if $n$ is 3 (resp.~1) mod 4. Indeed, the irreducibility is standard and the rest follows from Lemma \ref{lem:spin-inv-pairing} (with the pairing given as in the lemma). Lemma \ref{lem:spin-inv-pairing} also tells us that $\mu^\varepsilon_\lambda=\cN\circ \rho_{\pi,\iota_\lambda}$. By (A2), $\omega_\pi=\cN\circ \rho_{\pi,\iota_\lambda}$ so $\mu^\varepsilon_\lambda$ is associated with $\omega_\pi$. As in the proof of part 5 of \cite[Prop.~13.1]{GSp}, $\omega_\pi \otimes |\cdot|^{n(n-1)/2}$ corresponds to an even Galois character of $\Gamma$. (We change $n(n+1)/2$ in \cite{GSp}  to $n(n-1)/2$ here due to the difference in the definition of (L-coh).) It follows that $\mu_{\lambda,y}(c_y)=(-1)^{n(n-1)/2}$ for every $y|\infty$.
\end{proof}

Now we apply potential automorphy results to the weakly compatible system of $R^\varepsilon_{\pi,\lambda}$.

\begin{theorem}
Theorem~\ref{thm:meromorphic} is true.
\end{theorem}

\begin{proof}
  This follows from \cite[Thm.~A]{PatrikisTaylor}, which can be applied to the weakly compatible system $\{R^\varepsilon_{\pi,\lambda}\}$ thanks to the preceding proposition.
\end{proof}

\begin{remark}
  We cannot appeal to the potential automorphy as in \cite[Thm.~A]{BLGGT14-PA} as $R^\varepsilon_{\pi,\lambda}$ may be reducible. The point of \cite{PatrikisTaylor} is to replace the irreducibility hypothesis with a purity hypothesis (guaranteed by (iv) of Proposition \ref{prop:weakly-compatible}). We take advantage of this.
\end{remark}

\newpage

\appendix

\section{Extending a Galois representation}\label{app:extending}

Here we investigate the problem of extending a $\hat G$-valued Galois representation to an $^L G$-valued representation over a quadratic extension.

We freely use the notation and terminology of \S\ref{sect:Notation}. Let $E$ be a CM quadratic extension over a totally real field $F$ in an algebraic closure $\ol F$. Set $\Gamma=\Gamma_F:=\Gal(\ol{F}/F)$, $\Gamma_E:=\Gal(\ol{F}/E)$, and $\Gamma_{E/F}:=\Gal(E/F)=\{1,c\}$. Let $G$ be a quasi-split group over $F$ which splits over $E$. Let $\theta\in \Aut(\hat G)$ denote the action of $c$ on $\hat G$ (with respect to a pinning over $F$). By $\hat G(\lql) \rtimes \Gamma_{E/F}$, we mean the $L$-group relative to $E/F$, namely the semi-direct product such that $c g c = \theta(g)$ for $g\in \hat G(\lql)$.

Fix an infinite place $y$ of $F$. Write $c_y\in \Gamma_F$ for the corresponding complex conjugation (well defined up to conjugacy). Let $\rho': \Gamma_E\ra \hat G(\lql)$ be a Galois representation. Define $$^{c_y} \rho'(\gamma):=\rho'(c_y\gamma c^{-1}_y).$$ (Of course $c^{-1}_y=c_y$.) We will sometimes impose the following hypotheses.
\begin{enumerate}
\item[(H1)] $\tu{Cent}_{\hat G}(\tu{im}(\rho'))=Z(\hat G)$.
\item[(H2)] The map $Z(\hat G)\ra Z(\hat G)^\theta$ given by $z\mapsto z\theta(z)$ is a surjection on $\lql$-points.
\end{enumerate}

\begin{lemma}\label{lem:extend}
Consider the following statements.
\begin{enumerate}
\item $\rho'$ extends to some $\rho:\Gamma_F\ra \hat G(\lql) \rtimes \Gamma_{E/F}$.
\item $^{c_y}\rho'\simeq \theta\circ \rho'$.
\item there exists $g\in \hat G(\lql)$ such that $g\theta(g)=1$ and $\rho'(c_y \gamma c_y^{-1})=g \theta(\rho'(\gamma))g^{-1}$ for every $\gamma\in \Gamma_E$.
\end{enumerate}
Then (3)$\Leftrightarrow$(1)$\Rightarrow$(2). In particular if $\rho$ is as in (1) then the element $g$ such that $\rho(c_y)=g\rtimes c$ enjoys the property of (3).
If (H1) and (H2) are satisfied, then we also have (2)$\Rightarrow$(3), so all three statements are equivalent.
\end{lemma}

\begin{remark}
  We recommend  \cite[Section A.11]{BellaicheChenevier} as a useful guide to similar ideas.
\end{remark}

\begin{remark}
  Often (2) is the condition to verify to extend a Galois representation, as we did in Lemma \ref{lem:condition-to-extend} of this paper.
\end{remark}

\begin{proof}

(3)$\Leftrightarrow$(1): First we show (3)$\Rightarrow$(1). Define $\rho$ by $\rho|_{\Gamma_E}:=\rho'$ and $\rho(\gamma c_ y):=\rho'(\gamma)g c$ ($\gamma \in \Gamma_E$). Then
$$
\rho(c_ y ^2)=gc gc = g\theta(g)=1,
$$
$$
\rho(c_ y \gamma c_ y ^{-1}) = {}^{c_ y} \rho'(\gamma) = g \theta(\rho'(\gamma))g^{-1},
$$
and using this, one checks that $\rho$ is a homomorphism on the entire $\Gamma$. A similar computation shows (1)$\Rightarrow$(3) for $g$ such that $\rho(c_y)=g\rtimes c$.

\medskip

(1)$\Rightarrow$(2): Write $\rho(c_y)=gc$ with $g\in \hat G(\lql)$. For every $\gamma\in \Gamma_E$,
$$^{c_y}\rho'(\gamma)=\rho(c_y \gamma c^{-1}_ y) = g c \rho'(\gamma) c^{-1} g^{-1} = g \theta(\rho'(\gamma)) g^{-1}.$$

\medskip

(2)$\Rightarrow$(3), assuming (H1) and (H2): There exists $g\in \hat G(\lql)$ such that
\begin{equation}\label{eq:c-gamma-c}
\rho'(c_ y \gamma c^{-1}_ y) = g\theta(\rho'(\gamma)) g^{-1},\qquad \gamma\in \Gamma_E.
\end{equation}
Putting $c_ y \gamma c^{-1}_ y $ in place of $\gamma$, we obtain
$$ 
\rho'(\gamma)=\rho'(c_y^2 \gamma c_y^{-2}) = g \theta( g\theta(\rho'(\gamma)) g^{-1} ) g^{-1}
= g\theta(g) \rho'(\gamma) (g\theta(g))^{-1}.
$$
Hence $g\theta(g)\in Z(\hat G)$ as $\tu{Cent}_{\hat G}(\rho')=Z(\hat G)$ by (H1). As a central element,
$$
g\theta(g)= g^{-1}(g\theta(g)) g =\theta(g) g= \theta(g\theta(g)),
$$
namely $g\theta(g)\in Z(\hat G)^\theta$. By (H2), $g\theta(g)=z\theta(z)$ for some $z\in Z(\hat G)$. Replacing $g$ with $gz^{-1}$, we can arrange that 
$$
g \theta(g)=1.
$$
This does not affect \eqref{eq:c-gamma-c} so we are done.
\end{proof}

\begin{lemma}\label{lem:uniquely-extend}
 Assume (H1). Then the set of $\hat{G}$-conjugacy classes of extensions of $\rho'$ to $\Gamma$ is an $H^1(\Gamma_{E/F},Z(\hat G))$-torsor if nonempty.
\end{lemma}

\begin{proof}
Fix an extension $\rho_0$ of $\rho'$, which exists by Lemma \ref{lem:extend}. If $\rho$ is another extension of $\rho'$, then set $z:=\rho_0(c_y)\rho(c_y)^{-1}$. Writing $\rho_0(c_y)=g_0 \rtimes c$ and $\rho(c_y)=g\rtimes c$, we have $z g= g_0$, and both $g_0,g$ satisfy the condition of Lemma \ref{lem:extend} (3). It follows that $z$ centralizes $\theta(\tu{im}(\rho'))$, hence $z\in Z(\hat G)$, and also that $z\theta(z)=1$. Thus $z$ defines a $Z(\hat G)$-valued 1-cocycle on $\Gamma_{E/F}$, and by reversing the process, such a cocycle determines an extension of $\rho'$.

Let $\rho_z$ be the extension given by $z\in Z(\hat G)$ such that $z\theta(z)=1$. It remains to show that $\rho_z\sim \rho_0$ if and only if $z=\theta(x)/x$ for some $x\in Z(\hat G)$. If $\rho_z\sim \rho_0$ then $\rho_z=\tu{Int}(x) \rho_0$ for some $x\in \hat G$. By (H1), $x\in Z(\hat G)$. Evaluating at $c_y$, we obtain $z^{-1}\rho_0(c_y)= x \rho_0(c_y) x^{-1}$. Therefore $z=\theta(x)/x$. The converse direction is shown similarly by arguing backward.
\end{proof}

We illustrate assumptions (H1), (H2), and the lemmas in the following examples.

\begin{example}\label{ex:SO2n-extend}
Consider $\hat G=\SO_{2n}$ ($n\ge 3$) with $\theta$ being the conjugation by $\vartheta^\circ\in\tu{O}_{2n}(\lql)-\SO_{2n}(\lql)$ as in \eqref{eq:Elementw}. Assume that $\tu{im}(\rho')$ contains a regular unipotent of $\SO_{2n}(\lql)$. In this case $Z(\hat G)=Z(\hat G)^\theta = \{\pm1\}$. Then (H2) is trivially false but (H1) is true. To see this, by assumption, $\tu{std}\circ \rho'$ is either irreducible or the direct sum of an irreducible $(2n-1)$-dimensional representation and a character. In the former case (H1) is clear by Schur's lemma. In the latter case, again by Schur's lemma, a centralizer of $\tu{im}(\rho')$ in $\SO_{2n}(\lql)$ is contained in $(\begin{smallmatrix} a \cdot 1_{2n-1} & 0 \\ 0 & b \end{smallmatrix})$ with $a,b\in \{\pm 1\}$ up to $\OO_{2n}(\lql)$-conjugacy. Since the determinant equals 1, we deduce that $a=b$, i.e., the centralizer belongs to $Z(\hat G)$.

We easily compute $Z^1(\Gamma_{E/F},Z(\hat G))=H^1(\Gamma_{E/F},Z(\hat G))\simeq \Z/2\Z$, the nontrivial element sending $c$ to $-1$. In fact if $\rho$ extends $\rho'$ in the setup of the preceding lemmas, the other extension is easily described as $\rho\otimes \chi_{E/F}$, where $\chi_{E/F}:\Gamma \twoheadrightarrow \Gamma_{E/F}\isom \{\pm 1\}$.
\end{example}

\begin{example}\label{ex:GSpin2n-extend}
The main case of interest for us is when
\begin{itemize}
\item $\hat G=\GSpin_{2n}$ ($n\ge 3$) ,
\item $\theta$ is the conjugation by an element of $\GPin_{2n}(\lql)-\GSpin_{2n}(\lql)$,
\item $\tu{im}(\rho')$ contains a regular unipotent.
\end{itemize}
Since $Z(\hat G)^\theta=\G_m$ (identified with invertible scalars in the Clifford algebra underlying $\hat G$ as a $\GSpin$ group; see \S\ref{sect:CliffordGroups}), assumption (H2) is satisfied. (The squaring map $\G_m\ra \G_m$ is clearly surjective on $\lql$-points.) To check (H1), $\tu{Cent}_{\hat G}(\tu{im}(\rho'))$ is contained in the preimage of $\tu{Cent}_{\SO_{2n}}(\tu{im}(\rho^{\prime,\circ}))$ via $\pr^\circ:\GSpin_{2n}\ra \SO_{2n}$. Since the latter centralizer is $\{\pm1\}\subset \SO_{2n}(\lql)$, we see that  $\tu{Cent}_{\hat G}(\tu{im}(\rho'))\subset \tu{pr}^{\prime,-1}(\{\pm1\}) = Z(\hat G)$.

In the coordinates for $Z(\hat G)$ of Lemma \ref{lem:ComputeCenter}, $Z^1(\Gamma_{E/F},Z(\hat G))=\{(s_0,s_1): s_1\in\{\pm1\},~s_1=s_0^2\}\simeq \mu_4$, of which coboundaries are $\{(\pm1,1)\}\simeq \mu_2$. (The first identification is given by taking the image of $c$.) Hence $H^1(\Gamma_{E/F},Z(\hat G))\simeq \Z/2\Z$. Let $\zeta=(\zeta_4,-1)\in Z(\hat G)$, where $\zeta_4$ is a primitive fourth root of unity, cf.~Lemma \ref{lem:spin-center2}. If $\rho$ is an extension of $\rho'$, then the other extension (up to $\hat G$-conjugacy) is described as $\rho\otimes \chi$, where $\chi:\Gamma \ra Z(\hat G)\rtimes \{1,c\}$ is inflated from $\Gamma_{E/F}\isom \{1,\zeta\rtimes c\}$.  Notice that $\pr^\circ\circ \chi=\chi_{E/F}$, for $\chi_{E/F}$ as in the preceding example.
\end{example}

\begin{example}
When studying Galois representations arising from automorphic representations on a unitary group $U_n$ in $n$ variables, two target groups appear in the literature: the group $\mathcal{G}_n$ in \cite[\S2.1]{CHT08} and the C-group of $U_n$ in \cite{BuzzardGee}; the two are isogenous as explained in \cite[\S8.3]{BuzzardGee}. The latter is the $L$-group of a $\G_m$-extension of $U_n$; it does not satisfy (H2). The former is not an $L$-group, but still a semi-direct product $(\GL_n\times \GL_1)\rtimes \Gamma_{E/F}$, with $c(g,\mu)=(\mu g^{-t},\mu)$ for an anti-diagonal matrix $\Phi_n\in \GL_n$. As such, the discussion in this appendix goes through for $\mathcal{G}_n$. An easy computation shows that $\mathcal{G}_n$ satisfies (H2) and that $H^1(\Gamma_{E/F},Z(\GL_n\times \GL_1))=\{1\}$ for the given Galois action. Thus provided that $\rho'$ satisfies (H1) (e.g., if $\rho'$ is irreducible), an extension of $\rho'$ exists if and only if $^{c_y}\rho'\simeq \theta\circ \rho'$, and the extension is unique up to conjugacy. Compare this with \cite[Lem.~2.1.4]{CHT08} (which allows a general coefficient field of characteristic~$0$).
\end{example}

\section{On local $A$-packets of even special orthogonal groups}

In this appendix we study the $A$-packets of the trivial and Steinberg representations of quasi-split forms of $\SO_{2n}$, with $n\ge 3$, often following the notation and formulation of \cite{ArthurBook}.

Let $F$ be a finite extension of $\Q_p$. Suppose that $E=F$ or that $E$ is a quadratic extension of $F$. Let $\chi_{E/F}:F^\times \ra \{\pm1\}$ denote the quadratic character associated with $E/F$ via class field theory. Let $G := \SO^{E/F}_{2n}$ denote the quasi-split form of the split group $\SO_{2n}$ over $F$ twisted by $\chi_{E/F}$. Write $\tilde{\Out}_{2n}(G):=\OO_{2n}(\C)/\SO_{2n}(\C)$ for the outer automorphism group on $\SO_{2n}(\C)$. Denote by $\mathbf 1$ and $\St$ the trivial and Steinberg representations of $G(F)$. We aim to identify local $A$-packets containing each of $\mathbf 1$ and $\St$.

Let $\cL_F:=W_F\times \SU(2)$ denote the local Langlands group. Let $|\cdot|:W_F\ra \R^\times_{>0}$ denote the absolute value character sending a geometric Frobenius element to the inverse of the residue cardinality of $F$. By abuse, keep writing $|\cdot|$ for its pullback to $\cL_F$ via projection. 

Denote by $\Psi^+(G)$ the set of isomorphism classes of extended $A$-parameters, that is, continuous morphisms $\psi:\cL_{F}\times \SU(2)\ra {}^L G$ such that $\psi|_{\cL_F}$ is an $L$-parameter. (Two $A$-parameters are considered isomorphic if they are in the same $\hat G$-orbit.) An extended $A$-parameter $\psi\in \Psi^+(G)$ gives rise to an $L$-parameter: 
$$
\phi_{\psi}:\cL_F \ra {}^L G,\qquad \gamma\mapsto \psi(\gamma,\diag(|\gamma|^{1/2},|\gamma|^{-1/2})).
$$

Write $\Psi(G)$ for the subset of $\Psi^+(G)$ consisting of $\psi\in \Psi^+(G)$ such that the image of $\psi(\cL_F)$ in $\SO_{2n}(\C)\rtimes \Gamma_{E/F}$ is bounded. (Such a property is $\hat G$-invariant.) The set of $\tilde{\Out}_{2n}(G)$-orbits in $\Psi^+(G)$ (resp.~$\Psi(G)$) is denoted by $\tilde\Psi^+(G)$ (resp.~$\tilde\Psi(G)$). The group $\cL_F \times \SU(2)$ admits the involution permuting the two $\SU(2)$-components (acting as the identity on $W_F$). This involution induces an involution 
$$
\psi\mapsto \hat \psi \quad\mbox{on each of}\quad \tilde\Psi^+(G)~\mbox{and}~\tilde\Psi(G).
$$

We say $\psi\in \tilde\Psi^+(G)$ is square-integrable if $g\psi g^{-1}=\psi$ for at most finitely many elements $g\in \hat G$. Then $\psi$ lies in $\tilde\Psi(G)$. To see this, let $w\in W_F$ be a lift of (geometric) Frobenius. Then $\psi(w)^m$ centralizes the image of $\psi$ for some $m\in \Z_{\ge1}$ as in \cite[proof of Lem.~8.4.3]{DeligneEquationsFonctionelles}. It follows that, replacing $m$ with a suitable multiple, $\psi(w)^m$ has trivial image in $\SO_{2n}(\C)\rtimes \Gamma_{E/F}$.  Write $I_F\subset W_F$ for the inertia subgroup. Since $I_F\times \SU(2)\times \SU(2)\subset \cL_F\times \SU(2)$ has already bounded image in $\SO_{2n}(\C)\rtimes \Gamma_{E/F}$ under $\psi$, we see that $\psi\in \Psi(G)$. Denote by $\Psi_2(G)$ the subset of $\Psi(G)$ consisting of square-integrable members.

Define $\psi_{\tu{triv}} \colon \cL_F \times \SU(2) \to {}^L G$ as follows. On $\cL_F$ it is the composite map $\cL_F\twoheadrightarrow W_F \twoheadrightarrow \Gamma_{E/F}\subset {}^L G$ through the natural projections. On $\SU(2)$ (outside $\cL_F$), $\psi_{\tu{triv}}$ is a principal embedding $i_{\tu{pri}} \colon \SU(2) \to \widehat G$ that is $\theta^\circ$-invariant, i.e., $i_{\tu{pri}}$ commutes with the $L$-action of $\Gamma_{E/F}$ on $\hat G$. (Such an $i_{\tu{pri}}$ into $\widehat G$ can be realized as the $\SU(2)$-representation $\Sym^{2n-2}\oplus \mathbf 1$ into $\GL_{2n-1}\times \GL_1$, where the latter is identified with the centralizer of the element $\vartheta^\circ\in \GL_{2n}$ from \S\ref{sect:CliffordGroups}.
 Write $\psi_{\St}:=\hat \psi_{\tu{triv}}$. Then $\psi_{\tu{triv}},\psi_{\St}\in \Psi(G)$ and they are $\tilde{\Out}_{2n}(G)$-stable.

To every $\psi\in \tilde\Psi(G)$, Arthur \cite[Thm.~1.5.1]{ArthurBook} assigned an $A$-packet $\tilde\Pi(\psi)$, which is a certain finite multi-set consisting of $\tilde{\Out}_{2n}(G)$-orbits of irreducible unitary representations of $G(F)$. Below \emph{loc.~cit.}~he also defines $\tilde\Pi(\psi)$ for $\psi\in \tilde\Psi^+(G)$, consisting of $\tilde{\Out}_{2n}(G)$-orbits of parabolically induced representations of $G(F)$ (which need not be irreducible or unitary).

By a globalization $(\dot E/\dot F,\qq,\dot G)$ of $(E/F,G)$ as above, we mean an extension of number fields $\dot E/\dot F$, a finite place $\qq$, and a quasi-split form $\dot G$ of the split $\SO_{2n}$ over $\dot F$ such that $\dot E_\qq\simeq E$, $\dot F_\qq  \simeq F$, and $\dot G_\qq  \simeq G$. It is an elementary fact that such a globalization always exists. Recall that a (formal) global parameter $\dot \psi\in \tilde\Psi(\dot G)$ gives rise to a parameter $\dot \psi_v\in \tilde \Psi^+(\dot G_{\dot F_v})$ and a packet $\tilde \Pi(\dot \psi_v)$ at each place $v$ of $\dot F$.

\begin{proposition}\label{prop:A-parameter-1-St} Let $\psi\in \tilde \Psi^+(G)$. The following are true.
\begin{enumerate}
\item $\tilde\Pi(\psi_{\tu{triv}})=\{\mathbf 1\}$ and $\tilde\Pi(\psi_{\St})=\{\St\}$.
\item Assume $\psi\in \tilde \Psi(G)$.
If $\mathbf 1$ (resp.~$\St$) is a member of $\tilde\Pi(\psi)$ then $\psi=\psi_{\tu{triv}}$ (resp.~$\psi=\psi_{\St}$).
\item Assume that $\psi=\dot \psi_{\qq}\in \tilde \Psi^+(G)$ for global data $(\dot E/\dot F,\qq,\dot G)$ and $\dot \psi$ as above.
If $\mathbf 1$ (resp.~$\St$) is a subquotient of a member of $\tilde\Pi(\psi)$ then $\psi=\psi_{\tu{triv}}$ (resp.~$\psi=\psi_{\St}$).
\end{enumerate}
\end{proposition}

\begin{remark}
We use (3) in the main text. Part (3) would be subsumed by (2) if the generalized Ramanujan conjecture for general linear groups was known, cf.~proof of (3) below.
\end{remark}

\begin{proof}
(1)  
According to \cite[Lem.~7.1.1]{ArthurBook} (and the discussion following it), the involution $\psi\mapsto \hat \psi$ changes members of $A$-packets by the Aubert involution, which carries $\mathbf 1$ to $\St$ and vice versa. So it suffices to consider the case of $\psi_{\tu{triv}}$. Choose a globalization $(\dot E/\dot F,\qq,\dot G)$ of $(E/F,G)$. Arthur's global theorem \cite[Thm.~1.5.2]{ArthurBook} assigns a global parameter $\dot \psi$ whose packet contains the trivial representation $\mathbf 1_{\dot G}$ of $\dot G(\A_F)$. Considering the Satake parameters at almost all places, we identify
\begin{equation}\label{eq:dot-psi}
\dot \psi=(\mathbf{1}\boxtimes \nu_{2n-1})\boxplus (\chi_{E/F}\boxtimes \nu_{1})
\end{equation}
in Arthur's notation \cite[\S1.4]{ArthurBook}, where $\nu_i$ denotes the $i$-dimensional irreducible representation of $\SU(2)$. From this, we see that $\dot\psi_\qq=\psi_{\tu{triv}}$.

In our case, Arthur's global packet $\tilde\Pi(\dot\psi)=\otimes'_v \tilde\Pi(\dot\psi_v)$ consists of $\dot G(\A_{\dot F})$-representations $\dot\pi=\otimes'_v \dot\pi_v$ with $\dot\pi_v\in\tilde\Pi(\dot\psi_v)$ at each place $v$. Since the groups $\cS_{\dot \psi}$ and $\cS_{\dot \psi_v}$ are trivial, every member of $\tilde\Pi(\dot\psi)$ is automorphic by \cite[Thm.~1.5.2]{ArthurBook}. Moreover, each multi-set $\tilde\Pi(\dot\psi_v)$ contains $\mathbf 1$ with multiplicity one by \cite[Prop.~7.4.1]{ArthurBook} since $\mathbf 1$ is easily seen to be the unique member of the local $L$-packet for the unramified $L$-parameter $\phi_{\dot \psi_v}$.

Now let $\pi_{\qq}\in \tilde\Pi(\dot \psi_{\qq})$. Then $\dot \pi=\pi_{\qq}\otimes(\otimes'_{v\neq \qq} \mathbf 1)\in \tilde\Pi(\dot\psi)$, so it appears in the $L^2$-discrete automorphic spectrum. This implies that $\pi_{\qq}=\mathbf 1$ since $G(F)G(\A_F^{\qq})$ is dense in $G(\A_F)$ by weak approximation. Therefore $\tilde\Pi(\dot \psi_{\qq})=\{\mathbf 1\}$ (with multiplicity one) as desired.

\smallskip

(2) We deduce this from Moeglin--Waldspurger's description of $A$-packets for $\psi\in \tilde \Psi(G)$ (see \cite{MW2006,Moeglin2009} and the summary in \cite[\S5]{Xu-Apackets}), which coincides with Arthur's $A$-packets thanks to the main theorem of \cite{Xu-Apackets}.
If we write $W_F\times \Delta(\SU(2))$ for the obvious subgroup of $W_F\times \SU(2)\times \SU(2)$ via the diagonal embedding $\Delta$ on $\SU(2)$, then the first sentence in \cite[Rem.~7.6]{Moeglin2009} tells us that $\psi$ as in the statement of (2) has the property that $\psi|_{W_F\times \Delta(\SU(2))}$ is isomorphic to $\psi_{\tu{triv}}$, resp.~$\psi_{\St}$. This implies that $\psi\simeq \psi_{\tu{triv}}$, resp.~$\psi\simeq \psi_{\St}$, by a simple exercise with representations of $\SU(2)$. (In fact, we could start from scratch and deduce part (1) from Moeglin--Waldspurger as well, before part (2).)

Since no proof is provided for \cite[Rem.~7.6]{Moeglin2009} (the proof is given only when $\psi|_{W_F\times \Delta(\SU(2))}$ is a discrete parameter; see \cite[Prop.~4.6]{Moeglin2009}), we explain another way to deduce (2) by analyzing supercuspidal support. Viewing $\psi$ as a parameter for $\GL_{2n}$, decompose 
\begin{equation}\label{eq:psi-appendix}
\psi=\oplus_{i=1}^r l_i(\rho_i\otimes \nu_{a_i}\otimes \nu_{b_i}),\qquad l_i,a_i,b_i\in \Z_{\ge 1},
\end{equation}
where $\rho_i$ are irreducible unitary representations of $W_F$,
as in \cite[p.892]{Xu-Apackets}; if $i\neq i'$ then $\rho_i\ncong \rho_{i'}$ or $a_i\neq a_{i'}$ or $b_i\neq b_{i'}$. When $\psi$ is trivial on the $\SU(2)$ outside $\cL_F$, the $A$-packet for $\psi$ coincides with the $L$-packet. For arbitrary $\psi$, Moeglin--Waldspurger construct $A$-packets in increasing generality. We do not recall the construction, but only extract enough on the supercuspidal support $(M,\sigma)$ of $\pi\in \tilde\Pi(\psi)$ from \cite[\S5, pp.907--909]{Xu-Apackets}. In a nutshell, $(M,\sigma)$ has the following shape depending on $\psi$:
\begin{itemize}
\item $M=\SO^{E/F}_{2n_-}\times \prod_{j\in J} \GL_{n_j}$ is a Levi subgroup of $G$,
\item $\sigma=\sigma_-\otimes (\otimes_{j\in J}\sigma_j)$ is an irreducible supercuspidal representation of $M(F)$, and
\item $I_-\subset \{1,...,r\}$ is a possibly empty subset such that $\rho_i$ is self-dual of orthogonal type for each $i\in I_-$,
\end{itemize}
such that $\sigma_-$ belongs to the $L$-packet of $\oplus_{i\in I_-} (\rho_i\otimes\nu_1\otimes \nu_1)$, which defines a discrete $L$-parameter $\psi_-$ for $\SO^{E/F}_{2n_-}$, and for each $j\in J$, we have $1\le i \le r$ depending on $j$ such that
\begin{equation}\label{eq:sigmaj-appendix}
\sigma_j\simeq \tu{LL}(\rho_{i} \otimes|\cdot|^{x_j})\quad \mbox{for some}~x_j\in \R,
\end{equation}
where $\tu{LL}$ denotes the local Langlands correspondence for general linear groups. 
When \eqref{eq:sigmaj-appendix} holds, we will say that $\rho_i$ contributes to $\sigma_j$. 
Conversely, for every $1\le i\le r$, either or both of the following are true: $\rho_i$ contributes to $\sigma_j$ for some $j\in J$, or $\rho_i$ is isomorphic to $\rho_{i'}$ for some $i'\in I_-$. (This condition is obviously satisfied for $i\in I_-$ but possibly also for some $i\notin I_-$ since $\rho_i$'s need not be mutually non-isomorphic in \eqref{eq:psi-appendix}.)

We apply the above to the case where $\pi=\textbf{1}$ or $\pi=\St$. In either case, the supercuspidal support is $(T,\delta_B^{-1/2})$, where $B$ is a Borel subgroup of $G$ containing a maximal torus $T$. Hence we can identify $(M,\sigma)=(T,\delta_B^{-1/2})$, possibly after applying a Weyl group action, so in the notation above, 
\begin{itemize}
\item  $n_j=1$ for all $j\in J$,
\item if $E=F$ then $n_-=0$; otherwise $n_-=2$,
\item $\sigma_-=\textbf{1}$ and each $\sigma_j$ is a half-integral or integral power of the modulus character $|\cdot|$.
\end{itemize}
The last fact tells us that for each $1\le i\le r$, if $\rho_i$ contributes to $\sigma_j$ for some $j$ then $\rho_i=\textbf{1}$ since $\rho_i$ is unitary. On the other hand, if $E\neq F$ (thus $n_-=2$), $\psi_-$ is the parameter of the torus $SO^{E/F}_2$ containing $\textbf{1}$ as a summand, so $\psi_-=\textbf{1}\oplus \chi_{E/F}$. (If $E=F$ then $\psi_-$ is non-existent as $n_-=0$.) The argument so far shows that $\rho_i=\textbf{1}$ for every $i$, with exactly one exception if $E\neq F$, in which case there is one $i_0$ such that $\rho_{i_0}=\chi_{E/F}$. Since $\chi_{E/F}$ does not contribute to $\sigma_j$'s, it also follows from \cite[\S5, pp.907--909]{Xu-Apackets} that $l_{i_0}=a_{i_0}=b_{i_0}=1$.

It remains to check that there is an $i$ with $\rho_i=\textbf{1}$ such that $a_i=2n-1$ and $b_i=1$ or $a_i=1$ and $b_i=2n-1$. (Then it follows that $r=2$ and the complementary 1-dimensional factor in \eqref{eq:psi-appendix} is $\textbf{1}$ if $E=F$ and $\chi_{E/F}$ if $E\neq F$.)  Again we can read this off from \emph{loc.~cit.}, keeping in mind that $\mathbf{1}\otimes \nu_{a_i}\otimes \nu_{b_i}$ corresponds to $(\rho,A,B,\zeta)$ in the notation there with $\rho=\mathbf{1}$, $A=\tfrac{a+b}{2}-1$, $B=\tfrac{|a-b|}{2}$, and $\zeta=\tu{sgn}(a-b)$ if $a\neq b$ and $\zeta$ is arbitrary if $a=b$ (see \cite[p.901]{Xu-Apackets}). It follows from the successive procedures of \cite[\S5, pp.907--909]{Xu-Apackets} to embed $\pi$ in a parabolic induction  that each factor $\mathbf{1}\otimes \nu_{a_i}\otimes \nu_{b_i}$  contributes 
$$|\cdot|^{x_j} \quad\mbox{with}\quad x_j\in \tfrac12 \Z \quad\mbox{such that}\quad |x_j|\in \{\tfrac{a_i+b_i}{2} - 1, \tfrac{a_i+b_i}{2} - 2, ..., \}$$
 to the supercuspidal support of $\pi$ as $\sigma_j$ for some $j$'s. (Namely each $i$ may contribute several powers of the modulus character of the above form to the supercuspidal support.) On the other hand, calculation of $\delta_B^{-1/2}$ on $T$ shows that $\sigma_j=|\cdot|^{n-1}$ must appear in the supercuspidal support $(T,\delta_B^{-1/2})$ for some $j$. Hence
 $\tfrac{a_i+b_i}{2}-1-(n-1) = \tfrac{a_i+b_i}{2}-n\in \Z_{\ge 0}$ for some $i$.
 This is only possible when $(a_i,b_i)=(2n-1,1)$ or $(1,2n-1)$ due to the obvious constraint $a_i b_i\le 2n$ and $n\ge 3$. The proof of (2) is finished.

\smallskip

(3) We claim that $\psi\in \tilde \Psi^+(G)$ as in the statement belongs to $\tilde\Psi(G)$. To see this, we review the construction of the packet $\Pi(G)$ from \cite[1.5]{ArthurBook}, made explicit in \cite[App.~A]{BinXuLPackets}.

Since $\psi$ comes from a global parameter, $\psi|_{\cL_F}$ is contained by what local components of cuspidal automorphic representations of general linear groups can be. Following the same observations as in \cite[App.~A]{BinXuLPackets} we can express $\psi$ concretely as follows:
$$ 
\psi = \psi_{G_-} \oplus \bigoplus_{i=1}^m ( |\cdot|^{a_i} \psi_i \oplus  |\cdot|^{-a_i} \psi_i^\vee),\qquad 0<a_m<\cdots < a_1<1/2.
$$
where $\psi_{G_-}\in \tilde\Psi(\SO^{E/F}_{2n_-})$ and $\psi_i\in \Psi(\GL_{n_i})$ such that $n_-+\sum_{i=1}^m n_i = n$, and if we take 
\begin{equation}\label{eq:M-decomposition}
M=\SO^{E/F}_{2n_-}\times \prod_{i=1}^r \GL_{n_i},
\end{equation}
we have
$$
\psi_M := \psi_{G_-} \times \textstyle (\prod_{i=1}^r  \psi_i ) \in \tilde\Psi(M).
$$
Actually a weaker inequality $a_m\le \cdots\le a_1$ holds for the exponents in \cite{BinXuLPackets}. This is because he wants $\psi_i$ to be simple parameters. We only require $\psi_i$ to be bounded parameters, so the simple parameters with the same exponent will go into a single $\psi_i$ in our case.

Define a character $\chi:M(F)\ra \C^\times$ to be trivial on $\SO^{E/F}_{2n_-}(F)$ and $|\det|^{a_i}$ on $\GL_{n_i}(F)$. We can also view $\chi$ as a central $L$-morphism $\cL_F\ra {}^L M$. Then $\psi$ is the image of $\psi_M\otimes \chi$ under the natural map $\tilde\Psi^+(M)\ra \tilde\Psi^+(G)$. (The image of $\psi_M\otimes \chi$ is still $\psi$ if $|\det|^{a_i}$ is changed to $|\det|^{-a_i}$ for some $i$'s.) Write $\mathfrak a_M^*:=\Hom_F(M,\G_m)\otimes_{\Z} \R$. The character $\chi$ corresponds to $\nu\in \mathfrak a_M^*$, as in \cite[p.45]{ArthurBook}. We choose a parabolic subgroup $P$ with Levi factor $M$ such that $\nu$ is in the open chamber determined by $P$. Let $B$ be a Borel subgroup of $G$ contained in $P$, and denote by $T$ a maximal torus contained in $B$. The packet $\Pi_{\psi}$ consists of ($\Out_{2n}(G)$-orbits of) representations $\Ind^G_P(\sigma)$ for $\sigma$ in the $A$-packet $\Pi_{\psi_M}$ for $M$. Note that $\sigma$ is unitary since $\psi_M\in \tilde\Psi(M)$ (rather than $\psi_M\in \tilde\Psi^+(M)$).

Now suppose that $\mathbf 1$ is a subquotient of a member of $\Pi_{\psi}$, i.e., for some $\sigma\in \Pi_{\psi_M}$,
\begin{equation}\label{eq:1-in-Ind}
\mathbf 1\in \tu{JH}(\Ind^G_P(\sigma\otimes \chi)).
\end{equation}
To show that $\psi\in\tilde \Psi(G)$, let us verify the equivalent statement that $M=G$. 
As a representation of $M(F)$, we write $\sigma$ as the Langlands quotient $J^M_{P',\sigma',\chi'}$ in the convention of \cite[XI.2.9]{BorelWallach}, where $P'$ is a parabolic subgroup of $M$ containing $B\cap M$ with a Levi factor $M'$, $\sigma'$ is a tempered representation of $M'(F)$, and $\chi'$ lies in the open chamber of $a^*_{M'}$ given by $P'$. Then the Langlands quotient of $\Ind^G_P(\sigma\otimes \chi)$ is $J^G_{P'',\sigma',\chi\chi'}$, where $P''\supset B$ is a parabolic subgroup of $G$ with Levi factor $M'$. We observe that $\chi'$ is unitary on $Z(M')$ since $\sigma$ has unitary central character.

On the other hand, considering the supercuspidal support in \eqref{eq:1-in-Ind}, we see that every irreducible constituent of $\Ind^G_P(\sigma\otimes \chi)$ appears in $\Ind^G_B(\delta_{B}^{1/2})$. We know $\mathbf 1 = J^G_{B,\mathbf 1,\delta_B^{1/2}}$. On the other hand, the Langlands quotient in a parabolic induction is the unique extremal constituent in the sense of \cite[Lem.~2.13]{BorelWallach}. Applying this fact to $\Ind^G_P(\sigma\otimes \chi)$ and $\Ind^G_B(\delta_{B}^{1/2})$, we have 
$$
J_{B,\mathbf 1,\delta_B^{1/2}}=J_{P'',\sigma',\chi \chi'},
$$ 
which in turn implies that $B=P''$, $\mathbf 1=\sigma'$, and
$$
\delta_B^{1/2}=\chi\chi'
$$ 
by \emph{loc.~cit.} To prove $M=G$ by contradiction, suppose $M\neq G$. Then $r\ge 1$, and $M$ has $\GL_{n_1}$ as a direct factor. Consider elements of the form $m=(1,m_1,1,...,1)\in Z(M)$ according to \eqref{eq:M-decomposition}, where $m_1$ is a scalar matrix in $ \GL_{n_1}$. Then $|\chi(m)\chi'(m)|=|\chi(m)|=|m_1|^{a_1n_1}$, whereas there exists $d\in \Z$ (which depends on how $a_1$ compares to $a_2,...,a_r$ in terms of size) such that $\delta_B^{1/2}(m)=|m_1|^{dn_1/2}$. This is a contradiction since $0<a_1<1/2$.

If $\mathbf 1$ is a subquotient of a member of $\Pi_{\psi}$, we have shown that $M=G$, namely $\psi\in \tilde\Psi(G)$. Then every member of $\Pi_{\psi}$ is irreducible by Arthur's main local theorem \cite[Thm.~1.5.1]{ArthurBook}, so we conclude that $\psi=\psi_{\tu{triv}}$ by (2).

The remaining case is when $\St$ is a subquotient of a member of $\Pi_{\psi}$. Then
$$
\St\in \tu{JH}(\Ind^G_P(\sigma\otimes \chi)),\qquad  \sigma \in \Pi_{\psi_M}.
$$
Let us apply the Aubert involution, denoted by the hat symbol on both parameters and representations. From the compatibility of the involution with endoscopy and parabolic induction  by \cite[App.~A]{Xu-Apackets} and \cite[Thm.~1.7]{Aubert}, we deduce that 
$$
\mathbf 1\in \tu{JH}(\Ind^G_P(\hat\sigma\otimes \chi)),\qquad  \hat\sigma \in \Pi_{\hat \psi_M}.
$$
Since $\hat \psi_M\in \Psi(M)$, the argument in the preceding case carries over with $\hat \psi_M$ and $\hat \sigma$ in place of $\psi_M$ and $\sigma$. Thereby we conclude again that $M=G$. Then $\psi=\psi_{\tu{St}}$ by (2).
\end{proof}

\bibliographystyle{amsalpha}

\providecommand{\bysame}{\leavevmode\hbox to3em{\hrulefill}\thinspace}
\providecommand{\MR}{\relax\ifhmode\unskip\space\fi MR }
\providecommand{\MRhref}[2]{\href{http://www.ams.org/mathscinet-getitem?mr=#1}{#2} }
\providecommand{\href}[2]{#2}

\end{document}